\documentclass{amsart}
\usepackage{amsmath,amsthm}
\usepackage{marvosym}
\usepackage{graphicx}
\usepackage{color}
\usepackage{epsfig}
\usepackage{enumitem}

\newcommand{\appsection}[1]{\let\oldthesection\thesection
  \renewcommand{\thesection}{\oldthesection}
  \section{#1}\let\thesection\oldthesection}

\theoremstyle{plain}
\newtheorem{thm}{Theorem}[section]
\newtheorem{cor}[thm]{Corollary}
\newtheorem{conj}[thm]{Conjecture}
\newtheorem{prop}[thm]{Proposition}
\newtheorem{lemma}[thm]{Lemma}
\newtheorem{claim}[thm]{Claim}

\newtheorem*{CHANGING}{Theorem~\ref{thm:changingHS}}

\newtheorem*{caseA}{Case A}
\newtheorem*{caseB}{Case B}
\newtheorem*{caseC}{Case C}
\newtheorem*{caseD}{Case D}
\newtheorem*{question}{Question}
\newtheorem*{claim*}{Claim}
\newtheorem*{MAIN}{Theorem~\ref{thm:main}}
\newtheorem*{dyckscon}{Conjecture~\ref{con:dyckssurface}}
\newtheorem*{dycksthm}{Theorem~\ref{thm:3rp2s}}

\theoremstyle{definition}
\newtheorem{defn}[thm]{Definition}
\newtheorem{remark}[thm]{Remark}

\newcommand{\hatF}{\ensuremath{{\widehat{F}}}}
\newcommand{\hatT}{\ensuremath{{\widehat{T}}}}
\newcommand{\hatQ}{\ensuremath{{\widehat{Q}}}}

\newcommand{\calL}{\ensuremath{{\mathcal L}}}
\newcommand{\calT}{\ensuremath{{\mathcal T}}}
\newcommand{\calN}{\ensuremath{{\mathcal N}}}
\newcommand{\calO}{\ensuremath{{\mathcal O}}}
\newcommand{\calD}{\ensuremath{{\mathcal D}}}

\newcommand{\comment}[1]{}

\newcommand{\bdry}{\ensuremath{\partial}}
\DeclareMathOperator{\Int}{Int}

\DeclareMathOperator{\nbhd}{N}%{N\!d}%{N\!b\!h\!d}

\newcommand{\R}{\ensuremath{\mathbb{R}}}

\newcommand{\RP}{\ensuremath{\mathbb{RP}}}

\newcommand{\mobius}{M\"{o}bius }

\newcommand{\arc}[1]{\ensuremath{(#1)}}  %Would like to use a ``wideparaenthesis''
\newcommand{\edge}[1]{\ensuremath{\overline{#1}}}

\newcommand{\cut}{\ensuremath{\backslash}}

\newcommand{\conditionI}{{\sc Situation no scc}}
\newcommand{\conditionII}{{\sc Situation scc}}
\newcommand{\situationnscc}{{\sc Situation no scc}}
\newcommand{\situationscc}{{\sc Situation scc}}
\newcommand{\SC}{SC}
\newcommand{\SCs}{SCs}
\newcommand{\ESC}{ESC}
\newcommand{\ESCs}{ESCs}
\newcommand{\FESC}{FESC}
\newcommand{\FESCs}{FESCs}

\newcommand{\ttT}{{\tt T}}
\newcommand{\ttS}{{\tt S}}
\newcommand{\ttB}{{\tt B}}
\newcommand{\ttM}{{\tt M}}
\newcommand{\ttG}{{\tt G}} 
\newcommand{\ttg}{{\tt g}}

\definecolor{light-gray}{gray}{0.75}

\begin{document}
\title{Obtaining genus 2 Heegaard splittings from Dehn surgery}

\author{Kenneth L.\ Baker}
\address{Department of Mathematics \\ University of Miami \\ Coral Gables, FL 33146}
\email{kb@math.miami.edu}

\author{Cameron  Gordon}
\address{Department of Mathematics \\ The University of Texas at Austin \\ Austin, TX 78712}
\email{gordon@math.utexas.edu}

\author{John Luecke}
\address{Department of Mathematics \\ The University of Texas at Austin \\ Austin, TX 78712}
\email{luecke@math.utexas.edu}

\begin{abstract}

Let $K'$ be a hyperbolic knot in $S^3$ and suppose that some Dehn surgery on $K'$ with distance at least $3$ from the meridian yields a $3$-manifold $M$ of Heegaard genus $2$.  We show that if $M$ does not contain an embedded Dyck's surface (the closed non-orientable surface of Euler characteristic $-1$), then the knot dual to the surgery is either $0$-bridge or $1$-bridge with respect to a genus $2$ Heegaard splitting of $M$.  In the case that $M$ does contain an embedded Dyck's surface, we obtain similar results.   As a corollary, if $M$ does not contain an incompressible genus $2$ surface, then the tunnel number of $K'$ is at most $2$.
\end{abstract}

\maketitle

%\today

%\tableofcontents

\section{Introduction}

Let $M = K'(\gamma)$ be the manifold obtained by Dehn surgery on a knot $K'$ in 
$S^3$ along a slope $\gamma$. In $K'(\gamma)$, the core of the attached solid 
torus is a knot which we denote by $K$. It is natural to consider the 
properties of $K$ as a knot in $M$. In this paper we are interested in the 
relationship between $K$ and the Heegaard splittings of $M$; more specifically, 
if $\hatF$ is a Heegaard surface in $M$ of genus $g$, what can we say about the 
bridge number $br(K)$ of $K$ with respect to $\hatF$? Assume $K'$ is a 
hyperbolic knot, meaning that its complement $S^3 - K'$ admits a complete 
Riemannian metric of constant sectional curvature $-1$. It follows from 
\cite{rs1} (see also \cite{moriahrubenstein} and \cite{rieck}) that for all 
but finitely many slopes $\gamma$, $K$ can be isotoped to lie on $\hatF$, i.e.\ 
$br(K) = 0$. Let $\Delta = \Delta(\gamma,\mu)$ be the {\it distance\/} of the 
surgery, in other words the minimal geometric intersection number on $\partial 
N(K')$ of the slope $\gamma$ and the meridian $\mu$ of $K'$. 
Since the 
trivial Dehn surgery $K'(\mu) = S^3$ represents the maximal possible 
degeneration of Heegaard genus, one would expect the Heegaard splittings of 
$K'(\gamma)$ to reflect those of the exterior of $K'$ as $\Delta$ gets large. Indeed, it follows from 
\cite{rieck} that for any Heegaard surface $\hatF$ of $K'(\gamma)$ of genus $g$ if $\Delta\ge 18(g+1)$ then $br(K) = 0$, and so after at most one stabilization $\hatF$ is isotopic to a Heegaard surface for the exterior of $K'$.  
Also, in \cite{bgl:tnvhg} we show that if $\Delta \ge 2$, $\gamma$ is not a boundary slope for $K'$, and $M$ has a strongly irreducible Heegaard splitting of genus $g$, then 
the bridge number $br(K)$ of $K$ with respect to {\it some} genus $g$ splitting of $M$ is bounded above by a universal linear function of $g$.  
In contrast, this is not true for $\Delta=1$:  By Teragaito \cite{teragaito} there exists a family of knots  $K'_n$ and a $\gamma$ with $\Delta(\gamma,\mu)=1$ such that $K'_n(\gamma)$ is the same small Seifert fiber space $M$ for all $n$, and we show in \cite{bgl:teragaitobridge} that the set of bridge numbers of the corresponding cores $K_n$ with respect to any genus $2$ Heegaard splitting of $M$ is unbounded.

Turning to small values of $g$, note that the impossibility of getting $S^3$ by 
non-trivial Dehn surgery on a non-trivial knot \cite{gl:kadbtc} 
can be expressed as saying 
that if $g = 0$ and $\Delta > 0$ then $br(K) = 0$. When $g=1$, $K'(\gamma)$ 
is a 
lens space and here the Cyclic Surgery Theorem \cite{cgls:dsok} 
says that if $\Delta > 1$ 
then $K'$ is a torus knot, which is easily seen to imply $br(K) = 0$, while if 
$\Delta = 1$ and $K'$ is hyperbolic the Berge Conjecture \cite{berge} asserts that $br(K) = 
1$. In the present paper we consider the case $g = 2$ and show that if $\Delta 
> 2$ then, generically, $br(K)\le 1$ (with respect to 
some genus $2$ splitting). In fact we consider $1$-sided as well as 
$2$-sided genus $2$ Heegaard splittings of $M$; recall that such a splitting is 
defined by a closed (connected) non-orientable surface of Euler characteristic
$-1$ in $M$, the complement of an open regular neighborhood of which is a 
genus $2$ handlebody. Such a surface is a connected sum of three projective
planes and is also known as a {\it cross cap number $3$ surface} or as a
{\it Dyck's surface}; in this paper we shall adopt the latter terminology.

\begin{MAIN}
Let $K'$ be a hyperbolic knot in $S^3$ and assume $M=K'(\gamma)$ has 
a 1- or 2-sided Heegaard splitting of genus $2$. Assume that 
$\Delta(\gamma,\mu) \geq 3$ where $\mu$ is
the meridian of $K'$. 
Denote by $K$ the core of the attached solid torus 
in $M$. Then either

\begin{enumerate} 

\item 
$K$ is $0$-bridge or $1$-bridge 
with respect to a 1- or 2-sided, genus $2$ Heegaard splitting of $M$. 
In this case, the tunnel number of $K'$ is at most two.
\end{enumerate}
or
\begin{enumerate}[resume]
\item $M$ contains a Dyck's surface, $\widehat{S}$, such that the orientable
genus $2$ surface $\hatF$ that is the boundary of a regular
neighborhood of $\widehat{S}$ is incompressible in $M$.
Furthermore, $K$ can either be isotoped onto $\widehat{S}$ as an 
orientation-reversing curve or can be isotoped to intersect $\widehat{S}$
once. In the latter case, the intersection of $\widehat{F}$ with the exterior
of $K$ (which is also the exterior of $K'$) gives a twice-punctured, 
incompressible, genus $2$ surface in that
exterior. 

\end{enumerate}

\end{MAIN}

Conclusion (2) is an artifact of the proof and probably not necessary, but
allowing it simplifies an already lengthy argument. Similarly,
the assumption that $K'$ is hyperbolic simplifies the argument; we will consider
the case where $K'$ is a satellite knot elsewhere. %\cite{bgl:hksfds}.

As a warning, the Heegaard splitting of conclusion (1) may be different than
the one you started with -- for example, starting with a 2-sided genus
$2$ Heegaard splitting of $K'(\gamma)$, the proof of
Theorem~\ref{thm:main} may produce a 1-sided splitting with 
respect to which $K$ is 1-bridge.

Theorem~\ref{thm:main} fails dramatically when $\Delta=1$. For the Teragaito examples mentioned above \cite{teragaito},
Theorem~\ref{SFSDycks} of Appendix~\ref{sec:appendix} shows that the ambient Seifert fiber space, $M$, contains no Dyck's surface; thus conclusion $(2)$
of Theorem~\ref{thm:main} does not apply and every genus $2$ splitting of $M$ is 2-sided. On the other hand, \cite{bgl:teragaitobridge}
shows there are knots in the Teragaito family with arbitrarily large bridge number with respect to any
genus $2$ splitting of $M$.

Theorem~\ref{thm:main} says that there exists a Heegaard splitting of $M$ with 
respect to which $K$ is most $1$-bridge. If the bridge
number is more than one, the proof of Theorem~\ref{thm:main} constructs a 
new genus $2$ splitting
with respect to which the bridge number is smaller. By keeping a track of
when such a modification is necessary, we see that the proof typically shows 
that $K$ is at most $1$-bridge with respect to any genus $2$ splitting of $M$.
We make this precise in Theorem~\ref{thm:changingHS} below. For this we need
the following definitions.

\begin{defn}~\label{def:addremove}
Let $H_B \cup_{\hatF} H_W$ be a genus $2$ ($2$-sided) Heegaard splitting of 
$M$. Assume there is a \mobius band on one side of the Heegaard surface 
$\hatF$ whose boundary is a primitive curve on the other side of $\hatF$. 
A new Heegaard splitting of $M$, of the same genus, can be formed by 
removing a neighborhood
of the \mobius band from one side of $\hatF$ and adding it to the other side. 
We say that this new splitting
is obtained from the old by {\em adding/removing a \mobius band}.
\end{defn}

\begin{defn}
Let $M$ be a Seifert fiber space over the $2$-sphere with three exceptional fibers.
A {\em vertical} Heegaard splitting of $M$ is a genus $2$ splitting for 
which one of the Heegaard handlebodies
is gotten by tubing together the neighborhoods of two exceptional fibers, where
the tube connecting them is the neighborhood of a co-core arc of a vertical annulus
connecting the neighborhoods of these exceptional fibers.
\end{defn}

\begin{CHANGING}%\ref{thm:changingHS}
Let $K'$ be a hyperbolic knot in $S^3$.
Let $H_B \cup_{\hatF} H_W$ be a genus $2$ ($2$-sided) Heegaard splitting of 
$M=K'(\gamma)$. Assume that 
$\Delta(\gamma,\mu) \geq 3$ where $\mu$ is
the meridian of $K'$. Furthermore assume that $M$ does not contain a 
Dyck's surface.  
Denote by $K$ the core of the attached solid torus 
in $M$. Then either

\begin{enumerate} 

\item 
$K$ is
$0$-bridge or $1$-bridge with respect to a Heegaard splitting of $M$
obtained from $H_B \cup_{\hatF} H_W$ by a (possibly empty) sequence of 
adding/removing \mobius bands; or

\item $M$ is a Seifert fiber space over the disk with three exceptional fibers, one of which has order $2$ or $3$, and 
$K$ is $0$-bridge or $1$-bridge with respect to a Heegaard splitting gotten
from a vertical Heegaard splitting of the Seifert fiber space $M$ which has been changed
by a (possibly empty) sequence of adding/removing \mobius bands; or

\item $M$ is $n/2$-surgery on a trefoil knot, $n$ odd, and $K$ is $0$-bridge
or $1$-bridge with respect
to the Heegaard splitting on $M$ coming from the genus $2$ splitting of the
trefoil knot exterior. Note that in this case $M$ is a Seifert fiber
space over the $2$-sphere with three exceptional fibers, one of order $2$
and a second of order $3$.

\end{enumerate} 

In particular, if $M$ is not a Seifert fiber space over the $2$-sphere with
an exceptional fiber of order $2$ or $3$, and if the Heegaard surface
$\hatF$ has no \mobius band on one side whose boundary is a 
primitive curve on the other, then $K$ must be $0$-bridge or
$1$-bridge with respect to the given splitting $H_B \cup_{\hatF} H_W$.

\end{CHANGING}

\begin{remark}
The situations in which the Heegaard splitting $H_B \cup_{\hatF} H_W$ must
be altered in Theorem~\ref{thm:changingHS} are special. The situation
when $M$ contains a Dyck's surface is discussed in more detail below, for example in Theorem~\ref{thm:3rp2s} (see also Appendix~\ref{sec:appendix}). It is conjectured that 
the second and third conclusions of Theorem~\ref{thm:changingHS} never hold, 
that a Seifert fiber space never arises
by non-integral surgery on a hyperbolic knot. Finally, the existence of 
a \mobius band in one Heegaard handlebody of $H_B \cup_{\hatF} H_W$ whose
boundary is primitive on the other is a special 
case of this Heegaard splitting having Hempel distance $2$ \cite{hempel, thompson:tdcpag2m} 
--- which
also places restrictions on what $M$ can be. Presumably these exceptions 
are artifacts of the proof, and that in fact $K$ is at most $1$-bridge 
with respect to any genus $2$ splitting when $\Delta \geq 3$.
\end{remark}

Our results give information on the relationship between the Heegaard genus of 
$M$ and that of $X = S^3 - N(K')$, the exterior of $K'$. Recall that a Heegaard 
splitting of $X$ is a decomposition $X = V \cup_{S} W$, 
where $V$ is a handlebody 
with $\partial V = S$ and $W$ is a compression body with $\partial W = 
S\sqcup\partial X$. The {\it Heegaard genus\/} $g(X)$ of $X$ is the minimal 
genus of $S$ over all such decompositions.

In this context one often talks 
about the {\it tunnel number\/} $t(K')$ of $K'$, the minimum number of arcs 
(``tunnels'') that need to be attached to $K'$ so that the complement of an open 
regular neighborhood of the resulting 1-complex is a handlebody. It is easy to 
see that $g(X) = t(K') + 1$. For any slope $\gamma$, $V\cup_{S}W(\gamma)$ is a 
Heegaard splitting of $M = K'(\gamma)$; in particular $g(M)\le g(X)$. In fact, 
by \cite{rs2}, generically we have $g(M) = g(X)$. More precisely, recall 
that for all but finitely many slopes $\gamma$, $br(K) = 0$ with respect to 
any Heegaard surface $\hatF$ of $M$. 
Taking $\hatF$ to have minimal genus, it is then easy to see that when $br(K)=0$ either 
$g(M) = g(X) = t(K') + 1$ or 
$g(M) = g(X)- 1 = t(K')$. See \cite{rieck} for details. By \cite{rs2}, the 
second possibility can happen for only a finite number of lines of slopes 
(where a {\it line} of slopes is a set of slopes $\gamma$ such that 
$\Delta(\gamma,\gamma_0) 
= 1$ for some fixed slope $\gamma_0$). We know no examples where the 
Heegaard genus 
of $K'(\gamma)$ ($\gamma\ne \mu$) is less than $t(K')$.

\begin{question} Is $t(K')\le g(K'(\gamma))$ for all $\gamma\ne \mu$?
\end{question}
Now it is easy to see that an upper bound on $br(K)$ in $M$, with respect to a 
1- or 2-sided Heegaard surface, gives an upper bound on $t(K')$. In particular 
part (1) of Theorem~\ref{thm:main} gives

\begin{cor}\label{cor:tunnel}
Let $K'$ be a hyperbolic knot in $S^3$ and suppose $K'(\gamma)$ has Heegaard 
genus 2 and does not contain an incompressible genus 2 surface, where 
$\Delta(\gamma,\mu)\ge 3$. Then the tunnel number of $K'$ is at most 2.
\end{cor}

Corollary~\ref{cor:tunnel} is sharp: 
there exist hyperbolic tunnel number 2 knots 
$K'$ having non-Haken Dehn surgeries $K'(\gamma)$ of Heegaard genus 2 with 
$\Delta(\gamma,\mu)$ arbitrarily large. To see this, let $K'$ be a knot that 
lies on a standard genus 2 Heegaard surface in $S^3$, and let $\lambda$ be the 
(integral) slope on $\partial N(K')$ induced by the surface. Then for any 
$\gamma$ such that $\Delta(\gamma,\lambda) = 1$, $K'(\gamma)$ has a (2-sided) 
Heegaard splitting of genus 2. Note that the tunnel number of $K'$ is at most 
2; on the other hand one can arrange that it is 2, and that $K'(\gamma)$ is 
non-Haken. Explicit examples are provided by the pretzel knots $K' = P(p,q,r)$ 
where $|p|,|q|,|r|$ are distinct odd integers greater than 1. Such a knot $K'$ 
lies on the standard genus 2 surface in $S^3$, with $\lambda$ the canonical 
longitude (slope 0). Hence $K'(\gamma)$ has a genus 2 Heegaard splitting 
for all 
$\gamma$ of the form $1/n$ (with the usual parametrization of slopes for 
knots). Note that $\Delta(\gamma,\mu) = |n|$ can be arbitrarily large. By 
\cite{trotter}, $K'$ is non-invertible, 
and therefore does not have tunnel number 1. 
The double branched cover of $K'$ is a Seifert fiber space over $S^2$ 
with three 
exceptional fibers, which does not contain an incompressible surface, and hence 
by \cite{GLi2} $S^3 - K'$ contains no closed essential surface. It 
follows that $K'$ is hyperbolic. It also follows that if $K'(\gamma)$ is Haken 
then $\gamma$ is a boundary slope. Since any knot has only finitely many 
boundary slopes \cite{hatcher}, $K'(1/n)$ will be non-Haken for all 
but finitely many 
values of $n$. (Other pretzel knots provide similar examples, using 
\cite{msy} to ensure that they have tunnel number 2.)

One reason we are interested in the genus 2 case is that this includes the 
situation where $M$ is a Seifert fiber space over $S^2$ with three exceptional 
fibers. Here it is expected that (when $K'$ is hyperbolic) $\Delta = 1$, 
although to date the best known upper bound is 8 \cite{lackenbymeyerhoff}.  The techniques of this article ought to enable further restrictions on non-integral, Seifert fibered surgeries on hyperbolic knots in $S^3$.  We will explore this elsewhere.

% In a sequel \cite{bgl:hksfds} to the present paper we will show that, for $M$ a Seifert fiber space, $\Delta\le 4$ when $t(K')>1$.

We derived the bound on the tunnel number $t(K')$ from the bound on the 
bridge number $br(K)$ in $K'(\gamma)$ given in Theorem 2.4. 
%As the examples in \cite{bgl:teragaitobridge}, the latter bound is stronger, at least when $\Delta =1$.
We point out that the latter bound is %at least {\it a priori}
 stronger: for example 
for any $t \ge 1$ there are knots in $S^3$ with tunnel number $t$ whose 
bridge number with respect to the genus $t$ splitting of $S^3$ is 
arbitrarily high \cite{mms}.  
Also, although Teragaito's family of knots \cite{teragaito} mentioned above have tunnel number $2$, we show that the set of their bridge numbers with respect to any genus $2$ Heegaard splitting of the small Seifert fiber space is unbounded \cite{bgl:teragaitobridge}.  
At any rate, the bound on bridge number in Theorem 
2.4 allows us to use a result of Tomova \cite{tomova} to get a statement about 
the distance of splittings of exteriors of knots with genus 2 Dehn 
surgeries. If $S$ is a Heegaard surface for some 3-manifold, we denote by 
$d(S)$ the {\it (Hempel) distance} of the corresponding splitting; see \cite{hempel}.

\begin{cor}\label{cor:distance} 
Let $K'$ be a hyperbolic knot in $S^3$ whose exterior has a 
Heegaard splitting $S$ with $d(S) > 6$. Let $\gamma$ be a slope with 
$\Delta(\gamma,\mu)\ge 3$, where $\mu$ is a meridian of $K'$,
and suppose the manifold $K'(\gamma)$ does not 
contain a Dyck's surface and has Heegaard genus 2. Then $S$ has 
genus 2.
\end{cor}

Thus the distance of a splitting of a knot exterior is putting a limit on 
the degeneration of Heegaard genus under Dehn filling. For instance, this 
applies to the examples of \cite{mms}. First note that the condition that 
$K'(\gamma)$ not contain a Dyck's surface (or indeed any closed 
non-orientable surface) can be easily ensured by taking $\gamma = p/q$ 
with $p$ odd. Now by \cite{mms}, for any $g \ge 3$, there are knots $K'$ in 
$S^3$ whose exteriors have genus $g$ Heegaard splittings $S$ with $d(S) > 
6$, in fact with $d(S)$ arbitrarily large (such knots are necessarily 
hyperbolic). Corollary 1.2 says that for such a knot $K'$, if $q \ge 3$ and 
$p$ is odd, $K'(p/q)$ does not have Heegaard genus 2.

\begin{proof} {\em (Corollary~\ref{cor:distance})}
Let $K', {\gamma}, S$ be as in the hypothesis. By 
Theorem 2.4, the bridge number of $K$ with respect to some genus 2 
Heegaard surface $\hatF$ of $K'(\gamma)$ is at most 1. Thus $K$ can be put 
in bridge position with respect to $\hatF$ so that $2 - {\chi}(\hatF - K) 
= 2 - (-2 - 2) = 6$. Since $d(S) > 6$ by assumption, the main result of 
\cite{tomova} implies that, in $K'(\gamma)$, $\hatF$ is isotopic to a 
stabilization of $S$. Hence $S$ has genus 2 (and $\hatF$ is isotopic to 
$S$ in $K'(\gamma)$).                                                  
\end{proof}

In the course of proving Theorem~\ref{thm:main}, 
we consider Dehn surgeries that produce Dyck's surfaces, leading to conclusion
(2) of that theorem. 
If a knot $K'$ in $S^3$ has a maximal Euler characteristic spanning surface $S$ with $\chi(S)=-1$ (so that $K'$ has genus $1$ or cross cap number $2$) then surgery on $K'$ along a slope $\gamma$ of distance $2$ from $\bdry S$ produces a manifold with Dyck's surface embedded in it.  There is a \mobius band embedded in the surgery solid torus whose boundary coincides with $\bdry S$ so that together they form an embedded Dyck's surface $\tilde{S}$.  The core of the surgery solid torus is the core of the \mobius band, and hence the surgered knot lies as a simple closed curve on $\tilde{S}$.  Furthermore, such a surgery slope $\gamma$ may be chosen so that it has any desired odd distance $\Delta=\Delta(\gamma,\mu)$ from the meridian $\mu$ of  $K'$. Any knot with (Seifert) genus more than $1$
and crosscap number $3$ has an integral surgery containing a Dyck's surface
that does not come from this construction, and there are many such hyperbolic
knots, the smallest being $6_3$ (see e.g.\ the tables \cite{knotinfo}). However,
we conjecture that this is the only way a
Dyck's surface arises from a non-integral (i.e.\ $\Delta > 1$)
Dehn surgery on a hyperbolic knot:

\begin{dyckscon}
Let $K'$ be a hyperbolic knot in $S^3$ and assume that $K'(\gamma)$ contains an
embedded Dyck's surface. If $\Delta(\gamma,\mu)>1$, where $\mu$ is 
a meridian of $K'$, then there is an embedded Dyck's surface, $\widehat{S}
\subset K'(\gamma)$, such that the core of the attached solid torus in 
$K'(\gamma)$ can be isotoped to an orientation-reversing curve in $\widehat{S}$.
In particular, $K'$ has a spanning surface with Euler 
characteristic $-1$.
\end{dyckscon}

In section~\ref{sec:dycks}, we prove the following, which goes a long way 
towards verifying this conjecture.

\begin{dycksthm}
Let $K'$ be a hyperbolic knot in $S^3$ and assume that $M=K'(\gamma)$ 
contains an 
embedded Dyck's surface. 
If $\Delta(\gamma,\mu) > 1$, where $\mu$ is a meridian of $K'$, 
then there is an embedded Dyck's surface in $M$ that 
intersects the core of the attached solid torus in $M$ transversely once. 
\end{dycksthm}

Conjecture~\ref{con:dyckssurface} fits in well with earlier 
results on small surfaces
in Dehn surgery on a knot in the 3-sphere. When $\Delta \geq 2$, $M$ cannot 
contain an essential sphere (\cite{gl:oidscyrm}), an embedded projective plane (\cite{gl:oidscyrm}, \cite{cgls:dsok}), or an embedded Klein bottle (\cite{gl:dsokcetII}). When $\Delta \geq 3$
(as in fact must be the case when $M$ contains an embedded, closed, 
non-orientable surface and $\Delta > 1$),  
$M$ cannot contain an essential torus
(\cite{gl:dsokcetI}).

\subsection{Sketch of the argument for Theorem~\ref{thm:main}}

The idea of the proof of Theorem~\ref{thm:main} is as follows. Assume 
$M=K'(\gamma)$
has a 2-sided, genus $2$ Heegaard splitting. Assume $K$ has the smallest 
bridge number with respect to this splitting, among all $2$-sided, genus $2$ 
splittings of $M$. The typical situation is when this bridge position of 
$K$ is also
a thin position of $K$ with respect to this splitting (see 
section~\ref{sec:basics}). This thin presentation of $K$ in $M$ and
one of $K'$ in $S^3$ allow us to find a genus $2$ 
Heegaard surface $\hatF$ of the splitting of $M$ and a genus $0$ Heegaard 
surface $\hatQ$ of 
$S^3$ such that $F=\hatF - \nbhd(K)$ and $Q=\hatQ - \nbhd(K')$ intersect
essentially. The arcs of $F \cap Q$ form graphs $G_F,G_Q$ on $\hatF,\hatQ$.
Then $t=|K \cap \hatF|$ is twice the bridge number of $K$ in $M$.  
We show that $t \leq 2$,
thereby implying that $K$ is $0$-bridge or $1$-bridge with respect to this
splitting. We do this typically by showing that if $t >2$ then we
can thin the presentation (i.e.\ find one with smaller bridge number)
with respect to some genus $2$ Heegaard splitting in $M$.
To find such ``thinnings'' of $K$, we show that $G_Q$ has
a special subgraph, $\Lambda$, called a 
great 2-web (section~\ref{sec:greatwebs}). Disk faces of $\Lambda$ are
thought of as disks properly embedded $M - \nbhd(K \cup \hatF)$ (at least
when there are no simple closed curves of $F \cap Q$). 
Within $\Lambda$ we look for configurations of small faces that can be
used to locate $K$ in its bridge presentation with respect to $\hatF$.
For example, a configuration called an ``extended Scharlemann cycle'' 
(an \ESC, see Figure~\ref{fig:basicscharlemanncycles}) leads to a ``long \mobius
band'' (Figure~\ref{fig:basicextendedmobius}), which, when long enough,
leads to an essential torus in $M$ (which does not happen since $\Delta \geq
3$) or to a thinning of $K$ (e.g.\ Lemmas~\ref{lem:LMB} and ~\ref{lem:3PCC}).
For $t \leq 6$, configurations of bigons and trigons at ``special 
vertices'' of $\Lambda$ (section~\ref{sec:specialvertices}) are often used 
to construct a new Heegaard splitting of $M$ with respect to which 
$K$ has smaller bridge number.

As a note to the reader, the generic argument (showing that $K$ is at most
4-bridge with respect to some genus $2$ splitting of $M$) is given in 
sections~\ref{sec:setup}, ~\ref{sec:escandlongmb}, ~\ref{sec:greatwebs}, 
~\ref{sec:abundancebigons}, 
~\ref{sec:escbounds}, and ~\ref{sec:t<10}. The arguments get more complicated
as the supposed bridge number of $K$ in $M$ gets smaller. In particular,
almost half of the current paper is from section~\ref{sec:t4noscc} on,
showing that the minimal bridge number of $K$ is not $2$ (i.e.\ $t \neq 4$). 

\subsection{Notation}

By $\nbhd(\cdot)$ we denote a regular open neighborhood or its subsequent closure as the situation dictates.

Let $Y$ be a subset of the manifold $X$, typically a properly embedded submanifold (such as an arc or loop in a surface or a surface or handlebody in a $3$-manifold).  By $X \cut Y$ we denote $X$ {\em chopped} or {\em cut} along $Y$.  That is, $X\cut Y$ may be viewed as either $X - \Int \nbhd(Y)$ or the closure of $X-Y$ in the path metric.  

For $Y$ a connected codimension 1 properly embedded submanifold of $X$, any newly created maximal connected submanifold of the boundary of $X \cut Y$ is an {\em impression} of $Y$.  In other words, an impression of $Y$ is a component of the closure of $\bdry (X\cut Y) - \bdry X$.  Note that the impressions of $Y$ form a double cover of $Y$.  Suitably identifying $\bdry (X \cut Y)$ along them will reconstitute $X$ with $Y$ inside.  Alternatively $X$ with $Y$ may be reconstituted by suitably attaching $\nbhd(Y)$ to $X \cut Y$.

\subsection{Acknowledgements}
In the course of this work KB was partially supported by NSF Grant DMS-0239600, by the University of Miami 2011 Provost Research Award, and by a grant from the Simons Foundation (\#209184 to Kenneth Baker).  KB would also like to thank the Department of Mathematics at the University of Texas at Austin for its hospitality during his visits.  These visits were supported in part by NSF RTG Grant DMS-0636643.

%%\newpage

\section{Thin-bridge position, $G_Q, G_F$, and the proof of Theorem~\ref{thm:main}} \label{sec:setup}

\subsection{Heegaard splittings, thin position, and bridge position}\label{sec:basics}
Given a (2-sided) Heegaard surface $\Sigma$ of a closed $3$-manifold $Y$ there is a product $\Sigma \times \R \subset Y$ so that $\Sigma = \Sigma \times \{0\}$ and the complement of the product is the union of spines for each of the 
two handlebodies.  This defines a height function on the complement of spines for each of the handlebodies.  Consider all the circles $C$ embedded in the product that are Morse with respect to the height function and represent the knot type of $J$.  The following terms are all understood to be taken with respect to the Heegaard splitting.

Following \cite{gabai:fatto3mIII} (see also \cite{thompson:tpabnfkit3s}), the {\em width} of an embedded circle $C$ is the sum of the number of intersections $|C \cap \Sigma \times \{y_i\}|$ where one regular value $y_i$ is chosen between each pair of consecutive critical values.  The {\em width} of a knot $J$ is the minimum width of all such embeddings.   However, if $J$ can be isotoped to a curve embedded in a level surface $\Sigma \times \{y\}$, we define such an embedding as having width $0$.  An embedding realizing the width of $J$ is a {\em thin position} of $J$, and $J$ is said to be {\em thin}.  If the critical point immediately below $y_i$ is a minimum and the critical point immediately above $y_i$ is a maximum, then the level $\Sigma \times \{y_i\}$ is a {\em thick level}.

The minimal number of maxima among Morse embeddings of $C$ is the 
{\em bridge number} of $J$, and denoted $br(J)$.  
An embedding realizing the bridge number of $J$ 
may be ambient isotoped so that all maxima lie above all 
minima, without introducing any more extrema.  
The resulting embedding is a {\em bridge position} of $J$, and 
$J$ is said to be {\em bridge}.
If $J$ can be isotoped into a  level surface $\Sigma \times \{y\}$,
	    we define such an embedding as having bridge number $0$.

With $J$ in bridge position, the arcs of $J$ intersecting a Heegaard 
handlebody are collectively $\bdry$-parallel.  There is an embedded collection of disks in the handlebody such that the boundary of each is formed of one arc on $\Sigma$ and one arc on $J$.  A single such disk is called a {\em bridge disk} for that arc of $J$, and the arc is said to be {\em bridge}.

A thin position for a knot may have smaller width than that of its 
bridge position,
with respect to the same Heegaard splitting. That is, thin position may not
be bridge position.  
However, this only happens when the meridian of the knot in the ambient 
manifold is a boundary slope of the knot
exterior.

\begin{defn}\label{def:2bdryslope}
Let $E$ be an orientable 3-manifold with a single torus boundary. Let 
$\gamma$ be the isotopy class of a non-trivial curve on $\partial E$. 
Then $\gamma$ is said to be a {\em boundary slope} for $E$ if there
is an incompressible, $\partial$-incompressible, orientable surface, $P$, 
properly embedded in $E$ with non-empty boundary, 
such that each component of $\partial P$
is in isotopy class $\gamma$. $\gamma$ is said to be a {\em $g$-boundary slope}
if there is such a surface $P$ with genus at most $g$.
\end{defn}

\begin{lemma}~\label{lem:thin=bridge}
Assume $J$ is a knot in a 3-manifold $M$. If $J$ has a thin position which 
is not a bridge position with
respect to a genus $g$ Heegaard splitting of $M$, then the meridian of $J$ 
is a $g$-boundary
slope for the exterior of $J$.
\end{lemma}

\begin{proof}
This is proved in \cite{thompson:tpabnfkit3s} when $g=0$. 
The same proof works here.
We sketch it for the convenience of the reader. 

Let $\Sigma$ be the Heegaard surface of a genus $g$ splitting of $M$ with
respect to which $J$ is in thin position but not bridge position.
Then there must be a {\em thin level} --- a level surface $\Sigma \times \{y\}$ 
at a regular value of the height function
such that the first critical level below the surface is a maximum and the
first critical level above the surface is a minimum. There can be no bridge 
disks for $J$ to the thin level surface, else such a disk 
would give rise to a thinner
presentation of $J$. Maximally compress $(\Sigma \times \{y\}) - \nbhd(J)$
in the exterior of $J$. Either some component of the result is an 
incompressible, $\partial$-incompressible surface of genus at most $g$ whose
boundary components are meridians of $J$, or the result is a non-empty 
collection of
boundary parallel annuli along with some closed surfaces. But each boundary
parallel annulus gives rise to a bridge disk of $J$ onto $\Sigma \times \{y\}$,
which is not possible. Thus the meridian is a $g$-boundary slope for the 
exterior of $J$.
\end{proof}

We tend to consider the situation where thin position is not bridge 
position as non-generic. For example we have the following useful
result.

\begin{lemma}\label{lem:mnosfs}
Let $K'$ be a hyperbolic knot in $S^3$ with meridian $\mu$. 
Assume there is a Heegaard
splitting of $M=K'(\gamma)$ with respect to which the core of the attached
solid torus, $K$, has a thin position
which is not a bridge position. If $\Delta(\gamma,\mu) \geq 2$ then 
$M$ is not Seifert fibered.
\end{lemma}

\begin{proof}
Assume $M$ is Seifert fibered. By Corollary~1.7 of \cite{boyerzhang} or 
Theorem~1.1 of \cite{gl:nitds},
$M$ is non-Haken.  
Considering $K$ in $M$, Lemma~\ref{lem:thin=bridge} says that
$\gamma$ is a boundary slope for the exterior of $K$. But this contradicts
Theorem~2.0.3 of \cite{cgls:dsok} ($M$ is irreducible and $K(\mu)$ is
non-Haken).
\end{proof}

We now give the proof of the main theorem, which defines the graphs
$G_Q,G_F$ studied throughout the rest of the paper.  

\begin{thm}\label{thm:main}
Let $K'$ be a hyperbolic knot in $S^3$ and assume $M=K'(\gamma)$ has 
a 1- or 2-sided Heegaard splitting of genus $2$. Assume that 
$\Delta(\gamma,\mu) \geq 3$ where $\mu$ is
the meridian of $K'$. Denote by $K$ the core of the attached solid torus 
in $M$. Then either

\begin{enumerate} 

\item 
$K$ is $0$-bridge or $1$-bridge 
with respect to a $1$- or $2$-sided, genus $2$ Heegaard splitting of $M$. 
In this case, the tunnel number of $K'$ is at most two.
\end{enumerate}
or
\begin{enumerate}[resume]
\item $M$ contains a Dyck's surface, $\widehat{S}$, such that the orientable
genus $2$ surface $\hatF$ that is the boundary of a regular
neighborhood of $\widehat{S}$ is incompressible in $M$.
Furthermore, $K$ can either be isotoped onto $\widehat{S}$ as an 
orientation-reversing curve or can be isotoped to intersect $\widehat{S}$
once. In the latter case, the intersection of $\widehat{F}$ with the exterior
of $K$ gives a twice-punctured, incompressible, genus $2$ surface in that
exterior. 

\end{enumerate}

\end{thm}

\begin{remark}
In this proof and throughout the article, since $K'$ is hyperbolic and $\Delta \geq 3$, 
 $M$ cannot 
contain an essential sphere (\cite{gl:oidscyrm}), an embedded projective plane (\cite{gl:oidscyrm}, \cite{cgls:dsok}),  an embedded Klein bottle (\cite{gl:dsokcetII}), or an essential torus
(\cite{gl:dsokcetI}).
\end{remark}

\begin{proof}
Let $K'$ be a hyperbolic knot in $S^3$, and let $M=K'(\gamma)$.
Assume $\Delta = \Delta(\gamma,\mu) \geq 3$. Let $K$ be
the core of the attached solid torus in $M=K'(\gamma)$.

If $M$ contains an embedded Dyck's surface, then the theorem follows from 
Corollary~\ref{cor:dyckstunnel2}. This includes the case where $M$ has
a $1$-sided genus $2$ Heegaard splitting. We assume hereafter that
$M$ contains no embedded Dyck's surface.

Thus $M$ has $2$-sided, genus $2$ Heegaard splitting. Note that any such
splitting is irreducible, since $M$ is neither a lens space nor a connected
sum (\cite{cgls:dsok}, \cite{gl:oidscyrm}). Consequently, such a splitting
is also strongly irreducible (the disjoint disks can
be taken to be separating, hence to have isotopic boundaries).

Assume we have a genus $2$ Heegaard splitting of $M$ for which $K$ does not have bridge number $0$.
Take $K$ to be in bridge position. 
By Theorem~\ref{thm:thinsurface}, we may assume that $K$ is also in 
thin position with respect to this Heegaard splitting of
$M$. 
In $S^3$, put $K'$ into thin position with respect to the genus $0$ 
Heegaard splitting.  
By Theorem 6.2 of \cite{rieck} (by assumption $K,K'$ cannot be isotoped onto 
their Heegaard surfaces), there exist thick level surfaces,
$\hatF$ of $M$ and $\widehat{Q}$ of $S^3$ such that
\begin{itemize}
\item[{\bf (*)}] {\bf each arc of $F \cap Q$ is essential in each of 
$F = \hatF - \nbhd(K)$ and $Q = \widehat{Q}-\nbhd(K')$.}  
\end{itemize}
As the exterior of $K'$ is irreducible, after an isotopy we may assume:
\begin{itemize}
\item[{\bf (**)}] {\bf there are no simple closed curves of $F \cap Q$ 
trivial in
both $F$ and $Q$.}
\end{itemize}

On $\widehat{Q}$ and $\hatF$ form the {\em fat vertexed graphs of intersection} $G_Q$ and $G_F$, respectively, consisting of the {\em fat vertices} that are the disks $\overline{\nbhd(K')} \cap \widehat{Q}$ and $\overline{\nbhd(K)} \cap \hatF$ and {\em edges} that are the arcs of $F \cap Q$.

Choosing an orientation on $K \subset M$, we may number the intersections of $K$ with $\hatF$, and hence the vertices of $G_F$, from $1$ to $t=|K \cap \hatF|$ in order around $K$.  Similarly, if $|K' \cap \widehat{Q}| = u$, by choosing an orientation on $K' \subset S^3$ we may number the intersections of $K'$ with $\hatQ$ and hence the vertices of $G_Q$ from $1$ to $u$ in order around $K'$.

Each component of $\bdry F$ intersects each component of $\bdry Q$ a total of $\Delta$ times.  Thus a vertex of $G_Q$ has valence $\Delta t$ and a vertex of $G_F$ has valence $\Delta u$.  Since each component of $\bdry F \cap \bdry Q$ is an endpoint of an arc of $F \cap Q$, each endpoint of an edge in $G_Q$ may be labeled with the vertex of $G_F$ whose boundary contains the endpoint.  Thus around the boundary of each vertex of $G_Q$ the labels $\{1, \dots, t\}$ appear in order $\Delta$ times.  Similarly around the boundary of each vertex of $G_F$ the labels $\{1, \dots, u\}$ appear in order $\Delta$ times. 

Now $t/2$ is the bridge number of $K$ with respect to the Heegaard surface
$\hatF$. We show that $t \leq 2$,
thereby implying that $K$ is $0$-bridge or $1$-bridge with respect to this
genus $2$ splitting.

The arguments typically divide into the two cases:
\begin{itemize}
\item {\bf \situationnscc}: There are no closed curves of $Q \cap F$ 
in the interior of disk faces of $G_Q$.  
\item
{\bf \situationscc}: There are closed curves of $Q \cap F$ in the interior of 
disk faces of $G_Q$.  The strong irreducibility of the Heegaard splitting allows
us then to assume (section~\ref{sec:scc}) that 
any such closed curve must be non-trivial on $\hatF$ 
and bound a disk on one side of $\hatF$. 
\end{itemize}

In \situationscc\ there is then a meridian disk on one side of the genus 2
splitting that is disjoint from $K$ and $Q$. This imposes strong restrictions
on the graph $G_F$. Typically then, the arguments are simpler (though different)
than those for \situationnscc.

Now assume that $\hatF$ is a Heegaard surface for $M$ for which $K$ has the smallest bridge number among genus $2$ splittings of $M$.  The paper is divided into sections ruling out various values of $t$, which
are necessarily even as $\hatF$ is separating.
Theorems~\ref{thm:tleq8}, ~\ref{thm:tleq6}, ~\ref{thm:tnot6} show in sequence
$t<10, t<8, t<6$  in both \situationnscc\ and 
\situationscc. Theorem~\ref{prop:tnot4} then implies that $t \leq 2$ in 
\situationnscc, and Theorem~\ref{thm:scctnot4} that $t \leq 2$ in \situationscc.
That is, $K$ is at most $1$-bridge with respect to the genus $2$ splitting
$\hatF$. 

To see that $K$ (and hence $K'$) has tunnel number at most $2$, write $K$ in $M$
as the union of an arc in $\hatF$ and a trivial arc in a handlebody $H$
on one side of $\hatF$ (this can be done if $K$ is $0$-bridge as well).
Attaching two tunnels to $K$ to form core curves of $H$ thickens to a genus
$3$ handlebody whose complement is a handlebody in $M$. Thus the tunnel number
of $K$ is at most two.
\end{proof}

%Begin replacement of rmk2.6
Keeping track of when and how we are forced to modify the Heegaard 
splitting in the proof of Theorem~\ref{thm:main} gives the following:

\begin{thm}\label{thm:changingHS}
Let $K'$ be a hyperbolic knot in $S^3$.
Let $H_B \cup_{\hatF} H_W$ be a genus $2$ ($2$-sided) Heegaard splitting of 
$M=K'(\gamma)$. Assume that 
$\Delta(\gamma,\mu) \geq 3$ where $\mu$ is
the meridian of $K'$. Furthermore assume that $M$ does not contain a 
Dyck's surface.  
Denote by $K$ the core of the attached solid torus 
in $M$. Then either

\begin{enumerate} 

\item 
$K$ is
$0$-bridge or $1$-bridge with respect to a Heegaard splitting of $M$
obtained from $H_B \cup_{\hatF} H_W$ by a (possibly empty) sequence of 
adding/removing \mobius bands (Definition~\ref{def:addremove}); or

\item $M$ is a Seifert fiber space over the disk with three exceptional fibers, one of which has order $2$ or $3$, and 
$K$ is $0$-bridge or $1$-bridge with respect to a Heegaard splitting gotten
from a vertical Heegaard splitting of the Seifert fiber space $M$ which has been changed
by a (possibly empty) sequence of adding/removing \mobius bands; or

\item $M$ is $n/2$-surgery on a trefoil knot, $n$ odd, and $K$ is $0$-bridge
or $1$-bridge with respect
to the Heegaard splitting on $M$ coming from the genus $2$ splitting of the
trefoil knot exterior. Note that in this case $M$ is a Seifert fiber
space over the $2$-sphere with three exceptional fibers, one of order $2$
and a second of order $3$.

\end{enumerate} 

In particular, if $M$ is not a Seifert fiber space over the $2$-sphere with
an exceptional fiber of order $2$ or $3$, and if the Heegaard surface $\hatF$ has no \mobius band on one side whose boundary is a 
primitive curve on the other, then $K$ must be $0$-bridge or
$1$-bridge with respect to the given splitting $H_B \cup_{\hatF} H_W$.

\end{thm}

\begin{proof} Let $H_B \cup_{\hatF} H_W$ be a genus $2$ 
Heegaard splitting of $M$ 
for which $K$ is not $0$-bridge or $1$-bridge. As in the proof
of Theorem~\ref{thm:main}, the arguments of 
sections \ref{sec:escbounds}--\ref{sec:tis4scc} show,
or can be adapted to show, that either
\begin{itemize}
\item $M$ contains a Dyck's surface; or
\item $H_B \cup_{\hatF} H_W$ can be altered by
adding/removing a \mobius band so that we get a new 
genus $2$ splitting for which $K$ has smaller bridge number; or
\item $M$ is a Seifert fiber space over the $2$-sphere with an exceptional fiber
of order $2$ or $3$ and we can find a vertical splitting of this Seifert fiber
space for which $K$ has smaller bridge number. 
\item $M$ is an $n/2$-surgery on the trefoil knot, $n$ odd, and $K$ is shown to
be at most $1$-bridge with respect to a genus $2$ splitting of $M$ coming from the
Heegaard splitting of the trefoil exterior (i.e.\ 
remove a neighborhood of the unknotting tunnel from the exterior of the trefoil 
for one handlebody
of the splitting of $M$, then the filling solid torus in union with a neighborhood of 
the unknotting tunnel is the other). This conclusion only occurs at the very end of 
section~\ref{sec:tis4scc}.
\end{itemize}
In sections \ref{sec:escbounds}--\ref{sec:tis4scc}, there are a few places where the argument given
needs to be altered slightly to see that in fact one of the items above occurs. 
We have included remarks to that end when necessary.
Repeated applications of the above alternatives leads to a genus
$2$ splitting of $M$ with respect to which $K$ is $0$-bridge or $1$-bridge as claimed by Theorem~\ref{thm:changingHS}. Note that
the statement there when $M$ is a Seifert fiber space with an exceptional fiber
of order $2$ or $3$ follows by starting with a vertical splitting of $M$.
\end{proof}
%End of replacement to rmk 2.6

We finish this section with the proof of Theorem~\ref{thm:main} in
the special case that thin position is not bridge position. Here the arguments 
of the preceding proof are applied to thin level surfaces
rather than thick. 

\begin{thm}\label{thm:thinsurface}
Let $K'$ be a hyperbolic knot in $S^3$. Assume there is a genus two Heegaard
splitting of $M=K'(\gamma)$ with respect to which $K$, the core of the attached
solid torus, has a thin position
which is not a bridge position. 
If $\Delta(\gamma,\mu) \geq 3$ then 
$M$ contains an embedded Dyck's surface.
\end{thm}

\begin{proof}
Let $K',K, M$ be as given. 
Assume $M$ has a genus $2$ Heegaard splitting
with respect to which $K$ (in $M$) has a thin presentation which is not a
bridge presentation.  Note that this implies $K$ is not isotopic onto the Heegaard surface of the splitting.   As $M$ is neither a lens space nor a connected sum,
the splitting is irreducible and therefore strongly irreducible. 
Let $\hatF$ be a thin level surface --- a level surface 
at a regular value of the height function
such that the first critical level below the surface is a maximum and the
first critical level above the surface is a minimum.

\begin{lemma}\label{lem:monogon}
Let $\hatF$ be a thin level surface in a thin presentation of $K$.
There is no {\em trivializing disk} $D$ for a subarc $\alpha$ of $K$ 
with respect
to $\hatF$. That is, there is no embedded disk $D \subset M$ such that
\begin{enumerate} 
\item
The interior of $D$ is disjoint from $K$, and
\item
$\partial D = \alpha \cup \beta$ where $\alpha$ is a subarc of $K$ and
$\beta$ lies in $\hatF$.
\end{enumerate}
\end{lemma}

\begin{proof}
After an isotopy we may assume that $D$ lies above $\hatF$ near $\beta$
and otherwise $D$ intersects $\hatF$ transversely. 
Among all the arcs of $\Int D \cap \hatF$,
let $\beta'$ be outermost with respect to $\beta$, and let $D'$ be the 
outermost disk that it cuts from $D$.  (If none exists take $\beta' = \beta$ and $D' = D$.) Then $\partial D' = \alpha' \cup \beta'$
where $\alpha'$ is a component of $K-\hatF$. $D'$ guides an isotopy of $\alpha'$
to $\beta'$, giving a positioning of $K$ with smaller width, a contradiction. 
\end{proof} 

In $S^3$, put $K'$ into thin position with respect to the genus $0$ 
Heegaard splitting.  
The thin position argument of \cite{gabai:fatto3mIII}, shows that
there exists a level surface $\widehat{Q}$ of $S^3$ such that 
$F = \hatF - \nbhd(K)$ and $Q = \widehat{Q}-\nbhd(K')$
intersect transversely and 
each arc of $F \cap Q$ is essential $F$. Furthermore, Lemma~\ref{lem:monogon}
shows that each arc of $F \cap Q$ is essential in $Q$ ($\partial F, \partial Q$
are taken to intersect minimally on the boundary of the knot exterior).  
As the exterior of $K'$ is irreducible, after an isotopy we may further
assume there are no simple closed curves of $F \cap Q$ 
that are trivial in
both $F$ and $Q$.

We set up the fat vertexed graphs of intersection $G_Q$ in $\hatQ$ and
$G_F$ in $\hatF$ as in the proof of Theorem~\ref{thm:main}, recording the
intersection patterns of $F$ and $Q$. Let $t = |\hatF \cap K|>0$.

Exactly as in the context of Theorem~\ref{thm:main}, there are two cases
to consider:
\begin{itemize}
\item {\bf \situationnscc}: There are no closed curves of $Q \cap F$ 
in the interior of disk faces of $G_Q$.  
\item
{\bf \situationscc}: There are closed curves of $Q \cap F$ in the interior of 
disk faces of $G_Q$.  The strong irreducibility of the Heegaard splitting allows
us then to assume (section~\ref{sec:scc}) that 
any such closed curve must be non-trivial on $\hatF$ 
and bound a disk on one side of $\hatF$. 
In \situationscc, there is then a meridian disk on one side of the genus 2
splitting that is disjoint from $K$ and $Q$.
\end{itemize}

The arguments of sections~\ref{sec:fatgraphs},
~\ref{sec:escandlongmb}, and ~\ref{sec:combin} apply just as in the
context of Theorem~\ref{thm:main} giving rise to \ESCs\ and \SCs,
and their corresponding long \mobius bands and \mobius bands. The constituent
annuli and \mobius bands of the long \mobius bands are almost properly 
embedded on either side of $\hatF$,  and they are properly embedded in 
\situationnscc.

We now have the following stronger versions of Lemmas~\ref{lem:LMB}
and ~\ref{lem:esc}.

\begin{lemma}\label{lem:thinLMB} 
Let $\sigma$ be a proper $(n-1)$-\ESC\ in $G_Q$.
Let $A=A_1 \cup \dots \cup A_n$ be the corresponding long \mobius band and
let $a_i \in a(\sigma)$ be $\bdry A_i - \bdry A_{i-1}$ for each $i=2 \dots n$ and 
$a_1 = \bdry A_1$.
Assume that, for some $i < j$, $a_i, a_j$ cobound 
an annulus $B$ in $\hatF$.  Then $K$ must intersect the interior of $B$.
\end{lemma}

\begin{proof}

The context of section~\ref{sec:escbounds} is that of Theorem~\ref{thm:main},
that $K$ is in a bridge position that is also thin. 
However, the proof of Lemma~\ref{lem:LMB} proves the above, using a
thin presentation
of $K$, and inserting 
Lemmas~\ref{lem:monogon} and ~\ref{lem:mnosfs} when necessary. 
In particular, the final conclusion of Lemma~\ref{lem:LMB},
that $V$ guides an isotopy of $A_j$ to $B$, contradicts Lemma~\ref{lem:monogon}.
\end{proof}

\begin{lemma}\label{lem:thinesc} 
Assume $M$ contains no Dyck's surface.
If $G_Q$ contains a proper $r$-\ESC\ then $r \leq 1$. Furthermore, if
$\sigma$ is a proper $1$-\ESC\ then the two components of $a(\sigma)$ are 
not isotopic on $\hatF$.
\end{lemma}

\begin{proof}
Let $\sigma$ be a proper $(n-1)$-\ESC\ in $G_Q$ for which
$n$ is largest.  We assume $n \geq 2$.
Let $A=A_1 \cup A_2 \cup \dots \cup A_n$ be the long \mobius band associated
to $\sigma$.  Let $a(\sigma)$ be the collection of simple closed curves 
$a_i = \bdry A_i \cap \bdry A_{i+1}$.  
If no two elements of $a(\sigma)$ are isotopic on $\hatF$, 
then either $n = 3$ and $a_1,a_2, a_3$
cobound a 3-punctured sphere in $\hatF$, contradicting 
(Lemma~\ref{lem:disjtmobiusannulus}) that $M$ contains
no Dyck's surface, or $n=2$ and we satisfy the second
conclusion. 
Thus we assume $a_i, a_j$ are isotopic on $\hatF$ for some $i<j$.
Let $B$ be the annulus cobounded by $a_i, a_j$ on $\hatF$. We may assume
that the interior of $B$ is disjoint from $a(\sigma)$.

Lemma~\ref{lem:thinLMB} shows that there is a vertex $x$ of $K \cap \Int B$.  Since, by Corollary~\ref{cor:bigonsforall}, $\Lambda_x$ contains a bigon,  there is a proper \ESC, $\nu$, and a corresponding long \mobius band $A^x$ whose boundary is a curve comprising two edges of $\Lambda_x$ meeting at $x$ and one other vertex.  Therefore this curve cannot transversely intersect $\bdry B$ and thus must be contained in $B$.  By 
Lemma~\ref{lem:PLMB}, $\nu,\sigma$ must have the same core labels. 
But this contradicts the maximality of $n$. 
\end{proof}

Finally, observe

\begin{lemma}\label{lem:1esc+sc}
If $G_Q$ contains a $1$-\ESC, $\sigma$, and an \SC, $\tau$, on disjoint
label sets, then $M$ contains a Dyck's surface.
\end{lemma}

\begin{proof}
Let $A = A_1 \cup A_2$ be the long \mobius band 
corresponding to $\sigma$ and $A_3$ the almost properly embedded
\mobius band corresponding to
$\tau$. By Lemma~\ref{lem:thinesc}, the components of $\partial A_2$
are not isotopic on $\hatF$. Neither is isotopic to $\partial A_3$,
else $M$ would contain a Klein bottle. By Lemma~\ref{lem:disjtmobiusannulus},
$M$ contains a Dyck's surface.
\end{proof}

To finish the proof of the theorem,
assume $M$ contains no Dyck's 
surface. Lemmas~\ref{lem:thint<8}, ~\ref{lem:thint<6},
~\ref{lem:thint<4}, and ~\ref{lem:thint<2} now eliminate the possibilities for
$t$.

\begin{lemma}\label{lem:thint<8} $t<8$
\end{lemma}

\begin{proof}
By Corollary~\ref{cor:bigonsforall} and Lemma~\ref{lem:thinesc}, 
each label of $G_Q$ belongs to a $1$-\ESC\ or to an \SC. 
Assume $t \geq 8$. If $G_Q$ contains no $1$-\ESC, then there are 
three \SCs\ on disjoint label sets, and 
Lemma~\ref{lem:3disjointmobiusbands} contradicts that $M$ contains
no Dyck's surface.

So assume $G_Q$ contains a $1$-\ESC, $\sigma$, on labels, say, $\{1,2,3,4\}$
-- i.e.\ whose core is a $\arc{23}$-\SC. By Lemma~\ref{lem:1esc+sc}, the 
label $7$ of $G_Q$ belongs to a $1$-\ESC\ on labels $\{7,8,1,2\}$. Similarly,
label $6$ must belong to a $1$-\ESC\ on labels $\{3,4,5,6\}$. The latter
$1$-\ESCs\ contradict Lemma~\ref{lem:1esc+sc}.
\end{proof}

\begin{lemma}\label{lem:thint<6}
$t \neq 6$.
\end{lemma}

\begin{proof}

\begin{claim}\label{clm:thinno2escs} 
With $t=6$, $G_Q$ cannot have two $1$-\ESCs\ on different label sets
whose core \SCs\ lie on the same side of $\hatF$.
\end{claim}

\begin{proof}
WLOG assume $\sigma, \sigma'$ are 1-\ESCs\ on labels 
$\{1,2,3,4\}, \{3,4,5,6,\}$ (resp.). Let $A = A_1 \cup A_2, A' = A_1' \cup
A_2'$ be the long mobius bands corresponding to $\sigma,\sigma'$. 
First assume \situationnscc. Then $A_2,A_2'$ are (non-separating) incompressible
annuli in a handlebody on one side of $\hatF$ intersecting is the single
arc $\arc{34}$ of $K$. A boundary compressing disk of $A_2$ can be
taken disjoint from $A_2'$ (or vice versa). This disk can be used to construct
a trivializing disk for $\arc{34}$, contradicting Lemma~\ref{lem:monogon}.

So assume we are in \situationscc\ and let $D$ be a meridian disjoint from 
$Q$ and $K$. As each component of $\partial A_2$ intersects $\partial A_2'$
in a single point, $\partial D$ must be separating in $\hatF$. In particular,
one component of $\hatF - \partial D$ contains vertices $\{2,3,6\}$ of $G_F$
and the other contains vertices $\{4,5,1\}$. But the arcs $\arc{34},\!\arc{61}$
of $K$ contradict that $D$ is separating on one side of $\hatF$.
\end{proof}

Assume $t=6$.
By Corollary~\ref{cor:bigonsforall} and Lemma~\ref{lem:thinesc}, 
each of the six labels of $G_Q$ belong
to either a $1$-\ESC\ or \SC\ in $G_Q$. If $G_Q$ contains no $1$-\ESC,
then $G_Q$ must have three \SCs\ on disjoint label sets. 
Lemma~\ref{lem:3disjointmobiusbands} shows that $M$ contains a Dyck's surface.
So assume $G_Q$ contains a $1$-\ESC\ on labels, say, $\{1,2,3,4\}$. 

If $G_Q$
also contains a $1$-\ESC\ on a different label set, then by 
Claim~\ref{clm:thinno2escs} and Lemma~\ref{lem:1esc+sc}, we may assume it
is on labels $\{2,3,4,5\}$. Now label $6$ must belong to a $1$-\ESC\ or an 
\SC. A $1$-\ESC\ contradicts Claim~\ref{clm:thinno2escs},
an \SC\ contradicts Lemma~\ref{lem:1esc+sc}. 

So we assume all $1$-\ESCs\ are on label set $\{1,2,3,4\}$. 
Corollary~\ref{cor:bigonsforall} then implies there is a $\arc{45}$-\SC\ 
and a $\arc{61}$-\SC\  (a $\arc{56}$-\SC\ contradicts 
Lemma~\ref{lem:1esc+sc}). But then Lemma~\ref{lem:3disjointmobiusbands} says
that $M$ contains a Dyck's surface.
\end{proof}

\begin{lemma}\label{lem:thint<4} $t \neq 4$.
\end{lemma}

\begin{proof}
Let $t=4$.  
By Corollary~\ref{cor:bigonsforall}, $G_Q$ either contains a $1$-\ESC\ or
two \SCs\ on disjoint label sets. First assume we have \situationnscc.
In the case of a $1$-\ESC\ a boundary compression of
the associated incompressible annulus and in the case of two \SCs\ 
a boundary compression
of one \mobius band disjoint from the second, gives rise to a trivializing 
disk for an arc of $K - \hatF$, contradicting Lemma~\ref{lem:monogon}.

So assume we are in \situationscc, and let $D$ be a meridian on one side
of $\hatF$ disjoint from $Q$ and $K$. Let $\calN$ be the solid torus or tori
obtained by surgering the handlebody in which $D$ lies along $D$, and that
have non-empty intersection with $K$. 

Assume $G_Q$ has a $1$-\ESC, and
let $A = A_1 \cup A_2$ be the associated long \mobius band. 
After an isotopy we may take $A_1, A_2$ as properly embedded in $\calN$ or
its exterior. If $A_2$ lies in $\calN$, then a boundary compression of 
$A_2$ in $\calN$ gives a trivializing disk for an arc of $K-\hatF$.
Thus $A_2$ lies outside of $\calN$.  If both components of $\partial A_2$ lie on the same
component of $\calN$, then $\calO = \nbhd(\calN \cup A_2)$ is Seifert 
fibered over the annulus 
with an exceptional fiber of order two. 
As the exterior of $K$ is atoroidal and irreducible. Some 
component of $M - \Int \calO$ bounds a solid torus $\calT$.
As $M$ is irreducible and atoroidal and as $M$ is not
a Seifert fiber space (Lemma~\ref{lem:mnosfs}), 
$\calO \cup \calT$ is a solid torus whose exterior in $M$ has
incompressible boundary. Again as the exterior of $K$ is
atoroidal and irreducible, $K$ must be isotopic to a core
of the solid torus $\calO \cup \calT$ and consequently to the core of 
$\calN$. But then $K$ can be isotoped to lie on the 
Heegaard surface -- a contradiction. So we may assume
$N$ consists of two solid tori, each containing a component of $\partial A_2$.
But then a boundary compression of  $A_1$ in the solid torus component
containing it, gives rise to a trivializing disk for an arc of 
$K - \hatF$.

So it must be that $G_Q$ contains \SCs\ on disjoint labels sets.
Let $A,A'$ be the corresponding almost properly embedded \mobius bands. 
As $M$ contains no Klein
bottle or projective plane, $\partial A, \partial A'$ must lie on different
components of $\calN$. Then $D$ is a separating meridian of one side of 
$\hatF$ and must lie on the same side as the $A, A'$ (by the separation of
vertices of $G_F$). After surgering away
simple closed curves of intersection, $A$ and $A'$ can be taken to be
properly embedded \mobius bands in separate components of $\calN$. Then
a boundary compression of either gives rise to a trivializing disk for
an arc of 
$K - \hatF$.
\end{proof}

\begin{lemma}\label{lem:thint<2}
$t \neq 2$
\end{lemma}

\begin{proof}
By Corollary~\ref{cor:bigonsforall}, $G_Q$ contains an \SC. 
Let $A$ be the corresponding almost properly embedded \mobius band.
In \situationnscc, $A$ is properly embedded on one side of $\hatF$.
A boundary compression of $A$ then gives rise to a trivializing disk
for an arc of $K - \hatF$, contradicting Lemma~\ref{lem:monogon}.
So we assume \situationnscc, and let $D$ be a meridian disjoint from
$Q$ and $K$. Let $\calN$ be the solid torus
obtained by surgering the handlebody in which $D$ lies along $D$ and taking
that component containing $\partial A$. We may surger the interior of
$A$ off of $\partial \calN$, so that $A$ is properly embedded in $\calN$
or its exterior. If $A$ now lies in $\calN$, then a boundary compression
of it will give rise to a trivializing disk for an arc of $K - \hatF$.
So we assume $A$ is properly embedded in the exterior of $\calN$ and
set $\calO = \nbhd(\calN \cup A)$. If $\partial A$ is longitudinal in
$\calN$, then $\calO$ is a solid torus containing $K$. As $M$ is not a
lens space and the exterior of $K$ is atoroidal and irreducible, $K$ must
be isotopic to a core of $\calO$. On the other hand, the core $L$ of $\calN$
is a $(2,1)$-cable of the core of $\calO$, and hence of $K$. As $L$ has
tunnel number one in $M$, Claim~\ref{claim:ckt1} implies that $K$ can be isotoped to lie on the Heegaard surface.

Thus we assume $\partial A$ is not longitudinal in $\calN$. Then $\calN$ 
is a Seifert fiber space over the disk with two exceptional fibers 
($M$ contains no projective planes).
As both $M$ and the exterior of $K$ are irreducible and atoroidal, the
exterior of $\calO$ is a solid torus, and $M$ is a Seifert fiber space.
This contradicts Lemma~\ref{lem:mnosfs}.
\end{proof}

This completes the proof of Theorem~\ref{thm:thinsurface}.
\end{proof}

\section{More on $G_Q, G_F$ and simple closed curves 
of $F \cap Q$}~\label{sec:fatgraphs}

Assume $K'$ is a hyperbolic knot in $S^3$ and $K'(\gamma)$ has a 2-sided genus
$2$ Heegaard splitting. 
Let $\hatF,F,\hatQ,Q$ be as in the proof of Theorem~\ref{thm:main}. 
Let $G_Q,G_F$ be the labelled graphs of intersection defined there.
In this section we define some terminology for $G_Q,G_F$, and discuss 
simple closed curves of intersection between $Q$ and $F$.
 
On each of $G_Q$ and $G_F$, if the labels around two vertices occur in the 
same direction (equivalently: the oriented intersections of $K'$ with $\widehat{Q}$ or $K$ with $\hatF$ at those spots have the same signs) then we say the vertices are {\em parallel}; otherwise they are {\em anti-parallel}. The orientability of $F$ and $Q$ and the knot exterior gives the
following

\medskip

\noindent {\bf Parity Rule:} 
{\em An edge connects parallel vertices on one of $G_F, G_Q$
if and only if it connects anti-parallel vertices on the other.}

\medskip

We may refer to an edge of $G_F$ or $G_Q$ with endpoints labeled $1$ and $2$, for example, as a $\edge{12}$-edge.  We will also say that $\{1,2\}$ is the {\em label pair} of the edge.

In $M$, the Heegaard surface $\hatF$ bounds two genus $2$ handlebodies $H_B$ and $H_W$:  $M= H_B \cup_{\hatF} H_W$.  We refer to $H_B$ as {\em Black} and $H_W$ as {\em White} and similarly color the objects inside them.

A {\em face} of $G_Q$ is a component of the complement of the edges of $G_Q$ in $Q$.  We color it Black or White according to the side of $\hatF$ on which a small collar neighborhood of its boundary lies.  The arcs of intersection between the boundary of a face and a vertex are the {\em corners} of the face; a vertex is chopped into corners.  We shall refer to both the corners of $G_Q$ between labels $2$ and $3$ and the arc of $K \subset M$ from intersection $2$ to $3$, for example, as $\arc{23}$, as a $\arc{23}$-corner,  or as a $\arc{23}$-arc.  For a contiguous run of corners $\arc{t1},\!\arc{12},\!\arc{23}$ around a vertex or arcs of $K$ we may write $\arc{t123}$.

Two edges of $F \cap Q$ are {\em parallel} on $F$ or on $Q$ if they cobound an embedded bigon in that surface (with corners on the vertices). We also refer to
such edges as parallel on $G_F$ or $G_Q$.
Two faces $g$ and $g'$ of $G_Q$ are {\em parallel}  if there is an embedding of $g \times [0,1]$ into $M-\nbhd(K)$ such that $g \times \{0\} = g$, $g \times \{1\} = g'$ and the components of $\bdry g \times [0,1]$ are alternately composed of rectangles on $\bdry \nbhd(K)$ and parallelisms on $F$ between edges of $g$ and $g'$ .

\subsection{Simple closed curves of $Q \cap F$.}\label{sec:scc}

The intersection graphs $G_Q, G_F$ are given by the arc components of
$F \cap Q$. However, there may also be simple closed curves in $Q \cap F$.
By (**) of the proof of Theorem~\ref{thm:main}, we may assume no such curve is
trivial on both $Q$ and $F$. 
We show in this subsection that any such that is trivial on $Q$ must, WLOG,
be a meridian on one side of $\hatF$. 

\begin{lemma}~\label{lem:AEntscc}
No simple closed curve of $Q \cap F$ that is 
trivial in $Q$ is trivial in $\hatF$.
\end{lemma}
\begin{proof}
Otherwise let ${\widehat D} \subset \hatF$ be the disk bounded by such a 
simple closed curve. Let $G_D$ be $G_F$ restricted to $\widehat D$. 
By (**), $G_D$ is non-empty. Then
there are no 1-sided faces in $G_D$, and no 1-sided faces in the
subgraph of $G_Q$ corresponding to the edges of $G_D$. The argument of 
Proposition 2.5.6 of \cite{cgls:dsok}, 
along with the assumption that $\Delta \geq 3$,
implies that one of $G_D$ or $G_Q$ contains a Scharlemann cycle. Such
a Scharlemann cycle would imply the contradiction that either $S^3$ or
$M$ contains a lens space summand. 
\end{proof}

\begin{cor}~\label{cor:AEntscc}
Any simple closed curve of $F \cap Q$ that is trivial on $Q$ is a meridian
of either $H_W$ or $H_B$.
\end{cor}
\begin{proof}
This follows immediately from Lemma~\ref{lem:AEntscc}, Lemma~\ref{lem:AEGor}
below, and the fact that $H_W \cap_{\hatF} H_B$ is a genus 2, strongly
irreducible Heegaard splitting of $M$.
\end{proof}

For Corollary~\ref{cor:AEntscc}, we need the following which generalizes
Proposition 1.5 and Lemma 2.2 of \cite{sch2}.

\begin{lemma}\label{lem:AEGor}
Let $M=H_W \cup_{\hatF} H_B$ be a Heegaard splitting, where $M$ is a closed
$3$-manifold other than $S^3$. Let $C$ be a simple closed curve in $\hatF$ such
that 
\begin{enumerate}
\item $C$ does not bound a disk in $H_W$ or $H_B$, and
\item $C$ lies in a $3$-ball in $M$. 
\end{enumerate}
Then the splitting $H_W \cup_{\hatF} H_B$ is weakly reducible.
\end{lemma}

\begin{remark}
By the uniqueness of Heegaard splittings of $S^3$, Lemma~\ref{lem:AEGor} 
also holds when $M$ is $S^3$, provided $g(\hatF) \neq 1$. %(Because all Heegaard splittings of $S^3$ of genus \geq 2 are weakly reducible.)
\end{remark}

\begin{proof}
Since $M$ is not $S^3$, the boundary of the 3-ball containing $C$ is 
essential in $M - C$, so $M - C$ is reducible. Since $M-C = H_W \cup_{{\hatF}-C}
H_B$, and $H_W$ and $H_B$ are irreducible, this implies that $\hatF - C$ is 
compressible
in $H_W$ or $H_B$, say $H_W$.

Let $\mathcal D$ be a maximal (with respect to inclusion) disjoint union of
properly embedded disks in $H_W$ such that 
$\bdry \mathcal D \subset {\hatF}-C$, no
component of $\bdry \mathcal D$ bounds a disk in $\hatF-C$, and no pair of 
components
of $\bdry {\mathcal D}$ cobound an annulus in $
\hatF-C$. Note that ${\mathcal D} 
\neq \emptyset$.
Let $H_{W_0} \subset H_W$ be the compression body determined by $\mathcal D$, i.e.
$H_{W_0}$ is obtained from a regular neighborhood 
$\nbhd(\hatF \cup {\mathcal D})$ 
in $H_W$
by capping off any 2-sphere boundary components of 
$\bdry \nbhd(\hatF \cup {\mathcal D})$ with 3-balls in $H_W$. Let 
$\bdry_{\_}H_{W_0} =  
{\bdry H_{W_0}} -\hatF$. Since $C$ does not bound a disk in $H_W$ by hypothesis, 
$C$ is
not contained in any 2-sphere component of 
$\bdry \nbhd(\hatF \cup {\mathcal D})$. 
By the maximality of $\mathcal D$, it follows that $\bdry_{\_}H_{W_0}$ has 
exactly one
component, $G$, say, $C$ is contained in $G$, and no component of $\bdry 
\mathcal D$
bounds a disk in $\hatF$. 

Let $H_{W_1} \subset H_W$ be the handlebody bounded by $G$, and isotop $C$ into
$\Int H_{W_1}$. By the maximality of $\mathcal D$, $G$ is incompressible
in $H_{W_1} - C$. This together with the irreducibility of $H_{W_1}$, implies
also that $H_{W_1} - C$ is irreducible. Let $M_0 = H_B \cup H_{W_0}$.
Since $M-C \cong M_0 \cup_G (H_{W_1} - C)$ is reducible, either $M_0$ is
reducible or $G$ is compressible in $M_0$. This implies that the splitting
of $M_0$ given by $H_{W_0} \cup_{\hatF} H_B$ is reducible or weakly reducible, 
by \cite{haken} or
\cite{cassongordon}, respectively. Hence the same holds for 
$H_W \cup_{\hatF} H_B$.
\end{proof}

Many of the arguments in later sections 
naturally divide themselves
into the two basic cases:
\begin{itemize}
\item {\bf \situationnscc}: There are no closed curves of $Q \cap F$ 
in the interior of disk
faces of $G_Q$. In the later sections, this assumption will allow us
to think of the faces
of $G_Q$ as disks in a Heegaard handlebody of $M$. 
\item
{\bf \situationscc}: There are closed curves of $Q \cap F$ in the interior of 
disk faces of $G_Q$.  By Corollary~\ref{cor:AEntscc}, any such curve must 
be non-trivial on $\hatF$ 
and bound a disk on one side of $\hatF$. A disk face of $G_Q$ containing 
such a curve does not sit in one Heegaard handlebody of $M$, hence some of
the arguments applied in \situationnscc\ will not apply. However, an 
innermost such curve will supply a meridian disk $D$ of either
$H_W$ or $H_B$ which is disjoint from both $K$ and $Q$. This places 
strong restrictions on $G_F$ and yet the combinatorics of the faces of 
$G_Q$ remain the same. Also, one can usually think of the faces of 
$G_Q$ then as living in the exterior of $\hatF$ surgered along $D$. Together,
these facts allow simpler, though somewhat different arguments in 
\situationscc.

\end{itemize}

\section{Scharlemann cycles,  (forked) extended Scharlemann cycles, and
long  \mobius bands}\label{sec:escandlongmb}

\subsection{\SC, \ESC, \FESC}\label{sec:scycles}

A {\em Scharlemann cycle (of length $n$)} is a disk face of $G_Q$ or $G_F$
with $n$ edges, all with the same labels \{$a,b$\}, and all connecting 
parallel vertices of the graph. We use the same term for the set of edges
defining the face. Typically, the Scharlemann cycles considered in this paper
are on
$G_Q$ and of length $2$, so we designate such by the abbreviation \SC. For
specificity, an $\arc{ab}$-\SC\ is one whose edges have labels $\{a,b\}$.
A $\arc{23}$-\SC\ whose corners are on the vertices $x$ and $y$ 
is depicted in Figure~\ref{fig:basicscharlemanncycles}(a).  Though it is 
a rectangle, by virtue of alternatingly naming its sides `corners' and 
`edges', we call it a {\em bigon}.  A {\em Scharlemann cycle of length $3$} is
 shown in Figure~\ref{fig:basicscharlemanncycles}(b).
 Its face is a {\em trigon}. 

\begin{figure}
\centering
\input{basicsharlemanncycles.pstex_t}
\caption{}
\label{fig:basicscharlemanncycles}
\end{figure}

For $n \ge 0$, an {\em $n$-times extended Scharlemann cycle of length $2$}, 
abbreviated
$n$-\ESC, is a set of $2(n+1)$ adjacent parallel edges and the $2n+1$ bigon faces they delineate  between two parallel vertices of a fat vertexed graph such that the central bigon is a Scharlemann cycle of length $2$. This central bigon
is referred to as the {\em core Scharlemann cycle} for the $n$-\ESC.  
When $n>0$, we sometimes refer to an $n$-\ESC\
as simply an ``extended Scharlemann cycle'', abbreviated as ``\ESC''.  
Figure~\ref{fig:basicscharlemanncycles}(c) shows a $2$-\ESC\ on the 
corner $\arc{t12345}$. An $n$-\ESC\ is called {\em proper} if in its corner 
no label appears more than once.   As with \SCs, to emphasize the labels along the corner of an \ESC, we will also call an \ESC\ on the corner $\arc{t123}$, for example, an $\arc{t123}$-\ESC.

A {\em forked $n$-times  extended Scharlemann cycle} is an $(n-1)$-times 
extended Scharlemann cycle of length $2$ with an extra bigon and trigon at 
its two ends.   Figure~\ref{fig:basicscharlemanncycles}(d) shows a forked 
$1$-time extended Scharlemann cycle.  In this paper, a ``forked extended 
Scharlemann cycle,'' which is abbreviated ``\FESC,''  means a forked $1$-time
extended Scharlemann cycle.

We will often use the letters $\sigma$ and $\tau$ to refer to the sets of edges of these various sorts of \SCs\ and the letters $f$, $g$, and $h$ to refer to the faces within them.

\subsection{Almost properly embedded surfaces, long 
\mobius bands}\label{sec:ape}

\begin{defn} Let $H$ be a handlebody on one side of $\hatF$.
A surface, $A$, in $M$ is 
{\it almost properly embedded in $H$} if
\begin{enumerate}
\item $\bdry A \subset \hatF$ and $A$ near $\partial A$ lies in $H$;
\item $\Int A$ is transverse to $\hatF$ and 
$A \cap \hatF$ consists of $\partial A$ along with a collection of simple
closed curves, referred to as $\partial_I A$. Each component of $\partial_I A$ 
is trivial in $A$, essential in $\hatF$, and bounds a disk on one side of
$\hatF$ (i.e.\ is a meridian for $H_W$ or $H_B$).
\end{enumerate}
\end{defn}

We use the disk faces of $G_Q$ to build almost properly embedded surfaces
in $H_W, H_B$.

Assume $\arc{23}$ is a White arc of $K \cap H_W \subset M$.  By $\nbhd(\arc{23})$ we indicate the closed $1$-handle neighborhood $I \times D^2$ of $\arc{23} \subset H_W$ that is a component of $H_W - \Int (M-\nbhd(K))$.

\begin{figure}
\centering
\input{basicmobius.pstex_t}
\caption{}
\label{fig:basicmobius}
\end{figure}
   Let $g$ be the bigon face of a $\arc{23}$-\SC\ of $G_Q$ shown in Figure~\ref{fig:basicmobius}(a).  Then in $M$ the two corners of $g$  both run along the $1$-handle $\nbhd(\arc{23}) \subset H_W$ extending radially to the $\arc{23}$-arc of $K$.  This forms a White \mobius band $A_{23} = g \cup \arc{23}$.  Refer to Figure~\ref{fig:basicmobius}(b).  If $\Int g$ is disjoint from $\hatF$, then $A_{23}$ is properly embedded in $H_W$; otherwise, by 
Corollary~\ref{cor:AEntscc}, it is almost properly embedded
in $H_W$.
  
\begin{figure}
\centering
\input{basicextendedmobius.pstex_t}
\caption{}
\label{fig:basicextendedmobius}
\end{figure}   
Assume the two Black $\arc{12},\!\arc{34}$-bigons $f$ and $h$ flank $g$ as in Figure~\ref{fig:basicextendedmobius}(a).  Identifying their corners to the arcs $\arc{12}$ and $\arc{34}$ of $K$ accordingly in $M$ forms a Black annulus $A_{12,34} = f \cup \arc{12} \cup h \cup \arc{34}$, which by Corollary~\ref{cor:AEntscc}
is almost properly embedded in $H_B$.  As $\bdry A_{23}$ is a 
component of $\bdry A_{12,34}$, together $A_{23} \cup A_{12,34}$ is a 
\mobius band.  
We regard it as a {\em long \mobius band} where the annulus $A_{12,34}$ 
extends the \mobius band $A_{23}$.  See Figure~\ref{fig:basicextendedmobius}(b).  Note that the arc $\arc{1234}$ is a spanning arc of the long \mobius band.

More generally, given $\sigma$, an $(n-1)$-times \ESC\ ($n\geq2$), 
we may again 
form a long \mobius band $A_1 \cup A_2 \cup \dots \cup A_n$  where $A_1$  is an
almost properly embedded \mobius band arising from the core Scharlemann cycle 
and each 
$A_i$, $i\geq2$, is an (almost properly embedded) extending annulus formed 
from successive pairs of 
flanking bigons.  The $A_i$ with odd indices $i$ will have one color and those with even indices will have the other color.  
Let $a_i$ denote the boundary component $\bdry A_i \cap \bdry A_{i+1}$.   
Let $a(\sigma) = \{a_i | i=1, \dots, n-1\}$. Denote by $L(\sigma)$, the label
set for $\sigma$, the set of labels appearing on a corner of $\sigma$. The
{\em core labels} for $\sigma$ are the two labels of its core Scharlemann
cycle. For
example, if $\sigma$ is as in Figure~\ref{fig:basicextendedmobius}(a),
$L(\sigma) = \{1,2,3,4\}$ and the core labels of $\sigma$ are $\{2,3\}$.

Generically we will use this notation, $A_1 \cup A_2 \cup \dots \cup A_n$,
for a long \mobius band and its constituent annuli and \mobius band, but when 
$n \leq 3$ we will often use the notation $A_{23}, A_{12,34}, \dots$
described above to 
emphasize the arc of $K$ on the long \mobius band or its constituent annuli.  

The consideration of long \mobius bands
falls into two basic contexts (see section~\ref{sec:ape}):
\begin{itemize}
\item {\bf \situationnscc}: There are no closed curves of $Q \cap F$ 
in the interior of disk
faces of $G_Q$.  Thus the annuli, \mobius band constituents of a long \mobius 
band are each properly embedded in $H_W$ or $H_B$.
\item
{\bf \situationscc}: There are closed curves of $Q \cap F$ in the interior of 
disk faces of $G_Q$.  By Corollary~\ref{cor:AEntscc}, any such curve must 
be non-trivial on $\hatF$ 
and bound a disk on one side of $\hatF$. In this case the annuli, \mobius
band constituents of a long \mobius band are each almost properly embedded
on one side of $\hatF$. Furthermore, there is a meridian disk $D$ of either
$H_W$ or $H_B$ which is disjoint from both $K$ and $Q$.
\end{itemize}

The fact that in \situationscc, the constituent annuli of the long \mobius
bands are almost properly embedded rather than properly embedded, complicates
the picture of these surfaces. On the other hand, the existence of the meridian
disk $D$ in this case (disjoint from $Q$), greatly restricts what $Q$ can look like, and usually
simplifies the arguments considerably.

We finish this subsection by describing some properties of long \mobius bands.

\begin{lemma}\label{lem:LMBess}
Let $\sigma$ be an $n$-\ESC, $n \geq 0$. Then no component of $a(\sigma)$ 
bounds a disk on either side of $\hatF$.
\end{lemma}

\begin{proof}
Otherwise, the long \mobius band corresponding to $\sigma$ coupled
with the meridian disk bounded by the component of $a(\sigma)$ can
be used to create an embedded projective plane in $M$. Since $M$ is 
$K'(\gamma)$, 
where $\Delta \geq 3$, this contradicts either \cite{gl:oidscyrm} or
\cite{cgls:dsok}.
\end{proof}

\begin{lemma}\label{lem:PLMB}
Let $\sigma$ be a proper $n_1$-\ESC\ and $\tau$ a proper $n_2$-\ESC\ of $G_Q$.
If there are components $a_{\sigma},
a_{\tau}$ of $a(\sigma),a(\tau)$ (resp.) that are isotopic on $\hatF$,
then $\sigma,\tau$ have the same core labels.

{\em Addendum:} Let $D$ be a meridian disk of $H_B,H_W$ disjoint from 
$K$ and $Q$. Let $F^*$ be $\hatF$ surgered along $D$. 
If components $a_{\sigma},
a_{\tau}$ of $a(\sigma),a(\tau)$ (resp.) are isotopic on $F^*$,
then $\sigma,\tau$ have the same core labels.
\end{lemma}

\begin{proof}
The argument for the Addendum is the same as that for the Lemma with $F^*$ replacing
$\hatF$, so we give only the argument for the Lemma itself.

For the proof of this Lemma, we use $\ESC$ to refer to an $n$-\ESC\ for
which $n \ge 0$.
Let $A(\sigma),A(\tau)$ be the long \mobius bands corresponding to
$\sigma,\tau$. By possibly working with \ESCs\ within $\sigma,\tau$,
we may assume $a_{\sigma} = \bdry A(\sigma)$ and $a_{\tau} = \bdry A(\tau)$.
We write $A(\sigma) = E_{\sigma} \cup F_{\sigma}, A(\tau) = 
E_{\tau} \cup F_{\tau}$ where $F_{\sigma},F_{\tau}$ is the union of
faces of $\sigma,\tau$ (resp.) (thought of as disks in, $X_K$, 
the exterior of $K$)
and where $E_{\sigma},E_{\tau}$ are rectangles in $\nbhd(K)$ describing
an extension of $F_{\sigma},F_{\tau}$ across $\nbhd(K)$ to form the
long \mobius band. Thus, $\bdry E_{\sigma} \cap \bdry X_{K} = \bdry F_{\sigma} 
\cap \bdry X_{K}$ and $\bdry E_{\tau} \cap \bdry X_{K} = \bdry F_{\tau} 
\cap \bdry X_{K}$. In all but Case IV' below (and Case II when $\sigma,\tau$
have the same core labels), we will choose $E_{\sigma},E_{\tau}$ to be
disjoint, making $A(\sigma),A(\tau)$ disjoint long \mobius bands whose
boundaries are isotopic on $\hatF$. Such long \mobius bands can be used
to construct an embedded projective plane or Klein bottle in $M$ (note
that each component of $A(\sigma) \cap \hatF$ is either a 
component of $a(\sigma)$ or
a meridian of $H_W,H_B$; the same for $A(\tau)$). This 
contradicts either  \cite{cgls:dsok}, \cite{gl:oidscyrm}, or
\cite{gl:dsokcetI}. Thus in each case below, it suffices to show how to
construct the desired $E_{\sigma}, E_{\tau}$.

Let $\{x,z\}$ be the extremal labels of $\sigma$, and $\{y,w\}$ the extremal 
labels of $\tau$. Let $\{\alpha, \beta\}$ be the corners of $\sigma$,
$\{\alpha',\beta'\}$ the corners of $\tau$ (thought of as arcs in $\bdry
X_K$). See Figure~\ref{fig:plmb1}.

\begin{figure}[h]
\centering
\input{plmb1.pstex_t}
\caption{}
\label{fig:plmb1}
\end{figure}

Let $L(\sigma), L(\tau)$ denote the label set of $\sigma, \tau$.

\medskip

{\bf Case I:} $L(\sigma) \cap L(\tau) = \emptyset$

\medskip

In this case $E_{\sigma}, E_{\tau}$ are automatically disjoint, and 
$A(\sigma), A(\tau)$ can be used
to construct the forbidden projective plane or Klein bottle.

\medskip

{\bf Case II:} $L(\tau) \subset L(\sigma)$

\medskip

We may assume that, say, $y \neq x,z $. 
Let $b_{\sigma}$ be the component of 
$a(\sigma)$ through vertex $y$ -- connecting $y$ to another vertex $r$.
If $r=w$, then $\sigma,\tau$ have the same core labels and we are done.

Thus we assume $r \neq y,w,x,z$ ($r \neq y$  by the Parity Rule). 
Then $b_{\sigma}$ intersects $a_{\tau}$
in a single point (at the vertex $y$). Since $b_{\sigma}$ is disjoint from
$a_{\sigma}$, and $a_{\sigma}$ is isotopic to $a_{\tau}$ on $\hatF$,
$b_{\sigma}$ must intersect $a_{\tau}$ tangentially. That is, as one transverses
around (fat) vertex $y$ of $G_F$ the labels $\{\alpha, \beta\}$ are not
separated by the labels $\{ \alpha', \beta' \}$. Thus in $\nbhd(K)$, we
may choose disjoint $E_{\sigma},E_{\tau}$ as pictured in Figure~\ref{fig:plmb2} 
(thereby making $A(\sigma),A(\tau)$ disjoint).

\begin{figure}[h]
\centering
\input{plmb2.pstex_t}
\caption{}
\label{fig:plmb2}
\end{figure}

\medskip

{\bf Case III:} $L(\sigma) \cap L(\tau)$ is a single interval of labels
$\arc{xy}$ (including a point interval).

\medskip

First we consider the case of a point interval, that is, when $x=y$ (and
Case II does not hold).
Then $a_{\sigma},a_{\tau}$ intersect in a single point (at vertex $x=y$).
As $a_{\sigma},a_{\tau}$ are isotopic on $\hatF$, they must be non-transverse
around vertex $x$ on $G_F$. This means that as one reads around vertex
$x$ on $G_F$, labels $\{\alpha, \beta\}$ do not
separate the labels $\{\alpha', \beta'\}$. We choose disjoint $E_{\sigma},
E_{\tau}$ as pictured in Figure~\ref{fig:plmb4} (with $x=y$), making $A(\sigma),A(\tau)$
disjoint.

Thus we may assume that $\{x,z\} \cap \{y,w\} = \emptyset$. Let $b_{\sigma}$
be the component of $a(\sigma)$ through vertex $y$. Then $b_{\sigma}$ intersects
$a_{\tau}$ in a single point (at $y$). Again $b_{\sigma}$ is disjoint from
$a_{\sigma}$ which is isotopic to $a_{\tau}$, so $b_{\sigma}$ must intersect
$a_{\tau}$ tangentially. Thus, as one transverses vertex $y$ in $G_F$ the 
$\{\alpha,\beta\}$ labels do not separate $\{\alpha',\beta'\}$. We then
may choose disjoint $E_{\sigma},E_{\tau}$ as pictured in Figure~\ref{fig:plmb4}.

\begin{figure}[h]
\centering
\input{plmb4.pstex_t}
\caption{}
\label{fig:plmb4}
\end{figure}

\medskip

{\bf Case IV:} $L(\sigma) \cap L(\tau)$ contains all labels of $G_Q$, and
$L(\sigma)$ overlaps $L(\tau)$ in two intervals of labels: $\arc{xy}$ and
$\arc{wz}$. See Figure~\ref{fig:plmb5}.

\begin{figure}[h]
\centering
\input{plmb5.pstex_t}
\caption{}
\label{fig:plmb5}
\end{figure}

\medskip

If $\{x,z\} = \{y,w\}$ then $\bdry A(\sigma), \bdry A(\tau)$ are isotopic
on $\hatF$ and both go through vertices $x,z$ of $\hatF$. Thus $A(\sigma),
A(\tau)$ can be amalgamated along their boundary to create an embedded
Klein bottle.

Next assume that $x=y$ but $z \neq w$. Let $b_{\sigma}$ be the component of 
$a(\sigma)$ through vertex $w$. As $b_{\sigma}$ is disjoint from $a_{\sigma}$
and intersects $a_{\tau}$ once at $w$, $b_{\sigma}$ and $a_{\tau}$ intersect
non-transversely. Thus around vertex $w$, the labels $\{ \alpha, \beta \}$ do 
not separate $\{ \alpha', \beta' \}$. Similarly, as $a_{\sigma}, a_{\tau}$
intersect in a single point at vertex $x$ and yet are isotopic, their
intersection is non-transverse. That is, around vertex $x$, the labels 
$\{ \alpha, \beta \}$ do 
not separate $\{ \alpha', \beta' \}$. Figure~\ref{fig:plmb7} (with $x=y$) shows that we
can choose disjoint $E_{\sigma},E_{\tau}$.

Thus we may assume $\{x,z\} \cap \{y,w\} = \emptyset$. Let $b_{\sigma}$ be
the component of $a(\sigma)$ through vertex $y$, and let $r$ be the other 
vertex of $G_F$ to which $b_{\sigma}$ is incident. Then $r \neq x,z$.

Assume $r \neq w$. Then as $b_{\sigma}$ intersects $a_{\tau}$ once and is 
disjoint from $a_{\sigma}$, it must intersect $a_{\tau}$ non-transversely.
That is, around vertex $y$ the labels $\{ \alpha, \beta \}$ do not
separate $\{ \alpha', \beta' \}$. Let $c_{\sigma}$ be the component of
$a(\sigma)$ through vertex $w$. Again, $c_{\sigma}$ must intersect $a_{\tau}$
non-transversely at $w$. Hence around $w$ in $G_F$, the labels 
$\{ \alpha, \beta \}$ do not separate $\{ \alpha', \beta' \}$. Thus we
may choose disjoint disks $E_{\sigma},E_{\tau}$ in $\nbhd(K)$ as pictured
in Figure~\ref{fig:plmb7}.

\begin{figure}[h]
\centering
\input{plmb7.pstex_t}
\caption{}
\label{fig:plmb7}
\end{figure}

This leaves us with the case that $r = w$, whose argument is slightly different
from the preceding ones.

{\bf Case IV':} In Case IV above $r=w$.

\medskip

This is the case when the core labels of $\sigma,\tau$ are ``antipodal''
labels. Let $b_{\sigma}$ be the component of $a(\sigma)$ through vertices
$y$ and $w$ of $G_F$. Then $b_{\sigma}$ and $a_{\tau}$ intersect twice.
Since $b_{\sigma}$ is disjoint from $a_{\sigma}$ which is isotopic to $a_{\tau}$,
the algebraic intersection number of $b_{\sigma}$ and $a_{\tau}$ is $0$.
Thus we can choose $E_{\sigma}, E_{\tau}$ in $\nbhd(K)$ so that they are
either (1) disjoint or (2) intersect in exactly two arcs. 
See Figure~\ref{fig:plmb8}. This follows since the labels $\{\alpha,\beta\}$
must separate $\{\alpha',\beta'\}$ either (1) around neither vertices $y,w$
or (2) around both vertices $y,w$.

\begin{figure}[h]
\centering
\input{plmb8.pstex_t}
\caption{}
\label{fig:plmb8}
\end{figure}

Let $B$ be the annulus on $\hatF$ between $a_{\sigma}$ and $a_{\tau}$. If 
$\Int B$ contains a vertex $u$ of $G_F$ (i.e.\ a vertex other than
$x,y,w,z$), then there must be another, $v$, such that $u,v$ lie on the
same component of $a(\sigma)$ or $a(\tau)$. If only one, say $a(\sigma)$,
then we let $\sigma'$ be the \ESC\ within $\sigma$ on labels $u,v$. Then
we may apply the argument of Case I to $\sigma',\tau$. If $u,v$ lie on 
components of both $a(\sigma)$ and $a(\tau)$, then we let $\sigma',\tau'$
be the \ESC\ within $\sigma,\tau$ (resp.) with labels $\{u,v\}$. We apply
the argument at the beginning of Case III (when $\{x,z\}=\{y,w\}$).

Thus we may assume $\Int B$ is disjoint from the vertices of $G_F$. 
Consider $A(\sigma)=E_{\sigma} \cup F_{\sigma}, A(\tau)=E_{\tau} \cup F_{\tau}$
where $F_{\sigma},F_{\tau}$ is the union of faces of $\sigma,\tau$. Then
$A(\sigma),A(\tau)$ are either (1) disjoint or (2) intersect in two double
arcs (from $x$ to $y$ and $w$ to $z$ along $K$). If (1), $M$ contains 
an embedded Klein bottle. If (2), 
$S=A(\sigma) \cup B \cup A(\tau)$ is a Klein bottle that 
self-intersects in a single double-curve (note that the core
of $B$ cannot be a meridian of $H_W,H_B$ else we obtain a projective
plane from $A(\sigma)$. Thus we may assume $A(\sigma),A(\tau)$ are 
disjoint from $B$ except along $a(\sigma),a(\tau)$). The two preimage curves
are disjoint from the cores of each of $A(\sigma), A(\tau)$, and $B$,
and consequently bound disjoint disks, \mobius bands in the pre-image.
We may surger along the double curve to
obtain an embedded projective plane or Klein bottle in $M$.
\end{proof}

\section{Combinatorics}\label{sec:combin}

Let $G_Q,G_F$ be the graphs of intersections defined in the proof of
Theorem~\ref{thm:main}.

\subsection{Great webs.}\label{sec:greatwebs}

  Say a label around a vertex of a subgraph $\Lambda$ of $G_Q$ is a {\em ghost label} of $\Lambda$ if no edge of $\Lambda$ is incident to the vertex at that label.  A {\em ghost edge} is an edge of $G_Q$ incident to a ghost label.  Let $\ell$ denote the number of ghost labels, or equivalently the number of ghost edges counted with multiplicity.  Recall $t = |K \cap \hatF|$ and hence is the number
of vertices of $G_F$.

  A {\em $g$-web} $\Lambda$ is a connected subgraph of $G_Q$ whose vertices are parallel (section~\ref{sec:fatgraphs}) and has at most $t+2g-2$ ghost labels: $\ell \leq t+2g-2$.  If $U$ is a component of $\widehat{Q}-\Lambda$ then we say $D = \widehat{Q}-U$ is a {\em disk bounded by $\Lambda$}.  A {\em great $g$-web} is a $g$-web $\Lambda$ such that there is a disk bounded by $\Lambda$ containing only vertices of $\Lambda$.  When $\Lambda$ is a great $g$-web this disk is unique (since there must be vertices of $G_Q$ anti-parallel to those in $\Lambda$) and so we say it is {\em the} disk bounded by $\Lambda$.

For each label $x$, the subgraph of a great $g$-web $\Lambda$ consisting of all vertices of $\Lambda$ and edges of $\Lambda$ with an endpoint labeled $x$ is denoted $\Lambda_x$.
  Say an $x$-label on a vertex of $\Lambda_x$ is a {\em ghost $x$-label} if the edge incident to it does not belong to $\Lambda_x$.  Let $\ell_x$ denote the total number of ghost $x$-labels of $\Lambda_x$.  Observe that a ghost $x$-label of $\Lambda_x$ is a ghost label of $\Lambda$.

   Given a great $g$-web $\Lambda$ and a disk $D$ bounded by $\Lambda$ containing only vertices of $\Lambda$, the disk $D_\Lambda$ that is the closure of $\widehat{Q} -D$ is the {\em outside face} of $\Lambda$; all other faces of $\Lambda$ are {\em ordinary faces} and are contained in $D$.  A corner 
(section~\ref{sec:fatgraphs}) of a vertex $v$ of $\Lambda$ is {\em outside} or {\em ordinary} according to whether it is the corner of an outside or ordinary face.   A vertex $v$ of $\Lambda$ is an {\em outside} vertex if and only if it has an outside corner, otherwise it is an {\em ordinary} vertex.

\begin{lemma}[\cite{gordon:cmids}, Theorem~6.1]\label{lem:greatweb} Since $\Delta \geq 3 > 2$ and $\hatF$ has genus $2$, then $G_Q$ contains a great $2$-web $\Lambda$.
\end{lemma}

\subsection{The abundance of bigons.}\label{sec:abundancebigons}

Let $\Lambda$ be a great $g$-web of $G_Q$.  Its existence is ensured by Lemma~\ref{lem:greatweb}.  

\begin{lemma}\label{lem:bigonsabound}
If $\Delta \geq 3$ and $t \geq g-1$  then either $\Lambda_x$ contains a bigon for each label $x$ or $\Lambda$ has just one vertex and $t=g-1$.
\end{lemma}

\begin{proof}
First consider the case that $\Lambda$ has just one vertex. Then $\Lambda$ may have no edges, else $G_Q$ would have a monogon.  Therefore $\Delta t = \ell \leq t+2g-2$.  Since $\Delta \geq 3$, this implies $t \leq g-1$.  Thus if $t \geq g-1$ and $\Lambda$ has just one vertex, then $t=g-1$.

Now fix a label $x$.  
We will show if $\Lambda_x$ does not contain a bigon but has more than one vertex, then $g-2 \geq t$.

Regard $\Lambda_x$ as a graph contained in $\widehat{Q}$.   First we assume  $\Lambda_x$ is connected.  Refer to the sole face of $\Lambda_x \subset \widehat{Q}$ that is not contained in the disk bounded by $\Lambda$ as the {\em outside face} of $\Lambda_x$.  Assume the outside face has $k \geq 1$ corners.  We count vertices (corners) and edges in a face locally, i.e.\ the same edge or 
vertex of $G_Q$ may contribute more than once to $k$.   

 Let $V$, $E$, and $F$ denote the number of vertices, edges, and faces of $\Lambda_x$. By the Parity Rule (section~\ref{sec:fatgraphs}),
$E = \Delta V - \ell_x$.     Let $k_i$ be the number of corners in the outside face of $\Lambda_x$ with exactly $i$ ghost $x$-labels.  Then $k = \sum_{i=0}^{\Delta} k_i$ and $\ell_x = \sum_{i=1}^{\Delta} i k_i$. (Recall that a vertex has at most $\Delta$ $x$-labels.)

Suppose $\Lambda_x$ contains no bigons.  Then 
\begin{align*}
 2E &\geq 3(F-1) + k = 3F + (k-3)\\
  F & \leq 2/3 E - 1/3 (k-3) \\
  2=V-E+F &\leq V -E +2/3 E - 1/3 (k-3) \\
  E  &\leq 3V - (k+3)  
\end{align*}
Hence $(\Delta-3)V+(k+3) \leq \ell_x$.  Because the outer face of $\Lambda_x$ has $k$ corners, some corner(s) of the outer face must have more than one ghost $x$-label.  

Let $V_2$ be the number of corners in the outer face with exactly 2 ghost $x$-labels.  Let $V_{\geq3}$ be the number of corners in the outer face with at least $3$ ghost $x$-labels.   Since a corner may have at most $\Delta$ ghost $x$-edges, $\Delta \geq 3$, and $ k+\Delta \leq (\Delta-3)V+(k+3)  \leq \ell_x$, then 
\[\tag{ $\dagger$} V_2 + 2 V_{\geq3} \geq 3    \quad \mbox{ and }   \quad  V_2 +V_{\geq3} \geq 2. \]
(The count $k+\Delta \leq \ell_x$ shows that, at worst, each of the $k$ corners has at least one ghost $x$-edge though there are at least $\Delta$ more.  Since a corner has at most $\Delta$ ghost $x$-labels, there must be at least $2$ corners with more than one ghost $x$-label.  Hence $V_2+V_{\geq3}\geq2$.  Since $\Delta \geq 3$, the possible distributions of these last $\Delta$ ghost $x$-labels implies $V_2 + 2 V_{\geq3} \geq 3$.)

Since there are $t-1$ labels between two consecutive ghost $x$-labels on a corner of the outer face, there are at least $(t+1)V_2 $ ghost labels on all corners of the outer face with exactly $2$ ghost $x$-labels. (No ghost labels of 
$\Lambda$ are separated by a cycle of edges in $\Lambda$.)  Similarly there are at least $(2t+1)V_{\geq3}$ ghost labels on all corners of the outer face with at least $3$ ghost $x$-labels.  Therefore if $\ell$ is the total number of ghost labels for $\Lambda$, then
\[ \tag{$\ddagger$} t+2g-2 \geq \ell \geq (t+1)V_2  + (2t+1) V_{\geq3} = t (V_2+2V_{\geq3}) + (V_2+V_{\geq3}) \geq 3t+2.\]
Hence $g-2 \geq t$.

 Now assume $\Lambda_x$ is not connected.  Observe that ($\dagger$) holds for each connected component of $\Lambda_x$ with at least $2$ vertices, using $V_2$ and $V_{\geq 3}$ to count the corners of the component's outside face.   
 
 Consider a component $\Lambda_x^a$ of $\Lambda_x$.   To each interval between consecutive ghost $x$-labels on a corner of its outside face we associate at least $t-1$ different ghost labels of $\Lambda$:  If all edges incident to the interval are actually ghost labels of $\Lambda$, then we use these.  If there is an edge of $\Lambda$ in this interval, then  (because the ghost $x$-labels bounding the interval cannot be separated by a cycle in $\Lambda$) removing from $\Lambda$ all edges, $\mathcal{E}$, incident to this interval disconnects $\Lambda$,  one component of which contains our initial component $\Lambda_x^a$.  Another component of $\Lambda - \mathcal{E}$ contains a component $\Lambda_x^b$ of $\Lambda_x$ where all the edges of $\Lambda - \Lambda_x$ incident to the outside face of $\Lambda_x^b$ lie in a single interval on a corner of its outside face between consecutive $x$-labels of the vertex giving that corner.  If $\Lambda_x^b$ has at least two vertices, then by ($\dagger$) there must be another corner of its outside face with two consecutive ghost $x$-labels yielding at least $t+1$ ghost edges that we may associate to the original interval.  If $\Lambda_x^b$ has only one vertex then its labels outside this single interval, at least $\Delta t - (t-1) \geq 2t +1$ of them, are all ghost labels that we may associate to the original interval.   
 
 If $\Lambda_x^a$  has at least two vertices, then ($\ddagger$) still holds with $V_2$ and $V_{\geq3}$ counting corners of $\Lambda_x^a$ and the associations above.  If $\Lambda_x^a$ has just one vertex, then since it has no edges (else $G_Q$ would have a monogon) and there is another component of $\Lambda_x$ we may associate, as above, more than $\Delta t$ ghost labels to $\Lambda_x^a$.  In either case we may conclude that $t < g-1$.
\end{proof}

\begin{cor}\label{cor:bigonsforall}
If $\Delta \geq 3$ and $g=2$ then $\Lambda_x$ contains a bigon for each label $x$.
Thus, for each label $x$, $\Lambda$ contains a proper \ESC\ or \SC\ with 
(outermost) 
label $x$.
\end{cor}

\begin{proof}
Apply Lemma~\ref{lem:bigonsabound} with $g=2$ to get the first statement.  Note that if $\Lambda$ were to have just one vertex, then $t=1$ which cannot be.
The second statement follows immediately from the first, as a bigon face
of $\Lambda_x$ corresponds to an \ESC\ or \SC\ of $\Lambda$. 
\end{proof}

\subsection{Special vertices.}\label{sec:specialvertices}

By Lemma~\ref{lem:greatweb} we have a great $2$-web $\Lambda \subset G_Q$ that resides in the sphere $\widehat{Q}$.

Here we seek the existence of a so-called {\em special} vertex of our great $2$-web $\Lambda$ with a large number of ordinary corners incident to bigons, though permitting fewer bigons at the expense of a greater number of trigons.  

\begin{defn}  Let $\mathcal{V}$ be the set of vertices of $\Lambda$.
At each vertex $v\in \mathcal{V}$ let $\phi_i(v)$ count the number of its ordinary corners incident to $i$-gon faces of $\Lambda$.   We have that $\sum_i \phi_i(v) \leq \Delta t$ for each vertex $v$.  This an equality if and only if $v$ is ordinary.  Since only the outside face of $\Lambda$ may be a monogon, $\phi_1(v) = 0$ for all $v$.  Thus we shall write $\phi(v) = (\phi_2(v), \phi_3(v), \dots)$.
\end{defn}

\begin{defn}
Let $F_i$ denote the number of faces of $\Lambda$ (including the outside face) that are $i$-gons; let $\overline{F}_i$ denote the number of ordinary faces of $\Lambda$ that are $i$-gons.  The total number of faces of $\Lambda$ is thus  $F = \sum_i F_i =1+ \sum_i \overline{F}_i$.   Furthermore $i \overline{F}_i = \sum_{v \in \Lambda} \phi_i(v)$ for each $i$ and also $2E = \sum_i iF_i$.
\end{defn}

\begin{defn}\label{def:specialvertex}
Let $\rho = (\rho_2, \rho_3, \rho_4, \dots)$ be a sequence of non-negative integers.  We say that $\rho$ is of {\em type} $[k_2, \dots, k_m]$ if 
\[\rho_2=k_2,\quad \dots, \quad \rho_{m-1} = k_{m-1},  \mbox{ and } \rho_m \geq k_m.\]

A vertex $v$ is said to be of {\em type} $[k_2, \dots, k_m]$ if $\phi(v)$ is of type $[k_2, \dots, k_m]$.

Each integer $N\geq 2$ gives a {\em weight} to which we associate a measure of a sequence of integers $\rho = (\rho_2, \rho_3, \dots)$:
\[ \alpha_N(\rho) = \sum_{i=2}^N \left( \frac{N-i}{i}\right) \rho_i\]

We say that $v$ is a {\em special vertex (of weight $N$)} of $\Lambda$ if 
\[\alpha_N (\phi(v)) >  \frac{N-2}{2}\Delta t - N. \]
\end{defn}

Recall that since $\Lambda$ is a $2$-web, the number of ghost labels $\ell$ is at most $t+2$.   Hence setting $V$ to be the total number of vertices of $\Lambda$ and $E$ to be the number of edges, then $2E=\Delta t V - \ell$.

\begin{prop}\label{prop:generalspecialvertex}
Assume the outside face of $\Lambda$ is a $k$-gon.  Then for any integer $N\geq2$ there exists a vertex $v$ of $\Lambda$ with 
\[ \alpha_N (\phi(v)) \geq \left(\left(\frac{N-2}{2}\right)\Delta t - N\right) + \frac{k+N-\left(\frac{N - 2}{2}\right)\ell}{V} \]
with equality only if $\overline{F}_i = 0$ for $i >N$.
In particular, using that $\ell \leq t+2$,
\[ \alpha_N (\phi(v)) \geq  \left(\left(\frac{N-2}{2}\right)\Delta t - N\right)  + \frac{k+2 - \left(\frac{N - 2}{2}\right)t}{V}. \]
\end{prop}

\begin{proof}
Multiplying the equation $\sum F_i = F = E-V+2$ by $N$ and subtracting $\sum i F_i = 2E$ yields:
\begin{align*}
 \sum (N-i)F_i &= (N-2)E-NV+2N\\
\sum_{i=2}^N (N-i)F_i &= \left( \frac{N-2}{2} \right) 2E -NV +2N + \sum_{i>N}(i-N)F_i\\
\sum_{i=2}^N (N-i)\overline{F}_i &= \left( \frac{N-2}{2} \right) (\Delta t V - \ell) -NV +2N + \sum_{i>N}(i-N)\overline{F}_i + (k-N)\\
\sum_{i=2}^N (N-i)\overline{F}_i &\geq \left(\left( \frac{N-2}{2} \right)\Delta t -N\right)V + \left(k+N-\left( \frac{N-2}{2} \right)\ell\right)  \\
\end{align*}
with equality only if $\overline{F}_i = 0$ for $i >N$.

Since $i \overline{F}_i = \sum_{v \in \mathcal{V}} \phi_i(v)$ for all $i$,
\[\sum_{i=1}^N (N-i) \overline{F}_i = \sum_{v \in \mathcal{V}} \sum_{i=1}^N \left(\frac{N-i}{i}\right) \phi_i(v) =\sum_{v \in \mathcal{V}} \alpha_N(\phi(v)). \]
Hence
\[\sum_{v \in \mathcal{V}} \alpha_N(\phi(v)) \geq \left(\left( \frac{N-2}{2} \right)\Delta t -N\right)V + \left(k+N-\left( \frac{N-2}{2} \right)\ell\right). \tag{$\diamondsuit$}\]
Therefore there exists a vertex $v$ such that 
\[ \alpha_N(\phi(v)) \geq 
\left(\left( \frac{N-2}{2} \right)\Delta t -N\right) + \frac{k+N-\left( \frac{N-2}{2} \right)\ell}{V}\]
as claimed.

Furthermore, using that $\ell \leq t+2$,
\[ \alpha_N (\phi(v)) \geq
\left(\left( \frac{N-2}{2} \right)\Delta t -N\right) + \frac{k+2-\left( \frac{N-2}{2} \right)t}{V}\]
\end{proof}

\begin{prop}\label{prop:countingoutervertices}

Assume there are $j$ distinct outside vertices $v_1, \dots, v_j$ with $v_i$ contributing $k_i$ corners to the outside face of $\Lambda$  (so that the outside face is a $k$-gon where $k=\sum_{i=1}^j k_i$).
If $\Lambda$ has an ordinary vertex, then for any integer $N\geq2$ there exists an ordinary vertex $v$ with 
\[ \alpha_N(\phi(v)) > \left( \left( \frac{N-2}{2} \right)\Delta t - N \right) +\frac{N(1+\frac{1}{2}k - j)}{V-j}.\]
\end{prop}
\begin{proof}
Let $\ell_{v_i}$ be the number of ghost labels incident to the outside vertex $v_i$.  If $v_i$ contributes $k_i$ corners to the outside face, then $v_i$ has $\Delta t - \ell_{v_i} - k_i$ ordinary corners.  
Since there can be no ordinary monogons, 
\[\alpha_N(\phi(v_i)) \leq  \left(\frac{N-2}{2}\right) (\Delta t - \ell_{v_i} - k_i) \]
where equality is only possible in the event that every ordinary face incident to $v_i$ is a bigon.  Assuming $\Lambda$ has an ordinary vertex, then this cannot be an equality for every outside vertex.  This induces the strict inequality in the calculation below.

Continuing from $(\diamondsuit)$ in the proof of Proposition~\ref{prop:generalspecialvertex} (which is an equality only if every ordinary face has $N$ sides or less),
\[\sum_{\mathcal{V} \backslash v_1, \dots, v_j} \alpha_N(\phi(v)) + \sum_{i=1}^j \alpha_N(\phi(v_i)) \geq \left(\left( \frac{N-2}{2} \right)\Delta t -N\right)V + \left(k+N-\left( \frac{N-2}{2} \right)\ell\right). \]
Thus
\begin{align*}
\sum_{\mathcal{V} \backslash v_1, \dots, v_j} \alpha_N(\phi(v)) & \geq \left(\left( \frac{N-2}{2} \right)\Delta t -N\right) (V-j) +  \left(\left( \frac{N-2}{2} \right)\Delta t -N\right)j \\&
\quad \quad\quad \quad
+ \left(k+N-\left( \frac{N-2}{2} \right)\sum_{i=1}^j \ell_{v_i} \right) - \sum_{i=1}^j \alpha_N(\phi(v_i))\\
           & >  \left(\left( \frac{N-2}{2} \right)\Delta t -N\right) (V-j) +   \left(\left( \frac{N-2}{2} \right)\Delta t -N\right)j 
           \\& \quad \quad \quad \quad
+ \left(N+k-\left(\frac{N-2}{2}\right) \sum_{i=1}^j \ell_{v_i} \right) - \sum_{i=1}^j \left(\frac{N-2}{2}\right) (\Delta t - \ell_{v_i} - k_i)\\
         & = \left(\left( \frac{N-2}{2} \right)\Delta t -N\right) (V-j)-Nj+(N+k)+ \left(\frac{N-2}{2}\right)k \\
         & = \left(\left( \frac{N-2}{2} \right)\Delta t -N\right) (V-j) 
           	+ N(1+\frac{1}{2}k - j).\\ 
\end{align*}
Thus there exists an ordinary vertex $v \in \mathcal{V} \backslash v_1, \dots, v_j$ such that
\[ \alpha_N(\phi(v)) > \left(\left( \frac{N-2}{2} \right)\Delta t -N\right) +\frac{N(1+\frac{1}{2}k - j)}{V-j}.\]
\end{proof}

\subsubsection{The existence of special vertices.}
In most of the following lemmas, we conclude that our great $2$-web $\Lambda$ either has a special vertex or a large number of mutually parallel edges.  In the applications of these lemmas such numbers of mutually parallel edges will be prohibited thereby implying the existence of a special vertex.

\begin{lemma}\label{lem:t8trulyspecial}
If $t=8$ then $\Lambda$ either has a special vertex $v$ of weight $N=3$ or $19$ mutually parallel edges.
\end{lemma}
\begin{proof}

By Proposition~\ref{prop:generalspecialvertex} there exists a vertex $v \in \Lambda$ such that
\[ \alpha_3 (\phi(v)) \geq (4\Delta-3) + \frac{k-2}{V}\]
where $k$ is the length of the outside face.
Thus if $k \geq 3$ then $v$ is special.

If $k=1$ or $2$, then let $j \leq k$ be the number of vertices contributing to the $k$ outside corners.  If $\Lambda$ has an ordinary vertex, then by Proposition~\ref{prop:countingoutervertices} there exists an ordinary vertex $v \in \Lambda$ such that
\[ \alpha_3(\phi(v)) > (4\Delta-3) +\frac{3(1+\frac{1}{2}k-j)}{V-j}\]  
This is the special vertex.

If $\Lambda$ has no ordinary vertices and $j=1$, then each edge of $\Lambda$ must bound a monogon.  This cannot occur. 

Thus we now assume $\Lambda$ has no ordinary vertices and $(k,j) = (2,2)$.  
Then all ordinary faces of $\Lambda$ are bigons.  Since there may be at most $10$ ghost edges, $\Lambda$ consists of two vertices and at least $8\Delta - 5 \geq 19$ mutually parallel edges. 
 \end{proof}

\begin{lemma}\label{lem:t6trulyspecial}
If $t=6$ then $\Lambda$ either has a special vertex of weight $N=4$ or $8$ mutually parallel edges.
\end{lemma}
\begin{proof}

By Proposition~\ref{prop:generalspecialvertex} there exists a vertex $v \in \Lambda$ such that
\[ \alpha_4 (\phi(v)) \geq (6\Delta-4) + \frac{k-4}{V}\]
where $k$ is the length of the outside face.
Thus if $k \geq 5$ then $v$ is special.
Therefore assume $k = 1, 2,3$, or $4$ and $j \leq k$ is the number of vertices contributing to the $k$ outside corners.

If $\Lambda$ has no ordinary vertex then the only vertices of $\Lambda$ are the $j$ outside vertices.  Consider the reduced graph of $\Lambda$ obtained by amalgamating mutually parallel edges in the disk bounded by $\Lambda$.
Assuming $\Lambda$ does not have $8$ mutually parallel edges, each edge of this reduced graph represents at most $7$ edges.  Thus a vertex (of this reduced graph) of valence $n$ must have at least $6\Delta -7n$ ghost edges.  Since $\Delta \geq 3$, a valence $1$ vertex has at least $11$ ghost edges and a valence $2$ vertex has at least $4$ ghost edges.  Since the the total number of incidences of ghost edges to $\Lambda$ is at most $8$ there can be no valence $1$ vertices and no more than two valence $2$ vertices.  Hence it must be that $(k,j) = (4,4)$ where there the reduced graph has two valence $2$ vertices. The remaining two vertices of $\Lambda$ have
no ghost labels and must be of valence $3$. Both these vertices have $\phi = (6 \Delta-3,2)$ and are thus special vertices of
weight $N=4$.

Now assume there is an ordinary vertex in $\Lambda$.   Then by Proposition~\ref{prop:countingoutervertices} there exists an ordinary vertex $v \in \Lambda$ such that
\[ \alpha_4(\phi(v)) > (6\Delta-4) +\frac{4(1+\frac{1}{2}k-j)}{V-j}.\]  
Hence $v$ is necessarily special unless $(k,j)$ is $(4,4)$ or  $(3,3)$.
For these two situations we must push the proof of Proposition~\ref{prop:countingoutervertices} further:  

Let $v_1, \dots, v_k$ be the outside vertices of $\Lambda$.  Since $k=j$, each outside vertex has just one outside corner.  Again $\ell_{v_i}$ denotes the number of ghost labels on the outside corner of $v_i$.  Then $v_i$ has $6\Delta  - \ell_{v_i} - 1$ ordinary corners which occur consecutively. 

Assume $\Lambda$ does not have $8$ mutually parallel edges.  Then there may be at most $6$ consecutive bigons. Let $n_i$ be the number of 
ordinary corners of $v_i$ 
that do not belong to bigons. Then
 \begin{align*}
 \alpha_4(\phi(v_i)) &\leq \left( \frac{N-2}{2} \right) (\Delta t - \ell_{v_i} - 1-n_i) +  \left( \frac{N-3}{3} \right) n_i \\
 & = \left( \frac{N-2}{2} \right) (\Delta t - \ell_{v_i} - 1) - \frac{N}{6} n_i
 \end{align*}
and hence
 \[\alpha_4(\phi(v_i)  \leq (6\Delta - \ell_{v_i}-1)-\frac{2}{3}n_i.\]

If for some $i$, $\ell_{v_i}=0$ and $n_i \leq 2$, then $v_i$ is a
special vertex of weight $N=4$. So we assume this is not the case.
Since $\Delta \geq 3$, $n_i \leq 1$ implies that $n_i = 1$ and $\ell_{v_i} \geq
4$. If there are two $v_i$ with $n_i=1$, then, as $\ell = \sum \ell_{v_i} \leq t+2=8$, 
all others have no ghost labels. 
Together these observations mean that $\sum_{i=1}^k n_i \geq 2k-1$.

Hence
\begin{align*}
\sum_{i=1}^k \alpha_4(\phi(v_i)) &\leq \sum_{i=1}^k \left((6\Delta - \ell_{v_i}-1)-\frac{2}{3}n_i \right)\\
&= (6\Delta-1)k -\sum_{i=1}^k \ell_{v_i} - \sum_{i=1}^k \frac{2}{3} n_i\\
&\leq (6\Delta-1)k - \ell - \frac{2}{3}(2k-1).
\end{align*}
Thus, again continuing from $(\diamondsuit)$ in the proof of Proposition~\ref{prop:generalspecialvertex} (as we did in the proof of Proposition~\ref{prop:countingoutervertices}) where now $j=k$,
\begin{align*}
\sum_{\mathcal{V} \backslash v_1, \dots, v_k} \alpha_N(\phi(v)) & \geq \left(\left( \frac{N-2}{2} \right)\Delta t -N\right) (V-k) +  \left(\left( \frac{N-2}{2} \right)\Delta t -N\right)k \\&
\quad \quad\quad \quad
+ \left(k+N-\left( \frac{N-2}{2} \right)\ell\right) - \sum_{i=1}^k \alpha_N(\phi(v_i))
\end{align*}
so that
\begin{align*}
\sum_{\mathcal{V} \backslash v_1, \dots, v_k} \alpha_4(\phi(v)) & \geq(6\Delta-4)(V-k) + (6\Delta-4)k + (4+k-\ell) - \sum_{i=1}^k \alpha_4(\phi(v_i))\\
&\geq (6\Delta-4)(V-k) + (6\Delta-4)k + (4+k-\ell)\\ & \quad\quad\quad\quad - (6\Delta-1)k+\ell+\frac{2}{3}(2k-1) \\
& = (6\Delta-4)(V-k) -4k+4+k+k+\frac{4}{3}k-\frac{2}{3} \\
&= (6\Delta-4)(V-k) +\frac{10-2k}{3}.
\end{align*}
Therefore there is an ordinary vertex $v\in\mathcal{V} \backslash v_1, \dots, v_k$ such that 
\[\alpha_4(\phi(v)) \geq (6\Delta-4)+\frac{10-2k}{3(V-k)}.\]
Since $k=3$ or $4$, $10-2k >0$.  Hence
\[\alpha_4(\phi(v)) > 6\Delta-4.\]
Thus $v$ is a special vertex.
This completes the proof of Lemma~\ref{lem:t6trulyspecial}.
 \end{proof}

\begin{lemma}\label{lem:t4trulyspecial}
If $t=4$ then $\Lambda$ either has a special vertex $v$ of weight $N=4$ or $9$ mutually parallel edges. 
\end{lemma}

\begin{proof}
By Proposition~\ref{prop:generalspecialvertex} there exists a vertex $v \in \Lambda$ such that
\[ \alpha_4 (\phi(v)) \geq (4\Delta-4) + \frac{k-2}{V}\]
where $k$ is the length of the outside face.
Thus if $k \geq 3$ then $v$ is special.

If $\Lambda$ has an ordinary vertex, then 
Proposition~\ref{prop:countingoutervertices} implies there is a special
vertex of weight $N=4$ if $k=1,2$. Thus we assume $\Lambda$ has 
no ordinary vertex and $k=1,2$. Let $j \leq k$ be the number of 
vertices contributing to the outside corners.

Since $\Lambda$ contains no ordinary vertices, any loop edge must bound
a monogon (1-sided face) -- which does not happen in $G_Q$.
Thus $j \neq 1$. Thus $(k,j) = (2,2)$, $\Lambda$ consists of two vertices, and
all the ordinary faces are bigons between the two vertices. 
Since there may be at most 
$6$ ghost edges, these bigons must be induced by at least $4\Delta - 3 \geq 9$ mutually parallel edges.   
\end{proof}

\subsubsection{Types of special vertices.}

See Definition~\ref{def:specialvertex} for vertex type.

\begin{lemma}\label{lem:N3specialtypes}
A special vertex of weight $N=3$ has type $[t\Delta-5]$.
\end{lemma}
\begin{proof}
If $v$ is a special vertex of weight $N=3$, then 
\[\alpha_3(\phi(v)) = \frac{1}{2} \phi_2(v) > \frac{1}{2} t \Delta - 3.\]
Hence $\phi_2(v) > t \Delta - 6$.  Thus $\phi_2(v) \geq t \Delta - 5$ and $v$ is of type $[t \Delta - 5]$.  
\end{proof}

\begin{lemma}\label{lem:N4trulyspecialtypes}
A special vertex of weight $N=4$ has type $[\Delta t-5,4]$, $[\Delta t - 4, 1]$, or $[\Delta t-3]$.
\end{lemma}
\begin{proof}
If $v$ is a special vertex of weight $N=4$, then 
\[\alpha_4(\phi(v)) = \phi_2(v) + \frac{1}{3} \phi_3(v) >\Delta t - 4.\]
Hence $3 \phi_2(v) + \phi_3(v) > 3 \Delta t - 12$. 
Then since $\phi_2(v) + \phi_3(v) \leq \Delta t$, we have $2 \phi_2(v) \geq 2 \Delta t - 11$.
Thus $\phi_2(v) \geq  \Delta t - 5$.

In order to maintain that $\alpha_4(\phi(v))> \Delta t - 4$, 
\begin{itemize}
\item if $\phi_2(v) = \Delta t - 5$ then $\phi_3(v) \geq 4$;
\item if $\phi_2(v) = \Delta t- 4$ then $\phi_3(v) \geq 1$; and 
\item if $\phi_2(v) = \Delta t - 3$ then $\phi_3(v) \geq 0$.
\end{itemize}
The conclusion now follows.
\end{proof}

\section{Elementary surfaces in genus $2$ handlebodies}

Handlebodies are irreducible.
Every properly embedded connected surface in a handlebody is either compressible, $\bdry$-compressible, the sphere, or the disk.

Throughout this article we will repeatedly be considering disks, annuli, and \mobius bands that are properly embedded in a genus $2$ handlebody $H$ and the results of chopping the handlebody along these surfaces.

\subsection{Definitions and notation}

\begin{figure}
\centering
\input{curvesonhandlebody2.pstex_t}
\caption{}
\label{fig:curvesonhandlebody}
\end{figure}
 
On the boundary of a solid torus $\calT$ a non-separating simple closed curve $c$ is: 
\begin{itemize}
\item {\em meridional} if it bounds a (meridional) disk in $\calT$;
\item {\em longitudinal} (or {\em primitive}) if it transversely intersects a meridian of $\calT$ once and thus runs once around $\calT$; or 
\item {\em cabled} if it is neither meridional nor longitudinal and thus runs more than once around $\calT$.
\end{itemize}

Analogously, there are three notable types of non-separating simple closed curves $c$ on the boundary of a genus $2$ handlebody $H$ depicted in Figure~\ref{fig:curvesonhandlebody}.
\begin{itemize}
\item If $c$ bounds a disk in $H$, then $c$ is {\em meridional} or a {\em meridian}.  The disk is a compressing disk which, in this case, we also describe as {\em meridional}. 
(Note:  In later sections we refer to any compressing disk for the handlebody to be a meridian, regardless of whether or not it is separating.)
\item If there is a compressing disk of $H$ whose boundary transversely intersects $c$ once, then $c$ is {\em primitive}.  We say such a meridional compressing disk is a {\em primitivizing} disk for $c$.  Given a primitivizing disk for a primitive curve there is necessarily a meridional compressing disk disjoint from both.
\item If $c$ is neither meridional nor primitive and there is a disjoint meridional compressing disk for $H$ then $c$ is {\em cabled}.  A meridional disk of $H$ whose boundary transversely intersects $c$ non-trivially and coherently (with respect to some chosen orientations) is a {\em cabling} disk if there is another meridional disk disjoint from both it and $c$.
\end{itemize}
Indeed, in each of the three cases there is a non-separating compressing disk $D$ for $H$ that is disjoint from $c$.  Then $c$ is an essential simple closed curve on the boundary of the solid torus $H \cut D$.  Hence $c$ is either meridional on $H \cut D$ and $H$, longitudinal on $H \cut D$ and primitive on $H$, or wound $n>1$ times longitudinally on $H \cut D$ and cabled on $H$.

Denote the attachment of a $2$-handle to $H$ along $c$ by $H \langle c \rangle$.
\begin{itemize}
\item If $c$ is primitive, then $H \langle c \rangle$ is a solid torus.
\item If $c$ is cabled, then $H \langle c \rangle$ is the connect sum of a solid torus and a non-trivial lens space.
\end{itemize}

\subsection{Disks in genus $2$ handlebodies}
\begin{figure}
\centering
\input{disksinhandlebody.pstex_t}
\caption{}
\label{fig:disksinhandlebody}
\end{figure}

Let $D$ be a disk properly embedded in the genus $2$ handlebody $H$.
Then we have the following trichotomy depicted in Figure~\ref{fig:disksinhandlebody}:
\begin{itemize}
\item $D$ is a non-separating compressing disk; $H \cut D$ is one solid torus.
\item $D$ is a separating compressing disk; $H \cut D$ is two solid tori.
\item $D$ is $\bdry$-parallel; $H \cut D$ is one genus $2$ handlebody and one $3$-ball.
\end{itemize}

\subsection{Annuli}\label{sec:annuli}

\begin{figure}
\centering
\input{annulusinhandlebody-nonsep.pstex_t}
\caption{}
\label{fig:annulusinhandlebody-nonsep}
\end{figure}

\begin{figure}
\centering
\input{annulusinhandlebody-sep.pstex_t}
\caption{}
\label{fig:annulusinhandlebody-sep}
\end{figure}

Let $A$ be an incompressible annulus properly embedded in the genus $2$ handlebody $H$.  %Then one of the following occurs:
Then we have the following trichotomy:
\begin{itemize}
\item $A$ is a non-separating annulus. In this case, $\partial A$ is non-separating on $\partial H$.
\item $A$ is a separating but not $\bdry$-parallel annulus. In this case,
$\partial A$ also bounds an annulus on $\partial H$. 
\item $A$ is a $\bdry$-parallel annulus. Again, $\partial A$
bounds an annulus on $\partial H$.
\end{itemize}
Examples of the first two situations are depicted in Figures~\ref{fig:annulusinhandlebody-nonsep} and \ref{fig:annulusinhandlebody-sep}.

Let $d$ be a $\bdry$-compressing disk for $A$.  Then $\bdry \nbhd(A \cup d) 
- \partial H$ is a properly embedded disk $D$ and a parallel copy of $A$ in $H$.
Let $A_+$ be the impression of $A$ on the side of $H \cut A$ containing $d$.  Let $A_-$ be the other impression of $A$.  Then one of the following occurs (situations (2) and (3) are not exclusive):
\begin{enumerate}
\item If $D$ is non-separating, then $A$ is non-separating; $H \cut A$ is a genus $2$ handlebody on which $A_+$ is primitive and $A_-$ either primitive or cabled.  See Figure~\ref{fig:annulusinhandlebody-nonsep}.
\item If $D$ is separating, then $A$ is separating; $H \cut A$ is a genus $2$ handlebody on which $A_+$ is primitive and a solid torus $\calT$ on which $A_-$ is either primitive or cabled.  See Figure~\ref{fig:annulusinhandlebody-sep}.
\item If $D$ is $\bdry$-parallel, then $A$ is $\bdry$-parallel; $H \cut A$ is a genus $2$ handlebody on which $A_-$ lies and a solid torus $\calT$ on which $A_+$ is primitive.
\end{enumerate}
In each of these situations $d$ becomes a primitivizing disk for $A_+$ in $H \cut A$.

We say an annulus, $A$, in a handlebody, $H$, is {\it primitive} if 
there is a meridian disk of $H$ that intersects $A$ in a single essential arc. 
Note that an annulus is primitive if and only a component of its 
boundary is primitive in the ambient handlebody.

\subsection{\mobius bands}

\begin{figure}
\centering
\input{mobiusinhandlebody2.pstex_t}
\caption{}
\label{fig:mobiusinhandlebody}
\end{figure}

Let $A$ be an incompressible \mobius band properly embedded in the genus $2$ handlebody $H$.  Let $d$ be a $\bdry$-compressing disk for $A$.  Then $\bdry \nbhd(A \cup d) - \partial H$ is a properly embedded disk $D$.  The disk $D$ is separating and not $\bdry$-parallel in $H$.  Therefore $H \cut D$ is two solid tori, one of which contains the \mobius band $A$.  Because there is a unique embedding of a \mobius band in a solid torus (up to homeomorphism):
\begin{itemize}
\item Up to homeomorphism, there is a unique embedding of a \mobius band $A$ in a genus $2$ handlebody $H$;  $H \cut A$ is a genus $2$ handlebody on which the annular impression of $A$ is primitive.  
 \end{itemize}
 A $\bdry$-compressing disk for $A$ in $H$ becomes a primitivizing disk for the impression of $A$ in $H \cut A$.  This is depicted in Figure~\ref{fig:mobiusinhandlebody}.

\subsection{Cores of Handlebodies}
A curve embedded in the interior of the solid torus $D^2 \times S^1$ is a {\em core} if it is isotopic to $\{z\} \times S^1$ for some point $z \in D^2$.
A curve $c$ embedded in the interior of a handlebody $H$ is a {\em core} if it is the core of a solid torus connect summand of $H$.  This is equivalent to saying $c$ is isotopic to a primitive curve in $\bdry H$.

%%%%%%%%%%%%%%%%%%%%%%%%%%%%%%%%%%%%%%%%%%%%
\section{Obtaining Dyck's surface by surgery.}\label{sec:dycks}

A closed, connected, compact, non-orientable surface with Euler characteristic $-1$ is the connect sum of three projective planes.  It is known as a {\em cross cap number $3$ surface} and as {\em Dyck's surface} \cite{dyckref}.

If a knot $K'$ in $S^3$ has maximal Euler characteristic spanning surface $S$ with $\chi(S)=-1$ (so that $K'$ has genus $1$ or cross cap number $2$) then surgery on $K'$ along a slope $\gamma$ of distance $2$ from $\bdry S$ produces a manifold with Dyck's surface embedded in it.  There is a \mobius band embedded in the surgery solid torus whose boundary coincides with $\bdry S$ so that together they form an embedded Dyck's surface $\tilde{S}$.  The core of the surgery solid torus is the core of the \mobius band, and hence the surgered knot lies as a simple closed curve on $\tilde{S}$.  Furthermore, such a surgery slope $\gamma$ may be chosen so that it has any desired odd distance $\Delta=\Delta(\gamma,\mu)$ from the $S^3$ meridian $\mu$ of  $K'$. We conjecture that this is the only way a
Dyck's surface arises from a non-integral Dehn surgery on a hyperbolic knot:

\begin{conj}\label{con:dyckssurface}
Let $K'$ be a hyperbolic knot in $S^3$ and assume that $K'(\gamma)$ contains an
embedded Dyck's surface. If $\Delta=\Delta(\gamma,\mu)>1$, where $\mu$ is 
a meridian of $K'$, then there is an embedded Dyck's surface, $\widehat{S}
\subset K'(\gamma)$, such that the core of the attached solid torus in 
$K'(\gamma)$ can be isotoped to an orientation-reversing curve in $\widehat{S}$.
In particular, $K'$ has a spanning surface with Euler 
characteristic $-1$.
\end{conj}

The following goes a long way towards verifying this conjecture.

\begin{thm}\label{thm:3rp2s}
Let $K'$ be a hyperbolic knot in $S^3$ and assume that $M=K'(\gamma)$ 
contains an 
embedded Dyck's surface. 
If $\Delta=\Delta(\gamma,\mu) > 1$, where $\mu$ is a meridian of $K'$, 
then there is an embedded Dyck's surface in $M$ that 
intersects the core of the attached solid torus in $M$ transversely once. 
\end{thm}

\begin{proof}
The proof of this Theorem occupies most of this section.  Initially it closely follows \S6 and \S7 of \cite{gl:dsokcetI} where an analogous theorem is proven for a Klein bottle.  We refer the reader to these sections and note where the proofs diverge in our situation.

For homological reasons (see e.g.\ Lemma~6.2, \cite{gl:dsokcetI}), 
or just by explicit construction of a closed non-orientable surface 
in the exterior of $K'$, if $M$ were to contain an embedded closed non-orientable surface, then $\Delta$ cannot be even.  Hence we assume $\Delta \geq 3$  and odd.

Assume that a Dyck's surface does embed in $M=K'(\gamma)$. 
Note that any embedding of Dyck's surface in $M$ must be incompressible since otherwise a compression would produce an embedded Klein bottle or projective plane; neither of these may occur since $\Delta > 1$.

Let $K$ be the core of the attached solid torus in $M=K'(\gamma)$.
As $S^3$ contains no embedded Dyck's surface there is no such surface in
$M$ that is disjoint from $K$. Thus if $K$ can be isotoped onto a Dyck's
surface in $M$, it must be as an orientation-reversing curve, and can thus be
perturbed to intersect the surface transversely once. So we may assume this does
not happen. Among all embeddings of Dyck's surfaces in $M$ that intersect 
$K$ transversely, take $\widehat{S}$ to be one that intersects $K$ minimally. Let $\widehat{T}$ be the closed orientable genus $2$ surface that is the boundary of a regular neighborhood of $\widehat{S}$.  Let $S$ and $T$ be the intersection of $\widehat{S}$ and $\widehat{T}$ respectively with the exterior $E$ of $K'$.  Let $t = |\bdry T| = 2|\bdry S|$. As mentioned above, we assume $t>0$.
The goal is to show that $t = 2$.

Let $\widehat{Q}$ be a $2$-sphere in $S^3$.  As in \cite{gl:dsokcetI}, via thin position we may assume $\widehat{Q}$ intersects $K'$ (in $S^3$) transversely so that $Q = \widehat{Q} \cap E$ intersects $S$ transversely and no arc component of $Q \cap S$ is parallel in $Q$ to $\bdry Q$ or parallel in $S$ to $\bdry S$.  Moreover, as $T$ ``double covers'' $S$,  $Q$ intersects $T$ transversely and no arc component of $Q \cap T$ is parallel in $Q$ to $\bdry Q$ or parallel in $T$ to $\bdry T$.  We may now form the labeled fat vertexed graphs of intersection $G_Q$ in $\widehat{Q}$ and $G_T$ in $\widehat{T}$ whose edges are the arc components of $Q \cap T$ as well as the graphs $G_Q^S$ in $\widehat{Q}$ and $G_S$ in $\widehat{S}$ whose edges are the arc components of $Q \cap S$.
Furthermore, the incompressibility of $\widehat{S}$ allows us to assume no disk face of either $G_Q^S$ or $G_Q$ contains a simple closed curve component of $Q \cap S$ or $Q \cap T$ respectively.

The proofs in \S2 of \cite{gl:dsokcetI} go through for the pair $G_Q$ and $G_T$ after replacing ``web'' with ``$2$-web'' throughout to accommodate that $T$ has genus $2$ rather than $1$.  In particular, Theorem~6.3 of \cite{gl:dsokcetI} becomes
\begin{lemma}\label{lem:nonorgreatweb}
$G_Q$ contains a great $2$-web.
\end{lemma}

We refer to the side of $\widehat{T}$ containing $\widehat{S}$ as {\em Black} and the other side as {\em White}.  Correspondingly the faces of $G_Q$ are divided into Black and White faces.  Each Black face of $G_Q$ is a bigon and corresponds to an edge of $G_Q^S$.

\begin{lemma}[cf.\ Theorem~6.4 \cite{gl:dsokcetI}]\label{lem:nowhiteSC}
If $t \geq 4$ then no Scharlemann cycle in $G_Q$ of any length 
bounds a White face.
\end{lemma}

\begin{proof}
The proof of Theorem~6.4 \cite{gl:dsokcetI} goes through replacing the Klein bottle with Dyck's surface.
\end{proof}

\begin{lemma}[cf.\ Theorem~6.5 \cite{gl:dsokcetI}]\label{lem:toomanySC}
If $\sigma_1$, $\sigma_2$, $\sigma_3$, and $\sigma_4$ are \SCs\ in $G_Q$, then two of them must have the same pair of labels.
\end{lemma}
\begin{proof}
Assuming no two of the \SCs\ have the same label pair, it must be that $t \geq 4$.
Since the faces of each of these \SCs\ must be Black by Lemma~\ref{lem:nowhiteSC}, then their label pairs are all mutually distinct.  Hence they give rise to four disjoint \mobius bands properly embedded in the Black side of $\widehat{T}$.  Their intersections with $\widehat{S}$ form four mutually disjoint orientation reversing curves.  But this contradicts that $\widehat{S}$ is the connect sum of only three projective planes.
\end{proof}

\begin{lemma}[cf.\ Theorem~6.6 \cite{gl:dsokcetI}]\label{lem:noextend}
If $t \geq 6$ then $G_Q$ does not contain a $1$-\ESC\
(see section~\ref{sec:scycles}). 
\end{lemma}
\begin{proof}
Follow the proof of Theorem~6.6 \cite{gl:dsokcetI} until the last three sentences.  Recall there is a \mobius band $A$ such that $\bdry A = \hat{\alpha}$ and the core curve of $A$ is $\hat{\beta}$.  An arc of $K$ is a spanning arc of $A$.
 On $\widehat{S}$ the curves $\hat{\alpha}$ and $\hat{\beta}$ are disjoint, embedded non-trivial loops.  On $\widehat{S}$ $\hat{\alpha}$ is orientation preserving and $\hat{\beta}$ is orientation reversing.  A small neighborhood of $\hat{\beta}$ on $\widehat{S}$ is a \mobius band $B$.

If $\hat{\alpha}$ is separating, then on $\widehat{S}$ it must bound either a 
\mobius band, once-punctured Klein bottle, or once-punctured torus $P$ that is disjoint from $\hat{\beta}$.  If $P$ is a \mobius band, then $P \cup A$ is a Klein bottle.  By assumption (since $\Delta > 2$) this cannot occur.  If $P$ is a once-punctured Klein bottle or once-punctured torus, then $\widehat{P}=P \cup A$ is a closed non-orientable surface with $\chi = -1$.  We may now perturb $\widehat{P}$ to be transverse to $K$ and have fewer intersections with $K$.  This contradicts the minimality of $\widehat{S}$.

If $\hat{\alpha}$ is non-separating then consider the annulus $A' = A \backslash \hat{\beta}$.  Then $A'$ may be pushed off $A$ keeping $\bdry A'$ on $\widehat{S}$ so that $\bdry A'$ is a push-off of $\hat{\alpha}$ and $\bdry B$.  Then cutting $\widehat{S}$ open along $\hat{\alpha}$ and $\partial B$, we may attach $A$ and $A'$ to the resulting boundary components to form $\widehat{P}$, a new embedded instance of Dyck's surface.  Again we may now perturb $\widehat{P}$ to be transverse to $K$ and have fewer intersections with $K$.  This contradicts the minimality of $\widehat{S}$.
\end{proof}
Let $\calL$ be the set of labels of $G_Q$ that are labels of \SCs\ in $G_Q$.

\begin{lemma}[cf.\ Theorem~6.7 \cite{gl:dsokcetI}]\label{lem:labelcount}
If $t \geq 6$ then $| \calL| \geq (4t-2)/5$.
\end{lemma}
\begin{proof}
The proof is the same as that of Theorem~4.3 \cite{gl:dsokcetI}, using Lemma~\ref{lem:nonorgreatweb} instead of its Corollary~2.7 and Lemma~\ref{lem:noextend} instead of its Theorem~3.2.
\end{proof}

\begin{lemma}\label{lem:notmultipleof4}
$t$ is not a positive multiple of $4$.
\end{lemma}
\begin{proof}
If $t = 4k$, then $|K \cap \widehat{S}|=2k$.  Therefore $\widehat{S}$ may be tubed $k$ times along $K$ to form a closed non-orientable surface in the exterior of $K$.  This forms a closed, embedded, non-orientable surface in $S^3$: a contradiction.
\end{proof}

\begin{lemma} \label{lem:nonortleq6}
$t \leq 6$
\end{lemma}

\begin{proof}
By Lemma~\ref{lem:labelcount} if $t \geq 10$ then there must be at least seven labels that appear as labels of \SCs\ in $G_Q$.  This contradicts Lemma~\ref{lem:toomanySC}.  Lemma~\ref{lem:notmultipleof4} forbids $t=8$.
Hence $t \leq 6$.
\end{proof}

\begin{lemma}\label{lem:BWB}
If $t = 6$, then three consecutive bigons in $\Lambda$ must be Black-White-Black with a Black \SC.  In particular, there may be at most $4$ mutually parallel edges on $\Lambda$.
\end{lemma}

\begin{proof}
In a stack of three consecutive bigons, each corner has four labels.  Since $t=6$, the two sets of four labels of the two corners either completely coincide or overlap in just two labels. The former situation implies the stack is an \ESC; this violates Lemma~\ref{lem:noextend}.  The latter situation implies one of the outer bigons is an \SC.  By Lemma~\ref{lem:nowhiteSC}, this bigon must be Black.  The lemma at hand now follows. 
\end{proof}

\begin{lemma}\label{lem:nonorforkedextS2}
If $t=6$, there cannot be a forked (once) extended Scharlemann cycle
(see section~\ref{sec:scycles}).
\end{lemma}

\begin{figure}
\centering
\input{forkedextS2.pstex_t}
\caption{}
\label{fig:forkedextS2}
\end{figure}

\begin{proof}
Assume there is a forked extended Scharlemann cycle. By symmetry we may assume, 
without loss of
generality, that it has  labels and faces marked as in Figure~\ref{fig:forkedextS2}(a).  The edges of $\bdry f$ and $\bdry g$ form the subgraph of $G_T$ shown in Figure~\ref{fig:forkedextS2}(b).

Collapse $\nbhd(\widehat{S})$ back down to $\widehat{S}$ expanding the two faces $f$ and $g$ into $\overline{f}$ and $\overline{g}$ as shown in Figure~\ref{fig:nonorforkedext}(a).  Since the two $\edge{34}$-edges of Figure~\ref{fig:forkedextS2}(b) bound a single Black bigon, they are collapsed into one orientation reversing loop on $G_S$.  Because the other edges of $\bdry f$ and $\bdry g$ belong to distinct bigons, they remain distinct edges of $\bdry \overline{f}$ and $\bdry \overline{g}$.  In particular the two $\edge{25}$-edges continue to form an orientation preserving loop on $G_S$.  The corresponding subgraph of $G_S$ is shown in Figure~\ref{fig:nonorforkedext}.

\begin{figure}
\centering
\input{nonorforked.pstex_t}
\caption{}
\label{fig:nonorforkedext}
\end{figure}

A small collar neighborhood of the $\edge{34}$-edge on $G_S$ is a \mobius band.  Nearby, the faces $\overline{f}$ and $\overline{g}$ encounter the $3/4$ vertex as in Figure~\ref{fig:splittingvertex}(a).  To separate these faces
we may perturb $K$ near the vertex, introducing two new intersections with $\widehat{S}$ as shown in Figure~\ref{fig:splittingvertex}(b).  The perturbation is done so that the resulting five edges of $\bdry \overline{f}$ and $\bdry \overline{g}$ are disjoint. 

\begin{figure}
\centering
\input{splittingvertex.pstex_t}
\caption{}
\label{fig:splittingvertex}
\end{figure}

Now we surger $\widehat{S}$ along the two arcs of $K$ that form the corners of $\overline{f}$ and $\overline{g}$.  This produces a new closed non-orientable surface $\widehat{R}$ that $K$ intersects $2$ fewer times than $\widehat{S}$, though $\chi(\widehat{R}) = -5$.  Finally, since the boundaries of the faces $\overline{f}$ and $\overline{g}$ are disjoint on $\widehat{R}$ and non-separating both individually and together, they simultaneously give compressions of $\widehat{R}$ yielding a closed non-orientable surface $\widehat{S}'$ with $\chi(\widehat{S}') = -1$ that $K$ intersects transversely just once.  This contradicts the minimality assumption on $\widehat{S}$.
\end{proof}

\begin{lemma}\label{lem:justone12}
If $t=6$ and $\Lambda$ contains a Black $\arc{34},\!\arc{56}$-bigon, then there is only one parallelism class of Black $\arc{12}$-\SC.
\end{lemma}

\begin{proof}
 By Lemma~\ref{lem:labelcount} there must be at least five labels that appear as labels of \SCs\ in $G_Q$.    Hence all $6$ labels are labels of \SCs\!.  In particular, there are $\arc{12}$-, $\arc{34}$-, and $\arc{56}$-\SCs\!. 
Choose a representative \SC\ for each Black label pair.  Since these Black bigons are disjoint,  after their corners are identified along $K$, they project to three mutually disjoint orientation reversing simple curves on $\widehat{S}$.  Thus the complement of these three curves is a thrice-punctured sphere $P$.  
A Black $\arc{34},\!\arc{56}$-bigon projects to a properly embedded arc $a$ on $\widehat{S}$ connecting two of the punctures on $P$.  
Given a new $\arc{12}$-\SC, it projects to a properly embedded arc $b$ on $P$ disjoint from $a$ and connects the third puncture to itself.  Since $P \backslash a$ is an annulus, $b$ must be boundary parallel.  Hence the new $\arc{12}$-\SC\
 must be parallel to the original representative $\arc{12}$-\SC.
\end{proof}

\begin{thm}
$t \leq 2$
\end{thm}

\begin{proof}
By Lemma~\ref{lem:nonortleq6} we have $t \leq 6$.  Since $t \neq 4$ by Lemma~\ref{lem:notmultipleof4}, we assume for a contradiction that $t=6$.
Since bigons may occur in at most runs of three according to Lemma~\ref{lem:BWB}, Lemmas~\ref{lem:N4trulyspecialtypes} and \ref{lem:t6trulyspecial} imply 
that $\Lambda$ has a special vertex $v$ of type $[6\Delta-5,4]$, $[6\Delta-4,1]$, or $[6\Delta-3]$.  Thus there are bigons at at least $6\Delta-5$ corners of $v$.  By Lemma~\ref{lem:BWB} at most $\frac{3}{4}$ of the corners around a vertex may belong to bigons, however.  Hence $6\Delta-5 \leq \frac{3}{4} \Delta t = \frac{9}{2} \Delta$ and so $\Delta \leq \frac{10}{3}$.  Thus $\Delta = 3$.  Therefore $v$ has type $[13,4]$ or $[14]$ (which includes both types $[14,1]$ and $[15]$).

If $v$ is of type $[14]$ then there are at most $4$ faces around $v$ that are not bigons.  
Hence there must be some run of at least $4$ bigons.  
This contradicts Lemma~\ref{lem:BWB}.  

\begin{figure}
\centering
\input{type13-4.pstex_t}
\caption{}
\label{fig:type13-4}
\end{figure}
If $v$ is of type $[13,4]$ and not type $[14]$ then the $13$ bigons must appear around $v$ as in Figure~\ref{fig:type13-4} up to relabeling.  In Figure~\ref{fig:type13-4}, each ``gap'' marks a non-bigon (the two on the ends mark the same corner); at least $4$ mark a trigon.  Since each gap must correspond to a White corner at $v$, there are either three bigons or just one bigon between gaps.  Hence around $v$ there must be four runs of three consecutive bigons, as pictured, each containing a Black \SC\ by Lemma~\ref{lem:BWB}. 

Because at most one non-bigon around $v$ is not a trigon, at least two of the trigons lie between two of these runs of bigons.  Such a trigon is adjacent to $0$, $1$, or $2$ \SCs\ in the two runs of bigons.  Lemma~\ref{lem:nowhiteSC} 
prohibits such a trigon being adjacent to $0$ \SCs\!.  If such a trigon is adjacent to $1$ \SC, then it must be part of a forked extended Scharlemann cycle as in Figure~\ref{fig:forkedextS2}; Lemma~\ref{lem:nonorforkedextS2} prohibits this configuration.  Hence every such trigon must be adjacent to $2$ \SCs\!. This implies that 
there cannot be three consecutive runs of bigon triples at $v$ --- that the
central gap in Figure~\ref{fig:type13-4} is the one not filled by a trigon 
in $\Lambda$. The labeling is now completely forced, except for that of the
singleton Black bigon at the left of the figure. It must be a $\arc{12}$-\SC,
otherwise one of the trigons on either side is a White \SC, contradicting
Lemma~\ref{lem:nowhiteSC}. But then there are three $\arc{12}$-\SCs\ incident
to $v$, along with a $\arc{34},\!\arc{56}$-bigon. 
Lemma~\ref{lem:justone12} implies that there are $\edge{12}$ edges incident
to $v$ that are parallel on $G_T$. The argument of 
Lemma~\ref{lem:parallelwithsamelabel} applied to $G_T, G_Q$ shows that $K$
is a cable knot --- a contradiction.
\end{proof}

The above lemma provides the conclusion of Theorem~\ref{thm:3rp2s}.
\end{proof}

\begin{cor}\label{cor:dyckstunnel2}
Let $K'$ be a hyperbolic knot in $S^3$ and assume that $M=K'(\gamma)$ 
contains an 
embedded Dyck's surface. 
If $\Delta=\Delta(\gamma,\mu) > 1$, where $\mu$ is a meridian of $K'$, 
then there is an embedded Dyck's surface, $\widehat{S}$, 
in $M$ that 
intersects transversely once the core, $K$, of the attached solid torus.
Let $\widehat{N} = M - \nbhd(\widehat{S})$. Either
\begin{enumerate}
\item  $\partial \widehat{N}$ is incompressible in $\widehat{N}$ (hence in $M$).
Furthermore, either $K$ can be isotoped in $M$ onto $\widehat{S}$ as an 
orientation-reversing curve, or the twice-punctured, genus $2$ surface 
$\partial \widehat{N} - \nbhd(K)$ is incompressible
in the exterior of $K$.
\item  $\widehat{N}$ is a genus $2$ handlebody in which $K \cap \widehat{N}$
is a trivial arc. That is, $\widehat{S}$ gives a 1-sided Heegaard 
splitting for $M$ with respect to which $K$ is 
1-bridge. In this case, $K'$ has tunnel number at most $2$.
\end{enumerate}
\end{cor}
\begin{proof}
Theorem~\ref{thm:3rp2s} provides the Dyck's surface $\widehat{S}$ in $M$
that intersects $K$ at most once. 
If $K$ in $M$ can be isotoped
onto $\widehat{S}$, then it must be an orientation-reversing curve in that
surface (as $S^3$ admits no embedded, closed, non-orientable surfaces) 
and we are done (if $\widehat{N}$ has incompressible boundary, it is the
first conclusion, if not, it is the second as $M$ is atoroidal and $\widehat{S}$
is incompressible). 
So assume $K$ cannot be isotoped onto $\widehat{S}$.
Using a thin position of $K'$ in $S^3$, find surfaces $\widehat{S},S,
\widehat{T},T$ as at the beginning of the proof of Theorem~\ref{thm:3rp2s}.
Now we have $t=2$. By Lemma~\ref{lem:nonorgreatweb},
$G_Q$ contains a great $2$-web. This web must contain some White face, $f$,
which must be a Scharlemann cycle (though not necessarily a bigon). 
We view $f$ as giving
an essential disk in $N = \widehat{N} - \nbhd(K)$.   

Assume the twice-punctured, genus $2$ surface 
$\partial \widehat{N} - \nbhd(K)$ is incompressible in the exterior of
$K$. This is equivalent to its incompressibility in $\widehat{N} - \nbhd(K)$.
As $f$ gives a compressing disk for the boundary of $N$, 
Lemma 2.1.1 of \cite{cgls:dsok} (handle addition
lemma) implies that $\partial \widehat{N}$ is incompressible in $\widehat{N}$
and hence in $M$. This is one of the desired conclusions.

So we assume that $\partial \widehat{N} - \nbhd(K)$ is compressible 
in the exterior of $K$, hence in $N$. Compress
$\partial \widehat{N} - \nbhd(K)$ maximally in $N$.
As $K$ is hyperbolic, no component of the result can be 
an essential annulus in $N$. Thus 
$\partial \widehat{N} - \nbhd(K)$ must compress so that the component
containing its boundary is either a twice-punctured, essential
torus  or a boundary parallel annulus in $N$.
Assume first it is a twice-punctured, essential torus. That is, we may 
assume there is a
compressing $D$ for $\partial \widehat{N}$ that is disjoint from $K$,
such that some component of $\partial \widehat{N} - \nbhd(K)$ surgered along
$D$ is a
twice-punctured essential torus, $F$. Let $\hatF$ be the corresponding
torus component obtained by compressing $\partial \widehat{N}$ along $D$.
Note that $\hatF$ is incompressible on the side containing $\widehat{S}$
as any compressing disk could be taken disjoint from both $\widehat{S}$ 
and $D$. On the other hand, $\hatF$ is also incompressible on the side,
$\calO$, 
lying in $\widehat{N}$ by Lemma 2.1.1 of \cite{cgls:dsok}: surgering 
the disk $f$ off of $D$, gives
rise to an essential disk in $\calO - \nbhd(K)$. 
Thus, $\hatF$ is an incompressible torus in $M$, 
a contradiction.  

Thus $\partial \widehat{N} - \nbhd(K)$ must compress
to a boundary parallel annulus in $N$. 
Thus for the arc 
$\kappa = K \cap \widehat{N}$, there is a disk $D_{\kappa}$ in 
$\widehat{N}$  such that $\partial D_{\kappa} = \kappa \cup \delta$
where $\delta \subset \partial \widehat{N}$. That is, $D_{\kappa}$ 
is a ``bridge disk''
for $\kappa $ in $\widehat{N}$.

First, assume that $\partial \widehat{N}$ does not compress in $\widehat{N}$. 

\begin{claim}\label{clm:dissc}
Let $A$ be the annulus $N \cap \nbhd(\kappa)$, and $\alpha$ be the core of
$A$. There are disjoint disks $D_1,D_2$ properly embedded in $N$ such that
$\partial D_1$ intersects $\alpha$ once and $\partial D_2$ intersects $\alpha$
algebraically and geometrically $n>1$ times.
\end{claim}

\begin{proof}
Initially, set $D_1=D_{\kappa}, D_2=f$. 
Isotope $D_1$ so that it is 
disjoint from $D_2$ along $A$. Subject to this
condition, isotop $D_1$ to intersect $D_2$ minimally. 
If $D_1, D_2$ are disjoint, we are done.
Otherwise, there is an outermost arc of intersection, $\nu$, on $D_1$ 
cutting off a disk $d$
which is disjoint from $D_2$ except along $\nu$ and also disjoint from 
$\alpha$. By minimality, each side of 
$\nu$ in $D_2$ contains components of $\partial D_2 \cap A$.
If one side of $\nu$ contains a single component of $\partial D_2 \cap A$,
then add this side of $\nu$ in $D_2$ to $d$,
thereby getting a new disk $D_1$ disjoint from the disk $D_2$ as desired. 
Otherwise, surger $D_2$ along $d$ and take either component as the new
$D_2$. Then $D_1,D_2$ still satisfy the desired intersection properties
with $\alpha$ but have fewer components of intersection with each other.
Repeating, we eventually get disjoint $D_1,D_2$.
\end{proof}

Note that $N-\nbhd(D_1)$ is isotopic to $\widehat{N}$. Under this isotopy
the disk $D_2$ becomes a disk in $\widehat{N}$ whose boundary is easily seen
to be non-separating in $\partial \widehat{N}$. This contradicts the 
incompressibility of $\partial \widehat{N}$ in $\widehat{N}$. Thus it must
be that $K$ could be isotoped onto $\widehat{S}$.

Finally, assume $\partial \widehat{N}$ compresses in $\widehat{N}$.
Since $M$ is atoroidal,
$\widehat{N}$ is a genus $2$ handlebody. That is, $\widehat{S}$ is a 1-sided
Heegaard surface for $M$, and $D_{\kappa}$ says that $K$ is 1-bridge with
respect to this splitting. Adding the cores of $\widehat{N}$ as tunnels to $K$ 
gives a genus $3$ handlebody isotopic
to $\widehat{N} \cup \nbhd(K)$. Since the neighborhood of a punctured non-orientable surface (in an orientable $3$-manifold) is a handlebody, these two tunnels 
provide a tunnel system for $K$, hence also for $K'$.
\end{proof}

\section{Scharlemann cycles, \mobius bands, and annuli}\label{sec:escbounds}

See section~\ref{sec:escandlongmb} for definitions regarding extended Scharlemann cycles, long \mobius bands, and almost properly embedded surfaces. Recall that 
$M = K'(\gamma)$ with $\Delta=\Delta(\gamma,\mu) > 2$, where $\mu$ is a
meridian of $K'$. In particular, as $K'$ is hyperbolic this implies that $M$ does not contain an essential $2$-sphere, $2$-torus, projective plane, or Klein
bottle and is not a lens space. 
$M = H_B \cup_{\hatF} H_W$ is a strongly 
irreducible genus $2$ Heegaard
splitting of $M$. We assume that thin position of $K$, the core of the attached
solid torus in $M$, with respect
to this splitting is the minimal bridge position for $K$ among all genus 2 Heegaard splittings of $M$ and that we have surgered $Q$ to get rid of any simple closed curves of $Q \cap F$ that are trivial on both.

In this and subsequent sections, we will often need to divide the argument
into the two cases:
\begin{itemize}
\item {\bf \situationnscc}: There are no closed curves of $Q \cap F$ 
in the interior of 
faces of $G_Q$.  Thus the annuli, \mobius band constituents of a long \mobius 
band are each properly embedded on one side of $\hatF$.
\item
{\bf \situationscc}: There are closed curves of $Q \cap F$ in the interior of 
faces of $G_Q$.  Recall (Corollary~\ref{cor:AEntscc}) that 
any such must be non-trivial on $\hatF$ 
and bound a disk on one side of $\hatF$. In this case the annuli, \mobius
band constituents of a long \mobius band are each almost properly embedded
on one side of $\hatF$ (section~\ref{sec:ape}).
\end{itemize}

\begin{lemma}\label{lem:MBB} 
Assume $A$ is an almost properly embedded \mobius band in one
handlebody of a Heegaard splitting $H_W \cup_{\hatF} H_B$ of a 
3-manifold $M$. If a core curve of $A$
lies in a $3$-ball in $M$ then the Heegaard splitting is weakly reducible.
\end{lemma}

\begin{proof} $\partial A$ cannot be a meridian of either $H_W$ 
or $H_B$ since $M$ contains no projective planes.
But $\partial A$ can be isotoped into a neighborhood
of the core of $A$. Hence $\partial A$ lies in a 
3-ball in $M$, and Lemma~\ref{lem:AEGor} says the splitting is
weakly reducible.
\end{proof}

%%%%%%%%%%%%%%%%%%%%%%%%
\begin{lemma}\label{lem:1bdryslope}
The exterior of $K$ contains no properly embedded, essential, twice-punctured torus with boundary slope $\gamma$, the meridian of $K$ in $M$.
\end{lemma}

\begin{proof}
Assume $T$ is a properly embedded, essential, twice-punctured torus in $M-\nbhd(K)$.  Then $T$ caps off to a separating torus $\hatT$ in $M$ that is punctured twice by $K$, $T = \hatT-\nbhd(K)$.  
Color the two components of $M \cut \hatT$ Black and White and denote them $M_B$ and $M_W$ respectively.

The thin position argument of \cite{gabai:fatto3mIII} shows that 
we may find a thick sphere $\hat{Q}$ for $K' \subset S^3$ in thin position so that the fat-vertexed graph $G_Q$ of intersection on $\hat{Q}$ between $Q = \hat{Q}-\nbhd(K')$ and $F$ in the exterior of $K$ has no monogons.  (A monogon of $G_Q$ would give a bridge disk for an arc of $K \cut \hatT$ and hence give a compression of $T$.)     We may now follow %Corollary~2.8 and 
Lemmas~8.2 and 8.3 of \cite{gl:dsokcetI} to show that both $M_B-\nbhd(K)$ and $M_W-\nbhd(K)$ are genus $2$ handlebodies.

Since $M_B$ is recovered from the handlebody $M_B-\nbhd(K)$ by attaching a $2$-handle along the core of the annulus $\bdry(M_B-\nbhd(K)) - T$, the Handle Addition Lemma~2.1.1 of \cite{cgls:dsok} implies that $\hatT=\bdry M_B$ is incompressible in $M_B$.  The same argument shows $\hatT=\bdry M_W$ is also incompressible in $M_W$.  Thus the torus $\hatT$ is incompressible in $M$, a contradiction since $M$ is atoroidal.
\end{proof}

\begin{lemma}\label{lem:sSFS1bridge}
Let $N \subset M$ be a small Seifert fiber space over the disk with two exceptional fibers. Assume $N$ contains a properly embedded \mobius band $A$ such that  
$\partial A$ does not lie in a $3$-ball in $M$
(for example, $\partial A$ lies on a genus 2 Heegaard splitting of $M$, 
Lemma~\ref{lem:AEGor}). Furthermore, assume $K \cap N$ is a spanning arc of $A$.
Then there is a genus $2$ Heegaard splitting of $M$ in which $K$ is $0$-bridge.
\end{lemma}

\begin{remark} Note that the proof of Lemma~\ref{lem:sSFS1bridge} actually shows that under
the given hypotheses, $M$ is a Seifert fiber space with at most 
three exceptional fibers,
one of which has order $2$. Furthermore, the new splitting constructed
is a vertical splitting of the Seifert fiber space and $K$ is a core of
this vertical splitting.
\end{remark}

\begin{proof}
Since $\partial A$, hence $N$, does not lie in a $3$-ball in $M$, and $M$ is atoroidal, $M - \Int N$ must be a solid torus. 
Let $T = \bdry N$.
As $\partial \nbhd(A)- T$ must be an essential annulus in $N$, $T-\nbhd(K)$ is 
incompressible in $N - \nbhd(K)$. 
%and $\gamma$ (the meridian of $K$ in $M$) is not a 1-boundary slope, 
Lemma~\ref{lem:1bdryslope} implies
$T - \nbhd(K)$ must compress in $(M - \Int N) - \nbhd(K)$ to give a boundary parallel annulus.  This gives an isotopy of $K \cap (M - \Int N)$ onto $T$ through $M - \Int N$.

Attaching the $1$-handle $\nbhd(K) \cap N$ to $M - \Int N$ then forms a genus $2$ handlebody where $K$ is isotopic onto its boundary.  Since $N - \nbhd(A)$
must be a solid torus, $N - \nbhd(K)$ is a genus $2$ handlebody. Thus we have 
the desired Heegaard splitting.
\end{proof}

%%%%%%%%%%%%%%%%%%%%%%%%

Recall that an \ESC\ is called {\em proper} if in its corner 
no label appears more than once. Section~\ref{sec:escandlongmb} describes
how an \ESC\ gives rise to an almost properly embedded, long \mobius band.

\begin{lemma}\label{lem:LMB} 
Let $\sigma$ be a proper $(n-1)$-\ESC\ in $G_Q$.
Let $A=A_1 \cup \dots \cup A_n$ be the corresponding long \mobius band and
let $a_i \in a(\sigma)$ be $\bdry A_i - \bdry A_{i-1}$ for each $i=2 \dots n$ and 
$a_1 = \bdry A_1$.
Assume that, for some $i < j$, $a_i, a_j$ cobound 
an annulus $B$ in $\hatF$ that is otherwise disjoint from $K$. 
Then $j=i+1$ and $A_j$  cobounds a 
solid torus $V$ with $B$. 
Furthermore, $A_j$ is longitudinal in $V$, the
interior of $V$ is disjoint from $K$, and $V$ guides an
isotopy of $A_j$ to $B$.
 
{\em Addendum:} Let $D$ be a meridian disk of $H_B$ or $H_W$
disjoint from $K$ and $A$, and let $F^*$ be $\hatF$
surgered along $D$. If $a_i, a_j$ cobound an annulus $B$ 
of $F^*$ (rather than $\hatF$) that is otherwise disjoint from $K$, then the
above conclusion is still valid (i.e.\ $A_j=A_{i+1}$ is isotopic to
$B$). 
\end{lemma}

\begin{proof} 
The proof of the Addendum is the same as the proof for the Lemma, 
replacing $\hatF$ with $F^*$, after noting
that $A$ can be surgered off of $B$. So we proceed with the proof of 
the Lemma.

Let $B$ be the annulus on $\hatF$ cobounded by $a_i$ and $a_j$ whose interior is disjoint from $K$. Any simple closed curves of $A \cap B$ in the interior of $B$ must be meridians of either $H_B$ or $H_W$, and we could use such 
with $A$ to create a projective plane in $M$. Hence we may assume $A$ is disjoint from the interior of $B$. Then $T=A_{i+1} \cup \dots \cup A_j \cup B$ is an embedded $2$-torus in $M$. $M$ is atoroidal, so let $D$ be a compressing disk for $T$.  The proof now splits  into three cases depending on the 
relationship of $A_{i+1},A_j,D$ with respect to $\hatF$. 

{\bf Case I:} $A_{i+1}$ and $A_j$ lie on opposite sides of $\hatF$.

Compressing $T$ along $D$ gives a sphere which bounds a ball $B^3$ in $M$.
If $D$ is not contained in $B^3$ then $T$ bounds a solid torus to the side
containing $D$ (and $B^3$). In this situation, the unfurling move from Section~4.3 of \cite{baker:sgkilshsbn} applies to reduce the width of $K$. (Let $V$ be the solid torus bounded by $T$.  $K$ intersects $V$ as a single arc partitioned as a pair of spanning arcs $\kappa$ and $\kappa''$ on the annulus $A_{i+1} \cup \dots \cup A_j$ union an arc $\kappa'$ in $\Int V$ with its boundary on a single boundary component of this annulus.  With support in a small neighborhood of $V$, there is an isotopy of $K$ (which may be viewed as rotations of $V$) that returns $\kappa'$ to its original position and replaces $\kappa,\kappa''$ by spanning arcs of $B$ (these may be taken to be on $B \cap \hatF$ for the Addendum).  A further slight isotopy in a neighborhood of the new $\kappa, \kappa''$ puts $K$ in bridge position with respect to $\hatF$ again, but with smaller bridge number (width).)
Since this contradicts the presumed thinnest positioning of $K$, 
$D$ must be contained in $B^3$. On the other hand, if $B^3$ contains $D$ then
$T$ lies in $B^3$, hence $a_i$ does also. But this contradicts 
Lemma~\ref{lem:AEGor}.

{\bf Case II:} $A_{i+1}$ and $A_j$ lie on the same side of $\hatF$, and
$D$ near $B$ lies on the opposite side of $\hatF$.

Let $V$ be the closure of the component of $M \cut T$ containing $D$. 
As $a_i$ does not lie in a 3-ball by Lemma~\ref{lem:AEGor}, $V$ is a
solid torus. Isotop $K$ into the interior of $V$. Since $M$ is irreducible 
and not
a lens space, and since the exterior of $K$ is irreducible and atoroidal, 
$K$ must be a core of $V$. Now $A' = A_1 \cup \dots \cup A_i$ is a
\mobius band properly embedded in $V$. Thus $K$ is isotopic to the core
of $A'$, hence of $A_1$. The following contradicts either that $K$ has bridge number greater than zero
with respect to $\hatF$ or that $K$ is hyperbolic.

\begin{claim}\label{clm:a1t1}
The core of $A_1$ is isotopic to a core curve of $H_W$ or $H_B$ or has exterior 
which is a Seifert fiber space over the disk with at most 
two exceptional fibers.
\end{claim}

\begin{proof}
In \situationnscc, $A_1$ is a properly embedded \mobius band in one of the
Heegaard handlebodies. So the core of $A_1$ is a core curve of the handlebody. 
So assume we are in \situationscc. Then there
is a meridian disk $E$ of a Heegaard handlebody $H$ on one side of $\hatF$ 
that is disjoint from both
$K$ and $Q$. Let $\calN$ be the component of $H - \nbhd(E)$ containing
$\partial A_1$. We may isotop $A_1$ in $M$, fixing $\partial A_1$, so that 
its interior is
disjoint from $\partial \calN$. If $A_1 \subset \calN$ then the core of
$A_1$ is isotopic to a core of $H$. Thus we assume that $A_1$ is properly
embedded in the exterior of $\calN$ in $M$. Let $n$ be the number
of times $\partial A_1$ winds
around the core of $\calN$.  As $M$ contains no 
projective plane, $n>0$. If $n>1$ then, $U=\nbhd(\calN \cup A_1)$ is a Seifert
fiber space over the disk with two exceptional fibers. $\partial U$ must
compress in $M-U$. As $\partial A_1$ does not lie in a 3-ball by 
Lemma~\ref{lem:AEGor}, $M-U$ must be a solid torus. Thus the exterior
of the core of $A_1$ is a Seifert fiber space over the disk with at most two
exceptional fibers. Finally, assume $n=1$. 
Let $L$ be a core of $\calN$. Then $L$ is a $(2,1)$-cable of the core
of $A_1$.  Claim~\ref{claim:ckt1}
below shows that the core of $A_1$, since it is isotopic to $K$ and therefore
hyperbolic, is isotopic to a core of $H_B$ or $H_W$.
\end{proof}

\begin{claim}\label{claim:ckt1}
Let $L$ be a cable of a hyperbolic knot $K$ in a 3-manifold $M \not\cong
S^3$. Assume that $L$ is a core of $H_B$ in 
a strongly irreducible genus $2$ Heegaard splitting $H_B \cup_{\hatF} H_W$ 
of $M$. Then $K$ is isotopic to a core of either $H_B$ or $H_W$.
\end{claim}

\begin{proof} 
Let $Y=M-\nbhd(L)$. Let $A$ be the cabling annulus for $L$ considered as 
properly embedded in $Y$. Because $K$ is hyperbolic, $A$ is the unique essential
annulus in $Y$ up to isotopy. 
Let $E$ be a non-separating disk in $H_B$ disjoint from $L$, and
let $\alpha=\partial E \subset \partial H_W$. 
Then $A$ is the unique essential annulus in $H_W \cup \nbhd(E) = Y$. 
Now $\partial H_W - \alpha$ is incompressible in $H_W$ by the strong
irreducibility of the splitting.
%, either directly or by finding a curve in $\hatF$ that lies in 3-ball in $M$ and that does not bound  a disk in either $H_W$ or $H_B$ and applying Lemma~\ref{lem:AEGor}). 
Apply Theorem 1 of \cite{EMon32a} 
where $M=H_W$ and $M_{\alpha}=Y$. 
First assume (a) of that Theorem holds and let $A'$ be the $\alpha$-essential
annulus. By Proposition
C of \cite{EMon32a} (and the uniqueness of essential annuli in $Y$), 
$A'$ is isotopic to $A$ in $Y$. Then $A'$ is a separating essential annulus in 
$H_W$ and consequently
cobounds a solid torus $T$ with an annulus $A''$ on $\partial H_W$. As 
$A''$ is disjoint from $\partial E$, $T$ is isotopic to the solid torus
cobounded by $A$ and $\partial Y$.  
Thus $K$ can be isotoped in $Y$ to a core of $T$ and hence to a core of $H_W$.
So assume (b) of Theorem 1 holds and let $S$ be the essential annulus
of $Y$ described. Then again, $S$ is isotopic to $A$. Furthermore, the solid
torus $T$ there must be the cabling solid torus in $Y$ whose core is $K$.
As $\tau_1$ described in Theorem 1 is also a core of $T$, $K$ is isotopic to 
$\tau_1$. As $\tau_1$ is a core of $H_B$, so is $K$.  
\end{proof}

{\bf Case III:} $A_{i+1}$ and $A_j$ lie on the same side of $\hatF$, and
$D$ near $B$ lies on the same side of $\hatF$.

Let $V$ be the component of $M \cut T$ containing $D$. Then $K$ may be 
perturbed to miss $V$ completely. Since $K$ cannot lie in a $3$-ball
(by the irreducibility of the exterior of $K$), $V$ is a solid torus.
$A'=A_1 \cup \dots \cup A_i$ is a \mobius band properly embedded in
$M-V$. We may assume $\partial D$ intersects $\partial A'$ minimally
on $T$. Let $n$ be this intersection number. If $n=0$, then we may
use $A'$ and $D$ to construct a projective plane in $M$, a contradiction.
If $n=1$ then $B$ is longitudinal in $V$. If furthermore, $j>i+1$ then
we can use $V$ to thin $K$ (by reducing the bridge number), a contradiction. 
Thus, when $n=1$ we have the conclusion of the Lemma.

So assume $n > 1$. Let $N = \nbhd(V \cup A')$. Then $N$ is a Seifert fiber
space over the disk with two exceptional fibers of order $2,n$. 
  Lemma~\ref{lem:sSFS1bridge} now applies to give a genus $2$ Heegaard splitting of $M$ in which $K$ is $0$-bridge.  This contradicts the presumed minimal bridge position of $K$ with respect to the original splitting $H_B \cup_{\hatF} H_W$.
\end{proof}

We make the following useful observation:

\begin{lemma}\label{lem:annulusandbridgedisks}
Let $\Gamma$ be a bridge collection of arcs in a handlebody $H$. Let $A$ be an annulus or \mobius band
properly embedded in $H$ that is disjoint from $\Gamma$. Let $\kappa$ be a co-core of $A$. Then $\{\kappa\} \cup
\Gamma$ is a bridge collection of arcs in $H$.
\end{lemma}

\begin{proof}
Let $\calD$ be a collection of bridge disks in $H$ for $\Gamma$. If $\calD$ is disjoint from $A$, then
a $\partial$-compressing disk of $A$ (i.e.\ a disk intersecting $A$ in a single arc essential in $A$) can
be isotoped to give a bridge disk for $\kappa$. We may isotope this disk to be disjoint from $\calD$,
thereby showing that $\kappa \cup \Gamma$ is a bridge collection.

So we assume that $\calD$ can be chosen to meet $A$ in a non-empty collection of co-cores of $A$. An
outermost arc, $\kappa'$, of $A \cap \calD$ in $\calD$ cuts out an outermost disk $D$. After perturbing
$D$ slightly, $D$ becomes a bridge disk for $\kappa'$ disjoint from $D$, showing that $\kappa' \cup \Gamma$
is a bridge collection. As $\kappa$ is isotopic to $\kappa'$ in $A$, this proves the Lemma.
\end{proof}

\begin{lemma}\label{lem:ESCPA} 
Let $A=A_1 \cup \dots \cup A_n$ be the 
long \mobius band corresponding to a proper ESC, and assume
we are in \situationnscc. Assume, as in the conclusion of
Lemma~\ref{lem:LMB}, some $A_j$ cobounds a solid torus $V$ with
an annulus $B$ in $\hatF$, that $A_j$ is longitudinal in $V$, and the 
interior of $V$ is
disjoint from $K$. Then $j=n$.
\end{lemma}

\begin{proof}
Assume for contradiction that $j<n$. We use
$V$ to isotop $A_j$ to $B$ and then into the opposite
handlebody, $H_W$ say. Then $A_{j-1} \cup A_j \cup
A_{j+1}$ is a properly embedded, incompressible annulus
or \mobius band in $H_W$.  
Lemma~\ref{lem:annulusandbridgedisks} shows that
we can reduce the bridge number of $K$ by replacing the arcs $K \cap (A_{j-1} \cup A_{j} \cup A_{j+1})$
with the co-cores of this properly embedded annulus or \mobius band.
\end{proof}

\begin{lemma}\label{lem:3PCC} 
Let $a,a',a''$ be components of $a(\sigma)$ for a proper ESC, $\sigma$.
If $a,a'$ and $a',a''$ each cobound annuli on $\hatF$ with interiors disjoint
from $K$, then $K$ can be thinned.

{\em Addendum:} Let $D$ be a meridian disk of $H_B$ or 
$H_W$ disjoint from $K$ and
$Q$. Let $F^*$ be $\hatF$ surgered along $D$. If $a,a'$ and $a',a''$ each 
cobound annuli on $F^*$ (rather than $\hatF$) with interiors disjoint
from $K$, then $K$ can be thinned. 
\end{lemma}

\begin{proof}
The argument for the Addendum is the same as the argument below with $\hatF$
replaced by $F^*$. 

Assume $a,a'$ cobound an annulus $B$ on $\hatF$, and $a',a''$ cobound $B'$ on
$\hatF$, with int($B$), int($B'$) disjoint from $K$. Let $A=A_1 \cup \dots
\cup A_n$ be the long \mobius band associated to $\sigma$ and $a_i= \partial
A_i - \partial A_{i-1}$ be the components $a(\sigma)$. Then by 
Lemma~\ref{lem:LMB}
(and its Addendum  for the Addendum here), we can write $a=a_i, a'=a_{i+1},
a''=a_{i+2}$ for some $i$. Furthermore, $A_{i+1} \cup B, A_{i+2} \cup B'$
bound solid tori $V,V'$ whose interiors are disjoint from $K$ and 
which guide isotopies of $A_{i+1}, A_{i+2}$ to $B, B'$ (resp.). Together
these define an isotopy of the arcs $K \cap (A_{i+1} \cup A_{i+2})$ onto
$B,B'$. We can then perturb the resulting arcs off of $\hatF$, resulting 
in a thinning of $K$.
\end{proof}

\begin{lemma}%[Lemma DisjtMobiusBands]
\label{lem:3disjointmobiusbands}\quad
$M$ contains a Dyck's surface if either
\begin{enumerate}
\item there are three mutually disjoint \mobius bands in $M$, each 
almost properly embedded in $H_B$ or $H_W$.
\item there is a \mobius band in $M$ almost properly embedded in either $H_B$ or $H_W$ whose boundary is separating on $\hatF$.
\end{enumerate}
\end{lemma}

\begin{proof}
First assume the \mobius bands are properly embedded in the Heegaard handlebodies.  Note that $3$ mutually disjoint \mobius bands cannot all be properly embedded in a single genus $2$ handlebody.  Thus to have $3$ mutually disjoint \mobius bands in $M$ each properly embedded in either $H_B$ or $H_W$, two must lie to one side of $\hatF$ and one must lie to the other.  Furthermore their boundaries must be in different isotopy classes on $\hatF$, else $M$ contains an embedded
Klein bottle.

If the boundary of a \mobius band that is properly embedded in either of these handlebodies has separating boundary on $\hatF$, then it divides $\hatF$ into two once-punctured tori.  Capping off one of these with the \mobius band produces an embedding of Dyck's surface in the handlebody, a contradiction.
Thus the boundaries of these $3$ \mobius bands cut $\hatF$ into two thrice-punctured spheres.  Capping off one of these thrice-punctured spheres with the $3$ \mobius bands produces an embedding of Dyck's surface in $M$.  

Now assume that some \mobius band is almost properly embedded.  Then there is a meridian disk disjoint from all three \mobius bands.  After surgering $\hatF$ along this disk the hypotheses above guarantee that $M$ contains an embedded Klein bottle or projective plane (either some \mobius band boundary becomes trivial or two become isotopic), contrary to assumption.
\end{proof}

\begin{lemma}%[Lemma H2]
\label{lem:disjtmobiusannulus}
Assume there is an annulus, $A$, almost properly embedded in either $H_B$ or $H_W$ whose boundary components are in distinct, essential isotopy classes in $\hatF$ neither of which is a meridian of either handlebody. If there is an almost properly embedded \mobius band in either handlebody that is disjoint from $A$ and whose boundary is not isotopic on $\hatF$ to either boundary component of $A$, then $M$ contains a Dyck's surface.

\end{lemma}

\begin{proof}
Note that in fact the annulus and \mobius band are properly embedded, else the hypothesis would imply the existence of an embedded projective plane in $M$.  Let $A$ be the annulus and $B$ be the \mobius band.
By Lemma~\ref{lem:3disjointmobiusbands} we assume $\bdry B$ is not separating on $\hatF$. Since $A$ is non-separating and incompressible, 
each component of $\partial A$ is non-separating in $\hatF$
($A$ is disjoint from a non-separating meridian disk).
Therefore all three components of $\bdry(A \cup B)$ are non-separating.  The complement of these three curves on $\hatF$ is two copies of the thrice punctured sphere.  Let one be $P$.  Then $P \cup A \cup B$  is an embedding of Dyck's surface in $M$.  
\end{proof}

\begin{lemma}%[Lemma Esc]
\label{lem:esc} 
Assume $M$ does not contain a Dyck's surface. In \situationnscc  
\begin{enumerate}
\item If $G_Q$ contains a proper $r$-\ESC\ then $r \leq 2$.
\item The long \mobius band $A_1 \cup A_2 \cup A_3$ arising from any proper
2-ESC must have $\bdry A_2$ non-isotopic on $\hatF$, $\bdry A_3$ cobounding an annulus $B$ on $\hatF$. $A_3 \cup B$ cobounds a solid torus $V$ whose interior is disjoint from $K$ and in which $A_3$ is longitudinal.  That is, $V$ guides an isotopy of $A_3$ in $M$ to $B$.
\end{enumerate}
\end{lemma}

\begin{proof}
Consider an $(n-1)$-times extended Scharlemann cycle (an $(n-1)$-\ESC, see 
section~\ref{sec:escandlongmb}), $\sigma$, in $G_Q$ for which
$n$ is largest and that is still
proper.  Let $A=A_1 \cup A_2 \cup \dots \cup A_n$ be its associated long \mobius band.  Let $a_i = \bdry A_i \cap \bdry A_{i+1}$.  Assume $A_1$ is Black so that $A_i$ is White for $i$ even and Black for $i$ odd.

Assume there exists a proper 3-ESC so that $n \geq 4$ ($\sigma$ is maximal).
Since there are at most $3$ isotopy classes of mutually disjoint simple loops on $\hatF$, two curves of $a(\sigma)$ must be isotopic. Let $B$ be the annulus cobounded by adjacent ones. 

If the interior of $B$ is not disjoint from $K$ then there is a vertex $x$ of $K \cap \Int B$.  Since, by Corollary~\ref{cor:bigonsforall}, $\Lambda_x$ contains a bigon,  there is a proper extended Scharlemann cycle, $\nu$, and a corresponding long \mobius band $A^x$ whose boundary is a curve comprising two edges of $\Lambda_x$ meeting at $x$ and one other vertex.  Therefore this curve cannot transversely intersect $\bdry B$ and thus must be contained in $B$.  By 
Lemma~\ref{lem:PLMB}, $\nu,\sigma$ must have the same core labels. 
But this contradicts the maximality of $\sigma$.  

Hence $K \cap \Int B = \emptyset$.  Thus by Lemma~\ref{lem:LMB} and 
Lemma~\ref{lem:ESCPA}, since $K$ does not lie on a genus $2$ splitting of
$M$ and $\Int B \cap K = \emptyset$, $\bdry B = a_{n-1} \cup a_n$.
That is, $a_{n-1},a_n$ are the only components of $a(\sigma)$ parallel 
on $\hatF$. Since $n \geq 4$, $n=4$ and the components of $\bdry A_4$, 
$a_3$ and $a_4$, cobound an annulus on $\hatF$.  Moreover the curves $a_1,a_2,
a_3$ are in different isotopy classes on $\hatF$.  But then $A_1$ and $A_3$ 
contradict Lemma~\ref{lem:disjtmobiusannulus} (no $a_i$ bounds a disk else $M$ contains a projective plane).

Now assume there exists a 2-ESC so that $n=3$.  Then by 
Lemma~\ref{lem:disjtmobiusannulus} some pair of boundary components of $A_1$ and $A_3$ must be isotopic.  Lemma~\ref{lem:LMB}, Lemma~\ref{lem:ESCPA}, 
and the argument above (now a 2-ESC is maximal) shows this pair must be 
$\bdry A_3$ and proves part (2).
\end{proof}

\begin{lemma}%[Lemma 3DSC]
\label{lem:3disjtSC}
If there are $3$ \SCs\ in $G_Q$ with disjoint label pairs then $M$ contains a Dyck's surface.
\end{lemma}

\begin{proof}
Assume there are $3$ \SCs\ with disjoint label pairs.  These give rise to $3$ mutually disjoint \mobius bands each almost properly embedded in either $H_B$ or $H_W$.  By Lemma~\ref{lem:3disjointmobiusbands} $M$ contains a Dyck's surface.
\end{proof}

\begin{lemma}\label{lem:parallelwithsamelabel}
No two edges may be parallel in $G_F$ that meet a vertex at the same label.
\end{lemma}

\begin{proof}
Assume there were two such parallel edges in $G_F$. If the two vertices of $G_F$ that these edges connect are parallel, then there must be a length $2$ Scharlemann cycle in $G_F$ which can be used to create an embedded projective plane in the meridional surgery on $K$ --- a contradiction. If the two vertices that the parallel edges connect are anti-parallel, then the argument of \cite[\S5 Case (2)]{GLi}, implies that $K$ is a cable knot --- contradicting that $K$ is hyperbolic.
\end{proof}

\begin{lemma}\label{lem:bridgedisksdisjoitfrommobius}
In \situationnscc, assume there is a White $\arc{23}$-\SC.  Let $A_{23}$ be its corresponding \mobius band in $H_W$.  Then there are mutually disjoint bridge disks for all of the White arcs $K \cap H_W$ whose interiors are also disjoint from $A_{23}$.  
\end{lemma}

\begin{proof}
First take a bridge disk $D_{23}$ for \arc{23} disjoint from the from the other bridge disks $K \cap H_W$.
This disk may be chosen to have interior disjoint from $A_{23}$ since otherwise there is a compression, $\bdry$-compression, or a banding that will form a new bridge disk for $\arc{23}$ intersecting $A_{23}$ fewer times.

Then $\bdry \nbhd(D_{23} \cup A_{23}) - \bdry H_W$ is a separating meridian disk in
$H_W$. Any collection of bridge disks for the remaining arcs of $K \cap H_W$
may be pushed off this disk.
\end{proof}

\begin{lemma}\label{claim:arcsinmobius}
Two properly embedded, non-$\bdry$-parallel arcs in a \mobius band with the same boundary are isotopic rel-$\bdry$.
\end{lemma}

\begin{proof}
Let $a$ and $b$ be two properly embedded, non-$\bdry$-parallel  arcs in a \mobius band such that $\bdry a = \bdry b$.  Let $a'$ be a push-off of the arc $a$.  Isotop $b$ rel-$\bdry$ to minimize both $|a' \cap b|$ and $|a \cap b|$.  

If $|a' \cap b| = 0$, then $a$ and $b$ are isotopic rel-$\bdry$.  If $|a' \cap b| \neq 0$ then the two arcs of $b - a'$ sharing an end point with $a$ either lie on the same side of $a \cup a'$ or on different sides.  If they lie on the same side, then there must be a bigon with boundary composed of an arc in $a'$ and an arc in $b$ with interior disjoint from $a' \cup b$.  Thus there is an isotopy rel-$\bdry$ of $b$ to reduce $|a' \cap b|$ contrary to assumption.  If they lie on different sides, then their union must be $b$ with $b$ parallel into the boundary of the \mobius band.  This too is contrary to assumption.
\end{proof}

\section{ $t <10$.}\label{sec:t<10}

In this section we prove
\begin{thm}\label{thm:tleq8} 
Either $M$ contains a Dyck's surface or $t < 10$.
\end{thm}

\begin{proof}

This is Proposition~\ref{prop:tleq8nscc} of section~\ref{sec:nscct<10} 
when we are in 
\situationnscc,
and Proposition~\ref{prop:tleq8scc} of section~\ref{sec:scct<10} 
in \situationscc.
\end{proof}

\subsection{ $t \geq 10$  and \situationnscc}\label{sec:nscct<10}

\begin{prop}\label{prop:tleq8nscc} 
In \situationnscc, either $M$ contains a Dyck's surface or $t < 10$.
\end{prop}

\begin{proof}
Assume we are in \situationnscc, $M$ does not contain a Dyck's surface, 
and $t \geq 10$. 
By Lemma~\ref{lem:esc}, there are three cases to consider:
\begin{itemize}
\item[{\bf A.}] There is a 2-\ESC\ in $\Lambda$.
\item[{\bf B.}] There is an \ESC\ in $\Lambda$ but no 2-\ESC.
\item[{\bf C.}] There is no \ESC\ in $\Lambda$. 
\end{itemize}

\begin{caseA}  
There is a 2-\ESC\ in $\Lambda$.
\end{caseA}

Assume $G_Q$ contains $\tau$, the 2-\ESC\ depicted in Figure~\ref{fig:tauoflength6} (WLOG as labelled there and with Black and White as pictured).  It gives rise to a long \mobius band $A_1 \cup A_2 \cup A_3$ in which $A_1$ is a Black \mobius band, $A_2$ is a White annulus, and $A_3$ is a Black annulus.
By Lemma~\ref{lem:esc} the components of $\bdry A_2$ lie in two distinct isotopy classes on $\hatF$ whereas the components of $\bdry A_3$ are isotopic to each other.

\begin{figure}
\centering
\input{tauoflength6.pstex_t}
\caption{}
\label{fig:tauoflength6}
\end{figure}

\begin{lemma}
\label{lem:nodisjtSC} 
There is no \SC\ whose label set is disjoint from the labels $\{2, 3, 4, 5\}$.
\end{lemma}
\begin{proof}
Assume there is a \SC\ disjoint from the labels $\{2, 3, 4, 5\}$.  This gives rise to a \mobius band properly embedded in $H_B$ or $H_W$ which must be disjoint from the annulus $A_2$.  Since $M$ contains no Klein bottles, the boundary of this \mobius band cannot be isotopic to either component of $\bdry A_2$.  By Lemma~\ref{lem:disjtmobiusannulus}, however, this cannot occur.
\end{proof}

Recall Corollary~\ref{cor:bigonsforall}, that for each label $x$ the subgraph $\Lambda_x \subset \Lambda$ must contain a bigon and hence an \ESC\ or \SC.
By Lemma~\ref{lem:nodisjtSC}, the \SC\ in a bigon of $\Lambda_x$ must have label pair intersecting $\{2, 3, 4, 5\}$, and by Lemma~\ref{lem:esc}(1) an \ESC\ may be at most twice extended.  Thus $x$ can be no more than $3$ away from the label $2$ or $5$; at its furthest, $x=t-1$ or $x=8$.  Therefore $t=10$.   
 
 For $\Lambda_9$ the only possibility is a 2-\ESC\  with labels $\{9, 10, 1, 2, 3, 4\}$; it contains an \SC\ with label pair $\{1, 2\}$.  Similarly, for $\Lambda_8$ the only possibility is a 2-\ESC\ with labels $\{3, 4, 5, 6, 7, 8\}$; it contains an \SC\ with label pair $\{5, 6\}$.  The \SC\ in $\tau$ has label pair $\{3, 4\}$.  But now the existence of these three \SCs\ with disjoint label pairs contradicts Lemma~\ref{lem:3disjtSC}.  (Really this contradicts that there cannot be $3$ disjoint, properly embedded \mobius bands in a genus $2$ handlebody.)  This completes the proof in Case A.

\begin{caseB}
 There is an \ESC\ in $\Lambda$ but no 2-\ESC.
\end{caseB}

Assume $G_Q$ contains $\tau$, the \ESC\ depicted in Figure~\ref{fig:tauoflength4}.
\begin{figure}
\centering
\input{tauoflength4.pstex_t}
\caption{}
\label{fig:tauoflength4}
\end{figure}

\begin{lemma}%[Lemma 4 of Case(B) of $t \leq 8$]
\label{lem:disjtoppescs} 
There cannot be two \ESCs\ whose \SCs\ have opposite colors and for which the corresponding long \mobius bands are disjoint.
\end{lemma}

\begin{proof}
Assume otherwise.  Let $A_1$ and $A_2$ be the \mobius band and annulus respectively arising from one \ESC\ and $B_1$ and $B_2$ be the \mobius band and annulus respectively arising from the other.  We may assume $A_1$ and $B_2$ are Black while $A_2$ and $B_1$ are White.  No component of $\bdry A_2$  is isotopic on $\hatF$ to a component of $\bdry B_2$ since otherwise the two long \mobius bands
will form an embedded Klein bottle.
Then by Lemma~\ref{lem:disjtmobiusannulus} the components of $\bdry A_2$ must be isotopic as must the components of $\bdry B_2$.  
By Lemma~\ref{lem:LMB}, $A_2$ and $B_2$ are parallel into $\hatF$ (note that since these \ESCs\ are of maximal length, we may apply the argument of Lemma~\ref{lem:esc} to show that the annuli on $\hatF$  between the components of $\partial A_2$ and $\partial B_2$ respectively must be disjoint from $K$).
These two parallelisms however give a thinning of $K$.  This is a contradiction.
\end{proof}

We now consider the possible bigons of $\Lambda_7$ and $\Lambda_9$.  The possibilities are shown in Figure~\ref{fig:lambda79escs-alt}.  Lemma~\ref{lem:disjtoppescs} immediately rules out 7(d).

\begin{figure}
\centering
\input{lambda79escs-alt.pstex_t}
\caption{}
\label{fig:lambda79escs-alt}
\end{figure}

\begin{claim} 
9(c) is impossible.
\end{claim}

\begin{proof}
Consider the bigons of $\Lambda_1$.  The four possibilities are listed in Figure~\ref{fig:lambda1escs}.  With $\tau$ and 9(c), each of 1(a), 1(b), and 1(c) contradict Lemma~\ref{lem:3disjtSC}.  Together 9(c) and 1(d) contradict Lemma~\ref{lem:disjtoppescs}.
\begin{figure}
\centering
\input{lambda1escs.pstex_t}
\caption{}
\label{fig:lambda1escs}
\end{figure}
\end{proof}

\begin{claim}
7(a) and 7(c) are impossible.
\end{claim}
\begin{proof}
With $\tau$ and 7(a), each of 9(a), 9(b), and 9(d) (the remaining possible bigons of $\Lambda_9$)  contradict Lemma~\ref{lem:3disjtSC}.  Similarly with $\tau$ and 7(c), each of 9(a), 9(b), and 9(d) contradict Lemma~\ref{lem:3disjtSC}.
\end{proof}

\begin{claim}
7(b) is impossible.
\end{claim}
\begin{proof}
With $\tau$ and 7(b), each of 9(b) and 9(d) contradict Lemma~\ref{lem:3disjtSC}.  Therefore we must have 9(a).  Hence we have \SCs\ with label pairs $\{3,4\}$, $\{7,8\}$, and $\{8,9\}$.  Again, $\Lambda_1$ must have one of the bigons listed in Figure~\ref{fig:lambda1escs}. Each of the \SCs\ contained within 1(a),1(b), and 1(c) form, along with two of those with labels pairs $\{3,4\}$, $\{7,8\}$, and $\{8,9\}$, a triple of mutually disjoint \SCs.  This contradicts Lemma~\ref{lem:3disjtSC}.

So we assume we have 1(d) along with $\tau$, 7(b), and 9(a). The 
possible bigons of $\Lambda_6$ are 7(a), 9(c), an \SC\ on 
labels $\{5,6\}$,
and a $1$-\ESC\ on labels $\{3,4,5,6\}$. The first two have already been ruled 
out. Each of the two remaining gives rise to an \SC\ that joins with
those above to contradict Lemma~\ref{lem:3disjtSC}.
\end{proof}

This completes the proof in Case B.

\begin{caseC}
 There is no \ESC\ in $\Lambda$. 
\end{caseC}

In this case every label belongs to an \SC.  Lemma~\ref{lem:3disjtSC} then forces $t \leq 6$ contrary to the assumed $t \geq 10$.  This completes the proof in Case C and thus the proof of Proposition~\ref{prop:tleq8nscc}. 
\end{proof} 

\subsection{ $t \geq 10$  and \situationscc}\label{sec:scct<10}

\begin{prop}\label{prop:tleq8scc} 
In \situationscc, either $M$ contains a Dyck's surface or $t < 10$.
\end{prop}

\begin{proof}
Assume we are in \situationscc. Then there is a meridian disk $D$ of $H_W$
or $H_B$ disjoint from $K$ and $Q$. Let $F^*$ be $\hatF$ surgered along $D$.
Then $F^*$ is one or two tori. For contradiction, assume $M$ does not 
contain a Dyck's surface, and $t \geq 10$. 

\begin{lemma}\label{lem:t8nesc}
If $G_Q$ contains an $r$-ESC then $r \leq 3$.
\end{lemma}

\begin{proof}
Let $r$ be the largest value such $G_Q$ contains a proper $r$-\ESC, $\sigma$.
Assume for contradiction, $r \geq 4$. Then $|a(\sigma)| \geq 5$ and there
must be at least three components of $a(\sigma)$ that are isotopic on 
$F^*$. Let $B$ be an annulus between two
components of $a(\sigma)$ on $F^*$ whose interior
is disjoint from $a(\sigma)$. Any vertex of $G_F$ in $\Int B$ must belong 
to a component, $a$, of $a(\tau)$ for some $r'$-\ESC, $\tau$, of $\Lambda$. 
Since $a$ 
intersects $a(\sigma)$ at most once, it must lie in $B$ and
(Lemma~\ref{lem:LMBess}) be isotopic 
to the components $\partial B$ of $a(\sigma)$. But this would contradict
the Addendum to Lemma~\ref{lem:PLMB} and the maximality of $r$. Thus
$\Int B$ must be disjoint from $K$. That is,
there are components
$a_1, a_2, a_3$ of $a(\sigma)$ such that $a_1,a_2$ and $a_2, a_3$ cobound
annuli in $F^*$ whose interiors are disjoint from $K$. This contradicts
the Addendum to Lemma~\ref{lem:3PCC}.
\end{proof}

\begin{lemma}\label{lem:n3escscc}
$G_Q$ contains no 3-\ESC.
\end{lemma} 

\begin{proof}
Suppose $\sigma$ is a 3-\ESC. As argued in
the preceding lemma, the 
Addendum to Lemma~\ref{lem:PLMB} and 
the maximality of $\sigma$ show that if $B$ is an annulus of $F^*$ cobounded by
components of $a(\sigma)$ such that $\Int B$ is disjoint from $a(\sigma)$,
then $\Int B$ must be disjoint from $K$.
Then the Addendum to
Lemma~\ref{lem:3PCC} shows that at most two components of $a(\sigma)$
are isotopic on $F^*$. Since $|a(\sigma)|=4$, $F^*$ must be two tori
with exactly two components of $a(\sigma)$ on each. But the argument
above then says that every vertex of $G_F$ must lie on $a(\sigma)$  ---
contradicting that $t \geq 10$.
\end{proof}

\begin{lemma}\label{lem:t8n2esc}
There is no 2-\ESC.
\end{lemma}

\begin{proof}
Let $\sigma$ be a 2-\ESC. The argument of Lemma~\ref{lem:n3escscc}
coupled with its conclusion that there is no 3-\ESC, implies that
$F^*$ must consist of two tori: $T_1$ containing two components of
$a(\sigma)$ and $T_2$ containing one. Again, the argument of 
Lemma~\ref{lem:n3escscc}, shows that the only vertices of $G_F$
on $T_1$ are those lying on the two components of $a(\sigma)$.

Assume $\sigma$ is given by Figure~\ref{fig:tauoflength6}. By 
Corollary~\ref{cor:bigonsforall} there is a bigon of $\Lambda_{8}$.
This can be taken to be a proper $r$-\ESC, $\tau$ (where $r=0$ means an \SC). 
Then $r \leq 2$. 
By the Addendum to 
Lemma~\ref{lem:PLMB}, each component of $a(\tau)$ must
intersect a component of $a(\sigma)$. Enumerating the possibilities
for the labels of $\tau$ consistent with these conditions we have
\begin{itemize}
\item[a.] $\{5,6,7,8\}$
\item[b.] $\{8,9,10,1,2,3\}$
\item[c.] $\{3,4,5,6,7,8\}$
\end{itemize}
But (a) is not possible as vertices $7,8$ of $G_F$ must lie on $T_2$, but
the corresponding components of $a(\tau)$ intersect two different components
of $a(\sigma)$. The same argument with vertices $8,9$ rules out (b). So we
assume the labels of $\tau$ are given by (c). Since vertices $7,8$ of $G_F$
lie in $T_2$, then vertices 3,4 must also lie in $T_2$ while vertices
$1,2,5,6$ must be those in $T_1$.

Now take a bigon of $\Lambda_{10}$ giving a proper $n$-\ESC, $\nu$. 
By the Addendum
to Lemma~\ref{lem:PLMB}, as argued above, each component of $a(\nu)$
must intersect both $a(\sigma)$ and $a(\tau)$. Furthermore, 
$n \leq 2$. These conditions guarantee that the label set
for $\nu$ is \{10,1,2,3,4,5\}. But this contradicts that vertex 2 lies
on $T_1$ and vertex 3 on $T_2$.
\end{proof}

\begin{lemma}
There is no 1-\ESC.
\end{lemma}

\begin{proof}
Assume there is an \ESC, $\sigma$, on the labels $\{1,2,3,4\}$.
By Corollary~\ref{cor:bigonsforall}, there is a proper $r$-\ESC, $\tau$,
coming from a bigon of $\Lambda_5$, and a proper $n$-\ESC, 
$\nu$, coming
from a bigon of $\Lambda_9$. Furthermore, $0 \leq r,n \leq 1$.
A simple enumeration shows the possible label pairs of the core \SC\ for
$\tau$ are: $\{3\,4, 4\,5, 5\,6, 6\,7\}$. The possible labels for the core \SC\ of
$\nu$ are: $\{7\,8, 8\,9, 9\,10, 10\,*\}$ (where the label $*$ means either 1 or 11).
Three \SCs\ on disjoint
label pairs would allow us to use $F^*$ to form a Klein bottle in $M$. 
Thus the label pairs of the core \SCs\ of two of $\{\sigma, \tau, \nu\}$
must intersect. The possibilities are
\begin{itemize}
\item[a.] $\sigma,\tau$ where $\tau$ has label set $\{2,3,4,5\}$
\item[b.] $\tau,\nu$ where $\tau$ has label set $\{5,6,7,8\}$ and
$\nu$ has label set $\{6,7,8,9\}$.
\end{itemize}
Both lead to the same contradiction. We consider (b). The edges of $\tau,
\nu$ force the vertices $5,6,7,8,9$ to lie on the same torus component
of $F^*$. There is a component of $a(\tau)$ disjoint from a component
of $a(\nu)$. Hence these components are isotopic on $F^*$. But this 
contradicts the Addendum to Lemma~\ref{lem:PLMB}.
\end{proof}

The preceding lemmas along with Corollary~\ref{cor:bigonsforall},
implies that every label of $\Lambda$ belongs to an \SC. But then
Lemma~\ref{lem:3disjtSC} along with the assumption that $t \geq 10$
implies that $M$ contains a Dyck's surface.
This contradiction concludes the proof of Proposition~\ref{prop:tleq8scc}.
\end{proof}

\section{$t < 8$}
By Theorem~\ref{thm:tleq8}, $t \leq 8$.  In this section we prove
\begin{thm}\label{thm:tleq6}
Either $M$ contains a Dyck's surface or $t<8$.
\end{thm}

\begin{proof}

This is Proposition~\ref{prop:tleq6nscc} of section~\ref{sec:tleq6nscc} 
when we are in 
\situationnscc,
and Proposition~\ref{prop:tleq6scc} of section~\ref{sec:tleq6scc} 
in \situationscc.
\end{proof}

\subsection{ $t = 8$  and \situationnscc}\label{sec:tleq6nscc}

\begin{prop}\label{prop:tleq6nscc} 
In \situationnscc, either $M$ contains a Dyck's surface or $t < 8$.
\end{prop}

\begin{proof} Assume $M$ does not contain a Dyck's surface and $t=8$.  
By Lemma~\ref{lem:esc}, there are three cases to consider:
\begin{itemize}
\item[{\bf A.}] There is a $2$-\ESC\ in $\Lambda$.
\item[{\bf B.}] There is a $1$-\ESC\ in $\Lambda$ but no $2$-\ESC.
\item[{\bf C.}] There is no \ESC\ in $\Lambda$. 
\end{itemize}

The proof in Case B relies upon Corollary~\ref{cor:make3bridge} and 
Proposition~\ref{prop:nodisjtESC} which are proven in subsections 
following the present proof.

\begin{caseA}  
There is a $2$-\ESC\ in $\Lambda$.
\end{caseA}

As in Case A of Theorem~\ref{thm:tleq8}, assume $G_Q$ contains $\tau$, the 
$2$-\ESC\ depicted in Figure~\ref{fig:tauoflength6}.  It gives rise to a long \mobius band $A_\tau=A_1 \cup A_2 \cup A_3$ in which $A_1$ is a Black \mobius band, $A_2$ is a White annulus, and $A_3$ is a Black annulus.  By Lemma~\ref{lem:esc} the components of $\bdry A_2$ are not isotopic on $\hatF$ whereas the components of $\bdry A_3$ are.  %Observe that $A_3$ is separating in the Black handlebody $H_B$.

\begin{figure}
\centering
\input{lambda7escs.pstex_t}
\caption{}
\label{fig:lambda7escs}
\end{figure}

Figure~\ref{fig:lambda7escs} lists all possible bigons of $\Lambda_7$ that are at most $2$-\ESCs\ (i.e.\ containing at most $6$ edges).  We proceed to rule out all of these bigons, thereby contradicting Corollary~\ref{cor:bigonsforall}.
\begin{claim}\label{claim:no7abd}
 7(a), 7(b), and 7(d) are impossible.
 \end{claim}
 \begin{proof}
 Each of these bigons contain an \SC\ whose associated \mobius band is disjoint from $A_2$.  The boundary of such a \mobius band must not be isotopic to a component of $\bdry A_2$, else there would be an embedded Klein bottle in $M$.  This contradicts Lemma~\ref{lem:disjtmobiusannulus}.
 \end{proof}
 \begin{claim}\label{claim:no7cf}
 7(c) and 7(f) are impossible.
 \end{claim}
 \begin{proof}
Each 7(c) and 7(f) contain an \SC\ whose associated \mobius band $B$ intersects $A_3$ along a component of $A_3 \cap K$.  Because  $A_3$ is separating in the Black
side of $\hatF$, the intersection of $B$ with $A_3$ is not transverse.  Therefore the \mobius band $B$ may be isotoped in $H_B$ to be disjoint from $A_3$ and hence $A_2$.  Since $\bdry B$ cannot be isotopic on $\hatF$ to either component of $\bdry A_2$, together $B$ and $A_2$ form a contradiction to Lemma~\ref{lem:disjtmobiusannulus}.
\end{proof}
\begin{claim}\label{claim:no7e}
7(e) is impossible.
\end{claim}
\begin{proof}
\begin{figure}
\centering
\input{lambda8escs.pstex_t}
\caption{}
\label{fig:lambda8escs}
\end{figure}
Figure~\ref{fig:lambda8escs} lists all possible bigons of $\Lambda_8$ that are at most $2$-\ESCs.  Analogously to Claims~\ref{claim:no7abd} and \ref{claim:no7cf}, we may rule out all but 8(e).  Yet now 7(e) and 8(e) cannot coexist as the proof of Claim~\ref{claim:no7cf} applies analogously with 7(e) and 8(e) in lieu of 7(c) and $\tau$ respectively.
\end{proof}
This completes the proof in Case A.

\begin{caseB}
 There is a $1$-\ESC\ in $\Lambda$ but no $2$-\ESC.
\end{caseB}

Assume $G_Q$ contains $\tau$, the \ESC\ depicted in 
Figure~\ref{fig:tauoflength4}.

\begin{lemma}\label{lem:noabbuttinglength4}
There cannot be two $1$-\ESCs\ whose label sets intersect in one label.
\end{lemma}

\begin{proof}
Assume otherwise.  Then their \SCs\ have opposite colors.
Let $A_1$ and $A_2$ be the \mobius band and annulus arising from one \ESC; let $B_1$ and $B_2$ be the \mobius band and annulus arising from the other.  Since $A_1$ and $B_1$ are on opposite sides, so are $A_2$ and $B_2$.  

By Lemma~\ref{lem:disjtmobiusannulus} some pair of curves of $\bdry A_1 \cup \bdry B_2$ must be isotopic as must some pair of curves of $\bdry B_1 \cup \bdry A_2$.  Since we may not form any embedded Klein bottles, the two components of $\bdry B_2$ must be isotopic as must the two components of $\bdry A_2$.  Then by Lemma~\ref{lem:LMB} it follows that $A_2$ and $B_2$ are each parallel into $\hatF$.  (The \ESCs\ are of maximum length, so the argument of Lemma~\ref{lem:esc} shows that the parallelism between, say, $\partial A_2$ is disjoint from $K$.)

 By assumption, $\bdry A_2$ and $\bdry B_2$ intersect in one point, and thus this intersection is not transverse.  Hence, as with Lemma~\ref{lem:disjtoppescs}, the parallelisms of $A_2$ and $B_2$ into $\hatF$ give a thinning of $K$. This is a contradiction.
 \end{proof}

Let us now consider the possible \ESCs, \SCs\ coming from 
bigons of $\Lambda_8$ 
and $\Lambda_7$.  
These possibilities are shown in Figure~\ref{fig:lambda78len4}.  
\begin{figure}
\centering
\input{lambda78len4.pstex_t}
\caption{}
\label{fig:lambda78len4}
\end{figure}

\begin{claim}\label{claim:no7d8d}
Neither 7(d) nor 8(d) may occur.
\end{claim}
\begin{proof}
Since each of these shares one label with $\tau$, Lemma~\ref{lem:noabbuttinglength4} rules them out.
\end{proof}

\begin{claim}\label{claim:no7c8c}
Neither 7(c) nor 8(c) may occur.
\end{claim}
\begin{proof}
Since 7(c) and 8(c) have disjoint labels, at most one may occur by 
Proposition~\ref{prop:nodisjtESC}.
Assume 7(c) does occur.  Then either 8(a) or 8(b) must also occur (because 8(d) cannot by the preceding Claim).  But then there will be three disjoint \SCs\!,  contradicting Lemma~\ref{lem:3disjtSC}.  A similar argument shows 8(c) cannot occur.
\end{proof}

\begin{figure}
\centering
\input{lambda16escstis8.pstex_t}
\caption{}
\label{fig:lambda16}
\end{figure}

Figure~\ref{fig:lambda16} shows the possible \ESCs, \SCs\ coming from bigons of 
$\Lambda_1$ and $\Lambda_6$.  These will be of use in the next two claims.

\begin{claim}\label{claim:no7a8b}
Neither 7(a) nor 8(b) may occur.
\end{claim}
\begin{proof}
Assume 7(a) occurs. With 7(a) and the \SC\ in $\tau$, each of 1(a) and 1(b) form a triple of disjoint \SCs\!, contradicting Lemma~\ref{lem:3disjtSC}.  1(c) 
violates Proposition~\ref{prop:nodisjtESC}.  Therefore 1(d) must occur.  

With relabeling (subtracting $1$ from each label), we may now apply Corollary~\ref{cor:make3bridge} to show that there is another genus 2 Heegaard splitting of $M$ with respect to which $K$ is $3$-bridge.  (The \SC\ in 1(d) plays the role of $\sigma$, $\tau$ is again $\tau$, and 7(a) is the \SC\ disjoint from the 
labels $\{2,3,4\}$.)  This contradicts our minimality assumptions ($3$-bridge means $t=6$).

A similar argument rules out 8(b), using $\Lambda_6$ in place of $\Lambda_1$.
\end{proof}

\begin{claim}\label{clm:n78sc}
$\Lambda$ does not contain a $\arc{78}$-\SC.

\end{claim}

\begin{proof}
Assume there is a $\arc{78}$-\SC\; i.e.\ 7(b) and 8(a) occur.  
By Proposition~\ref{prop:nodisjtESC}, this \SC\ cannot be contained within an 
\ESC.  We must consider the bigons of $\Lambda_1$ and $\Lambda_6$ shown in Figure~\ref{fig:lambda16}.

Proposition~\ref{prop:nodisjtESC} rules out 1(c) and 6(d).  Corollary~\ref{cor:make3bridge} rules out 1(d) and 6(c) as in the proof of Claim~\ref{claim:no7a8b}.  Lemma~\ref{lem:3disjtSC} forbids each of 1(b) and 6(a) as they are \SCs\ each disjoint from the \SC\ in $\tau$ and the $\arc{78}$-\SC.  

Thus 1(a) and 6(b) must occur.  But then 1(a), 6(b), and the \SC\ in $\tau$ form 3 mutually disjoint \SCs\ in violation of Lemma~\ref{lem:3disjtSC}.  
\end{proof}

The above claims imply that all bigons of $\Lambda_8$ are forbidden, 
contradicting Corollary~\ref{cor:bigonsforall}. This completes the proof of
Theorem~\ref{thm:tleq6} in Case B.

\begin{caseC}
 There is no $1$-\ESC\ in $\Lambda$. 
\end{caseC}

By Corollary~\ref{cor:bigonsforall}, every label belongs to an \SC.  Lemma~\ref{lem:3disjtSC} then forces $t \leq 6$ contrary to the assumption that $t = 8$.  This completes the proof in Case C. 

Given the the proofs of Corollary~\ref{cor:make3bridge} and
Proposition~\ref{prop:nodisjtESC} in subsequent subsections, 
the proof of Theorem~\ref{thm:tleq6} is now complete.
\end{proof}

\subsection{A proposition and a corollary for Claim~\ref{claim:no7a8b}.}

For Claim~\ref{claim:no7a8b} and Claim~\ref{clm:n78sc} above we use Corollary~\ref{cor:make3bridge} which is a consequence of Proposition~\ref{prop:intersectinglength4and2}.  In this subsection we prove the proposition and its corollary.

\begin{prop}\label{prop:intersectinglength4and2}
Assume we are in \situationnscc\ and there exists an \ESC\ $\tau$ and \SC\ 
$\sigma$ as in Figure~\ref{fig:length4and2}.  Let $A_{23}$ be the White \mobius band arising from the \SC\ in $\tau$, and let $A_{12}$ be the Black \mobius band arising from $\sigma$.  If $\bdry A_{23}$ intersects $\bdry A_{12}$ 
transversely on 
$\hatF$ then there is a new Heegaard splitting of $M$ in which $K$ is $3$-bridge.
\end{prop}

\begin{figure}
\centering
\input{length4and2.pstex_t}
\caption{}
\label{fig:length4and2}
\end{figure}

\begin{proof}
Let $A_{12,34}$ be the Black annulus arising from $\tau$ that extends $A_{23}$. 
We assume $\bdry A_{23}$ intersects $\bdry A_{12}$ transversely.  Let $E$ be a neighborhood in $\hatF$ of the union of the vertices $\{1,2,3,4\}$ and the edges of $\sigma, \tau$.  The labeling of these edges on $\hatF$ must be as in Figure~\ref{fig:mobiusintersectannulus}.  Let $A$ be $\hatF - E$. 

\begin{figure}
\centering
\input{mobiusintersectannulus.pstex_t}
\caption{}
\label{fig:mobiusintersectannulus}
\end{figure}

\begin{claim}
$A$ is an annulus in $\hatF$. 
\end{claim}
\begin{proof}

Let $C_1$ and $C_2$ be the two curves on $\hatF$ as shown that form 
$\bdry E = \bdry A$. Since $\chi(E)=-2=\chi(\hatF)$, we have $\chi(A)=0$.
If either $C_1$ or $C_2$ were to bound a disk  in the complement of $E$, 
then such a disk could be joined to itself across an edge of $f_4$ to form an annulus in $\hatF$.  Then this resulting annulus together with the annulus $A_{12,34}$ would form an embedded Klein bottle.  This cannot occur.  Hence $A$ is an annulus.
 \end{proof}

%%%%
Let $\calN = \nbhd(A_{12} \cup A_{12,34}) \subset H_B$.  Then $\bdry H_B - \calN = A$  and set $H_B - \calN = \calT$.

\begin{claim}%[Claim 1]
\label{claim:thinningannulus} %Claim1
$\calT$ is a solid torus and the annulus $A$ is longitudinal on $\bdry \calT$.
\end{claim}

\begin{proof}
Consider $\calD$, a disjoint collection of bridge disks for $K \cap H_B$. By considering
the intersections of these disks with the faces $f_1,f_3,f_4$, surgering along outermost arcs of intersection
in $\calD$, and banding along $f_1,f_3,f_4$, we can take the bridge disks $D_{12}, D_{34}$ for
$\arc{12},\!\arc{34}$ to have interiors disjoint from $A_{12} \cup A_{12,34}$.
Hence $H_B - \calN = \calT$ is a solid torus in which $A$ is longitudinal.
\end{proof}

By Claim~\ref{claim:thinningannulus}, $\calN$ is isotopic to $H_B$ through $\calT$.

\begin{claim} \label{claim:2}
The arcs $\arc{12}$, $\arc{34}$, $\arc{56}$, and $\arc{78}$ in $H_B$ have mutually disjoint bridge disks that lie in $\calT$ and provide an isotopy of these arcs onto $A$. 
 \end{claim}

\begin{proof} 
The above proof of Claim~\ref{claim:thinningannulus} shows that there are bridge disks $D_{12}$~and $D_{34}$ for Ê$\arc{12}$ and $\arc{34}$ respectively, disjoint from the other arcs of $K \cap H_B$, which lie in $\calT$ and provide an isotopy of these arcs
Êonto $A$. ÊIndeed these are meridional disks of $\calT$. ÊSince bridge disks $D_{56}$ and $D_{78}$ for the other two arcs $\arc{56}$ and $\arc{78}$ are disjoint from $D_{12}$ and $D_{34}$, the arcs of $(D_{56} \cup D_{78}) \cap (\bdry \calT - \Int A)$ may be either isotoped along $\bdry \calT - \Int A - (D_{12} \cup D_{34})$ onto $A$ or banded to $D_{12}$ or $D_{34}$ to form bridge disks for $\arc{56}$ and $\arc{78}$ as desired.
\end{proof}

Attach a neighborhood of the White \mobius band $A_{23}$ in $H_W$ to $H_B = \calN \cup \calT$.  Write $\calN' = H_B \cup \nbhd(A_{23}) = \nbhd(A_{12} \cup A_{12,34} \cup A_{23}) \cup \calT$ and $\hatF' = \bdry \calN'$.  
\begin{claim}
$M = \calN' \cup_{\hatF'} (M \cut \calN')$ is a genus $2$ Heegaard splitting.
\end{claim}

\begin{proof}
We must show that $\calN'$ and $M \cut \calN'$ are each genus $2$ handlebodies.

To see that $\calN'$ is a genus $2$ handlebody, we show that the curve $\bdry A_{23}$ on $\hatF$ is primitive in $H_B$.  It suffices to show that $\bdry A_{23}$ is primitive in $\calN$ since $A_{23}$ is disjoint from $\calT$ and $\calT$ provides an isotopy of $\calN$ to all of $H_B$.

In $\calN$ a cocore of the annulus $A_{12,34}$ (such as the arc $\arc{34}$) thickens to a meridian disk of $\calN$ and thus extends through $\calT$ to a meridian disk $D$ of $H_B$.  Since $\bdry A_{23}$ is a component of $\bdry A_{12,34}$, it intersects $D$ once.  Hence $\bdry A_{23}$ is primitive in $H_B$ and $\calN'$ is a genus $2$ handlebody.

To see that $M \cut \calN'$ is a genus $2$ handlebody, observe that it is the complement of a neighborhood of a \mobius band in $H_W$.
\end{proof}

To complete the proof of Proposition~\ref{prop:intersectinglength4and2} we must show that $K$ is $3$-bridge with respect to the Heegaard splitting $M = \calN' \cup_{\hatF'} (M \cut \calN')$.

Claim~\ref{claim:2} shows that the arcs $\arc{56}$ and $\arc{78}$ are bridge in $\calN'$.  
As $\arc{1234}= \arc{12} \cup \arc{23} \cup \arc{34}$ is a cocore of the properly embedded \mobius band $A_{12,34} \cup A_{23}$ in the handlebody $\calN'$, 
it is bridge as well.

By Lemma~\ref{lem:bridgedisksdisjoitfrommobius}, the arcs $\arc{45}$, $\arc{67}$, and $\arc{81}$ have mutually disjoint bridge disks in $H_W$ that are also disjoint from $A_{23}$.  Therefore they remain bridge in $H_W -\nbhd(A_{23})$ which is isotopic to $M \cut \calN'$.    

Hence $K$ is $3$-bridge with respect to this new Heegaard splitting.
\end{proof}% of side thm

\begin{cor}\label{cor:make3bridge}
Assume we are in \situationnscc. If there is an \ESC\ $\tau$ and an \SC\ $\sigma$ as in Figure~\ref{fig:length4and2} as well as an \SC\ disjoint from the labels $\{1,2,3\}$ then $K$ is $3$-bridge with respect to some genus $2$ Heegaard splitting of $M$.
\end{cor}
\begin{proof}
Given such a set-up, the boundaries of the \mobius bands arising from the \SCs\ in $\tau$ and $\sigma$ cannot be isotoped to be disjoint.  Otherwise there would be three disjoint \mobius bands contrary to Lemma~\ref{lem:3disjtSC}.  Proposition~\ref{prop:intersectinglength4and2} now applies.
\end{proof}

\subsection{ $t=8$  and \situationscc}\label{sec:tleq6scc}

\begin{prop}\label{prop:tleq6scc} 
In \situationscc, either $M$ contains a Dyck's surface or $t < 8$.
\end{prop}

\begin{proof}
Assume we are in \situationscc. Then there is a meridian disk $D$ of $H_W$
or $H_B$ disjoint from $K$ and $Q$. Let $F^*$ be $\hatF$ surgered along $D$.
Then $F^*$ is one or two tori. For contradiction, assume $M$ does not 
contain a Dyck's surface, and $t=8$. 

\begin{lemma}\label{lem:t6n3esc}
$G_Q$ contains no $3$-\ESC.
\end{lemma}

\begin{proof}
Otherwise, there is a $3$-\ESC, $\sigma$. Note that this is a maximal
proper \ESC\ when $t=8$. Thus
the argument of Lemma~\ref{lem:n3escscc} shows that $F^*$ must be two tori,
$T_1,T_2$, 
with exactly two components of $a(\sigma)$ on each.
WLOG assume the core \SC\ of $\sigma$ is a $\arc{45}$-\SC. 
Let $A=A_1 \cup \dots \cup A_4$ be the long \mobius band associated to
$\sigma$ and $a_i \in a(\sigma)$ be $\bdry A_i - \bdry A_{i-1}$. By the
Addendum to Lemma~\ref{lem:LMB} isotopic components of $a(\sigma)$ on 
$F^*$ must be consecutive in $A$. Thus we may take $a_1,a_2$ in $T_1$
and $a_3,a_4$ in $T_2$. That is, vertices \{3,4,5,6\} of $G_Q$ lie on $T_1$
and vertices \{1,2,7,8\} on $T_2$. Recall that $D$ is the meridian disk
along which $\hatF$ is surgered to get $T_1 \cup T_2$. Because $D$ is disjoint
from $K$, vertex $3$ lies on $T_1$, and vertex $2$ lies on $T_2$, $D$ must
lie on the opposite side of $\hatF$ to the $\arc{23}$-arc of $K$. Taking
$\arc{23}$ to lie in $H_W$, $D$ must lie in $H_B$.
Let $\calN = H_B - \nbhd(D) = \calN_1 \cup \calN_2$ where $\calN_1,\calN_2$ 
are solid tori
with $\bdry \calN_j = T_j$. Since the components of $a(\sigma)$ cannot bound
disks in either handlebody, the $A_i$ of the long \mobius band meet $F^*$
in their interiors in simple closed curves which are trivial on $F^*$. We
surger along these curves to make the $A_i$ properly embedded in either
$\calN$ or $M - \Int \calN$. By the separation of $a(\sigma)$ in $T_1,T_2$, 
$A_1, A_3$
lie in the exterior of $\calN$, $A_2$ lies in $\calN_1$, and $A_4$ in $\calN_2$.

\begin{claim}\label{clm:t6nsfs}
$A_{2i}$ is a longitudinal annulus in $\calN_i$ for $i=1,2$.
\end{claim}

\begin{proof}
Assume not. Then $U = \nbhd(\calN \cup A_1 \cup A_3)$ is a Seifert fiber
space over the disk with two or three exceptional fibers (the core of
$A_1$ being one). Furthermore $K \cap U$ lies as a cocore in the \mobius band
$A'=A_1 \cup A_2 \cup A_3 \cup A_4$ properly embedded in $U$, where $\bdry A'$ is
a Seifert fiber of $U$. 

In fact $U$ must be Seifert fibered with exactly two exceptional fibers.
Otherwise, $V=U - \nbhd(A')$ would be a Seifert fiber space over the
disk with two exceptional fibers that is disjoint from $K$. As $V$ does
not lie in a 3-ball by Lemma~\ref{lem:AEGor} and the exterior of $K$
is irreducible and atoroidal, then $V$ would be isotopic to the exterior of $K$
--- contradicting that $K$ is hyperbolic.

Now Lemma~\ref{lem:sSFS1bridge} applies to give a genus $2$ Heegaard splitting of $M$ in which $K$ is $0$-bridge, a contradiction.
\end{proof}

Let $U = \nbhd(\calN \cup A_1 \cup A_3)$. Then $K \cap U$ lies as a cocore in 
the \mobius band
$A'=A_1 \cup A_2 \cup A_3 \cup A_4$ properly embedded in $U$.
The preceding Lemma means that $U$ is a solid torus, and hence that $K \cap
U$ is isotopic onto $\bdry U$ fixing its endpoints. 
Let $W$ be the genus 2 handlebody  
$U \cup \nbhd(K)$. Then $K$ is isotopic onto $\bdry W$.
 
\begin{claim}\label{paralleltosigma} An edge of a Black bigon of $\Lambda$
is parallel in $T_1$ or $T_2$ to an edge of $\sigma$.
\end{claim}
\begin{proof}
Let $\tau$ be a black bigon with a $\arc{12}$-corner. The argument for the
other black corners is similar. 

Assume $\tau$ is a $\SC$. Let $A'$ be the almost properly embedded \mobius
band corresponding to $\tau$. After surgery along trivial disks in $T_2$, 
we may take $A'$ to be properly embedded in $\calN_2$. Consider the annulus
$A_4$ in $\calN_2$ and the edges of $\sigma$ in $T_2$ lying in $\partial A_4$.
Using the fact that $\calN_2$ contains
no Klein bottle, a close look at the labeling of the edges of $\sigma$
and $\tau$ on $T_2$ shows that $\partial A'$ can be perturbed to be disjoint
from $\partial A_4$. But this contradicts that $\partial A_4$ is longitudinal
in $\calN_2$.

So $\tau$ is not a $\SC$. As the edges of $G_F$ lie in either $T_1$ or $T_2$,
the edges of $\tau$ must be a $\edge{27}-edge$ and an $\edge{81}$-edge. 
Looking at the edges of $\sigma$ in $T_2$, we see the edges of $\tau$ must
be parallel to these.
\end{proof}

\begin{claim}\label{no81bigon}
There is no bigon in $\Lambda$ with an $\arc{81}$-corner.
\end{claim}

\begin{proof}
Let $\tau$ be such a bigon. An edge of $\tau$ must lie in $T_2$,
implying that $\tau$ is an $\arc{81}-\SC$. But then the 
corresponding almost properly embedded \mobius band could be surgered
to produce a properly embedded \mobius band in the complement of $\calN$
whose boundary was parallel to $\partial A_4$ on $T_2$. Along with $A$
we would see a Klein bottle in $M$.
\end{proof}

By Lemmas~\ref{lem:t8trulyspecial} and \ref{lem:N3specialtypes}, 
there is a special vertex $v$ in $\Lambda$
of type $[8\Delta - 5]$ (Claim~\ref{no81bigon} implies there can be no more
than eight consecutive bigons in $\Lambda$.) This means that all but five
corners at $v$ belong to bigons of $\Lambda$. By Claim~\ref{no81bigon}, 
no $\arc{81}$-corner belongs to a bigon of $\Lambda$. So there must be
a black corner, say $\arc{12}$, such that every $\arc{12}$-corner at $v$
belongs to a bigon of $\Lambda$. By Claim~\ref{paralleltosigma}, the edges of
these bigons incident to $v$ at label $2$ must be $\edge{27}$-edges parallel
in $T_2$ to the edges of $\sigma$. In particular, there are two 
parallel edges in $T_2$ both incident to vertex $2$ in $T_2$ with label
$v$. This contradicts Lemma~\ref{lem:parallelwithsamelabelAPE}.
\end{proof}

To finish the proof of Proposition~\ref{prop:tleq6scc},
we now follow the outline of the proof of Proposition~\ref{prop:tleq6nscc}, 
indicating
the necessary modifications.

By Lemma~\ref{lem:t6n3esc}, there are three cases to consider:
\begin{itemize}
\item[{\bf A.}] There is a $2$-\ESC\ in $\Lambda$.
\item[{\bf B.}] There is a $1$-\ESC\ in $\Lambda$ but no $2$-\ESC.
\item[{\bf C.}] There is no \ESC\ in $\Lambda$. 
\end{itemize}

\begin{caseA}  
There is a $2$-\ESC\ in $\Lambda$.
\end{caseA}

Assume $G_Q$ contains $\tau$, 
the $2$-\ESC\ depicted in Figure~\ref{fig:tauoflength6}.  It gives rise to 
a long \mobius band $A_\tau=A_1 \cup A_2 \cup A_3$ in which $A_1$ is, say, a 
Black \mobius band, $A_2$ is a White annulus, and $A_3$ is a Black annulus
(each almost properly embedded in $H_W$ or $H_B$). As argued in 
Lemma~\ref{lem:t8n2esc}, $F^*$ consists of two tori $T_1$ containing
two components of $a(\tau)$ and $T_2$ containing one. Furthermore, the
only vertices of $G_F$ on $T_1$ must be those lying on the two components
of $a(\tau)$ --- the other four are on $T_2$. Finally, by the Addendum 
of Lemma~\ref{lem:LMB}, components
of $a(\tau)$ that are isotopic on $F^*$ must cobound some $A_i$. Thus
the vertices of $G_F$ on $T_1$ are either 
\begin{itemize}
\item[(i)] \{1,2,5,6\}
\item[(ii)] \{2,3,4,5\}
\end{itemize}

Figure~\ref{fig:lambda7escs} lists all possible bigons of $\Lambda_7$ that 
are at most $2$-\ESCs\ (i.e.\ containing at most $6$ edges).  We proceed 
to rule out all of these bigons in subcases (i) and (ii), thereby 
contradicting 
Corollary~\ref{cor:bigonsforall}.

First, assume (i). Then 7(a),(d),(e) are impossible by the separation of
vertices of $G_F$. 7(b) is impossible as it can be used with the $\arc{34}$-\SC\ of $\sigma$ to create a Klein bottle in $M$. So let $f$ be the
face bounded by the core \SC\ of either 7(c) or 7(f).
$\Int A_3 , \Int f$ intersect $T_1$ in trivial curves (since $M$ contains
no projective
planes). We surger away these intersections. Let $B$ be an annulus on $T_1$
cobounded by the components of $a(\tau)$ and containing an edge of $f$.
Since $B \cup A_3$ is separating, $f$ must lie on one side. But this implies
that the \mobius band, $A_f$ corresponding to $f$ can be pushed off of $A_3$ so
that $\bdry A_f$ is parallel to $\partial A_3$ on $T_1$. Then the long \mobius band
$A_{\tau}$ can be combined with $A_f$ to construct a Klein bottle in $M$.
This rules out all possibilities in subcase (i).

So assume (ii). 7(c),(d),(e),(f) are ruled out by the separation of vertices.
7(b) is impossible as then we can combine its face with the long \mobius
band $A_{\tau}$ to see a Klein bottle in $M$. Thus we assume $\Lambda$ contains
the $\arc{67}$-\SC\ of 7(a). By Corollary~\ref{cor:bigonsforall} and
Lemma~\ref{lem:t6n3esc}, there is a bigon face of $\Lambda_8$ giving
rise to an $r$-\ESC\ with $r \leq 2$. The possibilities are listed in
Figure~\ref{fig:lambda8escs}. But 8(c),(d),(e),(f) are ruled out by the
separation of vertices. 8(b) is impossible, else it and 7(a) combine along
$T_2$ to make a Klein bottle in $M$.
Finally,
8(a) can be combined with $A_{\tau}$ to give a Klein bottle in $M$. This rules
out possibility 7(a), hence (ii).

\begin{caseB}  
There is a $1$-\ESC\ in $\Lambda$, but no $2$-\ESC.
\end{caseB}

Assume $G_Q$ contains $\tau$, the \ESC\ depicted in 
Figure~\ref{fig:tauoflength4}. We follow the sequence of lemmas for
Case B in \situationnscc, modifying their proofs as necessary.
Note that Proposition~\ref{prop:nodisjtESC} is proven in the next
section under both \situationnscc\ and \situationscc.

\begin{lemma}\label{lem:noabbuttinglength4scc}
There cannot be two \ESCs\ whose label sets intersect in one label.
\end{lemma}

\begin{proof}
Let $\sigma,\nu$ be such \ESCs.
As their core \SCs\ are on disjoint label sets, $F^*$ must consist
of two tori, each containing one of these \SCs. Then one of these
tori must contain all of $a(\sigma)$, say, and one component of
$a(\nu)$. Since this component of $a(\nu)$ intersects $a(\sigma)$
once, they must all be isotopic on $F^*$. But this contradicts 
the Addendum to Lemma~\ref{lem:PLMB}.
\end{proof} 

Let us now consider the possible bigons of $\Lambda_8$ and $\Lambda_7$.  
These possibilities are shown in Figure~\ref{fig:lambda78len4}.

\begin{claim}\label{claim:no7d8dscc}
Neither 7(d) nor 8(d) may occur.
\end{claim}
\begin{proof}
Since each of these shares one label with $\tau$, 
Lemma~\ref{lem:noabbuttinglength4scc} rules them out.
\end{proof}

\begin{claim}\label{claim:no7c8cscc}
Neither 7(c) nor 8(c) may occur.
\end{claim}
\begin{proof}
Since 7(c) and 8(c) have disjoint labels, at most one may occur by 
Proposition~\ref{prop:nodisjtESC}.
Assume 7(c) does occur.  Then either 8(a) or 8(b) must also occur 
(8(d) cannot). But then there will be three disjoint \SCs,  
contradicting Lemma~\ref{lem:3disjtSC}.  
A similar argument shows 8(c) cannot occur.
\end{proof}

Figure~\ref{fig:lambda16} shows the possible bigons of 
$\Lambda_1$ and $\Lambda_6$.  These will be of use in the next two claims.

\begin{claim}\label{claim:no7a8bscc}
Neither 7(a) nor 8(b) may occur.
\end{claim}
\begin{proof}
Assume 7(a) occurs. With 7(a) and the \SC\ in $\tau$ each of 1(a) and 
1(b) form a triple of disjoint \SCs, contradicting Lemma~\ref{lem:3disjtSC}.
1(c) 
violates Proposition~\ref{prop:nodisjtESC}.  

Therefore 1(d) must occur. Call this $1$-\ESC, $\nu$.  
Because of 7(a), $F^*$ must consist of two tori. By 
the Addendum to Lemma~\ref{lem:PLMB}, one of these, $T_1$, contains
$a(\tau)$, and the other, $T_2$, contains the edges of 7(a). But
then, $a(\nu)$ must also lie in $T_1$. But then the component of 
$a(\nu)$ containing vertex 4 of $G_F$ must be isotopic to the components
of $a(\tau)$, contradicting the Addendum to Lemma~\ref{lem:PLMB}.
This rules out 1(d), and hence 7(a).

A similar argument rules out 8(b), using $\Lambda_6$ in place of $\Lambda_1$.
\end{proof}

\begin{claim}
There cannot be a $\arc{78}$-\SC.
\end{claim}

\begin{proof}
Assume there is a $\arc{78}$-\SC; i.e.\ 7(b) and 8(a) occur.  
We must consider the bigons of $\Lambda_1$ and $\Lambda_6$ 
shown in Figure~\ref{fig:lambda16}.

Proposition~\ref{prop:nodisjtESC} rules out 1(c) and 6(d).  
The argument of Claim~\ref{claim:no7a8bscc} rules out 1(d) and 6(c).
Lemma~\ref{lem:3disjtSC} forbids each of 1(b) and 6(a) as they are 
\SCs\ each disjoint from the \SC\ in $\tau$ and the $\arc{78}$-\SC.  

Thus 1(a) and 6(b) must occur by Corollary~\ref{cor:bigonsforall}.  
But then 1(a), 6(b), and the \SC\ in $\tau$ form 3 mutually disjoint 
\SCs\ in violation of Lemma~\ref{lem:3disjtSC}.  
Thus there cannot be a $\arc{78}$-\SC.
\end{proof}

The above claims imply that all bigons of $\Lambda_8$ are forbidden.  This 
contradicts Corollary~\ref{cor:bigonsforall}.

\begin{caseC}
 There is no \ESC\ in $\Lambda$. 
\end{caseC}

In this case every label belongs to an \SC.  Lemma~\ref{lem:3disjtSC} then forces $t \leq 6$ contrary to the assumed $t = 8$.  This completes the proof in Case C.

Given the following subsection, the proof of Proposition~\ref{prop:tleq6scc} 
is now complete.
\end{proof}

\subsection{A proposition for the preceding subsections.}

This subsection is devoted to the proof, in both \situationnscc\ and 
\situationscc, of Proposition~\ref{prop:nodisjtESC} 
stated below. This proposition was used in the preceding subsections.

\begin{prop}\label{prop:nodisjtESC}
Assume $M$ contains no Dyck's surface and $t=8$.  
If there is no $2$-\ESC\ in $\Lambda$ then there cannot be two 
disjoint \ESCs.
\end{prop}

Throughout this subsection we assume that there is no $2$-\ESC\ and that there exists two disjoint $1$-\ESCs\ $\tau$ and $\tau'$ on the corners $\arc{1234}$ and $\arc{5678}$ as shown in Figure~\ref{fig:twodisjtESC} (with Black and White
faces as pictured).  At the end of this section we prove Proposition~\ref{prop:nodisjtESC} by obtaining a contradiction.  To do so we must first develop several lemmas.

Let $A_{23}$ and $A_{12,34}$ be the White \mobius band and Black annulus arising from $\tau$.  Let $A_{67}$ and $A_{56,78}$  be the White \mobius band and Black annulus arising from $\tau'$.  By Lemma~\ref{lem:disjtmobiusannulus} the two components of $\bdry A_{12,34}$ are parallel on $\hatF$ as are the two components of $\bdry A_{56,78}$ (as $M$
contains no Klein bottle and no Dyck's surface).     
By Lemma~\ref{lem:LMB} the two annuli $A_{12,34}$ and $A_{56,78}$ are parallel into $\hatF$ (the \ESCs\ are maximal, hence $K$ must be disjoint from their parallelism).

\begin{figure}
\centering
\input{twodisjtESC.pstex_t}
\caption{}
\label{fig:twodisjtESC}
\end{figure}

\begin{claim}\label{claim:SD} 
In \conditionII, there is a separating, meridian disk $D$ of $H_B$ disjoint
from $K$ and $Q$ (i.e.\ disjoint from $Q$ in the exterior of $K$)
such that $\bdry D$ separates $\bdry A_{12,34}$ from $\bdry A_{56,78}$.
\end{claim}

\begin{proof} 
Otherwise there is a meridian disk, $D$, on one side of $\hatF$ which
is disjoint from $K$ and $Q$ (see section~\ref{sec:ape}). In particular, $D$ is disjoint from $A_{23} \cup A_{12,34}$ and $A_{67} \cup A_{56,78}$.
But then $\partial D$ must separate $\partial A_{12,34}$ and 
$\partial A_{56,78}$ (else compressing $\hatF$ along $D$ gives a 2-torus which allows one to find a Klein bottle in $M$).  The disk $D$ cannot be in $H_W$ since it is disjoint from the arc $\arc{45}$ of $K \cap H_W$.  Thus $D$ lies in $H_B$ as a separating disk. 
\end{proof}

\begin{lemma}
\label{lem:twoBbigons}
The only possible Black bigons of $\Lambda$ are $\arc{12},\!\arc{34}$-bigons and $\arc{56},\!\arc{78}$-bigons.
\end{lemma}

\begin{proof}
Assume there exists a Black \SC.  It gives rise to a Black \mobius $A'$ band that meets either $A_{12,34}$ or $A_{56,78}$ along an arc of $K$.  Since these two annuli are separating,  this intersection cannot be transverse.  Hence $A'$ may be slightly nudged to be disjoint from both of these annuli.  Thus there are three mutually disjoint \mobius bands in $M$ each properly embedded in $H_B$ or $H_W$.  This is contrary to Lemma~\ref{lem:3disjointmobiusbands}.

The lemma now follows immediately in \conditionII\ since the disk $D$ of Claim~\ref{claim:SD} separates vertices $\{1,2,3,4\}$ from $\{5,6,7,8\}$ on $\hatF$.  So we assume that we are in \conditionI.  In particular, the above \mobius bands and annuli are properly embedded on the White or Black sides of $\hatF$.

Assume there exists a $\arc{34},\!\arc{56}$-bigon $g$.  (A similar argument works for \mbox{$\arc{34},\!\arc{78}$-,} $\arc{12},\!\arc{56}$-, and $\arc{12},\!\arc{78}$-bigons.)  Let $D_{34}$ and $D_{56}$ be bridge disks for the arcs $\arc{34}$ and $\arc{56}$ contained in the solid tori cut off from the Black handlebody $H_B$ by the annuli $A_{12,34}$ and $A_{56,78}$.  Then, since $g$ is not contained in either of these solid tori, together $D_{34} \cup g \cup D_{56}$ forms a primitivizing disk for $\bdry A_{23}$ (i.e.\ a disk in $H_B$ intersecting $\bdry A_{23}$ once).   Note $D_{34} \cup g \cup D_{56}$ also forms a primitivizing disk for $\bdry A_{67}$.

Since $\bdry A_{23}$ is primitive with respect to the Black handlebody $H_B$, $H_B' = H_B \cup \nbhd(A_{23})$ is again a handlebody.  Now $K$ intersects $H_B'$ in the arcs $\arc{1234} =\arc{12} \cup \arc{23} \cup \arc{34}$, $\arc{56}$, and $\arc{78}$.  Note that the bridge disks for $\arc{56}$, $\arc{78}$ may be taken to be disjoint from $\nbhd(A_{23})$ hence the arcs $\arc{56}$ and $\arc{78}$
are bridge in $H_B'$. The arc $\arc{1234}$ lies in the properly embedded \mobius band $A_{12,34} \cup A_{23}$ in $H_B'$ and hence has a bridge disk in $H_B'$
disjoint from the bridge disks for $\arc{56}$ and $\arc{78}$.  Hence the arcs for $K \cap H_B'$ are bridge in $H_B'$.

Furthermore since $A_{23}$ is a \mobius band, $H_W' = H_W - \nbhd(A_{23})$ is also a handlebody.   By Lemma~\ref{lem:bridgedisksdisjoitfrommobius}, the arcs $K \cap H_W'$ are bridge in $H_W'$.  Therefore $H_B'$ and $H_W'$ form a genus $2$ Heegaard splitting of $M$ with respect to which $K$ is at most $3$-bridge.  This contradicts that $t=8$. 
\end{proof}

Recall that two edges in a graph $G$ are {\em in the same edge class}
or {\em are parallel} if they cobound a bigon in the graph (not necessarily
a bigon face of the graph).

\begin{lemma}\label{lem:1234parallelism}
For one of the pairs $(x,y)$ among the set of pairs $\{(2,3), (4,1), (6,7),$ $ (5,8)\}$, there are at most two edge-classes between the vertices $x,y$ in $G_F$.
\end{lemma}

\begin{proof} Otherwise each pair has three such edge-classes. This contradicts that $\hatF$ is genus two. 
\end{proof}

\begin{lemma}
\label{lem:partB}
For one of either $i=1$ or $i=5$ the following holds: 
at each vertex of $\Lambda$,
there are at most two Black bigons of $\Lambda$ incident to its 
$(i,i+1)$-corners and at most two Black bigons of $\Lambda$ incident to
its $(i+2,i+3)$-corners.
\end{lemma}

\begin{proof}
After Lemma~\ref{lem:1234parallelism}, assume that there are only two edge classes in $G_F$ connecting vertices $4,1$ of $G_F$.  Assume there is a vertex $v$ of $\Lambda$ that has three $\arc{12}$-corners belonging to Black bigons of $\Lambda$. By Lemma~\ref{lem:twoBbigons}, the bigons incident at these $\arc{12}$-corners are $\arc{12},\!\arc{34}$-bigons.  In particular, each such corner
has a $\edge{41}$-edge incident at label $1$.  But this means that one of the
edges classes connecting vertices $4,1$ on $G_F$ has two edges incident to
vertex $1$ with label $v$.  This contradicts Lemma~\ref{lem:parallelwithsamelabel}. We get a similar contradiction if there is a vertex of $\Lambda$ that
has three $\arc{34}$-corners belonging to the same Black bigons of $\Lambda$.
\end{proof}

\begin{lemma}\label{lem:23and67areSC}
A White bigon of $\Lambda$ with a $\arc{23}$- or $\arc{67}$-corner is a \SC.
\end{lemma}
\begin{proof}
We may assume we are in \conditionI\ as otherwise by Claim~\ref{claim:SD} there are no edges of $G_F$ connecting vertices $\{1,2,3,4\}$ with $\{5,6,7,8\}$.

Assume $g$ is a bigon of $\Lambda$ with a $\arc{23}$-corner that is not 
an \SC\ (the argument for the $\arc{67}$-corner is analogous).  
Hence its other corner is a $\arc{45}$-, $\arc{67}$-, or an $\arc{81}$-corner.  It cannot be a $\arc{67}$-corner since then, by the orderings of labels
on $\hatF$ around vertices $2$ and $3$, its $\edge{27}$-edge and $\edge{36}$-edge must 
lie on different components of $\hatF - \bdry A_{12,34}$ --- contradicting
that vertices $6,7$ of $G_F$ are connected by an edge (of $\tau'$).  Let us therefore assume that $g$ has a $\arc{45}$-corner;  the argument for a $\arc{81}$-corner is similar.

Let $r$ be an arc in the annulus $A_{12,34}$ sharing endpoints with $\arc{34}$ that projects through the $\bdry$-parallelism of $A_{12,34}$ onto the $\edge{34}$-edge of $g$. Note that
up to isotopy rel endpoints, $r$ is just $\arc{34}$ twisted along 
$\bdry A_{12,34}$. So we may take $r$ to have a single critical value (indeed
the same as for $\arc{23}$) under the
height function on $M$ for the thin presentation of $K$. Let $r'$ be an arc in the annulus $A_{12,34}$ disjoint from $r$ and sharing endpoints with $\arc{12}$. Similarly $r'$ can be taken to
have a single critical value with respect to the height function on $M$. Then $r' \cup \arc{23} \cup r$ and $\arc{12} \cup \arc{23} \cup \arc{34}$ are two properly embedded, non-$\bdry$-parallel arcs in the \mobius band $A_{23} \cup A_{12,34}$ with the same boundary.  By Lemma~\ref{claim:arcsinmobius}, these two arcs are isotopic rel-$\bdry$ within this long \mobius band. 
After this isotopy the bridge arcs 
$\arc{34},\!\arc{23},\!\arc{12}$ are replaced with bridge arcs $r, \arc{23}, r'$.
We may now isotop $\arc{23} \cup r \cup \arc{45}$, rel $\bdry$, onto
the $\edge{25}$-edge of $g$: 
isotop $r$ onto the $\edge{34}$-edge of $g$ using the $\bdry$-parallelism of $A_{12,34}$, then use $g$ to guide the remainder of the isotopy.
Perturbing the result
slightly into $H_W$ gives a smaller bridge presentation of $K$.
\end{proof}

\begin{lemma}\label{lem:partD}
 There is a $\edge{23}$-edge class in $G_F$ that contains 
an edge of every $\arc{23}$-\SC\ of $\Lambda$.  The analogous statement for  
$\arc{67}$-\SCs\ also holds.
\end{lemma}
\begin{proof}
 We prove this for $\arc{23}$-\SCs. The same proof works
for $\arc{67}$-\SCs. 

We assume first that we are in \conditionI.
Let $e_1, e_2$ be the edges of the $\arc{23}$-\SC\ in $\tau$, and let $f$ be the face that they bound. We assume for contradiction that there is a 
$\arc{23}$-\SC, $\sigma_1$, with no edge parallel to $e_1$ in $G_F$, and a
$\arc{23}$-\SC, $\sigma_2$, of $G_Q$ with no edge parallel to $e_2$. 
Let $g_1, g_2$ be the faces of $G_Q$ bounded by $\sigma_1, \sigma_2$. 
Because of the orderings of the labels around vertices of $G_F$, 
one edge of $g_i$
must lie in the annulus of $\hatF$ bounded by $\bdry A_{12,34}$. This
implies that one edge of $g_i$ is parallel to $e_j$ where $\{i,j\}=\{1,2\}$
(note that the interior of this annulus is disjoint from $K$).   By 
identifying $f$ with $g_1$ along their parallel edges (in the class of $e_2$)
we get a disk $D_1$ properly embedded in $H_W$ whose boundary is given
by the curve $e_1 \cup e^1$ where $e^1$ is the edge of $\sigma_1$ not parallel
to $e_2$.  Similarly identifying $f$ with $g_2$, we get a meridian disk $D_2$ of $H_W$ whose boundary is the curve $e_2 \cup e^2$, where $e^2$ is the other edge of $\sigma_2$. By looking at the ordering of these edges around the vertices $2, 3$ of $G_F$ we see that we can take $D_1,D_2, A_{23}$ (along with $A_{67}$) to be disjoint.  Both $D_1,D_2$ must be separating in $H_W$ (else there is a Klein bottle or projective plane in $M$). Then $\bdry D_1$, say, must separate $\bdry D_2$ from $\partial A_{23}$.
But again looking at the ordering of the edges of these \SCs\
around vertices $2,3$ of $G_F$ shows that this does not happen. 

Thus we assume we are in \conditionII.  Let $D$ be the separating disk of $H_B$
given by Claim~\ref{claim:SD}.  Then there can be at most three edge-classes of $\edge{23}$-edges in $G_F$ (surgering $\hatF$ along $D$ gives two 2-tori, 
one containing $\bdry A_{12,34}$ and the other $\bdry A_{56,78}$).  If the Lemma is false, then there must be exactly three such edge-classes  and there must be three $\arc{23}$-\SCs\ of length two representing each pair of these edge-classes.  One of these \SCs\ is that of $\tau$, and again let $e_1,e_2$ be its edges. 
The other two of these \SCs\ $\sigma_1,\sigma_2$,
each have an edge in the $\edge{23}$-edge class, $\epsilon$,
not represented by $e_1,e_2$.  Let $f_1,f_2$ be the faces bounded by
$\sigma_1,\sigma_2$ in $G_Q$.  Identifying $f_1,f_2$ along their edges
in class $\epsilon$ gives a disk $D'$ which is almost properly embedded
in $H_W$ whose boundary is the curve $e_1 \cup e_2$ in $\hatF$. Then
Lemma~\ref{lem:AEGor} implies that
$\bdry D'$ bounds a disk, $D''$ on one side of $\hatF$. But then $A_{23}$ 
and $D''$ can be used to construct a projective plane in $M$, a contradiction. 
\end{proof}

\begin{cor}\label{cor:gapsat23and67}
At every vertex of $\Lambda$ there are at most two bigons of $\Lambda$ at $\arc{23}$-corners and at most two bigons  at $\arc{67}$-corners.
\end{cor}

\begin{proof}
Lemma~\ref{lem:23and67areSC} says that any bigon of $\Lambda$ at a $\arc{23}$- or $\arc{67}$-corner must be an \SC.  By Lemma~\ref{lem:partD}, there is a $\edge{23}$-edge class containing an edge of any $\arc{23}$-\SC\ and a $\edge{67}$-edge class containing an edge of any $\arc{67}$-\SC.  Hence if there were three $\arc{23}$-\SCs\ or three $\arc{67}$-\SCs\ at a vertex, then two of the edges would be parallel on $G_F$ meeting a vertex at the same label.   Lemma~\ref{lem:parallelwithsamelabel} forbids this.
\end{proof}

\begin{proof}[Proof of Proposition~\ref{prop:nodisjtESC}]
By Lemmas~\ref{lem:t8trulyspecial} and \ref{lem:N3specialtypes}, $\Lambda$ has a vertex $v$ of type $[8\Delta-5]$ (there are no $2$-\ESCs).  Hence $v$ has at most $5$ gaps, i.e.\ corners to which bigons of $\Lambda$ are not incident.  By Lemma~\ref{lem:partB}, without loss of generality there exists a gap at a $\arc{12}$-corner and a $\arc{34}$-corner of $v$.    By Corollary~\ref{cor:gapsat23and67}, there must be a gap at a $\arc{23}$-corner and a $\arc{67}$-corner.  This accounts for $4$ of the gaps.  We now enumerate and rule out the possibilities for the remaining gap.

Since $\Delta \geq 3$, $v$ has at least three runs of the sequences of labels $4567812$. The $\arc{23}$-gap and $\arc{34}$-gap are not contained
in these sequences.  Each such sequence must have at least one gap, else
by Lemma~\ref{lem:23and67areSC}, there would be a $2$-\ESC\ (contradicting our assumptions). Thus $\Delta=3$ and the $\arc{67}$-gap, the $\arc{12}$-gap, and the fifth gap must be in different runs of the sequence.  But the run containing the $\arc{12}$-gap will have five consecutive bigons on $456781$, and, as above, 
Lemma~\ref{lem:23and67areSC} guarantees a $2$-\ESC.
\end{proof}

%\newpage

\section{$t < 6$.}%It cannot be that $t=6$.}

The goal of this section is the proof of 
\begin{thm} \label{thm:tnot6}
Either $M$ contains a Dyck's surface or $t < 6$.
\end{thm}

\begin{proof}
%Assume $t=6$.
For contradiction we assume $M$ contains no Dyck's surface and, 
by the earlier sections, that $t=6$. Given Corollary~\ref{cor:bigonsforall},
there are four ways in which $\Lambda_x$ contains a bigon for each $x$:
  
\begin{itemize}
\item[{\bf A.}]  There is a $2$-\ESC\ in $\Lambda$.
\item[{\bf B.}]  There are two $1$-\ESCs\ in $\Lambda$ whose label sets overlap in exactly two labels.
\item[{\bf C.}] There are three \SCs\ whose corresponding \mobius bands are disjoint.
%There are three mutually disjoint \SCs\ in $\Lambda$.
\item[{\bf D.}] There is a $1$-\ESC\ and a disjoint \SC\ in $\Lambda$.
\end{itemize}

Hence we proceed to address these four cases. The arguments will need to
account for each of the two possibilities: \conditionI\ and \conditionII.

\begin{caseA}
There is a $2$-\ESC\ in $\Lambda$.
\end{caseA}

Assume $\sigma$ is a $2$-\ESC\ in $\Lambda$ with corner $\arc{612345}$ 
so that it contains a $\arc{23}$-\SC.  Let $A_{23}$ be the White, say, 
\mobius band, $A_{12,34}$ be the Black annulus, and $A_{61,45}$ be the White 
annulus of the long \mobius band corresponding to $\sigma$, where the 
subscripts indicate the subarcs of $K$ lying in these surfaces.
  
{\bf Subcase A(i).} \conditionI\ holds.

Then $A_{23},A_{12,34}, A_{61,45}$ are properly embedded on each side of 
$\hatF$. Furthermore, by Lemma~\ref{lem:esc},  $\bdry A_{12,34}$ is 
non-separating in $H_B$ and $A_{61,45}$ is isotopic in $H_W$ onto $\hatF$.

\begin{lemma}\label{t6nsd}
There is a non-separating disk $D$ in $H_B$ disjoint from $A_{12,34}$ and
all arcs of $K \cap H_B$.
\end{lemma}

\begin{proof}
One can find a bridge disk for the arc $\arc{12}$ of $K$ 
whose interior is disjoint from 
$A_{12,34}$ and $K$. Using this to $\bdry$-compress a push-off of $A_{12,34}$
gives the desired disk.
\end{proof}

Let $\calT$ be the solid torus obtained by cutting $H_B$ along $D$.  
Then $A_{12,34}$ is properly embedded in $\calT$ and $A_{23},A_{61,45}$ in
$M - \calT$.  Then $a(\sigma)$ is three parallel curves on $\bdry \calT$.
Let $B,B'$ be the annuli on $\calT$ between $\bdry A_{61,45}, \bdry A_{12,34}$
(resp.) on $\bdry \calT$ whose interiors are disjoint from $K$. Because
$A_{12,34}$ is non-separating in $H_B$, $B'$ intersects $\nbhd(D)$ in a 
disk. $B$ is disjoint from $\nbhd(D)$. The Addendum to 
Lemma~\ref{lem:LMB} applied to the long \mobius band $A_{23} \cup A_{12,34}$
(i.e.\ to the 1-\ESC), 
with $\bdry \calT$ as $F^*$, shows that
$B' \cup A_{12,34}$ bounds a solid torus, $V'$ whose interior is disjoint from
$K$ and in which $B'$ is longitudinal. That is, $V'$ guides an isotopy of
$A_{12,34}$ to $B'$, and, hence, an isotopy (rel endpoints) of the arcs 
$\arc{12}$ and $\arc{34}$
of $K$ onto $\hatF \cap B'$. At the same time, there is an isotopy in $H_W$
of $A_{61,45}$ onto $B$, and hence of the arcs $\arc{61}$ and $\arc{45}$
(rel endpoints) onto $B$. This allows us to thin $K$ to be 1-bridge,
contradicting that $t=6$ and thereby proving 
Theorem~\ref{thm:tnot6}
in Subcase A(i).

{\bf Subcase A(ii).} \conditionII\  holds.

There is a meridian disk, $D$, disjoint from $K$ and $Q$. Let $F^*$ be
$\hatF$ surgered along $D$.
By Lemma~\ref{lem:LMBess} and 
the Addendum to Lemma~\ref{lem:3PCC}, $F^*$ consists of 
two tori: $T_1$ containing
two components of $a(\sigma)$, and $T_2$ containing one. 
Finally, by the Addendum 
to Lemma~\ref{lem:LMB}, components
of $a(\sigma)$ that are isotopic on $F^*$ must be consecutive along the 
long \mobius band. Thus
the vertices of $G_F$ on $T_2$ are either 
\begin{itemize}
\item[(i)] \{2,3\}
\item[(ii)] \{5,6\}
\end{itemize}

Assume (i) holds.
The separation of vertices implies there are no bigons of $\Lambda$ with 
corner $\arc{56}$.  
(Since there are no $\edge{25}$-edges or 
$\edge{36}$-edges, such a bigon would have to be an \SC.  Its corresponding 
\mobius band would have boundary isotopic on $T_1$ to a component of 
$\bdry A_{61,45}$  
permitting the construction of an embedded Klein bottle.)  Furthermore, 
any face of $\Lambda$ containing a $\arc{23}$-corner must be a 
$\arc{23}$-\SC\, and, since $M$ contains no Klein bottles, 
the edges of any two such $\arc{23}$-\SCs\ 
must be parallel on $T_2$ (i.e.\ lie in two edge classes on $T_2$). 
Again by separation, any Black bigon of $\Lambda$ with either a $\arc{12}$-corner 
or a 
$\arc{34}$-corner must be a $\arc{12},\!\arc{34}$-bigon, and the $\edge{41}$-
edges of any such bigon must be parallel on $T_1$ to one of the $\edge{41}$-
edges of $\sigma$.  Thus by Lemma~\ref{lem:parallelwithsamelabel}, at most 
two $\arc{23}$-corners, at most two $\arc{12}$-corners, and at most two 
$\arc{34}$-corners of bigons of $\Lambda$ may be incident to a vertex of
$\Lambda$.  Since $\Delta \geq 3$, the above implies that 
a vertex of $\Lambda$ must have at least 
$6$ corners (in fact at least 9) not incident to bigons of $\Lambda$.  
So Lemmas~\ref{lem:t6trulyspecial} 
and \ref{lem:N4trulyspecialtypes} imply that $\Lambda$ must have an 
edge class containing $8$ edges.  But among $8$ consecutive mutually 
parallel edges of $\Lambda$ one must have a bigon at a $\arc{56}$-corner.

Thus we may assume that (ii) occurs. We are assuming $\arc{45}$ of $K$ lies
in $H_W$; thus, by separation, the meridian disk $D$
along which we surger to get $F^*$ must be a meridian of $H_B$. 
Then $H_B - \nbhd(D)$ is two solid tori, $\calN = \calN_1 \cup \calN_2$ where 
$\bdry \calN_i
= T_i$.
After surgery along trivial curves of intersection on $F^*$,
we may take $A_{23}, A_{12,34},A_{61,45}$ to be properly embedded in $\calN$ or
its exterior. 

First assume $\bdry A_{61,45}$ is longitudinal in each of the solid tori
$\calN_1$ and $\calN_2$. Then $W = \calN \cup \nbhd(A_{23} \cup A_{61,45})$ 
is a solid 
torus. Since $K$ lies entirely in $W$, the exterior of $K$ is irreducible 
and atoroidal, and $M$ is not a lens space, $K$ must be isotopic to a core 
of $W$. But then the core, $L$, of the solid torus $\calN_1$ is a (2,1)-cable 
of $K$. As $L$ is a core of $H_B$,  Claim~\ref{claim:ckt1} contradicts that $t=6$.

Next assume that $\bdry A_{61,45}$ is not longitudinal in $\calN_2$. 
Let $A'= A_{23} \cup A_{12,34} \cup A_{61,45}$ be the long \mobius band
properly embedded in the exterior of $\calN_2$
and set $U = \calN_2 \cup \nbhd(A')$. Then $U$ is Seifert fibered over the
disk with two exceptional fibers. We may isotop $K$ in $U$ so that $K$ is 
the union of two arcs: $\alpha$ in $\bdry U \cap \nbhd(A')$, and $\beta$ 
isotopic to the arc
$\arc{56}$ of $K \cap \calN_2$. Since $\arc{56}$ is bridge in $H_B$, it is
bridge in $\calN_2$. Let $\gamma$ be a cocore of the annulus
$\nbhd(A') \cap \calN_2$. Then $V = U - \nbhd(\gamma)$ is
a genus two handlebody in which $\beta$ is bridge (the intersections of a 
bridge disk with $\nbhd(A') \cap \calN_2$ can be isotoped onto 
$\nbhd(\gamma)$). Furthermore,
$M - U$ must be a solid torus (the exterior of $K$ is irreducible and 
atoroidal); hence, $M - V$ is genus 2 handlebody.
Thus $K$ is at most 1-bridge ($t \leq 2$) with respect to the 
Heegaard handlebody $V$ of $M$
--- contradicting that $t=6$. 

Thus $\bdry A_{61,45}$ must be longitudinal in $\calN_2$ and hence must not be
longitudinal in $\calN_1$. Let $U = \calN_1 \cup \nbhd(A_{23})$ and $A'=
A_{23} \cup A_{12,34}$. Then $U$ is a Seifert fiber space over the
disk with two exceptional fibers and $A'$
is a properly embedded \mobius band in $U$ whose boundary is a Seifert fiber.
As $K \cap U$ is a spanning arc of $A'$,
Lemma~\ref{lem:sSFS1bridge} contradicts that $t=6$.

This final contradiction finishes the proof of Theorem~\ref{thm:tnot6}
in Subcase A(ii).

\begin{caseB}
There are two $1$-\ESCs\ in $\Lambda$ whose label sets overlap in two labels.
\end{caseB}

First note that we may assume in Case B that \conditionI\ holds. For if \conditionII\ holds, there is a meridian disk $D$ of either $H_W$ or $H_B$ disjoint from $Q$ and $K$. WLOG we may assume the two \ESCs\ are as in Figure~\ref{fig:twoESC}. The edges of Figure~\ref{fig:twoESC} show that either $\bdry D$ on $\hatF$ must separate vertices $\{1,4,5\}$ from $\{2,3,6\}$ or there is a projective plane or Klein bottle in $M$. The first is impossible as $D$ is disjoint from $K$, the second since the surgery slope is non-integral.

Assume there are two \ESCs\ $\sigma$ and $\sigma'$ on the corners $\arc{1234}$ and $\arc{3456}$ respectively as shown in Figure~\ref{fig:twoESC}.  
\begin{figure}
\centering
\input{twoESC.pstex_t}
\caption{}
\label{fig:twoESC}
\end{figure}
Let $A_{23}$ and $A_{45}$ be the two White \mobius bands arising from the \SCs\ contained within $\sigma$ and $\sigma'$.  Let $A_{12,34}$ and $A_{34,56}$ be the two Black annuli arising from the remaining two pairs of bigons. As we are in 
\conditionI\, we have that $A_{23},A_{12,34},A_{45},A_{34,56}$ are properly embedded on their sides of $\hatF$.

If either $A_{12,34}$ or $A_{34,56}$ is separating in $H_B$, then since $A_{12,34} \cap A_{34,56} = \arc{34}$ they cannot intersect transversely.
Hence $A_{12,34}$ and $A_{34,56}$ may be slightly isotoped to be disjoint.  In particular, after this isotopy we may assume $\bdry A_{23}$ is disjoint from $\bdry A_{34,56}$ (and isotopic to neither component) and similarly $\bdry A_{45}$ is disjoint from $\bdry A_{12,34}$.  Then by Lemma~\ref{lem:disjtmobiusannulus} the components of $\bdry A_{12,34}$ must be parallel and the components of $\bdry A_{34,56}$ must be parallel.  Thus both of these annuli are separating.

Therefore either both $A_{12,34}$ and $A_{34,56}$ are separating in $H_B$ or both are non-separating in $H_B$.

\noindent {\bf Subcase B(i).} Both $A_{12,34}$ and $A_{34,56}$ are separating in $H_B$.

As noted above, $A_{12,34}$ and $A_{34,56}$ must intersect non-transversely along the arc $\arc{34}$ and can be perturbed to be disjoint.
The annulus $A_{12,34}$ separates $H_B$ into a solid torus $\calT$ and a genus $2$ handlebody. As $M$ contains no Klein bottle, $A_{34,56}$ lies outside of $\calT$. Surgering along innermost closed curves and outermost arcs of 
intersection, we can find a bridge disk, $D_{34}$, for $\arc{34}$ of $K \cap H_B$ that intersects $A_{12,34}$ only along $\arc{34}$.  We first assume $D_{34}$ lies outside $\calT$, i.e.\ it intersects it only in the arc $\arc{34}$. 
%Assume there exist a bridge disk $D_{12}$ for the arc $\arc{12}$ in $H_B$ such that $D_{12}$ meets $\calT$ only along $\arc{12}$.  
Then $\bdry \nbhd(D_{34} \cup A_{12,34}) \cut \bdry H_B$ is an essential disk $D$ and an annulus $A$.

The annulus $A$ chops $H_B$ into a solid torus $\calT'$ on which one impression of $A$ on $\bdry \calT'$ runs $n \geq 1$ times longitudinally and a genus $2$ handlebody $H_B'$ that contains $A_{34,56}$ and $A_{12,34}$ such that $A_{12,34}$ is $\bdry$-parallel to the other impression of $A$ on $\bdry H_B'$. Then $D_{34}$ marks $\bdry A_{23}$ as a primitive curve on $H_B'$.  Also note that by banding across $A_{12,34} \cup D_{34}$, the arc $\arc{56}$ has a bridge disk $D_{56}$ that is disjoint from $\nbhd(A_{12,34} \cup D_{34})$ and thus from $D$.

Attach $\nbhd(A_{23})$ to $H_B'$ along the annulus $\nbhd(\bdry A_{23})$ to form $H_B''$.  Since $\bdry A_{23}$ is primitive in $H_B'$, $H_B''$ is a genus $2$ handlebody.   Note that the arc $\arc{1234}$ of $K \cap H_B''$ lies in a \mobius band in $H_B''$ and hence is bridge. The arc $\arc{56}$ is bridge in $H_B''$ as it has a bridge disk disjoint from $D$.

By Lemma~\ref{lem:bridgedisksdisjoitfrommobius} both arcs $\arc{45}$ and $\arc{61}$ have bridge disks in $H_W$ disjoint from $A_{23}$.  Thus they both have bridge disks in the genus $2$ handlebody $H_W' = H_W - \nbhd(A_{23})$.  Attach $\calT'$ to $H_W'$ along the annulus $A' = \bdry \calT' \cut A = \bdry \calT' \cap \bdry H_W$ to form $H_W''$.  Since $A'$ has $\bdry A_{23}$ as one of its boundary components it is primitive on $H_W'$ and thus $H_W''$ is a genus $2$ handlebody.  Moreover, since $\calT'$ is disjoint from $K$, so is $A'$; thus the bridge disks for $\arc{45}$ and $\arc{61}$ may be assumed to be disjoint from $A'$ as well.  Hence these two arcs are bridge in $H_W''$.

Therefore $H_B'' \cup H_W''$ is a new genus $2$ Heegaard splitting for $M$ in which $K$ has a $2$-bridge presentation.  This is contrary to assumption.  Hence it must be that $D_{34}$ lies in $\calT$.

\begin{remark} 
If $\partial A_{23}$ is longitudinal in $\calT'$, 
the new Heegaard splitting, $H_B'' \cap H_W''$,  comes from the old (up to isotopy) by adding/removing
a primitive \mobius band as described in the proof of Theorem~\ref{thm:changingHS}. If $\partial
A_{23}$ is not longitudinal in $\calT'$, then $M$ is a Seifert fiber space over
the $2$-sphere with an exceptional fiber of order $2$. In this case, we could 
find a vertical splitting with respect to which $K$ has smaller bridge number
by applying Lemma~\ref{lem:sSFS1bridge} to $\nbhd(A_{23}) \cup \calT'$, a Seifert fiber space
over the disk. This would then be consistent with the proof of Theorem~\ref{thm:changingHS}.
\end{remark}

Let $\calT''$ be the solid torus that $A_{34,56}$ separates off $H_B$. Since
$M$ contains no Klein bottles, $\calT$ lies outside of $\calT''$; hence, so
is $D_{34}$. Apply the above argument to $A_{34,56}$ in place of $A_{12,34}$
to get a 2-bridge presentation of $K$.

\noindent {\bf Subcase B(ii).} Both $A_{12,34}$ and $A_{34,56}$ are non-separating in $H_B$.

Since no pair of the four components of $\bdry A_{12,34} \cup \bdry A_{34,56}$ may be isotopic on $\hatF$ (else we get a Klein bottle), $A_{12,34}$ and $A_{34,56}$ intersect transversely along $\arc{34}$.  Since $H_B$ is a handlebody, $H_B - \nbhd(A_{12,34} \cup A_{34,56})$ is a single solid torus on which the annulus $\bdry H_B \cut \bdry  (A_{12,34} \cup A_{34,56})$ is a longitudinal annulus. That is, $H_B$ is
isotopic to $\nbhd(A_{12,34} \cup A_{34,56})$. Then $\bdry A_{23}$ is
primitive in $H_B$. Furthermore, 
the arc $\arc{56}$ has a bridge disk in $H_B$ disjoint from $A_{12,34}$.

Since $\bdry A_{23}$ is primitive in $H_B$, we may form the genus $2$ handlebody $H_B'$ by attaching $\nbhd(A_{23})$ to $H_B$ along $\nbhd(\bdry A_{23})$.  Its complement $H_W' = H_W - \nbhd(A_{23})$ is also a genus $2$ handlebody.  Thus together $H_W'$ and $H_B'$ form a new genus $2$ Heegaard splitting for $M$.

The White arcs $K \cap H_W$ other than $\arc{23}$ continue to be bridge in $H_W'$.  Furthermore, since there is a bridge disk $D_{56}$ in $H_B$ for the Black arc $\arc{56}$ that is disjoint from $A_{12,34}$, $D_{56}$ continues to be a bridge disk for $\arc{56}$ in $H_B'$.  Finally, the arc $\arc{1234}$ is bridge in $H_B'$ as it lies in the \mobius band $A_{12,34} \cup A_{23}$.

Thus the handlebodies $H_W'$ and $H_B'$ form a new genus $2$ Heegaard splitting for $M$ in which $K$ is at most $2$-bridge.  This contradicts the assumption that $t=6$.

This completes the proof of Subcase B(ii) and hence Case B cannot occur.

\begin{caseC}
There are three mutually disjoint \SCs\ in $\Lambda$.
\end{caseC}

Lemma~\ref{lem:3disjointmobiusbands} (independent of \conditionI\ and \conditionII) implies Case C does not occur.

\begin{caseD}
There is a $1$-\ESC\ and disjoint \SC\ in $\Lambda$.
\end{caseD}

This case is considered in the following Section~\S\ref{sec:tis6caseD}.
Proposition~\ref{prop:tis6caseD} shows Case D cannot occur.

This completes the proof of Theorem~\ref{thm:tnot6}.
\end{proof}

\section{Case D of Theorem~\ref{thm:tnot6}.}\label{sec:tis6caseD}

In this section we show:
\begin{prop}\label{prop:tis6caseD}
 Case D of Theorem~\ref{thm:tnot6} cannot occur.
That is, if 
\begin{enumerate}
\item $t=6$, 
\item $\Lambda$ contains no $2$-\ESC,
\item $\Lambda$ contains a $1$-\ESC\ and an \SC\ on a disjoint label sets, 
\end{enumerate}
then $M$ contains a Dyck's surface.
\end{prop}

\begin{proof}
Assume $M$ does not contain a Dyck's surface. 
Assume there is an \ESC\ $\tau$ with labels $\{1,2,3,4\}$ and an \SC\ $\sigma$ with labels $\{5,6\}$ as shown in Figure~\ref{fig:tis6with4and2}.  Let $A_{23}$ and $A_{12,34}$ be the White \mobius band and Black annulus arising from $\tau$.  Let $A_{56}$ be the Black \mobius band arising from $\sigma$.  By Lemma~\ref{lem:disjtmobiusannulus} (and that $M$ contains no Klein bottle or projective plane) the components of $\bdry A_{12,34}$ must be parallel on $\hatF$ bounding an annulus $B_{12,34}$ in $\hatF$.  Then by Lemma~\ref{lem:LMB}, $A_{12,34}$ is isotopic in $M$ to $B_{12,34}$. 

\begin{figure}
\centering
\input{tis6with4and2.pstex_t}
\caption{}
\label{fig:tis6with4and2}
\end{figure}

We consider the arguments of this subsection under both possibilities, 
\conditionI\ and \conditionII.

\begin{claim}\label{claim:SD2}
%ClaimSD: 
In \conditionII, there is a separating, meridian disk $D$ of $H_B$ disjoint
from $K$ and $Q$. Let $F^*$ be $\hatF$ surgered along $D$. Then $F^*$
is two tori, $T_1,T_2$, where vertices $\{1,2,3,4\}$ lie on $T_1$ and 
vertices $\{5,6\}$ on $T_2$.
\end{claim}

\begin{proof}
Recall that \conditionII\ implies that there is a meridian disk on 
one side of $\hatF$ that
is disjoint from $K$ and from $Q$. First assume this disk was nonseparating.
Compressing $\hatF$ along it would produce a $2$-torus that intersected the 
interiors of  $A_{23},A_{56}$ only in trivial curves. Surgering away such 
intersections exposes either a projective plane or Klein bottle in $M$. So this disk must be separating on one side of $\hatF$. It is disjoint 
from  $B_{12,34}$ else it along with $A_{23}$ forms a projective plane in $M$. Thus the boundary of this disk must separate vertices $\{1,2,3,4\}$ in 
$\hatF$ from vertices $\{5,6\}$. Since the disk is disjoint from the 
$\arc{45}$-arc of $K$, it must be in $H_B$.
\end{proof}

\begin{lemma}%[Lemma 1]
\label{lem:twoparallismsforSC}
In \conditionI, the edges of all $\arc{23}$-\SCs\  of $\Lambda$ belong 
to two parallelism classes in $G_F$. In \conditionII, all $\arc{23}$-edges
belong to two parallelism classes in $G_{F^*}$, where $G_{F^*}$ 
is the graph induced 
on $F^*$ from $G_F$.
\end{lemma}

\begin{proof}
First assume \conditionII\ holds and Claim~\ref{claim:SD2} applies. 
Because the components of $\bdry A_{12,34}$ are disjoint, essential curves 
in $T_1$, any $\edge{23}$-edge
of $\Lambda$ is parallel on $G_{F^*}$ to one of the $\edge{23}$-edges of $\tau$.

Thus we assume we are in \conditionI.
Let $\nu$ be the core $\arc{23}$-\SC\ of $\tau$, and assume there is another 
$\arc{23}$-\SC, $\nu'$, with an edge that is not parallel
on $G_F$ to either edge of $\nu$. Let $f,f'$ be the faces bounded by
$\nu,\nu'$. 
Let $A_{23}'$ be the \mobius band corresponding to $f'$.   
By the ordering of the labels around the vertices $2,3$ of $G_F$, one edge, 
$e_1'$, of $\nu'$ must lie within $B_{12,34}$ and the other, $e_2'$, 
outside. That is, $e_1'$ is parallel in $G_F$ to an edge $e_1$ of $\nu$.

We may use the parallelism between $e_1,e_1'$, along with the
parallelisms of the corners of $f,f'$ along $\bdry \nbhd(\arc{23})$, to 
band together $f,f'$ to get a properly embedded disk $D'$ in $H_W$. Here
$\bdry D'$ is the curve $e_2 \cup e_2'$ on $\hatF$, where $e_2$ is the 
edge of $\nu$ other than $e_1$. By an inspection of the labelling around
vertices $2,3$ of $G_F$, one can see that the $D'$ can be taken to be 
disjoint from both $A_{23},A_{23}'$ and $K$. As $e_2,e_2'$ are not parallel
on $G_F$, $\bdry D'$ is not trivial on $\hatF$. $\bdry D'$ must separate
 $\bdry A_{23}, \bdry A_{23}'$ from $\bdry A_{56}$ on $\hatF$, since $M$ contains no Klein bottles.
But this contradicts that $D'$ is disjoint from the $\arc{45}$-arc of $K$.
\end{proof}

\begin{lemma}\label{lem:nospanningarc}
%Under \conditionI, n
No White bigon or trigon has an edge that is a spanning arc of $B_{12,34}$.
\end{lemma}

\begin{proof}
Let $f$ be a White bigon or trigon with such an edge $e$. We assume $e$ is
a $\edge{34}$-edge, the argument for when it is a $\edge{12}$-edge is the same.
There is a bridge 
disk for some arc of $K \cap H_W$ that is disjoint from $f$ 
(except possibly along that arc) and hence from $\Int e$ (after removing trivial arcs and simple closed 
curves of intersection of $f$ with an arbitrary bridge disk, $D$, an outermost 
arc of intersection on $D$ will cut out the desired bridge disk).

Let $r$ be an arc in the annulus $A_{12,34}$ sharing endpoints with $\arc{34}$ that projects through the $\bdry$-parallelism of $A_{12,34}$ onto the $\edge{34}$-edge of $f$. That is, there is a bridge disk for $r$, $D_r$, that intersects
$\hatF$ in $e$. Note that
up to isotopy rel endpoints, $r$ is just $\arc{34}$ twisted along 
$\bdry A_{12,34}$. So we may take $r$ to have a single critical value (indeed
the same as for $\arc{34}$) under the
height function on $M$ for the thin presentation of $K$. Let $r'$ be an arc in the annulus $A_{12,34}$ disjoint from $r$ and sharing endpoints with $\arc{12}$. Similarly $r'$ can be taken to
have a single critical value with respect to the height function on $M$. Then $r' \cup \arc{23} \cup r$ and $\arc{12} \cup \arc{23} \cup \arc{34}$ are two properly embedded, non-$\bdry$-parallel arcs in the \mobius band $A_{23} \cup A_{12,34}$ with the same boundary.  By Lemma~\ref{claim:arcsinmobius}, these two arcs are isotopic rel-$\bdry$ within this long \mobius band.  Perform this 
isotopy, 
then use the Black bridge disk $D_r$ along with the White bridge disk of the
preceding paragraph to give a thinner presentation of $K$ --- a contradiction.
\end{proof}

\begin{lemma}%[Lemma 2]
\label{lem:singlenonSCWbigon}
The only type of White bigon in $\Lambda$ that is not an \SC\ is a  
$\arc{45},\!\arc{61}$-bigon.
\end{lemma}

\begin{figure}
\centering
\input{nonSCWbigons.pstex_t}
\caption{}
\label{fig:nonSCWbigons}
\end{figure}

\begin{proof}
The three possible non-Scharlemann White bigons are shown in Figure~\ref{fig:nonSCWbigons}.  We will rule out the two that have a $\arc{23}$-corner.   Note that we may assume we have \conditionI\ because of the separation that comes 
from Claim~\ref{claim:SD2} in \conditionII.   Let us focus on the bigon $R$ with the $\edge{12}$-edge as the proof is analogous for the other.

Since $R$ has a $\arc{23}$-corner, the labelling around vertices $2,3$ of 
$G_F$ forces the edges $\edge{12}$ and $\edge{36}$ to be incident to 
the vertices $2$ and $3$ on opposite sides of $\bdry A_{23}$ in $\hatF$.  Therefore since the $\edge{36}$-edge must connect vertex $3$ to vertex $6$,  the $\edge{12}$-edge must lie in the annulus $B_{14,23}$. But this contradicts
Lemma~\ref{lem:nospanningarc}.
\end{proof}

\begin{cor}%[Cor1]
\label{cor:23cornersatavtx}
At most two $\arc{23}$-corners at a vertex belong to bigons of $\Lambda$.
\end{cor}
\begin{proof}
Assume there are three $\arc{23}$-corners at vertex $x$ of $\Lambda$ belonging to bigons of $\Lambda$.  By Lemma~\ref{lem:singlenonSCWbigon}, these bigons must all be \SCs.  By Lemma~\ref{lem:twoparallismsforSC} these six edges belong to two parallelism classes on either $G_F$ or $G_{F^*}$.  
Therefore two of these six edges must be incident to the same vertex of $G_F$
($G_{F^*}$) at the label $x$ and parallel.  This violates Lemma~\ref{lem:parallelwithsamelabel} (Lemma~\ref{lem:parallelwithsamelabelAPE}).  (If only two 
bigons are incident at these three corners, two of the edges will have label $x$ at both ends in $G_F$ and Lemma~\ref{lem:parallelwithsamelabel} 
(Lemma~\ref{lem:parallelwithsamelabelAPE}) is still violated).
\end{proof}

\begin{lemma}%[Lemma 3]
\label{lem:no12SCno34SC}
$\Lambda$ does not contain a $\arc{12}$- or a $\arc{34}$-\SC.
\end{lemma}

\begin{proof}
Assume there exists a $\arc{34}$-\SC.  Let $A_{34}$ be the corresponding \mobius band.  

Since the annulus $A_{12,34}$ cobounds a solid torus with $B_{12,34}$, the \mobius band $A_{34}$ must intersect $A_{12,34}$ tangentially.  Therefore $A_{34}$ may be isotoped to be disjoint from $A_{23}$.  Since it is also disjoint from $A_{56}$, Lemma~\ref{lem:3disjointmobiusbands} implies that $M$ contains a Dyck's 
surface.

A similar argument rules out the existence of a $\arc{12}$-\SC.
\end{proof}

\begin{lemma}%[Lemma 4]
\label{lem:no4561esc}
There cannot be an \ESC\ with labels $\{4,5,6,1\}$.
\end{lemma}

\begin{proof}
Assume there is such an \ESC\ and take $\sigma$ to be its core \SC.  Let $A_{45,61}$ be the corresponding annulus formed from the bigons that flank $\sigma$.  
Observe that $A_{45,61}$ is disjoint from the \mobius band $A_{23}$, and $A_{12,34}$ is disjoint from the \mobius band $A_{56}$.  Then by Lemma~\ref{lem:disjtmobiusannulus} (and that there are no Klein bottles in $M$) the components of $\bdry A_{45,61}$ must be parallel as must the components of $\bdry A_{12,34}$. This 
contradicts Claim~\ref{claim:SD2} in \conditionII. In \conditionI, 
Lemma~\ref{lem:LMB} implies these annuli must be isotopic into $\hatF$.  Together these parallelisms give a thinning of $K$.
\end{proof}

\begin{cor}%[Cor2]
\label{cor:nothree4561}
There cannot be three consecutive bigons around a vertex with a $\arc{4561}$-corner.
\end{cor}
\begin{figure}
\centering
\input{no4561bigons.pstex_t}
\caption{}
\label{fig:no4561bigons}
\end{figure}
\begin{proof}
Assume there were.  Then there are three possibilities according to whether opposite the $\arc{56}$-corner is the $\arc{56}$-, $\arc{34}$-, or $\arc{12}$-corner.  These are shown in Figure~\ref{fig:no4561bigons}.    The first is ruled out by Lemma~\ref{lem:no4561esc} since it is an \ESC\ with labels $\{4,5,6,1\}$.  The second and third are ruled out since they contain the non-\SC\ White bigons prohibited by Lemma~\ref{lem:singlenonSCWbigon}.
\end{proof}

\begin{lemma}\label{lem:no23trigon}
%Under \conditionI, t
There cannot be a White trigon with a single 
$\arc{23}$-corner.
\end{lemma}

\begin{proof}
If there were such a trigon, then its two edges incident to that corner must be incident to opposite sides of $\bdry A_{23}$ on $\hatF$.  Thus one of these edges must be a spanning arc of $B_{12,34}$.  This is prohibited by Lemma~\ref{lem:nospanningarc}.
\end{proof}

\begin{lemma}\label{lem:no23S3}
There cannot be a Scharlemann cycle of length $3$ on the labels $\{2,3\}$.
\end{lemma}

\begin{proof}
Assume $g$ is the trigon face of such a $\arc{23}$-Scharlemann cycle.  
The ends of two edges of $g$ incident to the same corner of $g$ must be incident to opposite sides of $\bdry A_{23}$ as they lie on $\hatF$.  Since the annulus $B_{12,34}$ has $\bdry A_{23}$ as a boundary component, around $\bdry g$ the edges are alternately in or not in $B_{12,34}$.  This of course cannot occur since $g$ has three edges.
\end{proof}

\begin{lemma}\label{lem:2323trigon}
In $\Lambda$, a trigon cannot have exactly two $\arc{23}$-corners.  
\end{lemma}

\begin{proof}
First we assume \conditionII\ holds. This means there is a
Black meridian disjoint from $G_F$ separating vertices $\{1,2,3,4\}$ and $\{5,6\}$.  
 The third corner of a trigon with two $\arc{23}$--corners must be either a $\arc{61}$-corner or a $\arc{45}$-corner.  But then $G_F$ has an edge joining either vertices $2$ and $5$ or vertices $3$ and $6$, a contradiction.

Thus we may assume we satisfy \conditionI. A trigon with a single 
$\arc{23}$-corner is prohibited by Lemma~\ref{lem:no23trigon}.  A trigon with three $\arc{23}$-corners is a Scharlemann cycle of length $3$ which cannot occur by Lemma~\ref{lem:no23S3}.  Thus any trigon with a $\arc{23}$-corner must have exactly one other $\arc{23}$-corner.  

Assume $g$ is a trigon with two $\arc{23}$-corners and, WLOG, one $\arc{61}$-corner.
We shall construct a bridge disk for $\arc{61}$ that does not intersect the interior of $B_{12,34}$.  Such a bridge disk for $\arc{61}$ is disjoint from a bridge disk for $\arc{34}$ and hence provides a contradictory thinning of $K$.

Because the $\edge{12}$-edge and the $\edge{36}$-edge of $g$ must be incident to the side of $\bdry A_{23}$ opposite from which the $\edge{23}$-edge is incident
(by the labelling around vertices $2,3$ of $G_F$), neither of them lie in the annulus $B_{12,34}$.  Furthermore, the $\edge{23}$-edge is parallel in $G_F$ to a $\edge{23}$-edge of the \SC\ in $\tau$.  Let $f$ be the bigon face of the \SC\ in $\tau$.  Let $\delta$ be the rectangle of parallelism on $G_F$ between the two $\edge{23}$-edges of $g$ and $f$; its other two sides are arcs of the vertices $2$ and $3$.  Let $\rho$ and $\rho'$ be the disjoint rectangles on $\bdry \nbhd(K)$ between the two $\arc{23}$-corners of $g$ and the two $\arc{23}$-corners of $f$; the other two sides of each of $\rho,\rho'$ being
arcs of the vertices $2$ and $3$.  Then $g \cup \delta \cup \rho \cup \rho' \cup f$ forms a disk $D_{61}$ whose boundary is composed of the $\arc{61}$-corner of $g$ and an arc on $\hatF$; see Figure~\ref{fig:61bridge}.  By a slight isotopy, the interior of this arc on $\hatF$ may be made disjoint from $B_{12,34}$.  Thus $D_{61}$ is the desired bridge disk for $\arc{61}$.   
\begin{figure}
\centering
\input{61bridge.pstex_t}
\caption{}
\label{fig:61bridge}
\end{figure}
\end{proof}

\begin{lemma}\label{lem:trigonswitha23corner}
No trigon in $\Lambda$ has a $\arc{23}$-corner.
\end{lemma}
\begin{proof}
This is a combination of Lemmas~\ref{lem:no23S3}, \ref{lem:2323trigon}, and \ref{lem:no23trigon}.
\end{proof}

\subsection{Special vertices for Case D of Theorem~\ref{thm:tnot6}.}
As $\Lambda$ contains no $2$-\ESCs\, there may be no more than $7$ edges that are mutually parallel.

Since $t=6$,  Lemmas~ \ref{lem:t6trulyspecial} and \ref{lem:N4trulyspecialtypes} imply there exists a vertex $v$ of $\Lambda$ of type $[6\Delta-5,4]$, $[6\Delta-4,1]$, or $[6\Delta-3]$.  We refer to a corner at $v$ that is not incident
to a bigon of $\Lambda$ as a {\em gap}. Thus there are at most $5$ gaps at $v$.  We will argue by contradiction, in each case showing that there must be more gaps than specified.

By Corollary~\ref{cor:nothree4561}, each of the $\Delta$ corners $\arc{4561}$ around $v$ must have a gap.  By Corollary~\ref{cor:23cornersatavtx} at least $\Delta-2$ $\arc{23}$-corners must have gaps as well.  Thus there must be at least $2\Delta-2$ gaps.  If $\Delta \geq 4$ then there must be at least $6$ gaps; a contradiction.  Hence $\Delta =3$.

When $\Delta = 3$ the vertex $v$ is of type $[13,4]$, $[14,1]$, or $[15]$.  Type $[15]$ is prohibited since using $\Delta = 3$ in the argument of the preceding
paragraph implies $v$ must have at least $4$ gaps.  We eliminate the remaining types in the following subsections.

\subsection{ Vertex $v$ is of type $[14,1]$.}
There are at most $4$ gaps at $v$.  By Corollary~\ref{cor:nothree4561}, each of the three sequences of the labels $\arc{4561}$ must have a gap.  By Corollary~\ref{cor:23cornersatavtx} one $\arc{23}$-corner must be a gap.  Thus there are two sequence of bigons with a $\arc{1234}$-corner.  Therefore by the following Lemma~\ref{lem:sequencesofbigons} there must be a gap at the remaining $\arc{12}$-corner or $\arc{34}$-corner at $v$.  This however requires $5$ gaps at $v$.

\begin{lemma}\label{lem:sequencesofbigons}                                                                                                
 %(the numbering has changed now):                                                                                                        
If there are two $\arc{23}$-corners at a vertex $x$ that belong to bigons of                                                             
$\Lambda$, then at the other $\arc{23}$-corner of  $x$ one of the adjacent corners does not belong to a bigon of $\Lambda$.              
\end{lemma}                                                                                                                               
                                                                                                                                          
\begin{proof}                                                                                                                             
Assume two $\arc{23}$-corners at $x$ belong to bigons of $\Lambda$. By                                                                   
Lemma~\ref{lem:singlenonSCWbigon} these bigons are \SCs, $\sigma_1, \sigma_2$ (we assume they are distinct else a similar argument holds).  Assume there is another $\arc{23}$-corner at $x$ such that there is a bigon incident to its adjacent $\arc{12}$-corner (if the $\arc{34}$-corner, the same argument applies). Let $e_2,e_3$ be the edges of $G_Q$ incident to this $\arc{23}$-corner (where $e_2$ has label $2$ at
$x$).  By 
Lemma~\ref{lem:no12SCno34SC}, $e_2$ is either a $\edge{23}$-edge or a $\edge{25}$-edge of $\Lambda$.  As an                                 
edge of $G_F$, $e_2$ must lie outside of $B_{12,34}$ (if it were a $\edge{23}$-edge, then                                                 
it along with edges of $\sigma_1$, $\sigma_2$ would, by                                                                                   
Lemma~\ref{lem:twoparallismsforSC}, violate Lemma~\ref{lem:parallelwithsamelabel} or Lemma~\ref{lem:parallelwithsamelabelAPE}). But then                                              
$e_3$ as an edge in $G_F$ must lie inside $B_{12,34}$ (by the ordering                                                              
of the labels around vertices $2$,$3$ of $G_F$ coming from $\arc{23}$-Scharleman                                                         
cycle of $\tau$ and $e_2, e_3$). If $e_3$ is also in a bigon of $\Lambda$,                                                                
then it must be a $\edge{23}$-edge (by Lemma~\ref{lem:no12SCno34SC} and since                                                             
vertex $6$ does not lie in $B_{12,34}$).                                                                                                  
As $e_3$ lies in $B_{12,34}$, it must be parallel to edges of $\sigma_1, \sigma_2$, which by                                              
Lemma~\ref{lem:twoparallismsforSC} would violate                                                                                          
Lemma~\ref{lem:parallelwithsamelabel} or 
Lemma~\ref{lem:parallelwithsamelabelAPE}.                                                                                                    
\end{proof}                                                                                        

\subsection{ Vertex $v$ is of type $[13,4]$.}
There are at most $5$ gaps around $v$; at least $4$ of these gap corners belong to trigons of $\Lambda$; there is at most one corner that may belong to neither a bigon nor trigon of $\Lambda$. Note that this implies that every edge incident
to $v$ lies in $\Lambda$.

By Corollary~\ref{cor:nothree4561} each of the three $\arc{4561}$-corner sequences around $v$ must be missing a bigon.  Thus among the three $\arc{1234}$-corner sequences, only two may be missing a bigon.  Let us distinguish these three $\arc{1234}$-corner sequences around $v$ by marking them as $v$, $v'$, and $v''$.  Furthermore let $e_i$, $e_i'$ and $e_i''$ be the edge of $\Lambda$ incident to $v$, $v'$, and $v''$ respectively at the label $i$ for $i=1,2,3,4$.

By Corollary~\ref{cor:23cornersatavtx} at least one $\arc{23}$-corner is a gap, say the one at $v$.  Since there is not a trigon with a $\arc{23}$-corner by Lemma~\ref{lem:trigonswitha23corner}, only $v$ may have a gap at its $\arc{23}$-corner.
%
%
%{\bf Case I:}  Assume 
Thus the two $\arc{23}$-corners at $v'$ and $v''$ have bigons.  %The $\arc{23}$-corner at $v$ has either a gap or a trigon.   
These bigon faces are \SCs\ by Lemma~\ref{lem:singlenonSCWbigon}.
By Lemmas~\ref{lem:twoparallismsforSC},
\ref{lem:parallelwithsamelabel}, and \ref{lem:parallelwithsamelabelAPE},
neither $e_2$ nor $e_3$ may be parallel on $G_F$ to one of the $\edge{23}$-edges in $\bdry B_{12,34}$.
Thus the labelling around vertices $2,3$ of $G_F$ forces 
one edge to be a spanning arc of $B_{12,34}$ and the other to lie outside $B_{12,34}$ and not parallel into $\bdry B_{12,34}$.  Without loss of generality, let us assume $e_3$ is a spanning arc of $B_{12,34}$ and thus is a $\edge{34}$-edge (by the Parity Rule of section~\ref{sec:fatgraphs}).

Since $e_3$ is a $\edge{34}$-edge, the adjacent $\arc{34}$-corner cannot belong to a bigon.  Otherwise such a bigon would be a Black $\arc{34}$-\SC.  By Lemma~\ref{lem:no12SCno34SC} this does not occur.  Thus the adjacent $\arc{34}$-corner must belong to a trigon $g$.  Since the edge $e_3$ of this Black trigon $g$ is a spanning arc of $B_{12,34}$, $g$ lies in the solid torus of parallelism of $A_{12,34}$ into $B_{12,34}$.  Hence $g$ has a second $\arc{34}$-corner and a $\arc{12}$-corner.  (It cannot be a $\arc{34}$-\SC\ as $A_{12,34}$ is parallel into $B_{12,34}$.)  Moreover edge $e_4$ of $g$ is a $\edge{14}$-edge.  Since $e_4$ is contained in $B_{12,34}$, it belongs to one of the two $\edge{14}$-edge classes of $\bdry B_{12,34}$.

Because there are at most $5$ corners of $v$ without bigons, the two $\arc{1234}$-corner sequences at $v'$ and $v''$ must entirely belong to bigons and thus form \ESCs. Furthermore, the $\arc{12}$-corner at $v$ must be a bigon $h$ in $\Lambda$.  It is either a $\arc{12},\!\arc{34}$-bigon or a $\arc{12},\!\arc{56}$-bigon.

Assume $h$ is a $\arc{12},\!\arc{34}$-bigon. 
At $v$ we have identified three $\edge{41}$-edges incident to $v$ at label $4$, $e_4,e_4',e_4''$, and three $\edge{23}$-edges incident to $v$ at label $2$. By Lemma~\ref{lem:parallelwithsamelabel}, the $\edge{41}$-edges (as well as the $\edge{23}$-edges) must be in distinct edges classes in $G_F$. Thus a neighborhood of $B_{12,34}$ and these edges is an essential $4$-punctured sphere, $S$, in 
$\hatF$. This immediately rules out \conditionII\  as there would have to be a separating Black meridian disjoint from $S$. We assume \conditionI. As $e_4$ is parallel to one of the $\edge{41}$-edges of $\bdry B_{12,34}$, it must be that $e_4'$, say, is not in one of the edge classes of $\bdry B_{12,34}$.
By Lemma~\ref{lem:twoparallismsforSC}, there must be a $\arc{12},\!\arc{34}$-bigon $f$ of $\tau$ and a $\arc{12},\!\arc{34}$-bigon $f'$ containing $e_4'$ such that the $\edge{23}$-edges of $f,f'$ are parallel
but the $\edge{41}$-edges are not. Banding $f,f'$ together along the parallelism of the $\edge{23}$-edges in $G_F$, along with the corresponding rectangles along the boundary of the knot exterior, gives a Black disks $D'$ whose boundary on $\hatF$ is the union of the $\edge{41}$-edges of $f,f'$. This disk can be taken disjoint from $K$, and from the \mobius bands formed from $\sigma$ and from the $\arc{23}$-\SC\ of $\tau$. Thus $\bdry D'$ must be separating in $\hatF$ (else we can form a Klein bottle or projective plane with these \mobius bands). But $\bdry D'$ can be isotoped to $\bdry S$, contradicting the fact that it is 
separating.

Thus we may assume $h$ is a $\arc{12},\!\arc{56}$-bigon. This immediately rules out \conditionII\ (vertices $2,5$ would have to be
separated).  Again let $e_4,e_4',e_4''$ be the $\edge{41}$-edges incident to $v$ with label $4$. By Lemma~\ref{lem:parallelwithsamelabel}, one of these (not $e_4$), say $e_4'$,
is not parallel to either of the $\edge{41}$-edges of $\bdry B_{12,34}$.
By Lemma~\ref{lem:twoparallismsforSC}, there must be a $\arc{12},\!\arc{34}$-bigon, $f$, of $\tau$ and $\arc{12,34}$-bigon,$f'$,
containing $e_4'$ whose $\edge{23}$-edges are parallel but whose $\edge{41}$-edges are not. Banding $f,f'$ along the parallelism of their $\edge{23}$-edges as above gives a Black disk $D'$ which is disjoint from the \mobius bands formed from $\sigma$ and from $\tau$. Thus $\bdry D'$ must be separating in $\hatF$.  On the other hand, $\bdry D'$ can be isotoped to the boundary of the
essential $4$-punctured sphere formed from a neighborhood in $\hatF$ of the $\edge{25}$-edge in $h$, $e_4'$, $\sigma$, and $B_{12,34}$; hence, it cannot be separating.

This completes the proof of Proposition~\ref{prop:tis6caseD}.
\end{proof}

\subsection{A generalization of Lemma~\ref{lem:parallelwithsamelabel} to $G_{F^*}$}

We finish with a generalizaton of Lemma~\ref{lem:parallelwithsamelabel} that is
needed for this section as well for section~\ref{sec:tis4scc}.

\begin{lemma}\label{lem:parallelwithsamelabelAPE}
Let $D$ be a meridian disk of $\hatF$ disjoint from $Q$ and $K$, and $F^*$
be $\hatF$ surgered along $D$ (hence is either one or two tori).
Let $G_{F^*}$ be the induced graph on $F^*$.  There cannot be parallel edges 
of $G_{F^*}$ that are incident to a vertex at the same label.
\end{lemma}

\begin{proof}
Let $e,e'$ be parallel edges on $G_{F^*}$ incident to a vertex $v$ of $G_{F^*}$
with the same label. If there are no monogons (1-sided faces) of $G_{F^*}$
in the parallelism between these two edges, then the proof of 
Lemma~\ref{lem:parallelwithsamelabel} directly applies (after possibly
surgering away simple closed curves of intersection).  
But the graph $G_{F^*}$ may contain monogons even though $G_F$ does not.  
Any monogon of $G_{F^*}$ must contain at least one impression of $D$.  
In particular, there may be at most two innermost monogons of $G_{F^*}$.
\begin{claim}
Any monogon of $G_{F^*}$ must be innermost.
\end{claim}
\begin{proof}
If there is a non-innermost monogon of $G_{F^*}$, then there is one that 
appears as one in 
Figure~\ref{fig:innermonogon}.  Each of these configurations gives a 
long disk\footnote{See Lemma~\ref{lem:fesc+1} here or Lemma~2.2 and Figure~1 of \cite{baker:sgkilshsbn} for the concept of a {\em long disk}.} for $K$ as a knot in $S^3$, contradicting its thinness there.
\end{proof}
\begin{figure}
\centering
\input{innermonogon.pstex_t}
\caption{}
\label{fig:innermonogon}
\end{figure}

\begin{claim}
If there are monogons of $G_{F^*}$ in the parallelism between $e$ and $e'$, 
then there are exactly two and they appear as in 
Figure~\ref{fig:parallelmonogon}(a).
\end{claim}
\begin{figure}[h]
\centering
\input{parallelmonogon.pstex_t}
\caption{}
\label{fig:parallelmonogon}
\end{figure}
\begin{proof}
Since there may be only two monogons and they are not nested, there are three possible configurations for monogons between parallel edges.  These are shown in Figure~\ref{fig:parallelmonogon}(a), (b), (c).  Configurations (b) and (c) give lopsided bigons\footnote{See Lemma~\ref{lem:fesc+1} here or the last two paragraphs of the of Lemma~6.15 \cite{baker:sgkilshsbn} for the concept of a {\em lopsided bigon}.} for $K$ as a knot in $S^3$, contradicting its thinness there.
\end{proof}

So we may assume there are two monogons as in 
Figure~\ref{fig:parallelmonogon}(a) among the parallel edges 
between $e$ and $e'$.  Note that $e$ and $e'$ have the same label pairs.  
Abstractly band the monogons together as in Figure~\ref{fig:monogonbanding} 
to take advantage of the arguments of \cite{GLi}. 
\begin{figure}[h]
\centering
\input{banding.pstex_t}
\caption{}
\label{fig:monogonbanding}
\end{figure}

We employ the notation of \cite{GLi}, 
substituting $F^* = P_\alpha$ and  $Q = P_\beta$.  
We set $n_\beta = |K \cap \hatQ| $ and may assume $n_\beta+1$ is the 
number of arcs from $e$ to $e'$.  Let $A_m$ and $A_{m+1}$ be the arcs formed from banding the two monogons together, $1<m<n_\beta$.  (Since $n_\beta$ is even, we have $n_\beta +1$ arcs, and neither $A_m$ nor $A_{m+1}$ is $e$ or $e'$, then we may relabel and take our $n_\beta$ parallel arcs so that neither $A_m$ nor $A_{m+1}$ is outermost among these $n_\beta$ arcs.)   Let $A'$ and $A''$ be the arcs of the original monogons. 

We first assume the vertices of $G_{F^*}$ connected by $e,e'$ 
have the same parity (are parallel). Using $A_m,A_{m+1}$ in place 
of $A',A''$, we apply the arguments in section 5 of \cite{GLi} 
in the case $(1) \, \epsilon = -1$. These show that these edges form
an \ESC\ on $G_{F^*}$ which we may assume is centered about the 
edges $A_m,A_{m+1}$ 
that form a \SC\ (else $S^3$ has an $RP^3$ summand). The edges of
this \ESC\ other than $A_m,A_{m+1}$ then come in pairs, forming 
disjoint simple closed curves on $\hatQ$. 
Some innermost pair of these edges then bounds a disk in $G_Q$  
(since the edges $A',A''$ connect the remaining vertices $\{m,m+1\}$).
The argument of \cite{GLi}, after possibly surgering away simple closed curves
of intersection, implies that $K$ is a $(1,2)$-cable knot --- 
contradicting its hyperbolicity.

We next assume the vertices of $G_{F^*}$ connected by $e,e'$ 
have the opposite parity. Again, using $A_m,A_{m+1}$ in place 
of $A',A''$, we apply the arguments in section 5 of \cite{GLi} 
in the case $(2) \, \epsilon = 1$. 
The map $\pi$ partitions the arcs $A_1, \dots, A_{n_\beta}$ into orbits of equal cardinality of at least $2$. Since the surface $\hatQ$ is separating, the map $\pi$ must have an even number of orbits ($i \equiv \pi(i) \mod 2$ by the Parity Rule).  In particular, $A_m$ and $A_{m+1}$ belong to different orbits.  Each orbit $\theta$ , other than the ones containing $A_m$ and $A_{m+1}$, gives rise to a simple closed curve $C_\theta$ on $\hatQ$.  Exchanging $A_m, A_{m+1}$ for $A',A''$ merges the two orbits containing vertices $m,m+1$, giving rise to a single 
simple closed curve $C'$ on $\hatQ$.  All of these simple closed curves are mutually disjoint on $\hatQ$.  

If there is a simple closed curve other than $C'$ (i.e.\ if there are more than $2$ orbits of $\pi$) then there is one that is innermost on $\hatQ$; let this be the $C_\theta$ that is used to complete the proof in \cite{GLi}.

If $C'$ is the only simple closed curve then $e$ and $e'$ must be parallel on $G_Q$, bounding a disk in $G_Q$ whose interior is disjoint from $C'$.  
The argument of \cite{GLi} in case $(1)$ (above) applies to show that 
$K$ is a (1,2)-cable knot, a contradiction.
\end{proof}

\section{When $t=4$ and no SCC.} \label{sec:t4noscc}

In this section we assume that $t=4$ and that we are in \situationnscc.

We use configurations of bigons and trigons at a special vertex of $\Lambda$
to either produce a Dyck's surface in $M$ or to find a new genus two Heegaard
splitting of $M$ with respect to which $K$ has bridge number $0$ or $1$
(i.e.\ making $t=0$ or $t=2$).

$\Lambda$ cannot have $9$ mutually parallel edges by Lemma~\ref{lem:noBBBBB}.  
Therefore by Lemmas~\ref{lem:t4trulyspecial} and \ref{lem:N4trulyspecialtypes} 
there exists a special vertex $x$ in $\Lambda$ of type $[4\Delta-3]$, 
$[4\Delta-4,1]$, or $[4\Delta-5,4]$. Recall from 
section~\ref{sec:specialvertices} that
a special vertex, $x$, of $\Lambda$ is of type $[a,b]$ if,  
of the $4\Delta$ corners at $x$, $a$ belong to bigons of $\Lambda$
and $b$ belong to trigons of $\Lambda$. Nothing is known of the faces to
which the remaining corners belong, indeed these faces might not 
even belong to $\Lambda$. We refer to the corners of $x$ which 
belong to these latter faces as {\em true gaps} at $x$. Thus all but $a+b$
corners of $x$ are true gaps. We refer to those corners at $x$ as {\em gaps} 
which are not known to belong to bigons of $\Lambda$ at $x$ (i.e.\ the true
gaps as well as the $b$ corners that belong to trigons of $\Lambda$).
Thus, all but $a$ corners at $x$ are gaps.
In sequence around $x$ we label the faces in $\Lambda$ as follows:
\ttB: bigon, \ttS: an \SC, \ttM: mixed bigon, 
\ttT: trigon. A mixed bigon of $\Lambda$ is one that is not an \SC. 
We label as \ttg: gap, \ttG: true gap.
If {\tt ABC} and {\tt XYZ} are two disjoint subsequences of faces around 
a vertex, we write {\tt ABC+XYZ} to indicate coherent ordering 
(orientation) without assuming relative positions.

\begin{lemma}\label{lem:t4thendelta3}
If $t=4$ and \situationnscc\ then either $M$ contains a Dyck's surface
or $\Delta \leq 3$.
\end{lemma}

\begin{proof}
Assume $t=4$ and \situationnscc, that $\Delta \geq 4$ and $M$ does not
contain a Dyck's surface. As mentioned above,  
Lemma~\ref{lem:noBBBBB} along with Lemmas~\ref{lem:t4trulyspecial} and 
\ref{lem:N4trulyspecialtypes} guarantee that there exists a special vertex 
$x$ in $\Lambda$ of type $[4\Delta-3]$, $[4\Delta-4,1]$, or $[4\Delta-5,4]$. 

{\bf  $x$ has type $[4\Delta-3]$.} Since $\Delta \geq 4$, there must be 
five consecutive bigons around $x$.  This contradicts Lemma~\ref{lem:noBBBBB}. 

{\bf  $x$ has type $[4\Delta-4,1]$.}  First assume there are $4$ consecutive
bigons at $x$. By Lemma~\ref{lem:noBBBBB}, these must be flanked by two gaps.  
By relabeling we may assume these four bigons contain a $\arc{1234}$-\ESC.
By Lemma~\ref{lem:esc4} all bigons or trigons of $\Lambda$ at 
$\arc{23}$-corners must actually be $\arc{23}$-\SCs; furthermore,
the edges of any two such bigons must come in parallel pairs on $G_F$.
By Lemma~\ref{lem:parallelwithsamelabel} then, all but at most two 
$\arc{23}$-corners at $x$ are (true) gaps.  As there are four gaps at $x$, 
two of which are contiguous to the four consecutive bigons above and hence
are not $\arc{23}$-corners, $\Delta = 4$. By Lemma~\ref{lem:noBBBBB} the 
positions of these two remaining gaps at $\arc{23}$-corners is forced 
and there must exist a $\arc{3412}$-\ESC\ at $x$.  But then there are four 
$\arc{41}$-\SCs\ at $x$.  Together Lemmas~\ref{lem:esc4} and \ref{lem:parallelwithsamelabel} provide a contradiction.  

Since there cannot be four consecutive bigons at $x$, we must have $\Delta = 4$ and there must be four triples of bigons at $x$ separated by single gaps.  Since one of the gaps is actually a trigon, Lemma~\ref{lem:dbft2} implies that the adjacent triples of bigons are \ESC s.  Then Lemma~\ref{lem:esc4} implies all four triples are \ESC s.  Since their \SCs\ all have the same labels, Lemmas~\ref{lem:esc4} and \ref{lem:parallelwithsamelabel} provide a contradiction. 

 {\bf  $x$ has type $[4\Delta-5,4]$.}  If there is a, say, $\arc{1234}$-\ESC, then it must be adjacent to a true gap by Lemma~\ref{lem:esc4g}.  By Lemmas~\ref{lem:esc4} and \ref{lem:parallelwithsamelabel} there must also be a true gap at some $\arc{23}$-corner, but $x$ has only one true gap.  Hence there is no \ESC\ at $x$.
 
 Since $\Delta \geq 4$ and there is no \ESC, there must appear {\tt BgSMSgB} around $x$.   Because there is only one true gap at $x$, applying Lemma~\ref{lem:L1} to those gaps that are actually trigons implies there are at most 10 corners around $x$, a contradiction.
\end{proof}

\begin{thm}\label{prop:tnot4}
If $\Delta \geq 3$, $t=4$, and \situationnscc\ 
then $M$ contains a Dyck's surface.

\end{thm}
\begin{proof}
Assume $t=4$ and \situationnscc, that $\Delta \geq 3$ and $M$ does not
contain a Dyck's surface. As mentioned above,  
Lemma~\ref{lem:noBBBBB} along with Lemmas~\ref{lem:t4trulyspecial} and 
\ref{lem:N4trulyspecialtypes} guarantee that there exists a special vertex 
$x$ in $\Lambda$ of type $[4\Delta-3]$, $[4\Delta-4,1]$, or $[4\Delta-5,4]$. 

By Lemma~\ref{lem:t4thendelta3} and our hypothesis, $\Delta =3$.
But then Theorems~\ref{thm:thm3} and \ref{thm:3.5} contradict each other.
\end{proof}

\subsection{The lemmas to complete $t=4$ and \situationnscc}

The goal of this section is to finish the proof of Theorem~\ref{prop:tnot4}
by proving Theorems~\ref{thm:thm3} and \ref{thm:3.5}. So for this subsection
we assume $t=4$, \situationnscc, $M$ contains no Dyck's surface, and $\Delta=3$. Thus
a special vertex of $\Lambda$ is one of type $[4], [8,1],$ or $[7,4]$.

\begin{lemma}%[L10]
\label{lem:noBBBB}
If $\Delta=3$ then a special vertex $x$ of $\Lambda$ cannot have {\tt BBBB}.
\end{lemma}

\begin{proof}
By Lemma~\ref{lem:noBBBBB} there cannot be five bigons in a row. 
By Lemma~\ref{lem:noBBBBT} there cannot be a trigon adjacent to these four bigons.  Hence four consecutive bigons must be flanked by true gaps.  Thus, assuming a special vertex $x$ of $\Lambda$ has {\tt BBBB}, it has type $[9]$ or $[8,1]$.

Up to relabeling, we may assume for this vertex we have 
\begin{center}
\begin{tabular}{c}
\begin{tabular}[h]{*{12}{c}}
1&2&3&4&1&2&3&4&1&2&3&4
\end{tabular}\\
\begin{tabular}[h]{c|*{11}{c|}c}
\hline
\ttG&\ttM &\ttS&\ttM &\ttS & \ttG & \textcolor{white}{\ttG}&\textcolor{white}{\ttG} 
&\textcolor{white}{\ttG}&\textcolor{white}{\ttG}&\textcolor{white}{\ttG}&\textcolor{white}{\ttG} &\textcolor{light-gray}{\ttG}\\
\hline
\end{tabular}
\end{tabular}
\end{center}
By Lemma~\ref{lem:esc4}  the remaining two $\arc{23}$-corners each have a \ttG\ or \ttS\ and the remaining $\arc{41}$-corner has a \ttG, \ttS, or Scharlemann
cycle \ttT. Lemmas~\ref{lem:parallelwithsamelabel} and ~\ref{lem:esc4} imply that one of these two $\arc{23}$-corners must have the last \ttG\ so that the $\arc{41}$-corner has an \ttS\ or a Scharlemann cycle \ttT:
\begin{center}
\begin{tabular}{lc}
&
\begin{tabular}[h]{*{12}{c}}
1&2&3&4&1&2&3&4&1&2&3&4
\end{tabular}\\
\begin{tabular}{c}
(1)\\
(2)
\end{tabular}
&

\begin{tabular}[h]{c|*{11}{c|}c}
\hline
\ttG&\ttM &\ttS&\ttM &\ttS & \ttG & \ttS&\textcolor{white}{\ttG} 
&\textcolor{white}{\ttG}&\textcolor{white}{\ttG}&\ttG &\textcolor{white}{\ttG} &\textcolor{light-gray}{\ttG}\\
\hline
\ttG&\ttM &\ttS&\ttM &\ttS & \ttG & \ttG&\textcolor{white}{\ttG} 
&\textcolor{white}{\ttG}&\textcolor{white}{\ttG}&\ttS &\textcolor{white}{\ttG} &\textcolor{light-gray}{\ttG}\\
\hline
\end{tabular}
\end{tabular}
\end{center}

If $x$ has type $[9]$, then the remaining corners must be bigons.  In line (2) above there must be five consecutive bigons, contrary to Lemma~\ref{lem:noBBBBB}.  In line (1) above the three bigons to the left of the {\tt G} (at $\arc{23}$) form an \ESC\ with labeling contrary to Lemma~\ref{lem:notwoescs}.

Thus $x$ has type $[8,1]$ and there must be a \ttT. First, assume 
this \ttT\ is at the remaining $\arc{41}$-corner, and hence is an \SC.
Then Lemma~\ref{lem:SMTM} contradicts both lines (1) and (2) above (where the
roles of labels $2,3$ and $4,1$ are interchanged). Thus we assume
this $\arc{41}$-corner must belong to an \ttS. Lemma~\ref{lem:notwoescs} 
implies that the \ttT\ must be adjacent to this \ttS. Since the bigon 
to the other side of this \ttS\ must be an \ttM, we have a \FESC\ 
whose presence violates Lemma~\ref{lem:esc+fesc}.
\end{proof}

\begin{thm}\label{thm:thm1}
If $\Delta=3$ then a special vertex $x$ of $\Lambda$ cannot have an \ESC.
\end{thm}

\begin{proof}
Assume there is an \ESC\ around $x$.
By Lemma~\ref{lem:esc4g} this \ESC\ must be adjacent to a true gap.  
WLOG we assume the \ESC\ is on the corner $\arc{1234}$ with the true gap 
to the left and, by Lemma~\ref{lem:noBBBB}, a gap to the right.

{\bf Case I.} The vertex $x$ has type $[7,4]$.

  By Lemma~\ref{lem:esc4}, all three $\arc{23}$ corners around $x$ have \SCs\ 
whose edges are parallel to that of the \ESC.  This gives a contradiction via Lemma~\ref{lem:parallelwithsamelabel}.

{\bf Case II.}  The vertex $x$ has type $[8,1]$ or $[9]$.

Up to relabeling we may assume we have one of the following:
\begin{center}
\begin{tabular}{lc}
&
\begin{tabular}[h]{*{12}{c}}
1&2&3&4&1&2&3&4&1&2&3&4
\end{tabular}\\
\begin{tabular}{c}
(1)\\
(2)
\end{tabular}
&

\begin{tabular}[h]{c|*{11}{c|}c}
\hline
\ttG&\ttM &\ttS&\ttM &\ttG & \textcolor{white}{\ttG} & \textcolor{white}{\ttG}&\textcolor{white}{\ttG} 
&\textcolor{white}{\ttG}&\textcolor{white}{\ttG}&\textcolor{white}{\ttG} &\textcolor{white}{\ttG} &\textcolor{light-gray}{\ttG}\\
\hline
\ttG&\ttM &\ttS&\ttM &\ttT & \textcolor{white}{\ttG} & \textcolor{white}{\ttG}&\textcolor{white}{\ttG} 
&\textcolor{white}{\ttG}&\textcolor{white}{\ttG}&\textcolor{white}{\ttG} &\textcolor{white}{\ttG} &\textcolor{light-gray}{\ttG}\\
\hline
\end{tabular}
\end{tabular}
\end{center}

By Lemmas~\ref{lem:esc4} and \ref{lem:parallelwithsamelabel} one of the remaining $\arc{23}$-corners has a \ttG\ and the other has an \ttS.  (If both $\arc{23}$-corners have a true gap then we are in (2) above, and the remaining corners belong to bigons
of $\Lambda$. Then by Lemma~\ref{lem:esc4} there is an \ESC\ containing a 
$\arc{41}$-\SC.  This contradicts Lemma~\ref{lem:notwoescs}.)  Furthermore Lemma~\ref{lem:esc4} implies the remaining $\arc{41}$-corner has a \ttG, \ttS, 
or Scharlemann cycle \ttT.

If $x$ is as in line (1) then one of the remaining $\arc{23}$-corners must take the last \ttG.  This gives a run of five yet to be accounted corners.  Without one of the last corners being a \ttT, there would be four consecutive bigons contrary to Lemma~\ref{lem:noBBBB}.  So $x$ cannot have type $[9]$ and must have type $[8,1]$.  Since the remaining $\arc{23}$-corner in this run of five must be an \ttS, the two possible placements of the \ttT\ give configurations {\tt GMSTSMG}  and {\tt GMSMTMG} (or {\tt GMTMSMG}) overlapping the original {\tt GMSMG} on one \ttG.  The former is forbidden by Lemma~\ref{lem:MSTS}.  The latter may be seen as a case of line (2).

We may now assume $x$ is as in line (2). Note that the pictured {\tt T} is a 
Scharlemann cycle by Lemma~\ref{lem:esc4}.
If the remaining $\arc{41}$-corner has an \ttS\ then Lemma~\ref{lem:noBBBB} implies that the final \ttG\ must be between this \ttS\ and the \ttS\ at whichever of the two remaining $\arc{23}$-corners.  In these two cases, the types of bigons may be determined at enough of the remaining corners for Lemma~\ref{lem:SMTM} to apply and be contradicted by the number of $\edge{23}$-edges at the vertex.  Therefore the remaining $\arc{41}$-corner has a \ttG.  Whichever of the remaining $\arc{23}$-corners gets an \ttS\ must then be flanked by \ttM\ bigons and all remaining bigons must be Black.  Hence we must have the configuration {\tt gMSMgMSMgBgB} which is forbidden by Lemma~\ref{lem:2ESC+2B}.
\end{proof}

\begin{thm}%[Thm2]  
\label{thm:thm2}
If $\Delta=3$ then at a special vertex $x$ of $\Lambda$ a triple of bigons must be adjacent to a true gap.
\end{thm}

\begin{proof}
Assume otherwise.  Then by Lemma~\ref{lem:noBBBB} we must have {\tt TBBBT} at $x$.  By Lemma~\ref{lem:esc4g}  we have {\tt TSMST}.   
Hence the vertex $x$ must have type $[7,4]$.  By symmetry we may assume we do not have the true gap immediately to the right, so that we have either  {\tt TSMSTB} or {\tt TSMSTT}.  
We then have the following possible configurations at $x$:
\begin{center}
\begin{tabular}{l c r}

\begin{tabular}{c}
(1)\\(2)\\(3)\\(4)\\(5)
\end{tabular}
&
\begin{tabular}[h]{c|*{7}{c|}c}
\hline
\textcolor{white}{\ttG}&\ttB &\ttT&\ttS &\ttM & \ttS & \ttT &\ttB  &\textcolor{white}{\ttG}\\
\hline
\textcolor{white}{\ttG}&\ttB &\ttT&\ttS &\ttM & \ttS & \ttT &\ttT  &\textcolor{white}{\ttG}\\
\hline\textcolor{white}{\ttG}&\ttT &\ttT&\ttS &\ttM & \ttS & \ttT &\ttT  &\textcolor{white}{\ttG}\\
\hline\textcolor{white}{\ttG}&\ttG &\ttT&\ttS &\ttM & \ttS & \ttT &\ttB  &\textcolor{white}{\ttG}\\
\hline\textcolor{white}{\ttG}&\ttG &\ttT&\ttS &\ttM & \ttS & \ttT &\ttT  &\textcolor{white}{\ttG}\\
\hline
\end{tabular}
\end{tabular}
\end{center}

Lines (1) and (4) contradict Lemma~\ref{lem:L1} (too many true gaps). Note,
as in all of these lemmas,
Lemma~\ref{lem:L1} applies equally well to the reverse ordering, {\tt BTSM}.

In line (2), the \ttB\ must be an \ttS\ by Lemma~\ref{lem:L1}.  There are three remaining corners of the same color as the \ttM\ shown.  Lemmas~\ref{lem:SMST}, \ref{lem:NFSC2}, and \ref{lem:L1} then imply the \ttG\ must be adjacent to the newly placed \ttS\  and the other two corners are filled with an \ttM\ and the last \ttT.  Regardless of this last choice, the remaining two corners are both \ttS s (Lemma~\ref{lem:fesc+1}).  Thus five corners of the same color have an \ttS.  The edge shared by the adjacent \ttT s is parallel to an edge of the \ttM\ shown in line (2) by Lemmas~\ref{lem:fesc+1} and \ref{lem:TB}.  Now Lemma~\ref{lem:3SC+1SC+3e} applies providing a contradiction to Lemma~\ref{lem:parallelwithsamelabel}.

In lines (3) and (5) there must be two more bigons the same color as the \ttM\ shown.  Lemma~\ref{lem:NFSC2} implies each of these must be an \ttM,  but this contradicts Lemma~\ref{lem:SMST}.
\end{proof}

\begin{thm}%[Thm3] 
\label{thm:thm3}
If $\Delta = 3$ then a special vertex $x$ of $\Lambda$ cannot contain a triple of bigons.
 \end{thm}

\begin{proof}
Assume there is {\tt BBB} at the special vertex $x$.   By Theorems~\ref{thm:thm1} and \ref{thm:thm2}, every {\tt BBB} is an {\tt SMS} adjacent to a gap.  In particular, we have {\tt GSMS}.

{\bf Case I.}  The vertex $x$ has type $[7,4]$.  Then by Lemma~\ref{lem:noBBBB} we must have {\tt GSMST}. Note that the true gap indicated is the only one 
at $x$. The \FESC\ must be type I by Lemma~\ref{lem:fesc+1}.  Lemma~\ref{lem:L1} implies we cannot have {\tt GSMSTB}.  Thus we must have {\tt GSMSTT}.  To the right of this there are $2$, $1$, or $0$ \ttB s before the next \ttT\ (Theorem~\ref{thm:thm2}).

{\bf Case Ia.} Assume we have {\tt GSMSTTBBT}.  If the {\tt BB} are {\tt SM} then we  contradict Lemma~\ref{lem:L5}.  If the {\tt BB} are {\tt MS} then by Lemmas~\ref{lem:dbft2} and \ref{lem:L5} the remaining three spots are filled with {\tt TMS}. Yet this contradicts Lemma~\ref{lem:SMST}.

{\bf Case Ib.}  Assume we have {\tt GSMSTTBT}.    The \ttB\ is the same color as the \ttM.  Since three of the remaining four positions get a bigon, one of these must be the same color as the \ttM\ too.  This however contradicts Lemma~\ref{lem:SMST} or \ref{lem:NFSC2}.

{\bf Case Ic.}  Assume we have {\tt GSMSTTT}.  Theorem~\ref{thm:thm2} permits only two positions for the final \ttT.
\begin{center}
\begin{tabular}{l c r}
\begin{tabular}{c}
(1)\\(2)
\end{tabular}
&
\begin{tabular}[h]{c|*{11}{c|}c}
\hline
\ttG&\ttS &\ttM&\ttS &\ttT & \ttT & \ttT 
&\textcolor{white}{\ttG} &\ttT&\textcolor{white}{\ttG}&\textcolor{white}{\ttG}&\textcolor{white}{\ttG} &\textcolor{light-gray}{\ttG}\\
\hline
\ttG&\ttS &\ttM&\ttS &\ttT & \ttT & \ttT 
&\textcolor{white}{\ttG} &\textcolor{white}{\ttG}&\ttT&\textcolor{white}{\ttG}&\textcolor{white}{\ttG} &\textcolor{light-gray}{\ttG}\\
\hline
\end{tabular}
\end{tabular}
\end{center}
Theorem~\ref{thm:thm1} labels the triple of bigons in line (1) as {\tt SMS}. Now Lemma~\ref{lem:L1} gives a contradiction.

In line (2), the bigons to each side of this last \ttT\ are the same color as the \ttM.  This contradicts Lemma~\ref{lem:SMST} or \ref{lem:NFSC2}.

{\bf Case II.}  The vertex $x$ has type $[8,1]$.

Thus we have four gaps (one trigon and three true gaps) and by 
Lemma~\ref{lem:noBBBB} we must have {\tt gSMSg} (where at least one of these
gaps is a true gap).  By Lemma~\ref{lem:SMSandSMS} we cannot have {\tt gSMSgSMSg}.  Thus, up to symmetry, we must have one of the following four configurations

\begin{center}
\begin{tabular}{lc}
&
\begin{tabular}[h]{*{12}{c}}
1&2&3&4&1&2&3&4&1&2&3&4
\end{tabular}\\
\begin{tabular}{c}
(1)\\
(2)\\
(3)\\
(4)
\end{tabular}
&

\begin{tabular}[h]{c|*{11}{c|}c}
\hline
\ttg&\ttS &\ttM&\ttS &\ttg & \ttg & \ttS&\ttM 
&\ttS&\ttg&\textcolor{white}{\ttG} &\textcolor{white}{\ttG} &\textcolor{light-gray}{\ttg}\\
\hline
\ttg&\ttS &\ttM&\ttS &\ttg & \textcolor{white}{\ttG}& \ttg&\textcolor{white}{\ttG}
&\textcolor{white}{\ttG}&\ttg&\textcolor{white}{\ttG} &\textcolor{white}{\ttG} &\textcolor{light-gray}{\ttg}\\
\hline
\ttg&\ttS &\ttM&\ttS &\ttg & \textcolor{white}{\ttG} & \ttg&\ttS 
&\ttM&\ttS&\ttg &\textcolor{white}{\ttG} &\textcolor{light-gray}{\ttg}\\
\hline
\ttg&\ttS &\ttM&\ttS &\ttg & \textcolor{white}{\ttG}&\textcolor{white}{\ttG}&\ttg&\textcolor{white}{\ttG}&\ttg &\textcolor{white}{\ttG}&\textcolor{white}{\ttG} &\textcolor{light-gray}{\ttg}\\
\end{tabular}
\end{tabular}
\end{center}

In line (1), filling in the last two blanks with either {\tt MS} or {\tt SM} produces a contradiction to Lemma~\ref{lem:dbfg}.  

In line (2), by Lemma~\ref{lem:dbft2} the initial \ttg\ must be \ttG\ so that the \ttT\ occurs at one of the remaining three.  Lemmas~\ref{lem:fesc+1} and \ref{lem:NFSC2} force configuration (i) below when the \ttT\ is 
at the second gap. When the \ttT\ is at the third gap,  
Lemmas~\ref{lem:fesc+1}, \ref{lem:NFSC2}, and \ref{lem:dbfg} force 
configurations (ii), (iii), and (iv) below. 
Lemma~\ref{lem:dbft2} determines the labelings of all the bigons but one if 
the \ttT\ is at the last \ttg, giving (v) below. 

\begin{center}
\begin{tabular}{lc}
&
\begin{tabular}[h]{*{12}{c}}
1&2&3&4&1&2&3&4&1&2&3&4
\end{tabular}\\
\begin{tabular}{c}
(i)\\
(ii)\\
(iii)\\
(iv)\\
(v)\\
\end{tabular}
&

\begin{tabular}[h]{c|*{11}{c|}c}
\hline
\ttG&\ttS &\ttM&\ttS &\ttT & \ttS& \ttG&\ttS
&\ttM&\ttG&\ttM &\ttS &\textcolor{light-gray}{\ttg}\\
\hline
\ttG&\ttS &\ttM&\ttS &\ttG & \ttS& \ttT&\ttS
&\ttM&\ttG&\ttM &\ttS &\textcolor{light-gray}{\ttg}\\
\hline
\ttG&\ttS&\ttM&\ttS&\ttG& \ttS & \ttT &\ttM &\ttS &\ttG &\ttS &\ttM
&\textcolor{light-gray}{\ttg}\\
\hline
\ttG&\ttS&\ttM&\ttS&\ttG& \ttM & \ttT &\ttM &\ttS &\ttG &\ttS &\ttM
&\textcolor{light-gray}{\ttg}\\
\hline
\ttG&\ttS &\ttM&\ttS &\ttG & \textcolor{white}{\ttB}& \ttG&\ttS
&\ttM&\ttT&\ttM &\ttS &\textcolor{light-gray}{\ttg}\\
\hline
\end{tabular}
\end{tabular}
\end{center}

Lemma~\ref{lem:3SC+1SC+3e} (with the roles of $1,2$ and $3,4$ 
interchanged) applies to each of configurations (i),(ii), 
and (v), giving a contradiction to Lemma~\ref{lem:parallelwithsamelabel}. 
(Notice that the \ttT\ in line (v) is an \SC.) For (iv), 
Lemma~\ref{lem:dbfg} implies that a neighborhood of 
the $\edge{41}$-edges of the {\tt MSGSM} subconfiguration is a 1-punctured
torus, and hence that the $\edge{23}$-edges of the $\arc{23}$-Scharlemann
cycle (bigon and
trigon) lie in 
a 1-punctured torus. This contradicts Lemma~\ref{lem:SMTM}. To eliminate
(iii), consider the two $\arc{12}$-\SCs, ${\tt S_1,S_2}$, in that 
configuration. The argument 
of Lemma~\ref{lem:scBsc+BB+BB} applied to the sub-configurations ${\tt S_1MS}$, 
{\tt SM}, and ${\tt S_2}$ (with labels $1,2,3,4$ relabelled $3,4,1,2$), 
shows that ${\tt S_1,S_2}$ are parallel 
bigons. Thus we
can think of the sub-configurations ${\tt S_1M,S_2T}$ together as
one \FESC. That is, applying the argument of Lemma~\ref{lem:nfsc} to these
faces and the $\arc{41},\!\arc{23}$-\SCs\ of (iii) shows that three
\mobius bands $A_{41},A_{23},A_{12}$ corresponding to these Scharlemann
cycle faces
can be perturbed to be disjoint. But Lemma~\ref{lem:3disjointmobiusbands}
contradicts that $M$ does not contain a Dyck's surface. We have eliminated
configurations (i)-(v), and line (2) does not occur.

In line (3), by symmetry we may assume the second \ttg\ (just to the left of the first blank) is actually \ttT\ and the other \ttg\ are all \ttG. 
Then Lemma~\ref{lem:L1} implies that the first blank is an \ttS\ and the
contiguous \FESC\ is of type I.
Applying Lemma~\ref{lem:SMST} to this \FESC\ and the mixed bigon in the
remaining {\tt SMS} configuration contradicts 
Lemma~\ref{lem:parallelwithsamelabel}.

In line (4) we examine where the \ttT\ may go.  It cannot be either of the first two gaps (at the $\arc{41}$-corners) by Lemma~\ref{lem:dbft2}.  So without loss of generality assume the trigon is the third gap (at the $\arc{34}$-corner). Now we have two cases according to whether the bigon between the trigon and fourth gap is \ttM\ or \ttS.  

If it is \ttM, then to the left of the trigon there must also be an \ttM.  Otherwise we must have {\tt MSTM} contradicting Lemma~\ref{lem:fesc+1}.  Thus around the trigon we have {\tt SMTM}.   But now the \SCs\ of the {\tt SMS} provide a 
configuration contrary to Lemma~\ref{lem:SMTM} (where the 
$\arc{34}$-corners and 
$\edge{12}$-edges play the role of the $\arc{23}$-corners and 
$\edge{41}$-edges of the Lemma).

If it is \ttS, then it is a  $\arc{41}$-\SC\ and we can apply Lemma~\ref{lem:dbfg}(4) to conclude we obtain {\tt MSTSGSM}.
But then Lemma~\ref{lem:NFSC2} gives a contradiction.

{\bf Case III.}  The vertex $x$ has type $[9]$.

Since there are no {\tt MSM} and no string of four bigons, 
there is just one configuration:
\begin{center}
\begin{tabular}{c}
\begin{tabular}[h]{*{12}{c}}
1&2&3&4&1&2&3&4&1&2&3&4
\end{tabular}\\
\begin{tabular}[h]{c|*{11}{c|}c}
\hline
\ttG&\ttS &\ttM&\ttS &\ttG& \ttS & \ttM 
&\ttS &\ttG&\ttS&\ttM&\ttS &\textcolor{light-gray}{\ttG}\\
\hline
\end{tabular}
\end{tabular}
\end{center}
This configuration is forbidden by Lemma~\ref{lem:3SC+1SC+3e} and 
Lemma~\ref{lem:parallelwithsamelabel}.
\end{proof}
 
\begin{thm}%[Thm3.5]
\label{thm:3.5}
If $\Delta = 3$, then a special vertex $x$ of $\Lambda$ must contain a triple of bigons.  
\end{thm}
\begin{proof}
Assume there is no {\tt BBB} at $x$.  Then there can be neither {\tt ggg} nor {\tt gg+gg} at $x$, or else there would be a {\tt BBB} at $x$.

Without loss of generality we may assume a $\arc{41}$-corner of $x$ has a \ttG, a true gap.  We will use this \ttG \  to mark the beginning and end of the sequence of faces around $x$ as follows:
\begin{center}
\begin{tabular}[h]{*{12}{c}}
1&2&3&4&1&2&3&4&1&2&3&4
\end{tabular}

\begin{tabular}[h]{c|*{11}{c|}c}
\hline
\ttG&\textcolor{white}{\ttT} &\textcolor{white}{\ttT} &\textcolor{white}{\ttT} &\textcolor{white}{\ttT} &\textcolor{white}{\ttT} &\textcolor{white}{\ttT} &\textcolor{white}{\ttT} &\textcolor{white}{\ttT} &\textcolor{white}{\ttT} &\textcolor{white}{\ttT} &\textcolor{white}{\ttT}  &\textcolor{light-gray}{\ttG}\\
\hline
\end{tabular}
\end{center}

The light grey \ttG \ at the end is a repeat of the initial \ttG.

{\bf Case I.}  The vertex has type $[7,4]$ and there exists \ttT\ttT.

We enumerate the possibilities for the placement of this pair up to symmetry:

\begin{center}
\begin{tabular}{l c r}
&
\begin{tabular}[h]{*{12}{c}}
1&2&3&4&1&2&3&4&1&2&3&4
\end{tabular}
& \\
\begin{tabular}{c}
(1)\\
(2)\\
(3)\\
(4)\\
(5)
\end{tabular}
&
\begin{tabular}[h]{c|*{11}{c|}c}
\hline\hline
\ttG&\ttT &\ttT &\textcolor{white}{\ttT} &\textcolor{white}{\ttT} &\textcolor{white}{\ttT} &\textcolor{white}{\ttT} &\textcolor{white}{\ttT} &\textcolor{white}{\ttT} &\textcolor{white}{\ttT} &\textcolor{white}{\ttT} &\textcolor{white}{\ttT}  &\textcolor{light-gray}{\ttG}\\
\hline
\ttG&\textcolor{white}{\ttT} &\ttT &\ttT  &\textcolor{white}{\ttT} &\textcolor{white}{\ttT} &\textcolor{white}{\ttT} &\textcolor{white}{\ttT} &\textcolor{white}{\ttT} &\textcolor{white}{\ttT} &\textcolor{white}{\ttT} &\textcolor{white}{\ttT}  &\textcolor{light-gray}{\ttG}\\
\hline
\ttG&\textcolor{white}{\ttT} &\textcolor{white}{\ttT} &\ttT  &\ttT &\textcolor{white}{\ttT} &\textcolor{white}{\ttT} &\textcolor{white}{\ttT} &\textcolor{white}{\ttT} &\textcolor{white}{\ttT} &\textcolor{white}{\ttT} &\textcolor{white}{\ttT}  &\textcolor{light-gray}{\ttG}\\
\hline
\ttG&\textcolor{white}{\ttT} &\textcolor{white}{\ttT} &\textcolor{white}{\ttT} &\ttT  &\ttT  &\textcolor{white}{\ttT} &\textcolor{white}{\ttT} &\textcolor{white}{\ttT} &\textcolor{white}{\ttT} &\textcolor{white}{\ttT} &\textcolor{white}{\ttT}  &\textcolor{light-gray}{\ttG}\\
\hline
\ttG&\textcolor{white}{\ttT} &\textcolor{white}{\ttT} &\textcolor{white}{\ttT} &\textcolor{white}{\ttT} &\ttT  &\ttT &\textcolor{white}{\ttT} &\textcolor{white}{\ttT} &\textcolor{white}{\ttT} &\textcolor{white}{\ttT} &\textcolor{white}{\ttT}  &\textcolor{light-gray}{\ttG}\\
\hline
\end{tabular}
&
\end{tabular}
\end{center}

In each line two more \ttT s must be placed with the remaining being \ttB s.  Any placement of these two \ttT s in lines (1) and (4) contradicts having no {\tt BBB}.  In line (2), having no {\tt BBB} forces the placement of the remaining two \ttT s.   One application of Lemma~\ref{lem:dbft2} to the bigons around the central \ttT\ renders the following:
\begin{center}
\begin{tabular}{l c r}
\begin{tabular}{c}
(2)
\end{tabular}
&
\begin{tabular}[h]{c|*{11}{c|}c}
\hline
\ttG&\ttB &\ttT &\ttT &\ttS & \ttM & \ttT &\ttM &\ttS &\ttT &\ttB&\ttB  &\textcolor{light-gray}{\ttG}\\
\hline
\end{tabular}
\end{tabular}
\end{center}
But now a second application of Lemma~\ref{lem:dbft2} around the rightmost \ttT\ gives a contradiction.

For line (5), one \ttT \ must be among the leftmost four spots and the other must be in the middle of the rightmost five.  Lemma~\ref{lem:dbft2} gives the labeling of these rightmost five as {\tt SMTMS}.  Note that this \ttT \ is a Scharlemann cycle.   The two possibilities for the leftmost four are:  (i) {\tt BBTB} and (ii) {\tt BTBB}.  Labeling (i) as {\tt MSTB} contradicts Lemma~\ref{lem:L1} so it must be labeled as {\tt SMTB}.  But now the  {\tt SMTM} (reading right to left) along
with the $\arc{12}$-\SC\ call upon Lemma~\ref{lem:SMTM} to contradict that there
are two more $\edge{34}$-edges at $x$. Hence we must have (ii).  Labeling it as {\tt BTMS} forms {\tt MSTTSM} which contradicts Lemma~\ref{lem:L5}.  Thus it must be labeled as {\tt BTSM} giving us, by Lemma~\ref{lem:L1}, the following configuration.  Lemma~\ref{lem:3SC+1SC+3e} (with the roles of labels $1,2$ and $3,4$ interchanged) then gives three parallel edges that provide a contradiction to Lemma~\ref{lem:parallelwithsamelabel}.

\begin{center}
\begin{tabular}{l c r}
\begin{tabular}{c}
(5)
\end{tabular}
&
\begin{tabular}[h]{c|*{11}{c|}c}
\hline
\ttG&\ttS &\ttT &\ttS &\ttM & \ttT & \ttT &\ttS &\ttM &\ttT &\ttM&\ttS  &\textcolor{light-gray}{\ttG}\\
\hline
\end{tabular}
\end{tabular}
\end{center}

For line (3) both remaining \ttT s must be on the right side, and there are three possible placements.  The two bigons on the left are either {\tt MS} or {\tt SM}.

\begin{center}
\begin{tabular}{l c r}
&
\begin{tabular}[h]{*{12}{c}}
1&2&3&4&1&2&3&4&1&2&3&4
\end{tabular}
& \\
\begin{tabular}{c}
(3)
\end{tabular}
&
\begin{tabular}[h]{c|*{11}{c|}c}
\hline\hline
\ttG&\ttM &\ttS &\ttT &\ttT & \textcolor{white}{\ttM}  & \ttT &\ttS &\ttM &\ttT &\ttM&\ttS  &\textcolor{light-gray}{\ttG}\\
\hline
\ttG&\ttM &\ttS &\ttT &\ttT & \ttS & \ttM &\ttT &\textcolor{white}{\ttM} &\ttT &\ttM&\ttS  &\textcolor{light-gray}{\ttG}\\
\hline
\ttG&\ttM &\ttS &\ttT &\ttT & \ttM & \ttS &\ttT&\ttB &\textcolor{white} \ttB &\ttT&\textcolor{white}{\ttM} &\textcolor{light-gray}{\ttG}\\
\hline
\ttG&\ttS &\ttM &\ttT &\ttT & \textcolor{white}{\ttM} & \ttT &\ttS &\ttM &\ttT &\ttM&\ttS  &\textcolor{light-gray}{\ttG}\\
\hline
\ttG&\ttS &\ttM &\ttT &\ttT & \ttS & \ttM &\ttT &\textcolor{white}{\ttS} &\ttT &\ttM&\ttS  &\textcolor{light-gray}{\ttG}\\
\hline
\ttG&\ttS &\ttM &\ttT &\ttT & \ttS & \ttM &\ttT &\ttM &\ttS &\ttT&\ttS  &\textcolor{light-gray}{\ttG}\\
\hline
\end{tabular}
&
\begin{tabular}{c}
(i)\\
(ii)\\
(iii)\\
(iv)\\
(v)\\
(vi)
\end{tabular}
\end{tabular}
\end{center}
For lines (i) and (iv), first apply Lemma~\ref{lem:dbft2} to get the pictured 
configurations, then note these contradict Lemma~\ref{lem:L1}.  To get the
configuration for line (ii) apply Lemma~\ref{lem:L1}. But this line contradicts Lemma~\ref{lem:L5}.    For line (iii), apply Lemma~\ref{lem:L5} to get the 
pictured configuration. This now contradicts Lemma~\ref{lem:L1}.

For line (v) apply Lemma~\ref{lem:L1} twice to obtain the pictured 
configuration. Then Lemma~\ref{lem:fesc+1} applied to the {\tt TSM} adjacent to the leftmost \ttT\ implies the leftmost \ttT\ is a Scharlemann cycle. Lemma~\ref{lem:TB} shows that
the $\edge{34}$-edges on either end of the FESC {\tt TSM} are parallel in
$G_F$. This parallelism allows us to apply the argument of Lemma~\ref{lem:SMTM} 
to the subconfiguration {\tt SMT} on the left-hand side and the middle 
{\tt M} (for the analog of {\tt SMTM}) along with $\arc{34}$-\SC\ 
at the right. But then the five $\edge{12}$-edges incident to $x$ contradict
the conclusion of that argument.
  
For line (vi), apply Lemma~\ref{lem:dbft2} and then Lemma~\ref{lem:L1} to obtain the configuration shown. Furthermore, Lemma~\ref{lem:L1} shows that the first
{\tt T} must be a Scharlemann cycle. Lemma~\ref{lem:3SC+1SC+3e} then gives three parallel edges that provide a contradiction to Lemma~\ref{lem:parallelwithsamelabel}.

{\bf Case II.} The vertex has type $[7,4]$ and no {\tt TT}.

{\bf Case IIa.}  Assume there exists {\tt MST}.

Lemma~\ref{lem:L6} gives three configurations to consider.

\begin{center}
\begin{tabular}{l c r}
&
\begin{tabular}[h]{*{12}{c}}
1&2&3&4&1&2&3&4&1&2&3&4
\end{tabular}
& \\
\begin{tabular}{c}
(1)\\
(2)\\
(3)\\
\end{tabular}
&
\begin{tabular}[h]{c|*{11}{c|}c}
\hline\hline
\ttG& \textcolor{white}{\ttM}& \textcolor{white}{\ttM}& \textcolor{white}{\ttM}& \textcolor{white}{\ttM}& \textcolor{white}{\ttM}& \textcolor{white}{\ttM} &\ttB &\ttT & \ttM & \ttS & \ttT &\textcolor{light-gray}{\ttG}\\
\hline
\ttG& \textcolor{white}{\ttM}& \textcolor{white}{\ttM}& \textcolor{white}{\ttM}& \textcolor{white}{\ttM}& \textcolor{white}{\ttM} &\ttB &\ttT & \ttM & \ttS & \ttT &\ttS &\textcolor{light-gray}{\ttG}\\
\hline\ttG& \textcolor{white}{\ttM} & \textcolor{white}{\ttM}& \textcolor{white}{\ttM}& \textcolor{white}{\ttM}&\ttB &\ttT & \ttM & \ttS & \ttT & \ttS &\tt T &\textcolor{light-gray}{\ttG}\\
\hline
\end{tabular}
\end{tabular}
\end{center}
Line (1) must have either {\tt TB} or {\tt BT} to the left of the labeled faces.  
In the former, since no {\tt TT}, we then have {\tt BTBBTB} which is either {\tt BTSMTB} or {\tt BTMSTB}.  The positions of each of these is incongruous with Lemma~\ref{lem:L1}.
In the latter, the placement of the fourth \ttT\ is forced and Lemma~\ref{lem:dbft2} gives:
\begin{center}
\begin{tabular}[h]{c|*{11}{c|}c}
\hline
\ttG& \ttS& \ttM& \ttT& \ttM& \ttS& \ttT &\ttB &\ttT & \ttM & \ttS & \ttT &\textcolor{light-gray}{\ttG}\\
\hline
\end{tabular}
\end{center}
The {\tt MSTB} here contradicts Lemma~\ref{lem:L1} too.

In line (2), Lemma~\ref{lem:NFSC2} forces the \ttB\ to be an \ttM.  Then Lemmas~\ref{lem:NFSC2} and \ref{lem:MSTS} imply that the remaining $\arc{23}$ and $\arc{41}$ corners take the final two \ttT s.  Hence the remaining three corners have bigons.  Being adjacent to an \ttM, the center $\arc{12}$ corner takes an \ttS.    Lemma~\ref{lem:L1} now gives a contradiction.

In line (3) the placement of the fourth \ttT\ is forced.  Lemma~\ref{lem:dbft2} then gives a labeling contrary to Lemma~\ref{lem:L1}.

{\bf Case IIb.}  There exists no {\tt MST}.

There must be {\tt BB}, and it must be between the \ttG\ and a \ttT.
Hence {\tt BB} must be either {\tt GSMT} or {\tt TMSG}.  This forces:
\begin{center}
\begin{tabular}[h]{*{12}{c}}
1&2&3&4&1&2&3&4&1&2&3&4
\end{tabular}\\
\begin{tabular}[h]{c|*{11}{c|}c}
\hline
\ttG&\ttS&\ttM&\ttT&\ttB&\ttT&\ttB&\ttT&\ttB&\ttT&\ttM&\ttS &\textcolor{light-gray}{\ttG}\\
\hline
\end{tabular}
\end{center}

If the left and the right \ttB\ are both an \ttS\  then we have the configuration of Figure~\ref{fig:caseIIb}.
Lemma~\ref{lem:dbfg} shows that no pair of $e_1$, \dots, $e_6$ are parallel on $G_F$.  But Lemma~\ref{lem:2SC2T} shows that either there is a parallelism among  $e_1$, $e_4$, and $e_5$ or there is a parallelism among $e_2$, $e_3$, and $e_6$.  Hence WLOG the leftmost \ttB\ is an \ttM.

\begin{figure}[h]
\centering
\input{caseIIb.pstex_t}
\caption{}
\label{fig:caseIIb}
\end{figure}

If at least one of the two remaining \ttB s is an \ttM, then it will provide a fifth $\edge{12}$-edge incident to $x$.  Then we will have a configuration contrary to Lemma~\ref{lem:SMTM}.  If the remaining \ttB s are each an \ttS, then we may apply Lemma~\ref{lem:dbfg} to conclude that the $\edge{34}$-edges lie in a 
3-punctured torus on $\hatF$ (since the $\edge{12}$-edges fill out a 1-punctured
torus). This contradicts again Lemma~\ref{lem:SMTM}.

{\bf Case III.} The vertex has type $[8,1]$.

There are $8$ bigons and $4$ gaps.
There cannot be {\tt gBg} since this would imply the existence of {\tt BBB}.  Hence the gaps are equally spaced.  Due to symmetry, we may designate any of these gaps to be the \ttT\ and the others to be \ttG s.  Apply Lemma~\ref{lem:dbft2} to label the two pairs of bigons around the \ttT. If any of the remaining 
bigons are $\arc{12}$-\SCs, then the resulting configuration will contradict 
Lemma~\ref{lem:SMTM} (with too many $\edge{34}$-edges). This leaves the configuration shown in Figure~\ref{fig:81caseIII}.
\begin{figure}[h!]
\centering
\input{81caseIII.pstex_t}
\caption{}
\label{fig:81caseIII}
\end{figure}
The faces $f_1, \dots, f_6$ give the configuration and labeling in Figure~\ref{fig:dbfgconfig} where $f_5$ is now a gap and the label $a$ is now $x'$.  The remaining three bigons are labeled $f_7, f_8, f_9$ and the trigon is labeled $g$.   

Let $A_{34}$ be the Black \mobius band arising from $f_6$.  Let $A_{12,34}$ be the Black annulus arising from $f_1$ and $f_4$.  By Lemma~\ref{lem:dbfg}, $A_{34}$ and $A_{12,34}$ intersect transversely along the $\arc{34}$-arc of $K$.  The subgraph of $G_F$ arising from their edges is shown in Figure~\ref{fig:dbfggraph3}.  As a neighborhood of this subgraph is a twice punctured torus, Lemma~\ref{lem:dbfg} (1) implies its complement in $\hatF$ is an annulus $B_{\hatF}$.
Furthermore,  the $\arc{34}$-arc of $K$ has a bridge disk otherwise disjoint from $A_{12} \cup A_{12,34}$ that meets this annulus along a single spanning arc.  Hence cutting $H_B$ open along $N=\nbhd(A_{12} \cup A_{12,34})$ forms a solid torus $\calT$ on which $B_{\hatF}$ is a longitudinal annulus.

Since $g$ is a properly embedded disk in $\calT$, and $\bdry g$ crosses 
three times along the impressions of the $1$-handle neighborhood of 
$\arc{12}$, $g$ must be a meridional disk of $\calT$. 
From this, one can see that two
of the edges of $g$ must be parallel into $\bdry B_{\hatF}$ and the third runs
from one component of $\bdry B_{\hatF}$ to the other. In particular, the two
edges that are boundary parallel cobound contiguous squares in $G_{\hatF}$ as
pictured in Figure~\ref{fig:81caseIIIg}(b). The labellings of the
endpoints of edges $e_7,e_8$ on vertices $3,4$ force one of $e_3$ or $e_4$
to lie in one of these squares in $G_F$. 
But this means that either $e_3$ is parallel
to one of $e_2,e_6$ or $e_4$ is parallel to one of $e_1,e_5$, contradicting
Lemma~\ref{lem:dbfg}.

\begin{figure}[h!]
\centering
\input{81caseIIIg.pstex_t}
\caption{}
\label{fig:81caseIIIg}
\end{figure}

{\bf Case IV.}  The vertex has type $[9]$.

There must be a triple of bigons contrary to hypothesis.
\end{proof}

%\clearpage

%\newpage
% \clearpage

\section{Thrice-punctured spheres, forked extended Scharlemann cycles, and an application when $t=4$.}

In this section we assume that $t=4$ and that we are in \situationnscc.

\subsection{Thrice-punctured spheres in genus $2$ handlebodies}

Let $P$ be an incompressible, separating, thrice-punctured sphere in a genus $2$ handlebody $H$.    Decompose $H$ along $P$ as $H=M_1 \cup_P M_2$.  Since $P$ is an incompressible surface in the handlebody $H$, it must be $\bdry$-compressible.  Assume $P$ is $\bdry$-compressible into $M_2$. It is easy to see that $M_1$
and $M_2$ are genus $2$ handlebodies.

\begin{lemma}\label{lem:3Psphere}
\[ M_2 = P \times [0,1] \cup_{A_1 \cup A_2} \calT \]
where $P$ is identified with $P \times \{0\}$, $A_1$ and $A_2$ are disjoint, non-isotopic, essential annuli in $P \times \{1\}$, and either
\begin{enumerate}
\item $\calT$ is the union of two solid tori,  $\calT_1$ and $\calT_2$, and  $A_i \subset \bdry \calT_i$ is incompressible for each $i=1,2$; or,
\item  $\calT$ is a solid torus and $A_1 \cup A_2 \subset \bdry \calT$ is incompressible.
\end{enumerate}
In either case, if $A_i$ is not longitudinal in $\calT$ ($\calT_i$), then the component $c_i$ of $\bdry P$ that is isotopic (through $P \times [0,1]$) to the core of $A_i$ is primitive in $M_1$.
\end{lemma}

\begin{proof}
Let $\delta$ be a $\bdry$-compressing disk for $P$ in $M_2$.  The $\bdry$-compression of $P$ along $\delta$ yields $A$, an incompressible annulus or pair of annuli. Cutting $M_2$ along $\delta$ yields $\calT$,
either one or two solid tori, with $A \subset \bdry \calT$. In the case that $A$ is a single annulus,
then $\calT$ is a single solid torus. Reversing the compression along $\delta$, we see    
$M_2 = P \times [0,1] \cup_A \calT$. To get description (1) above, set $\calT_1=\calT, A_1=A$ and pick
a second disjoint, essential annulus, $A_2 \subset P \times \{1\}$ along which we attach a solid torus 
$\calT_2$ longitudinally (giving a trivial decomposition). Otherwise, $A= A_1 \cup A_2$ and reversing the
compression along $\delta$ gives either (1) or (2).

To prove the final statement, collapse $M_2$ along $P \times [0,1]$ to write $H$ as $M_1 \cup_{A_1 \cup 
A_2} \calT$ and let $c_i$ be the core of $A_i$. $A_i$ must boundary compress in $H$, but if $A_i$ is not
longitudinal in $\calT$ such a compression can be taken disjoint from $\Int \calT$. This gives a meridian
disk in $M_1$ that marks $c_i$ as primitive in $M_1$.
\end{proof}

\subsection{Forked extended Scharlemann cycles}

\begin{figure}
\centering
\input{FESC.pstex_t}
\caption{}
\label{fig:FESC}
\end{figure}

Consider a \FESC\ $\tau$ in $G_Q$. Up to relabelling vertices of $G_F$ and $G_Q$, we may assume it is as illustrated in Figure~\ref{fig:FESC}(a). As shown, label the two Black faces $f$ and $g$.  Also label and orient the two edges $\alpha$ and $\beta$.  
The subgraph of $G_F$ induced by the edges of $\tau$ then appears on $\hatF$ as shown in Figure~\ref{fig:FESC}(b) or its mirror. 
We assume in this subsection that we are in \situationnscc, so that $f,g$ are properly embedded in
$H_B -\nbhd(K)$. The White \SC\ between $f$ and $g$ gives rise to a 
\mobius band, $A_{23}$, properly embedded in $H_W$.

 Contracting the remaining three edges of $\tau$ to a point $*$ in this subgraph of $G_F$, we may view the edges $\alpha$ and $\beta$ as oriented loops.  A neighborhood on $\hatF$ of this subgraph of $G_F$ induced by the edges of $\tau$ is a $3$-punctured sphere $P'$.  Its boundary components may be identified with the loops $\alpha$, $\beta$, and $\alpha\beta$ based at $*$ also indicated in Figure~\ref{fig:FESC}(b).   We aim to show that $\beta$ bounds a disk in $\hatF$
(Lemma~\ref{lem:TB}).

Form the genus $2$ handlebody $M_1 = \nbhd (\arc{12} \cup \arc{34} \cup f \cup g) \subset H_B$.  Since $\bdry M_1 \cap \bdry H_B = P'$, $P = \bdry M_1 \cut \bdry H_B$ is also a $3$-punctured sphere.  Thus we may write $\bdry M_1 = P \cup_{\{\alpha, \beta, \alpha\beta\}} P'$  and $H_B = M_1 \cup_P M_2$.  Figure~\ref{fig:FESC}(a) gives instructions for the assembly of $M_1$ which we may realize embedded in $S^3$ as in Figure~\ref{fig:thepictureofM1} with $f$ and $g$ thickened. Note that one may thus visualize $M_1$ as the trefoil complement with the neighborhood of an unknotting tunnel removed: $\alpha\beta$ is the cocore of the unknotting tunnel and $\alpha$ and $\beta$ result from a banding of $\alpha\beta$ to itself. Recall that if $O$ is a 3-manifold with boundary and $\gamma$ is a curve
in $\bdry O$, then $O\langle \gamma \rangle$ is $O$ with a 2-handle attached along $\gamma$.

\begin{figure}
\centering
\input{forkedcomplex.pstex_t}
\caption{}
\label{fig:thepictureofM1}
\end{figure}

\begin{figure}
\centering
\input{M1attachalphabeta.pstex_t}
\caption{}
\label{fig:M1attachalphabeta}
\end{figure}

\begin{claim}\label{claim:primitivity}
Let $\alpha, \beta, \alpha \beta \subset M_1$ be as above.
\begin{enumerate}
\item $\alpha$ and $\beta$ are primitive in $M_1$. Indeed $M_1$ contains disjoint meridian disks, 
one intersecting $\alpha$ once and disjoint from $\beta$, the other disjoint from $\alpha$ and intersecting $\beta$ once.
\item The arcs $\arc{12},\!\arc{34}$ of $K \cap M_1$ can be isotoped in $M_1$, fixing their
endpoints, to arcs on $\bdry M_1$ that are disjoint from $\alpha$ and $\beta$ and that intersect
$\bdry A_{23} \subset \bdry M_1$ only in their endpoints (at vertices $2,3$ of $G_F$). Furthermore, these
arcs are incident to the same side of $\bdry A_{23}$ in $\bdry M_1$ .
\item $M_1\langle \alpha\beta \rangle$ is homeomorphic to the exterior of the trefoil. In particular,
$\alpha\beta$ is neither primitive nor cabled in $\bdry M_1$.
\end{enumerate}

\end{claim}
\begin{proof}
In Figures~\ref{fig:FESC} and \ref{fig:thepictureofM1}, $\alpha$ and $\beta$ encircle the two visible 
holes of the embedded genus $2$ handlebody indicated by Figure~\ref{fig:thepictureofM1}. 
Let $f,g$ be the black faces of $\tau$ (Figure~\ref{fig:FESC}). 
Let $e_1,e_2$ be disjoint properly embedded arcs in $g$ 
parallel to the $\arc{34}$-corners of $g$ along vertices $x,z$ (resp.) of $G_Q$. 
In $M_1$, a product neighborhood of $e_1$ ($e_2$) is a disk $E_1$ ($E_2$) in $g \times I$ such that 
$\bdry E_1$ ($\bdry E_2$) intersects $\alpha$ ($\beta$) once but is disjoint from $\beta$ ($\alpha$). 
$E_1$ and $E_2$ verify (1).

Let $E_1'$ be
the disk component of $E_1 \cut e_1$ that is disjoint from $\alpha$. Band $E_1'$ to the $\arc{34}$-corner of $g$
(to which $e_1$ is parallel) to 
obtain a bridge disk $D_{34}$ in $M_1$ for the arc $\arc{34}$ of $K$. $D_{34}$ is disjoint from both
$\alpha$ and $\beta$ and intersects $\bdry A_{23}$ only in vertex $3$ of $G_F$. 
Band the $\arc{12}$-corner of $f$ along $f$ to $D_{34}$ to obtain a bridge
disk $D_{12}$ of $\arc{12}$ of $K$ in $M_1$ which is disjoint from $D_{34}$, $\alpha$, and $\beta$;
and which intersects $\bdry A_{23}$ only at vertex $2$. $D_{12},D_{34}$ guide the isotopies of $\arc{12},
\arc{34}$ described in (2).

Figure~\ref{fig:M1attachalphabeta} shows that $M_1\langle \alpha\beta \rangle$ is homeomorphic to the exterior 
of the trefoil.  Since this is neither a solid torus nor the connect sum of a solid torus and a lens space, 
$\alpha\beta$ cannot be primitive or cabled in $M_1$.
\end{proof}

\begin{claim}\label{claim:Pincomp}
$P$ is incompressible in $M_1$.
\end{claim}

\begin{proof}
If $P$ were compressible in $M_1$, then either $\alpha, \beta$, or 
$\alpha \beta$ would bound a disk in $M_1$.
Claim~\ref{claim:primitivity} shows this cannot be.
\end{proof}

\begin{claim}\label{claim:Pbdryincomp}
 There is no properly embedded disk $D$ in $M_1$ such that $\partial D$ meets
$P$ in a single essential arc.
 \end{claim}

\begin{proof}
Let $D$ be a properly embedded disk in $M_1$ such that $\partial D \cap P$ is
an essential arc in $P$.

First suppose $D$ separates $M_1$, and let $D_1,D_2$ be meridian disks of the
two solid tori $M_1\cut D$. Both points of $\partial D \cap \partial P$ belong 
to the same component of $\gamma$ of $\partial P$. The other two components
$\gamma_1, \gamma_2$ of $\partial P$ can be numbered so that $\gamma_i \cap
\partial D_j = \emptyset, \{i,j\} = \{1,2\}$. Hence $\{ \gamma_1, \gamma_2 \}
=\{ \alpha, \beta \}$ and $\gamma_i$ intersects $D_i$ in a single point,
$i=1,2$. But then $\gamma (= \alpha\beta)$ intersects $D_1$ (and $D_2$) in
a single point, contradicting the fact that 
$M_1\langle \alpha\beta \rangle$ is the trefoil
exterior.

Next suppose $D$ does not separate $M_1$. Let $D'$ be a disk in $M_1$ disjoint
from $D$ such that $M_1 \cut (D \cup D')$ is a 3-ball. If the two points
of $\partial D \cap \partial P$ belong to the same component of $\partial P$,
then the other two components are disjoint form $D$, and hence must be
$\alpha$ and $\beta$. But then  $\alpha \cup \beta$ is disjoint from $D$,
a contradiction. If the two ponts of $\partial D \cap \partial P$ belong to
different components of $\partial P$, then each of these components intersects
$D$ once, and hence they are $\alpha$ and $\beta$, so the third component
must be $\alpha\beta$. But this component is disjoint from $D$, contradicting
the fact that $M_1\langle \alpha\beta \rangle$ so the trefoil exterior.
\end{proof}

\begin{lemma}
\label{lem:TB}
Assume we are in \conditionI\ and there is a \FESC\ centered (WLOG) 
about a $\arc{23}$-\SC. Then the two $\edge{14}$-edges are parallel in $G_F$.

\end{lemma}

\begin{remark} We later use this lemma for an \FESC\ put together by an
{\tt MS} and {\tt ST} pair where the {\tt S} are parallel bigons (merging the 
\SCs\ to one
to give the faces of the {\FESC}). We will also use this 
(Lemma~\ref{lem:L1}) for an \FESC\ put together by an {\tt ST} and {\tt M}
where the leftmost edge of the {\tt ST} is parallel to an edge of {\tt M}.
\end{remark}

\begin{proof}
Assume there is a \FESC\ $\tau$, without loss of generality as shown in Figure~\ref{fig:FESC}.
Construct $M_1$ from $f,g$ and write $H_B = M_1 \cup_P M_2$ as above.
Observe that $\alpha$ is the boundary of the White \mobius band $A_{23} \subset H_W$ arising from the $\arc{23}$-\SC\ between $f$ and $g$ on $G_Q$.  If $P$ is incompressible in $H_B$, then by Claim~\ref{claim:Pbdryincomp}
it must boundary compress in $M_2$ and Lemma~\ref{lem:3Psphere} holds. If $P$ compresses in $H_B$,
it must compress in $M_2$ by Claim~\ref{claim:Pincomp}.

{\bf Case I: $P$ is incompressible in $H_B$ and (1) of Lemma~\ref{lem:3Psphere} holds.}

Then collapsing along $P \times [0,1]$, we may write $H_B = M_1 \cup_{(A_1 \cup A_2)} (\calT_1 \cup \calT_2)$. 
The cores of the annuli must be isotopic to either $\alpha, \beta$, 
or $\alpha \beta$ in $P$.
If the core, $c$, of either $A_1$ or $A_2$ is 
isotopic to $\alpha \beta$ then, again by Lemma~\ref{lem:3Psphere}, $c$ must be longitudinal in 
$\calT$ as $\alpha \beta$ is not primitive in $M_1$.
Thus if $\alpha, \beta$ are not the cores of some $A_1,A_2$ in the original decomposition, then we can replace 
the trivial decomposition along $\alpha \beta$ with a trivial decomposition along $\alpha$ or $\beta$. 
So we may assume that in $P$,
the core of $A_1$ is isotopic to $\alpha$ and the core of $A_2$ to $\beta$. 
By Claim~\ref{claim:primitivity}(1), both $\alpha$ and $\beta$ are jointly primitive curves in $M_1$, and 
$H_B' = \nbhd(A_{23}) \cup_\alpha M_1 \cup_{A_2} \calT_2$ is a genus $2$ handlebody.  
Since $\alpha$ is a primitive curve in $H_W \cut A_{23}$, 
$H_W' = (H_W \cut A_{23}) \cup_\alpha \calT_1$ is also a genus $2$ handlebody.  
Together $H_B'$ and $H_W'$ form a new genus $2$ Heegaard splitting for $M$.

Since $\arc{41}$ has a bridge disk $D_{41}$ in $H_W$ that is disjoint from $A_{23}$ 
(Lemma~\ref{lem:bridgedisksdisjoitfrommobius}), it continues to be bridge in $H_W \cut A_{23}$.  
Moreover since $D_{41}$ may be taken to be disjoint from $\alpha$, $D_{41}$ is a bridge disk for 
$\arc{41}$ in $H_W'$. 
By Claim~\ref{claim:primitivity}, arcs $\arc{12},\!\arc{34}$ can
be isotoped to $\bdry H_B'$, fixing their endpoints, so they intersect $\bdry A_{23}$ only at vertices
$2,3$ (resp.) of $G_F$ and
are incident to the same side of $\bdry A_{23}$ (the isotopy in $M_1$ is disjoint from $\alpha, \beta$).
Now we can write $K$ as the
union of two arcs $\arc{3412}$ that is a bridge arc of $H_W'$ and $\arc{23}$ which is a bridge arc
of $H_B'$: After isotoping $\arc{34},\!\arc{12}$ to $\bdry H_B'$, the arc $\arc{3412}$ is isotopic as
a properly embedded arc in $\bdry H_W'$ to $\arc{41}$ --- which is bridge in $H_W'$.
On the other hand, $\arc{23}$ can be isotoped as a properly embedded arc in
$H_B'$ to be a cocore of the annulus $\nbhd(A_{23}) \cap \bdry H_B$. The primitivity of this annulus in
$H_B$ now describes $\arc{23}$ as a bridge arc in $H_B'$. That is, $K$ is 1-bridge with
respect to the splitting $H_W' \cup H_B'$.
This contradicts that $t=4$.

\begin{remark}If $A_1$ is longitudinal in $\calT_1$ then $\partial A_{23}$ will
be primitive in $H_B$ and the new splitting is gotten from the old by adding/removing
a primitive \mobius band (in this case, adding $\calT_1$ to $H_W$ is isotopic
to the splitting where $\calT_1$ is not added). This is consistent with 
the proof of Theorem~\ref{thm:changingHS}. 
If $A_1$ is not longitudinal in $\calT_1$ then $M$ is a Seifert fiber space over
the $2$-sphere with an exceptional fiber of order $2$. In this case, we could 
find a vertical splitting with respect to which $K$ has bridge number $0$
by applying Lemma~\ref{lem:sSFS1bridge} to $\calT_1 \cup_{\alpha} A_{23}$, a Seifert fiber space
over the disk. This would then be consistent with  the proof of Theorem~\ref{thm:changingHS}.
\end{remark} 

{\bf Case II: $P$ is incompressible in $H_B$ and (2) of Lemma~\ref{lem:3Psphere} holds.} 

Collapsing along $P \times [0,1]$, we view $H_B$ as $M_1 \cup_{A_1 \cup A_2} \calT$. Then the cores of 
$A_1,A_2$ must be $\alpha,\beta$ in $M_1$. This follows from Lemma~\ref{lem:3Psphere} when $A_1$ (hence
$A_2$ as well) is not longitudinal in $\calT$, since $\alpha \beta$ is not primitive in $M_1$. When
$A_1,A_2$ are longitudinal on $\calT$, assume for contradiction that the core of $A_1$ is $\alpha \beta$.
As $A_1 \cup A_2$ must be $\bdry$-compressible in $H_B$, and $\alpha \beta$ is not primitive in $M_1$,
it must be that there is a meridian disk for $M_1$ that is disjoint from $A_1$ and crossing the core
of $A_2$ once. But then we obtain the contradiction that the trefoil knot exterior, 
$ M_1 \langle \alpha\beta \rangle$, has compressible boundary. 

So we may assume the core of $A_1$ is $\alpha$ and the core of $A_2$ is $\beta$ in $M_1$.

First consider the case where $A_1$ runs $n>1$ times longitudinally around $\calT$. There is an
annulus $B$ contained in $\bdry M_1$ which we may assume contains $\bdry A_{23}$ and $A_1$ and that
intersects $K$ only at vertices $2,3$ of $G_F$ (i.e.\ only along $\bdry A_{23}$). Let $\calN$ be
$\nbhd(B \cup A_{23} \cup \calT)$. Then $\calN$ is a Seifert-fibered space over the disk with
two exceptional fibers. 
Furthermore, $K \cap \calN$ lies as a co-core of the \mobius band $A_{23}$ properly embedded in $\calN$.   Lemma~\ref{lem:sSFS1bridge} then gives a genus $2$ splitting of $M$ in which $K$ is $0$-bridge, a contradiction.

Finally consider the case where $A_1, A_2$ are longitudinal in $\calT$. By Claim~\ref{claim:primitivity},
there are disjoint meridian disks $D_1,D_2$ of $M_1$ such that $D_i$ intersects the core of $A_i$ once
and is disjoint from $A_j$ where $\{i,j\}=\{1,2\}$. Then there is a disk $D_3$ in $\calT$ such that
$D = D_1 \cup D_2 \cup D_3$ forms a meridian disk in $H_B$ that intersects each of $\alpha$ and $\beta$
once. In particular, $\alpha$ is primitive in $H_B$. Then $H_B'= H_B \cup \nbhd(A_{23})$ is a genus 2 handlebody.
Also $H_W' = H_W - \nbhd(A_{23})$ is a genus 2 handlebody. Hence $H_B' \cup H_W'$ is a genus 2 Heegaard
splitting of $M$. We now show that $K$ has bridge number one with respect to this splitting, thereby
contradicting the assumption that $t=4$. By Claim~\ref{claim:primitivity}, arcs $\arc{12},\!\arc{34}$ can
be isotoped to $\bdry H_B$ so they intersect $\bdry A_{23}$ only at vertices
$2,3$ (resp.) of $G_F$ and
are incident to the same side of $\bdry A_{23}$ (the isotopy in $M_1$ is disjoint from $\alpha, \beta$).
Now we can write $K$ as the
union of two arcs, $\arc{3412}$ that is a bridge arc of $H_W'$ and $\arc{23}$ which is a bridge arc
of $H_B'$: After isotoping $\arc{34},\!\arc{12}$ to $\bdry H_B$, the arc $\arc{3412}$ is isotopic as
a properly embedded arc in $\bdry H_W'$ to $\arc{41}$ --- which is bridge in $H_W'$ 
(Lemma~\ref{lem:bridgedisksdisjoitfrommobius}).
On the other hand, $\arc{23}$ can be isotoped as a properly embedded arc in
$H_B'$ to be a cocore of the annulus $\nbhd(A_{23}) \cap H_B$. The primitivity of this annulus in
$H_B$ now describes $\arc{23}$ as a bridge arc in $H_B'$. That is, $K$ is 1-bridge with
respect to the splitting $H_W' \cup H_B'$.

%%%%%  

{\bf Case III: $P$ is compressible.}

Because $P$ is not compressible into $M_1$ by Claim~\ref{claim:Pincomp}, some component of $\bdry P$ bounds a disk $D$ in $M_2$. The following then proves the lemma in this case.

\begin{claim}\label{claim:TBIII}
Assume there is a disk $D$ properly embedded in $H_B$ disjoint from $M_1$ and with $\bdry D$ isotopic
to $\alpha, \beta$ or $\alpha \beta$ in $\bdry H_B$. Then $\bdry D$ must in fact be isotopic to
$\beta$, and $\beta$ must bound a disk in $\bdry H_B$.
\end{claim}

\begin{proof}
If $\bdry D$ were isotopic to $\alpha$, then $D \cup A_{23}$ forms an $\RP^2$; this is a contradiction.  
If $\bdry D$ were $\alpha\beta$, then $\nbhd(D) \cup_{\alpha\beta} M_1 = M_1 \langle \alpha\beta \rangle$ is 
a trefoil complement embedded in $H_B$ (Claim~\ref{claim:primitivity}).  Thus $M_1 \langle \alpha\beta \rangle$ must be contained in a $3$-ball.  
By Lemma~\ref{lem:AEGor}, $\alpha$ bounds a disk in $H_B$ or $H_W$ which, as above, cannot occur. 
Thus $\bdry D$ is isotopic to $\beta$. 

Now assume $\beta$, hence $\bdry D$, is essential in $\bdry H_B$. Let $\calO$ be the solid torus 
component of $H_B - \nbhd(D)$  
containing $M_1$. Let $\calN$ be ${\calO} \cup \nbhd(A_{23})$.
Using $D$, we may extend the isotopy from Claim~\ref{claim:primitivity}(b) of
arcs $\arc{12},\!\arc{34}$, fixing their endpoints, to $\bdry {\calO}$ so that the
resulting arcs $a,b$ are incident $\bdry A_{23}$ only at their endpoints and on the same side of $\bdry
A_{23}$ (alternatively, Lemma~\ref{lem:btfsc} constructs such an isotopy). 
Thus $K$ can be written as the union of two arcs: $\arc{34123}$, $\mu$. Arc $\arc{34123}$ is 
the union of the arcs $a,b$ on $\bdry {\calO}$, the arc $\arc{41}$ of $K \cap H_W$, and an arc on 
$\partial \nbhd(A_{23}) - {\calO}$ (a cocore of this annulus) running from vertex $2$ to vertex $3$. 
The arc $\mu$ is a cocore of the annulus $B=\nbhd(A_{23}) \cap {\partial \calO}$ on $\hatF$. Note that 
$\arc{34123}$ is the union of the $\arc{41}$-arc of $K$ with two arcs on $\partial {\calN}$. Pushing
$\arc{34123}$ slightly into the exterior of $\calN$, we have $K \cap {\calN} = \mu$. 

$B$ winds $n>0$ times around $\calO$. First assume $n>1$. Then $\calN$ is a Seifert fiber space
over the disk with two exceptional fibers.  
Furthermore, $\mu = K \cap \calN$ is a co-core of the annulus 
$B \subset \calN$, where $B$ is vertical under the Seifert fibration.   Lemma~\ref{lem:sSFS1bridge} applies to give a new genus $2$ Heegaard splitting of $M$ in which $K$ is $0$-bridge, contradicting that $t=4$.

So assume $n=1$. Then $\bdry A_{23}$ is primitive in $H_B$. So $H_B' = H_B \cup \nbhd(A_{23})$ is a 
genus two handlebody, as is its exterior $H_W' = H_W - \nbhd(A_{23})$. Then $K \cap H_W' = \arc{34123}$ is
properly isotopic to the bridge arc $\arc{41}$ of $H_W'$, hence is bridge in $H_W'$. $K \cap H_B' = 
\mu$ is properly isotopic to a cocore of $B$ whose core is primitive in $H_B$. Thus $\mu$
is a bridge arc in $H_B'$. That is, $K$ is $1$-bridge in the Heegaard splitting $H_B' \cup H_W'$,
contradicting that $t=4$.
\end{proof}
This completes the proof of Lemma~\ref{lem:TB}.
\end{proof}

\begin{lemma}
\label{lem:btfsc}
Assume \conditionI\ and that there is a \FESC\ centered, WLOG, about a $\arc{23}$-\SC, $f$. There are bridge 
disks $D_{12}$ and $D_{34}$ disjoint from the edges of the $\arc{23}$-\SC. These bridge disks guide isotopies of the arcs $\arc{12},\!\arc{34}$, fixing endpoints, onto
arcs of $\hatF$ that are incident to the same side in $\hatF$ of the curve formed by the edges of this \SC.
Let $A_{23}$ be the \mobius band associated to $f$. If $\partial A_{23}$ is primitive in $H_B$, then $K$
is $1$-bridge with respect to a genus two Heegaard splitting of $M$.
\end{lemma}

\begin{proof}
WLOG we may assume there is a \FESC\ $\tau$ as shown in Figure~\ref{fig:FESC}(a).  Its edges induce the subgraph of $G_F$ shown in Figure~\ref{fig:FESC}(b).

Let $E$ be a disk giving the parallelism guaranteed by Lemma~\ref{lem:TB}.  Let $\rho_{12}$ and $\rho_{34}$ be rectangles on $\bdry \nbhd(\arc{12})$ and $\bdry \nbhd(\arc{34}))$ respectively that are between $f$ and $g$ and meet $E$.  Then together $f \cup g \cup E \cup \rho_{12} \cup \rho_{34}$ form a bridge disk $D_{34}$ for $\arc{34}$ as shown in Figure~\ref{fig:FESCdisk}(a) that meets $\hatF$ as shown in Figure~\ref{fig:FESCdisk}(b).
\begin{figure}
\centering
\input{FESCdisk.pstex_t}
\caption{}
\label{fig:FESCdisk}
\end{figure}

\begin{figure}
\centering
\input{FESCdisk2.pstex_t}
\caption{}
\label{fig:FESCdisk2}
\end{figure}

Let $\rho_{34'}$ be a rectangle on $\bdry \nbhd(\arc{34})$ between the $x$ corner of $g$ and $y$ corner of $f$ and containing the $z$ corner of $g$.   Banding $D_{34}$ to $\arc{12}$ with the rectangle $\rho_{34}' \cup f$ produces the bridge disk $D_{12} = D_{34} \cup \rho_{34}' \cup f$ shown in Figure~\ref{fig:FESCdisk2}(a) which  may be made embedded and disjoint from $D_{34}$ by a slight perturbation.  
Figure~\ref{fig:FESCdisk2}(b) shows how $D_{12}$ and $D_{34}$ meet $\hatF$.  In particular, they are incident to the same side of $\bdry A_{23}$. 

Now assume $\bdry A_{23}$ is primitive in $H_B$. So $H_B' = H_B \cup \nbhd(A_{23})$ is a                
genus 2 handlebody, as is its exterior $H_W' = H_W - \nbhd(A_{23})$. Now argue as in the last paragraph
of Claim~\ref{claim:TBIII}. That is, $K \cap H_W' = \arc{34123}$ is           
properly isotopic to the bridge arc $\arc{41}$ of $H_W'$, hence is bridge in $H_W'$. 
$K \cap H_B' =                
\mu$ is properly isotopic to a cocore of the annulus $B$, a neighborhood in $\hatF$ of 
$\bdry A_{23}$. 
As $\bdry A_{23}$ is primitive in $H_B$, $\mu$                             
is a bridge arc in $H_B'$. That is, $K$ is 1-bridge in the Heegaard splitting $H_B' \cup H_W'$,                   
contradicting that $t=4$.                                                                            
\end{proof}

%%\newpage

%%%%%%%%%%%%%%%%%%%%%%%%%%%%%%%%%%
\section{\FESCs}

Throughout this section assume $t=4$, there are no Dyck's surfaces embedded in $M$, and we are in \situationnscc.

\subsection{Type I and II \FESCs}

\begin{defn} 
By Lemma~\ref{lem:TB} two of the edges bounding a \FESC\ are parallel 
on $G_F$.  A  \FESC\ along a vertex $x$ of $G_Q$ is {\em type I} or {\em type II} ({\em at $x$}) according 
to whether both or just one of these parallel edges are incident 
to the vertex.  
See Figure~\ref{fig:fesctypes} for an illustration of types I and II at the vertex $x$.
\end{defn}

\begin{figure}[h]
\centering
\input{fesctypes.pstex_t}
\caption{}
\label{fig:fesctypes}
\end{figure}

The boldface notation in the lemmas of this section refers to that of 
section~\ref{sec:t4noscc}. 

\begin{lemma}\label{lem:fesc+1}
${\tt MST_{\tt I\!I}} \implies {\tt MSTG}$

At a vertex $x$, the trigon of a type II \FESC\ cannot be further adjacent to another bigon or a trigon.  In particular, a type II \FESC\ must have its trigon adjacent to a true gap at $x$.
\end{lemma}

\begin{proof}
Assume there is a type II \FESC\ adjacent to another bigon or trigon.  In the case of a bigon we construct a long disk\footnote{See also Lemma~2.2 \cite{baker:sgkilshsbn}.} as in Figure~\ref{fig:fesc+bigon}.  In the case of a trigon we construct a lopsided bigon\footnote{See also the last two paragraphs of the of Lemma~6.15 \cite{baker:sgkilshsbn}.} as in Figure~\ref{fig:fesc+trigon}.  Hence in both cases there is a thinning of $K$.

\begin{figure}[h]
\centering
\input{fesc+bigon.pstex_t}
\caption{}
\label{fig:fesc+bigon}
\end{figure}

\begin{figure}[h]
\centering
\input{fesc+trigon.pstex_t}
\caption{}
\label{fig:fesc+trigon}
\end{figure}

In these figures $\delta$ is the disk of parallelism guaranteed by Lemma~\ref{lem:TB}, and $\rho_{ab}$ denotes a rectangle on the boundary of the $\arc{ab}$ handle.  Note that $\rho_{23}$ and $\rho_{23}'$ have disjoint interiors.  

The long disk may be taken to lie on the boundary of the neighborhood of the $2$-complex formed from the four faces and $K$ as they are embed in $M$.  The lopsided bigon will be embedded except at its short $\arc{a \, a+1}$-corner; nevertheless, the lopsided bigon guides an isotopy of $K$.  Both the long disk and the lopsided bigon run over both sides of the $\arc{23}$-\SC.

To verify these isotopies explicitly, one may construct models of these $2$-complexes, their neighborhoods and their intersections with $\hatF$. As the case when the adjoining face, $h$, is a trigon can be viewed as a ``splintering'' of the case when $h$ is a bigon, we begin with the model of the bigon case.

{\bf The long disk.}  Form a \mobius band out of the $\arc{23}$-\SC\ and the $\arc{23}$-arc of $K$.  Complete $K$ and take a small regular neighborhood.  The attachment of $f$ is unique.  The attachment of $g$ is unique up to a choice of placement of its $\arc{34}$-corner opposite the $\edge{23}$-edge.  These two choices give mirror images and are thus equivalent up to homeomorphism.  The boundary of $\delta$ is now set and we may attach it. The bigon $h$ may now also be attached along the $\edge{34}$-edge of $g$ in a unique manner.  Beginning from the corners of $\delta$, the choices for $\rho_{12}$, $\rho_{23}$, $\rho_{34}$, and $\rho_{23}'$ are determined.  One may now ``wrap'' the long disk around this complex to exhibit an isotopy of $\arc{2341}$ onto $\hatF$.  The graph on $\hatF$ induced by the edges of these faces and the arc onto which the isotopy lays down $\arc{2341}$ is shown in Figure~\ref{fig:fesc+bigon-graph}.  Since $K$ is isotopic to the arc $\arc{12}$ and an arc on $\hatF$, it is at most $1$-bridge.

\begin{figure}[h]
\centering
\input{fesc+bigon-graph.pstex_t}
\caption{}
\label{fig:fesc+bigon-graph}
\end{figure}

\begin{remark}
The long disk can also be pictured as the union of the bridge disk $D_{34}$ of 
Lemma~\ref{lem:btfsc} and a White bigon on corners $\arc{23},\!\arc{41}$ gotten
by banding $h$ and two disjoint copies of the $\arc{23}$-\SC\ along the 
boundary of a neighborhood of the $\arc{23}$-arc of $K$. This white bigon
and $D_{34}$ agree on $F$ along one edge of the bigon.
\end{remark}  

{\bf The lopsided bigon.}  Take the above constructed complex and break the $\edge{12}$-edge of $h$ by inserting a corner, thereby changing $h$ from a bigon into a trigon.  This new corner will be either a $\arc{23}$- or a $\arc{41}$-corner.  To complete this model, this corner must be attached to $K$.  The long disk isotopy now becomes an isotopy of $\arc{2341}$ onto two arcs of $\hatF$ and either the $\arc{23}$- or $\arc{41}$-arc of $K$.  There are seven 
possible ways of hooking up this new corner to its position on the complex:  three for $\arc{23}$ and four for $\arc{41}$.  When the new corner is a $\arc{23}$-corner, Figure~\ref{fig:fesc+trigon-graphs23} shows the three possible graphs on $\hatF$ and the resulting two arcs on $\hatF$ after the isotopy of $\arc{2341}$.  Figure~\ref{fig:fesc+trigon-graphs41} shows four possibilities when the new corner is a $\arc{41}$-corner.  
Note that $\edge{41}$-edge of $h$ cannot lie in $\delta$ by 
Lemma~\ref{lem:parallelwithsamelabel}. Since $K$ is isotopic to the union of arc $\arc{12}$, two arcs on $\hatF$, and one of the arcs $\arc{23}$ or $\arc{41}$, it is at most $1$-bridge.

\begin{figure}[h]
\centering
\input{fesc+trigon-graphs23.pstex_t}
\caption{}
\label{fig:fesc+trigon-graphs23}
\end{figure}

\begin{figure}[h]
\centering
\input{fesc+trigon-graphs41.pstex_t}
\caption{}
\label{fig:fesc+trigon-graphs41}
\end{figure}

\begin{remark}
As in the remark above for the long disk, the lopsided bigon can be pictured 
as the union of the Black bridge disk $D_{34}$ with a new White trigon gotten by
banding $h$ and two copies of the $\arc{23}$-\SC. The new trigon has the 
property that it matches the bridge disk along one of its edges (and disjoint
elsewhere).
\end{remark}
\end{proof}

\begin{lemma}
\label{lem:dbft2}
{\tt BBTBB} $\implies$ {\tt SMTMS}.   In particular, the \ttT\ is a Scharlemann
cycle.
\end{lemma}

\begin{proof}
If the \ttT\ were a Scharlemann cycle, then the desired conclusion would follow, so assume otherwise.  By Lemma~\ref{lem:fesc+1}, if one of the \ttB\ adjacent to the \ttT\ is an \ttS, then the other is too.  Hence we assume we have {\tt MSTSM} as shown, WLOG, in Figure~\ref{fig:dbftnew}(a).
\begin{figure}
\centering
\input{dbftnew.pstex_t}
\caption{}
\label{fig:dbftnew}
\end{figure}
By Lemma~\ref{lem:TB}, the $\edge{41}$-edges of $f_1$ and $g$ are parallel as are the $\edge{23}$-edges of $f_4$ and $g$.  Using these parallelisms we may form the annulus $f_1 \cup f_2 \cup f_3 \cup f_4$ shown in Figure~\ref{fig:dbftnew}(b).  Since the two boundary components of this annulus each run along $K$ once in opposite directions, joining them along $K$ forms an embedded Klein bottle.  This is a contradiction.
\end{proof}

\begin{lemma}%[L1]  
\label{lem:L1}
{\tt MSTB} $\implies$ {\tt MSTSG} or {\tt MSTSTG}. 
\end{lemma}
\begin{proof}
Given {\tt MSTB} at a vertex, $\ttB = \ttS$ since otherwise {\tt MST} would form a type II \FESC\ giving a contradiction to Lemma~\ref{lem:fesc+1}.  We cannot have {\tt MSTSB} since this contradicts Lemma~\ref{lem:dbft2}.  Thus we have either {\tt MSTSG} or {\tt MSTST}.  We continue to examine the latter.

\begin{figure}[h]
\centering
\input{MSTST.pstex_t}
\caption{}
\label{fig:MSTST}
\end{figure}
Since the initial {\tt MST} forms a type I \FESC, we have a parallelism 
$\delta$ on $G_F$  between the leftmost edge of the \ttM\ and the rightmost edge of the \ttT.  We may use this to attach the \ttM\ to the \ttS\ of the subsequent {\tt ST} to form a \FESC.  This is illustrated in Figure~\ref{fig:MSTST} (without loss of generality we may use the labeling shown).  By the proof of Lemma~\ref{lem:TB}
(see the remark there) there is a parallelism $\delta'$ of the $\edge{23}$-edge of the \ttM\  to a $\edge{23}$-edge of the \ttT.  Lemma~\ref{lem:parallelwithsamelabel} forces this $\edge{23}$-edge of \ttT\ to not be incident to the vertex and thus the unlabeled corner in Figure~\ref{fig:MSTST} is a $\arc{12}$-corner.  Following the proof of Lemma~\ref{lem:fesc+1} we may build either a long disk or lopsided bigon as shown in Figure~\ref{fig:MSTST+bigonortrigon}.  
\begin{figure}[h]
\centering
\input{MSTST+bigonortrigon.pstex_t}
\caption{}
\label{fig:MSTST+bigonortrigon}
\end{figure}
(Note that the regions $\rho_{341}$ and $\rho_{412}$ in $\bdry \nbhd(K)$ must be as in Figure~\ref{fig:MSTST+knotbdry}(a) and not (b).  The corner $x$ is labeled in each; in (b) two continue into $\delta'$ contrary to Lemma~\ref{lem:parallelwithsamelabel}.)  Therefore there cannot be a bigon or trigon incident to the $\edge{12}$-edge of this \ttT.  
\begin{figure}[h]
\centering
\input{MSTST+knotbdry.pstex_t}
\caption{}
\label{fig:MSTST+knotbdry}
\end{figure}
Hence if we have {\tt MSTST}, we then have {\tt MSTSTG}.

\begin{remark}
From the point of view of the remarks in the proof of Lemma~\ref{lem:fesc+1},
the White bigon or trigon constructed is the same as there, the difference
is in the construction of the Black bridge disk where the parallelism
(here given by $\delta$) is used to modify the bridge disk on the Black side
to line up
with the White bigon, trigon along an edge.
\end{remark}

\end{proof}

\begin{lemma}%[L5]
\label{lem:L5}
 {\tt MSTTSM} cannot occur.  
\end{lemma}
\begin{proof}
Either {\tt MST} or {\tt TSM} must be a type II \FESC.  Since the trigon is not adjacent to a gap, this is forbidden by Lemma~\ref{lem:fesc+1}.
\end{proof}

\begin{lemma}%[L6]
\label{lem:L6}
At a vertex of type $[7,4]$, if no {\tt TT} and no {\tt BBB} then {\tt MST} $\implies$ {\tt BTMST}$\left\{\begin{array}{l} {\tt G}\\ {\tt SG}\\ {\tt STG} \end{array}\right\}$.
\end{lemma}
\begin{proof}
First consider the faces to the right of {\tt MST}.
Since no {\tt TT}, we must have either {\tt MSTG} or {\tt MSTB}.  For the latter, Lemma~\ref{lem:L1} gives {\tt MSTSG} or {\tt MSTSTG}.  Now since there is only one true gap at this vertex and having no {\tt TT} and no {\tt BBB} implies {\tt BT} must be to the left of {\tt MST}.  
\end{proof}

\subsection{More with \FESC:  Configurations {\tt SMST} and {\tt MSTS}}

%%%%Lemma SMST 
\begin{lemma} \label{lem:SMST}
Assume there is an {\tt SMST} configuration incident to vertex $x$ for
which the {\tt MST} is a type I \FESC. WLOG
assume the bigon Scharlemann cycles of this configuration are on the White side.
Then any Black mixed bigon, $f$, must have an edge that is 
parallel on $G_F$ to
an edge in the {\tt MST} subconfiguration (the \FESC). 
Furthermore, if that edge is parallel to an edge of the {\tt M} in the
{\tt MST}, then $f$ is parallel to the \ttM.
In particular,
there is at most one more Black mixed bigon incident to $x$, other than $f$
and that in the given \FESC.

\end{lemma}

\begin{proof}
WLOG we assume the configuration {\tt SMST} and $f$ on $G_Q$ are as in 
Figure~\ref{fig:SMST1}. 

\begin{figure}[h]
\centering
\input{SMST1.pstex_t}
\caption{}
\label{fig:SMST1}
\end{figure}

Let $A_{23}, A_{41}$ be the \mobius bands in $H_W$ gotten from the bigon 
Scharlemann cycles of the {\tt SMST} configuration.
By Lemma~\ref{lem:btfsc}, $\partial A_{23}$ cannot be primitive in $H_B$.
A similar argument shows that $\partial A_{41}$ cannot be primitive: 
Otherwise consider the new genus two Heegaard splitting gotten by 
attaching $\nbhd(A_{41})$ to $H_B$.  
Then constructing the right bridge disks $\Delta_{12}$ and $\Delta_{34}$ as 
in Lemma~\ref{lem:btfsc} (see the left or right of Figure~\ref{fig:SMST2}, ignoring the
$\edge{23}$-edge with $a,b$ endpoints and setting $a=x'',b=y'$ on vertices $4$ and $1$), one sees that  
the $\arc{12}$-arc and 
$\arc{34}$-arc of $K$ can be isotoped (rel endpoints) to arcs on $\hatF$ that 
are incident to $\partial A_{41}$ on the same side (and otherwise disjoint
from it). We then get a $1$-bridge
presentation of $K$ with respect to the new splitting by 
isotoping it to a $\arc{12341}$-arc and an arc which
is a cocore of $\nbhd(\partial A_{41})$ --- a contradiction. 

We assume for contradiction
that neither edge of $f$ is parallel on $G_F$ to an edge of $f_2,f_3,f_4$.  Applying Lemma~\ref{lem:TB} to the \FESC, 
the edges of $f$ and of the \FESC\ 
must appear on $G_F$ as in one of the two configurations of 
Figure~\ref{fig:SMST2}.

\begin{figure}[h]
\centering
\input{SMST2.pstex_t}
\caption{}
\label{fig:SMST2}
\end{figure}

Let $A_{12,34}$ be the annulus gotten from the union of $f_2$ and $f$.
Since no two of the edges of these faces are parallel on $G_F$, 
each component of $\bdry A_{12,34}$ is essential in $\hatF$.  Furthermore,
$A_{12,34}$ must be incompressible in $H_B$, otherwise we get a Black disk that either makes $\partial A_{23}$ primitive in $H_B$ or compresses
$\hatF$ to induce the formation of a Klein bottle in $M$ from $A_{23},A_{41}$.

As in the proof of Lemma~\ref{lem:btfsc}, we construct a thinning disk
$\Delta_{12}$ from $f_2, f_4$. 
$\partial$-compressing
$A_{12,34}$ along $\Delta_{12}$, we get a Black disk, $D$, with the boundary
as in Figure~\ref{fig:SMST3}. In Case (A) of that figure, $\partial D$ 
intersects $\partial A_{23}$ once, implying that $\partial A_{23}$ is primitive.

\begin{figure}[h]
\centering
\input{SMST3.pstex_t}
\caption{}
\label{fig:SMST3}
\end{figure}

Thus we assume we are in Case (B), where $\partial D$ intersects 
$\partial A_{23}$ algebraically zero times and geometrically twice. If 
$D$ is non-separating, then we can construct a Dyck's surface in $M$ by
attaching to $A_{23}$ the once-punctured torus or Klein bottle in $H_B$ 
pictured in Figure~\ref{fig:SMST4}.

\begin{figure}[h]
\centering
\input{SMST4.pstex_t}
\caption{}
\label{fig:SMST4}
\end{figure}

Thus we may assume $D$ is separating in $H_B$. 
As $D$ is homologous to $A_{12,34}$, this annulus must be separating in $H_B$. 
Let $B$ be the annulus bounded by $\partial A_{12,34}$ on $\hatF$. 
Note that $\partial D$ is not trivial in $\hatF$, for if so an edge of $f$ would be parallel on $G_F$ to one of the edges of $f_2$ or $f_3$ contrary to assumption. 
Thus $A_{12,34}$
is an incompressible, separating annulus in $H_B$. Note that if $A_{12,34}$ is 
parallel to $\partial H_B$, %John's v298J had $\partial \hatF$
then each component of $\partial A_{12,34}$ is
primitive in $H_B$. Let $P$ be the 4-punctured
sphere that is the union in $\hatF$ of the edges of $f,f_2,f_3,f_4$, the
fat vertices of $G_F$, and the disk of parallelism on $G_F$ between the
$\edge{41}$-edges of $f_2$ and $f_4$. Then the closure of $\hatF - P$ is
two annuli, one of which is $B$. Call the other $B'$. See Figure~\ref{fig:SMST5}.

\begin{figure}[h]
\centering
\input{SMST5.pstex_t}
\caption{}
\label{fig:SMST5}
\end{figure}

\begin{claim}  
Let $e$ be the $\edge{41}$-edge that $f_1$ does not share with $f_2$.
Then $e$ on $G_F$ is either (i) the dotted line in Figure~\ref{fig:SMST5}, or
(ii) parallel to the $\edge{41}$-edge of $f$.
\end{claim}

\begin{proof}
If $e$ lies in $B$ on $G_F$ then it isotopic into $\partial B$ and hence
is parallel to either the $\edge{41}$-edge
of $f_2$ or $f$. The former cannot occur else $M$ has a lens space summand, the
latter is conclusion (ii).
If $e$ lies in $B'$ then it is isotopic into $\partial B'$ and hence either is
parallel to the $\edge{41}$-edge of 
$f$ yielding conclusion (ii), is
parallel to the $\edge{41}$-edge of $f_2$ (a contradiction as above), is parallel
to the dotted edge in Figure~\ref{fig:SMST5} yielding conclusion (i), or is such that $\partial A_{41}$
would be isotopic on $\hatF$ to $\partial A_{23}$ giving a Klein bottle in $M$.
\end{proof}

Assume $e$ is as in (i) of the Claim. Then $A_{41}, A_{12,34}$ and $D$ can be 
perturbed to be disjoint with boundaries as indicated in Figure~\ref{fig:SMST6}
(by forming these with the given faces and the appropriate rectangles along
$\partial \nbhd(K)$). $D$ divides $H_B$ into two solid tori $\calT \cup \calT'$
where $\partial T$ contains $\partial A_{41}$. Since $\partial A_{41}$ is not
primitive in $H_B$, it is not longitudinal in $\calT$. Let $\calN = \calT \cup
\nbhd(A_{41})$. Then $\calN$ is a Seifert fiber space over the disk with two
exceptional fibers. A close look at Figure~\ref{fig:SMST6} shows that we can perturb
$K$ so that $K \cap \calN$ is a single arc, $\eta$, (basically 
the $\arc{41}$-arc) which
is isotopic to the cocore of the \mobius band $A_{41}$.  
 Lemma~\ref{lem:sSFS1bridge} now produces a genus $2$ Heegaard splitting of $M$ in which $K$ is $0$-bridge, a contradiction.

\begin{figure}[h]
\centering
\input{SMST6.pstex_t}
\caption{}
\label{fig:SMST6}
\end{figure}

So assume $e$ is as in conclusion (ii) of the Claim. 
As $\partial A_{41}$ is isotopic to a component of $\partial A_{12,34}$, and
$\partial A_{41}$ is not primitive in $H_B$, $A_{12,34}$ is not parallel
into $\hatF$.
We can enlarge the annulus $B$ slightly in $\hatF$
so that it contains $\partial A_{41}$. Let $\calT$ be the solid torus bounded
by $B \cup A_{12,34}$ in $H_B$. Then $\calN = \nbhd(\calT \cup A_{41})$ is
a Seifert fiber space over the disk with two exceptional fibers. 
$K \cap \calN$
is a single arc which is a cocore of a properly embedded \mobius band,
$A_{12,34} \cup A_{41}$, 
in $\calN$.  Lemma~\ref{lem:sSFS1bridge} now applies to produce a genus $2$ Heegaard splitting of $M$ in which $K$ is $0$-bridge, a contradiction.

This last contradiction proves the first conclusion of the Lemma, 
that some edge of $f$ must be parallel in $G_F$ to an edge of $f_2,f_3,f_4$.
Furthermore, if one edge of $f$ is parallel to an edge of $f_2$, then, 
in fact,
$f$ is parallel to $f_2$. For otherwise,
banding $f$ and $f_2$ together along these parallel edges, and perturbing
slightly gives a disk
in $H_B$ whose non-trivial boundary intersects $\partial A_{23} \cup 
\partial A_{41}$ at most once. 
If this disk is disjoint from $\partial A_{23} \cup 
\partial A_{41}$, and the boundary of the disk is non-separating in $\hatF$,
then $\partial A_{23}$ and $\partial A_{41}$ will
be isotopic in $\hatF$ surgered along this disk, and
$M$ contains a Klein bottle. If disjoint and the boundary of the disk
is separating, then one of $\bdry A_{23}, \bdry A_{41}$ must be primitive in $H_B$ since $M$ is irreducible, atoroidal, and the Heegaard
splitting is strongly irreducible (Lemma~\ref{lem:AEGor}). If the disk intersects $\partial A_{23} \cup
\partial A_{41}$ once, then one of these \mobius bands will have primitive boundary in $H_B$.

Finally, assume $f$ is incident to vertex $x$.  Then Lemma~\ref{lem:parallelwithsamelabel} says that $f$ cannot be
parallel to $f_2$ (both $\edge{41}$-edges
of the \FESC\ are parallel on $G_F$).  Thus the
$\edge{23}$-edge of $f$ must be parallel in $G_F$ with the $\edge{23}$-edge
of $f_3$ that is not shared with $f_2$. Applying this argument to 
another mixed black bigon incident to
vertex $x$, will then contradict Lemma~\ref{lem:parallelwithsamelabel}.
\end{proof}

\begin{lemma} \label{lem:MSTS}
Assume there is an {\tt MSTS} configuration incident to vertex $x$. WLOG
assume the bigon Scharlemann cycles of this configuration are on the White side.
Then any Black mixed bigon, $f$, must have an edge which is
parallel on $G_F$ to
an edge in the {\tt MST} subconfiguration (the \FESC).
Furthermore, if that edge is parallel to an edge of the {\tt M} in the
{\tt MST}, then $f$ is parallel to $M$.
In particular,
there is at most one more Black mixed bigon incident to $x$.
\end{lemma}

\begin{proof}
This is the same as the proof for Lemma~\ref{lem:SMST}. Note that the
\FESC\ is of type I at $x$ and in both contexts one edge of the additional
White \SC\ has an edge parallel to both the {\tt M} and {\tt T} in the
\FESC, {\tt MST}.
\end{proof}

\begin{lemma}%[NFSC]
\label{lem:nfsc}
Assume $\Lambda$ contains a \FESC\ and an \SC\ on the side of $\hatF$ opposite to that of the \SC\ in the \FESC, then the 
corresponding \mobius bands can be perturbed to be disjoint. 

That is, WLOG assume we have 
the configurations of Figure~\ref{fig:nfsc} where one of $f_1, f_3$ is a bigon and the other\
is a trigon and where $f_4$ is a Black \SC\ ($f_4$ could equally well be a $\arc{12}$-\SC). 
Let $A_{23},A_{34}$ be the \mobius bands corresponding
to $f_2,f_4$. If $\bdry A_{23}$ and $\bdry A_{34}$ intersect transversely once, then $K$ is $1$-bridge
with respect to a genus two Heegaard splitting of $M$.
\end{lemma}   
\begin{figure}[h!]
\centering
\input{nfsc.pstex_t}
\caption{}
\label{fig:nfsc}
\end{figure}

\begin{proof}
Without loss of generality assume $f_1$ is a bigon and $f_3$ is a trigon. 
In this proof we consider faces of $G_Q$ as disks properly embedded in
$H_B - \nbhd(K), H_W - \nbhd(K)$. The proof of Lemma~\ref{lem:btfsc} shows that there are thinning disks $\Delta_{34}$, $\Delta_{12}$ for $\arc{34}$, $\arc{12}$ disjoint from $f_2$.

In fact, the thinning disk $\Delta_{34}$ may be chosen to be disjoint from $f_4$ as well as $f_2$:  
Isotop $\bdry \Delta_{34} \cap \nbhd(\arc{34})$ so that it is disjoint from $f_2$ and $f_4$, for example as in Figure~\ref{fig:nfscclaim1}.
\begin{figure}[h!]
\centering
\input{nfscclaim1.pstex_t}
\caption{}
\label{fig:nfscclaim1}
\end{figure}
After surgering $\Delta_{34}$, we may assume it intersects $f_4$ in transverse arcs (i.e.\ from one edge of $f_4$ to the other).  Band an outermost disk of intersection along $f_4$ to give a thinning disk disjoint from both $f_2$ and $f_4$.
%\end{proof}

Using $\Delta_{34}$ to $\bdry$-compress $A_{34}$ yields a Black disk intersecting $\bdry A_{23}$ 
transversely once. See, for example, Figure~\ref{fig:nfsc2}. Thus $\bdry A_{23}$ is primitive in $H_B$.
Apply Lemma~\ref{lem:btfsc}.
\end{proof}
 
\begin{figure}[h!]
\centering
\input{nfsc2.pstex_t}
\caption{}
\label{fig:nfsc2}
\end{figure}

%%%%%% NFSC2

\begin{lemma} %[NFSC2] 
\label{lem:NFSC2}
Assume $G_Q$ has a configuration {\tt SMST} where WLOG the \SCs\ are on the
White side. Then $G_Q$ contains no Black \SC.

The same conclusion holds for the configuration {\tt MSTS}.
\end{lemma}

\begin{proof}
WLOG we assume the {\tt SMST} configuration on $G_Q$ is as in 
Figure~\ref{fig:SMST1} (without the face $f$). Assume for contradiction 
that $G_Q$ also contains
a Black \SC, $h$. Denote by $A_{23},A_{41},A_h$, the \mobius bands 
that result from $f_1,f_3,h$ (resp.). 
As argued in Lemma~\ref{lem:SMST}, 
neither $\partial A_{23}$ nor $\partial A_{41}$
can be primitive in $H_B$. By Lemma~\ref{lem:nfsc}, 
$\partial A_h$ can be perturbed
to be disjoint from $\partial A_{23}$. Since $M$ contains no Dyck's surface,
$\partial A_{41}$ must intersect $\partial A_h$ transversely once  
(at either vertex $1$ or $4$). Now follow the argument of Lemma~\ref{lem:nfsc}.
Let $\Delta_{12}, \Delta_{34}$ be the bridge
disks constructed as in Lemma~\ref{lem:btfsc}.
These bridge disks can be taken disjoint from both $f_1$ and $h$. Then
boundary compressing $A_{h}$ along one of these disks gives a disk in
$H_B$ intersecting $\partial A_{41}$ once. But this implies $\partial A_{41}$
is primitive in $H_B$. 

Applying Lemma~\ref{lem:TB}, the same argument shows that configuration 
{\tt MSTS}  where the {\tt S} are White implies there are no Black \ttS.
\end{proof}

%\clearpage
%\newpage
\section{Bigons and trigons when $t=4$.}

Throughout this section assume $t=4$, there are no Dyck's surfaces embedded in $M$, and we are in \situationnscc. Recall that an {\it $\arc{ab}$-\SC} is a 
bigon Scharlemann cycle on the labels $a,b$.

\subsection{Embeddings of \SCs\ and mixed trigons in a handlebody.}

\begin{lemma}\label{lem:3SC+1SC+3e}
Given three $\arc{12}$-\SCs, one $\arc{34}$-\SC, and three more $\edge{34}$-edges, then either three of the six $\edge{12}$-edges are parallel in $G_F$ or three of the five $\edge{34}$-edges are parallel in $G_F$.  Furthermore, if the three extra $\edge{34}$-edges form a trigon Scharlemann cycle 
then three $\edge{12}$-edges are parallel.
\end{lemma}
\begin{proof}
Assume these \SCs\ are contained in the handlebody $H$.

If there exists a compressing disk $D$ in $H$ that separates the $\arc{12}$-\SCs\ from the $\arc{34}$-\SC, then in one of the solid tori of $H \cut D$ the three $\arc{12}$-\SCs\ are all parallel.  Hence three of their $\edge{12}$-edges are parallel in $G_F$.  

If there exists a compressing disk $D$ in $H$ disjoint from these Scharlemann cycles that is non-separating, then $H\cut D$ is a solid torus containing a $\arc{12}$-\SC\ and a $\arc{34}$-\SC.  Thus there are two disjoint \mobius bands in this solid torus, a contradiction.   

If no compressing disk of $H$ is disjoint from the $\arc{12}$-\SCs\ and the $\arc{34}$-\SC, then any pair of $\arc{12}$-\SCs\ are either parallel or have no parallel edges (else band two \SCs\ together along parallel edges).  
In particular, only two are parallel.  (If all three $\arc{12}$-\SCs\ were parallel there would be a disk separating them from the $\arc{34}$-\SC.)  The complement in $H$ of these $\arc{12}$-\SCs\ and the $\arc{12}$-arc of $K$ is then one or two solid tori, that meet $\hatF$ in annuli, and a 
ball (the parallelism).  Since the subgraph of $G_F$ consisting of vertices $3$ and $4$ and the five $\edge{34}$-edges must lie in one of the annuli, three $\edge{34}$-edges must be parallel.  

If the three extra $\edge{34}$-edges form a Scharlemann cycle trigon, then we must be in the former case of three parallel $\edge{12}$-edges, as the edges of a $\arc{34}$-\SC\ and a 
$\arc{34}$-Scharlemann cycle trigon cannot lie together in an annulus (e.g.\ Goda-Teragaito \cite{gt:dsokwylsagok}).
\end{proof}

\begin{lemma}\label{lem:4bsc}
Given two $\arc{12}$-\SCs\ and two $\arc{34}$-\SCs\, then either one pair is parallel or each pair has a pair of parallel edges. 
\end{lemma}

\begin{proof}
Assume no pair of edges of the $\arc{12}$-\SCs\ are parallel.  Then the complement of the graph of these four edges and the vertices $1$ and $2$ in the boundary of the handlebody must be a collection of annuli.   Hence the edges of the $\arc{34}$-\SCs\ lie in an annulus.   Since handlebodies are irreducible and the edges of these Scharlemann cycles cannot lie in a disk, the $\arc{34}$-\SCs\ are parallel.

Similarly, if no pair of edges of the $\arc{34}$-\SCs\ are parallel, then the $\arc{12}$-\SCs\ are parallel.
\end{proof}
\begin{lemma}\label{lem:trigonandtwomobius}
Given a $\arc{12}$-\SC, a $\arc{34}$-\SC, and a trigon of $\Lambda$ with two $\arc{12}$-corners and one $\arc{34}$-corner, then there are two embeddings in their genus $2$ handlebody $H$ up to homeomorphism.  One has a pair of parallel $\edge{12}$-edges; the other does not.  These are shown in Figure~\ref{fig:trigonandtwomobius} with $H$ cut along the two \SCs.
\end{lemma}

\begin{figure}[h]
\centering
\input{trigonandtwomobius.pstex_t}
\caption{}
\label{fig:trigonandtwomobius}
\end{figure}

\begin{proof}
Let $A_{12}$ and $A_{34}$ be the \mobius bands associated to the two \SCs\ in the handlebody $H$.  Then $H \cut (A_{12} \cup A_{34})$ is a genus $2$ handlebody $H'$.  The impressions $\tilde{A}_{12}$ and $\tilde{A}_{34}$ of the \mobius bands are primitive annuli in $H'$ and each has a primitivizing disk disjoint from the other annulus.  Attach a $2$-handle to $H'$ along the core of $\tilde{A}_{12}$ to form a solid torus $\calT$.  The primitivizing disk for $\tilde{A}_{12}$ extends to a disk $\delta$ giving a boundary-parallelism for the cocore $c$ of this $2$-handle.  Moreover $\delta$ is disjoint from the (now longitudinal) annulus $\tilde{A}_{34}$. 

Let $g$ be the trigon.  The two $\arc{12}$-corners of $g$ are identified along $c$ to form $\tilde{g}$ in $\calT$.  If $A_{12}$ and $g$ met transversely along $\arc{12}$ in $H$, then $\tilde{g}$ is an annulus.  Otherwise $\tilde{g}$ is a \mobius band.  In each situation, $c$ is a spanning arc of $\tilde{g}$, $\tilde{g}$ is properly embedded, and $\bdry \tilde{g}$ crosses the longitudinal annulus $\tilde{A}_{34}$ in $\bdry \calT$ just once.

If $\tilde{g}$ is a \mobius band, then its embedding in $\calT$ is unique up to homeomorphism.  If $\tilde{g}$ is an annulus, then one boundary component is disjoint from $\tilde{A}_{34}$ and trivial on $\bdry \calT$; because the spanning arc $c$ (on $\tilde{g}$) is trivial in $\calT$, the embedding of $\tilde{g}$ in $\calT$ is unique up to homeomorphism.  Recover $H'$ with the impression $\tilde{A}_{12}$ from $\calT \cut c$.  Carrying the two possibilities of $\tilde{g}$ along produces the two embeddings of $g$ in $H'$ shown in Figure~\ref{fig:trigonandtwomobius}.     Reconstitute $H$ and the two \mobius bands by sewing up  $\tilde{A}_{12}$ and $\tilde{A}_{34}$.  This gives the two claimed embeddings of $g$ in $H$.
\end{proof}

\begin{lemma}\label{lem:2SC2T}
Given a $\arc{12}$-\SC, a $\arc{34}$-\SC, a trigon of $\Lambda$ 
with two $\arc{12}$-corners and one $\arc{34}$-corner, and a trigon of $\Lambda$
with two $\arc{34}$-corners and one $\arc{12}$-corner,  then either a pair of $\edge{12}$-edges or a pair of $\edge{34}$-edges must be parallel.
\end{lemma}

\begin{proof}
Otherwise by Lemma~\ref{lem:trigonandtwomobius} each trigon lives in $H'=H \cut (A_{12} \cup A_{34})$ as pictured in the second part of Figure~\ref{fig:trigonandtwomobius}. These trigon faces form a meridian 
system for $H'$, where dual curves give generators $x,y$ of $H_1(H')$.  Up to swapping the generators and taking their inverses, the core of $\tilde{A}_{12}$ represents $xy^2$ in $H_1(H')$, and the core of $\tilde{A}_{34}$ may be oriented to then represent either $yx^2$ or $yx^{-2}$.   In either case, attaching $2$-handles to $H'$ along the cores of $\tilde{A}_{12}, \tilde{A}_{34}$ gives a manifold with non-trivial torsion in first homology.
But from Figure~\ref{fig:trigonandtwomobius}, one sees that attaching such $2$-handles gives a $3$-ball.
\end{proof}

 %\newpage
 
\subsection{Configurations containing an \ESC} 

Recall from section~\ref{sec:annuli}
 that an annulus is primitive if and only 
a component of its 
boundary is primitive in the ambient handlebody.

\begin{prop}%[Lemma]
\label{prop:primitivemobius}
If there is an \ESC\ such that the extending annulus is non-separating in its handlebody, then the boundary of the central \mobius band is primitive with respect to the extending annulus's handlebody.  Hence the extending annulus is primitive in its handlebody.
\end{prop}
%Note: restricting to t=4 makes this much shorter than we had in earlier versions.

\begin{proof}
Assume there is an \ESC\ on the corner $\arc{1234}$ giving rise to a central White \mobius band $A_{23}$ and an extending Black annulus $A_{12,34}$.   Assume $A_{12,34}$ is non-separating in $H_B$ and that $\bdry A_{23}$ is not primitive with respect to $H_B$.  

There exists a bridge disk $D_{12}$ for $\arc{12}$ that is disjoint from $A_{12,34}$.  Indeed, $D_{12}$ is a $\bdry$-compressing disk for the annulus $A_{12,34}$.  Performing the $\bdry$-compression on a push-off of this annulus produces a non-separating disk $D_B$ in $H_B$ that is disjoint from $A_{12,34}$ and $K$.
Let $\calT$ be the solid torus obtained by compressing $H_B$ along $D_B$.  Then $A_{12,34}$ is contained in $\calT$ and is $\bdry$-parallel into $\bdry \calT$. 

Either both curves of $\bdry A_{12,34}$ are primitive on $H_B$  or both are non-primitive.  
Since $\bdry A_{23}$ is a component of $\bdry A_{12,34}$, the former case is contrary to assumption.  Hence we may assume $\bdry A_{12,34}$ consists of two non-primitive curves in $H_B$.  Therefore in $\calT$ the $\bdry$-parallel annulus $A_{12,34}$ wraps $n > 1$ times longitudinally.  Then $\calN = \calT \cup \nbhd(A_{23})$ is a Seifert fiber space over the disk with two exceptional fibers of orders $2$ and $n$.
Furthermore $K \cap \calN$ is the arc $\arc{1234}$ that is the co-core of the long \mobius band $A_{23} \cup A_{12,34}$.   Lemma~\ref{lem:sSFS1bridge} now applies to produce a genus $2$ Heegaard splitting of $M$ in which $K$ is $0$-bridge, a contradiction.
\end{proof}

%%%%% totally shortened this by only considering t=4  .... removed claims, figures

\begin{lemma}%[ESC for $t=4$]
\label{lem:escfortis4}
If there is an \ESC\ then the extending annulus is $\bdry$-parallel in its handlebody but is not primitive.
\end{lemma}
\begin{proof}
Assume there is an \ESC\ on the corner $\arc{1234}$.  Let $A_{12,34}$ be the corresponding extending Black annulus.  

Assume $\bdry A_{23}$ is a primitive curve on $\bdry H_B$ with respect to $H_B$.  Then $H_B' = H_B \cup_{\bdry A_{23}} \nbhd(A_{23})$ is a handlebody in which $\arc{1234}$ is bridge.  Also $H_W' = H_W \cut A_{23}$ is a handlebody in which $\arc{41}$ remains bridge.  Thus $(H_B', H_W')$ is a Heegaard splitting of $M$ in which $K$ is $1$-bridge.  This contradicts the minimality assumption on $t$.  Hence $\bdry A_{23}$ cannot be primitive in $H_B$.  

Consequently, Proposition~\ref{prop:primitivemobius} also implies that $A_{12,34}$  must be separating in $H_B$.  Chopping $H_B$ along $A_{12,34}$ forms a genus $2$ handlebody $H_B'$ and a solid torus $\calT$.

We may assume $A_{12,34}$ is not longitudinal in $\calT$. Lemma~\ref{lem:sSFS1bridge} applied to the Seifert fiber space over the disk given by $\nbhd(A_{23}) \cup \calT$ contradicts that $t=4$.
\end{proof}

\begin{figure}[h!]
\centering
\input{esc4.pstex_t}
\caption{}
\label{fig:esc4}
\end{figure}

\begin{lemma}%[esc4]
\label{lem:esc4}
Given the \ESC\ of Figure~\ref{fig:esc4},
 then in $\Lambda$ any White bigon is an \SC\ and any White trigon is a $\arc{41}$-Scharlemann cycle.  %there are no White trigons.
Furthermore any such $\arc{23}$-\SC\ must have its edges parallel to those of $f_2$.
\end{lemma}

\begin{proof}
Given the \ESC\ on the corner $\arc{1234}$ as in Figure~\ref{fig:esc4} let $A_{23}$ be the corresponding White \mobius band and $A_{12,34}$ be the extending Black annulus.  By Lemma~\ref{lem:escfortis4} the annulus $A_{12,34}$ is parallel to an annulus $B_{12,34}$ on $\hatF$.

  The arguments of Lemma~\ref{lem:nospanningarc} prove that a White bigon must be a \SC, while the arguments of Lemma~\ref{lem:no23trigon} prove there is no White trigon with just one $\arc{23}$-corner.  Lemma~\ref{lem:no23S3}  shows there cannot be a $\arc{23}$-Scharlemann cycle of length $3$.  By an argument similar to that of Lemma~\ref{lem:2323trigon}, a trigon with two $\arc{23}$-corners and one $\arc{41}$-corner may be used in conjunction with the $\arc{23}$--\SC\ of the \ESC\ to form a bridge disk for $\arc{41}$ with interior disjoint from $B_{12,34}$; this provides a thinning of $K$.  Hence a White trigon must be a $\arc{41}$-Scharlemann cycle.

Let $\sigma$ be a $\arc{23}$-\SC\ and $f$ be the face it bounds.  One of the edges of $\sigma$ must lie in $B_{12,34}$, call it $e_1$, and the other, $e_2$, lies outside of $B_{12,34}$.  Then $e_1$ must be parallel to an edge $e_1'$ of $f_2$.  Let $e_2'$ be the other edge of $f_2$.  

\begin{figure}[h!]
\centering
\input{esc4-2.pstex_t}
\caption{}
\label{fig:esc4-2}
\end{figure}
We assume $e_2$, $e_2'$ are not parallel on $G_F$.  Then $f$, $f_2$ can be amalgamated along the parallelism of $e_1$, $e_1'$ to give a White meridional disk $D$ disjoint from $K$ and $B_{12,34}$.  See Figure~\ref{fig:esc4-2}.  But then $K$ can be isotoped into the solid torus $H_W-\nbhd(D)$ using the parallelism of $A_{12,34}$ to $B_{12,34}$, a contradiction.
\end{proof}

\begin{lemma}%[esc4g]
\label{lem:esc4g}
There must be a true gap contiguous to an \ESC\ of $\Lambda$.
\end{lemma}
\begin{proof}
Assume there is a bigon or trigon of $\Lambda$ on each side of an \ESC\ on the corner $\arc{1234}$ as in Figure~\ref{fig:esc4}.  We can find a bridge disk $D$ for either $\arc{23}$ or $\arc{41}$ which is disjoint (in the exterior of $K$) from both of these faces as well as the White face of the \ESC.  Let $B_{12,34}$ be the annulus on $\hatF$ to which the Black annulus $A_{12,34}$ (arising from the \ESC) is parallel by Lemma~\ref{lem:escfortis4}.  Since $\bdry D \cap \hatF$ is disjoint from the edges of the \ESC, it either lies inside $B_{12,34}$ and is isotopic to an edge of the \ESC\ or it lies entirely outside $B_{12,34}$.  In either case, the parallelism of $A_{12,34}$ to $B_{12,34}$ along with $D$ gives a thinning of $K$.
\end{proof}

\begin{lemma}\label{lem:noBBBBB}
There cannot be five consecutive bigons.
\end{lemma}
\begin{proof}
Assume there are five consecutive bigons.   Then by Lemma~\ref{lem:esc4g}, they appear as {\tt MSMSM}.  No two of the \ttM\ are parallel since otherwise either there would be a contradiction to Lemma~\ref{lem:parallelwithsamelabel} or the boundary of a \mobius band arising from one of the \SCs\ would bound a disk in $\hatF$.  Hence the two extending annuli of the two \ESC\ are not parallel.  In particular the annuli on $\hatF$ to which they are boundary parallel by Lemma~\ref{lem:escfortis4} have disjoint interiors.  But since the two extending annuli share a spanning arc, the two boundary parallelisms cause it two sweep out a compressing disk for the handlebody that contains it.  This disk however is a primitivizing disk for the annuli, contrary to Lemma~\ref{lem:escfortis4}.
\end{proof}

\begin{lemma}\label{lem:notwoescs}
There cannot be two \ESCs\ extending the same color but differently labeled 
\SCs.
\end{lemma}

\begin{figure}
\centering
\input{twoescs.pstex_t}
\caption{}
\label{fig:twoescs}
\end{figure}

\begin{proof}
Assume to the contrary that there are two \ESCs\ as shown in Figure~\ref{fig:twoescs}.  Lemma~\ref{lem:noBBBBB} accounts for when they share a Black bigon.   Indeed, using Lemma~\ref{lem:escfortis4} and the fact that the boundaries of 
two \mobius bands cannot be isotopic in $\hatF$ (no Klein bottle), a similar
proof works when they do not share a bigon.
\end{proof}

\begin{lemma}
\label{lem:parallelann}
Any two \ESCs\ of $\Lambda$ extending \SCs\ of the same labels must have their extending annuli parallel. In particular, the corresponding faces of these 
two \ESCs\ are parallel (see section ~\ref{sec:basics}).
\end{lemma}
\begin{proof}
By Lemma~\ref{lem:esc4}, the \SCs\ of these two \ESCs\ have their edges parallel.  Let $A$ and $A'$ be the extending annuli of the two \ESCs.  By Lemma~\ref{lem:escfortis4}, they are each $\bdry$--parallel to annuli $B$ and $B'$, respectively, in $\hatF$.  Let $D_{12}$ and $D'_{12}$ be bridge disks for $\arc{12}$ swept out by the parallelisms of $A$ to $B$ and $A'$ to $B'$ respectively.  Assuming $A$ and $A'$ are not parallel, $B \cup B'$ is a once-punctured torus.  In particular, $D_{12} \cup D'_{12}$ is a disk in the handlebody containing $A$ and $A'$ whose boundary transversally intersects each component of $\bdry A$ and $\bdry A'$ once.  Thus $A$ and $A'$ are primitive in their handlebody, contradicting Lemma~\ref{lem:escfortis4}.
\end{proof}

\begin{lemma}\label{lem:2ESC+2B}
If $\Delta=3$, then at a vertex of $\Lambda$ there cannot be two $\arc{1234}$-\ESCs\ and bigons at the remaining $\arc{12}$- and $\arc{34}$-corners.  That is, there cannot be the configuration {\tt gMSMgMSMgBgB} as shown in Figure~\ref{fig:2ESC+2M}.
\end{lemma}

\begin{figure}
\centering
\input{2ESC+2M.pstex_t}
\caption{}
\label{fig:2ESC+2M}
\end{figure}

\begin{proof}
Assume the configuration shown in Figure~\ref{fig:2ESC+2M} is around a vertex $x$ in $\Lambda$.
Let $A_{12,34}$ and $A'_{12,34}$ be the Black annuli extending the two \mobius bands arising from the two \SCs.  By  Lemma~\ref{lem:escfortis4} and Lemma~\ref{lem:parallelann} they are parallel to one another and they are both $\bdry$-parallel onto $\hatF$.  Let $B$ be the union of the annuli on $\hatF$ to which $A_{12,34}$ and $A'_{12,34}$ are $\bdry$-parallel.  The edges of the two \ESCs\ and the annulus $B$ are shown in Figure~\ref{fig:2ESCann} with the relevant labelings of edges.

\begin{figure}
\centering
\input{2ESCann.pstex_t}
\caption{}
\label{fig:2ESCann}
\end{figure}

If one of the two remaining Black bigons were an \ttM\ with an edge in $B$, then that edge would be parallel to an edge of each of the two \ESCs.  But then there would be three parallel edges that all have an endpoint labeled $x$, contradicting Lemma~\ref{lem:parallelwithsamelabel}.
If one of these bigons were an \ttS\ with an edge in $B$, then it would form a Black \mobius band with boundary in $B$.  This would imply the existence of an embedded Klein bottle, a contradiction.
Thus the $x''$ labels on all four vertices must be outside $B$.  This however creates an ordering violation.
\end{proof}

\begin{lemma}\label{lem:esc+fesc}
In $\Lambda$ there cannot be an \ESC\ and an \FESC\ such that the two interior 
\SCs\ are the same color but have different labels.
\end{lemma}

\begin{proof}
\begin{figure}
\centering
\input{esc+fesc.pstex_t}
\caption{}
\label{fig:esc+fesc}
\end{figure}
Assume otherwise.  
Up to relabeling we may assume the \ESC\ and \FESC\ appear as in Figure~\ref{fig:esc+fesc}.  This \FESC\ is the one shown in Figure~\ref{fig:FESC}(a).  As in Figure~\ref{fig:FESC}(a), let $f$ denote the Black bigon and $g$ denote the Black trigon of this \FESC.  Its White bigon forms a White \mobius band $A_{23}$.  Lemma~\ref{lem:TB} implies that the two $\edge{14}$-edges of it cobound a disk $\delta$ in $H_B$. 

The \ESC\ gives rise to a White \mobius band $A_{41}$ and a Black annulus $A_{34,12}$.
By Lemma~\ref{lem:escfortis4}, this annulus $A_{34,12}$ is $\bdry$-parallel onto an annulus $B_{34,12}$ on $\hatF$.  Either the trigon $g$ is contained within this solid torus of parallelism $\mathcal{T}$ between $A_{34,12}$ and $B_{34,12}$ or it is not.

{\bf Case I:}  The trigon $g$ lies within $\mathcal{T}$.  Then the edges of $g$ lie within the annulus $B_{34,12}$.  
The bigon $f$ can neither lie within $\mathcal{T}$ nor also be a bigon of the \ESC.  Otherwise $\bdry A_{23}$ would lie in $B_{34,12}$ and we could form either an embedded $\RP^2$ if it were inessential or an embedded Klein bottle if it were essential (since it would be parallel to $\bdry A_{41}$).

Since the $\edge{41}$-edge of $g$ lies in $B_{34,12}$, it is parallel  to a $\edge{41}$-edge of a Black bigon, say $h$, of the \ESC.  (By the preceding paragraph, $h$ is necessarily distinct from $f$.)  Then, since the $\edge{41}$-edges of $f$ and $g$ cobound the disk $\delta$, there must be a disk $\delta'$ that the $\edge{41}$-edges of $f$ and $h$ bound.  Furthermore we may assume the interior of $\delta'$ is disjoint from $A_{34,12}$.  

Because $f$ lies outside $\mathcal{T}$, there are rectangles $\rho_{12}$ and $\rho_{34}$ on the boundaries of the $1$-handle neighborhoods $\nbhd(\arc{12})$ and $\nbhd(\arc{34})$ between the corners of $f$ and $h$ that have interiors disjoint from $\mathcal{T}$.
Then together $f \cup \delta' \cup h \cup \rho_{12} \cup \rho_{34}$ forms a disk $D$ whose boundary is the union of the $\edge{23}$-edges of $f$ and $h$ (and arcs of the boundaries of the fat vertices $2$ and $3$).  We may now slightly lift the interior of $D$ into $H_B$ off $\hatF$ so that it is disjoint from $A_{41}$.   Attach $D$ to the White \mobius band $A_{23}$ along the $\edge{23}$-edge of $f$.  Then $D \cup A_{23}$ is an embedded \mobius band in $M$ that is disjoint from $A_{41}$ and has boundary (formed of the $\edge{23}$-edges of $g$ and $h$) lying in $B_{34,12}$.  As argued earlier, if this boundary were inessential we could form an embedded $\RP^2$, and if it were essential we could form an embedded Klein bottle.  Neither of these may occur.

{\bf Case II:} The trigon $g$ is not contained in $\mathcal{T}$.  Then the edges of $g$ meet the annulus $B_{34,12}$ only at the vertices.  

 Assume $f$ does not lie in $\mathcal{T}$ (so that $f$ is also not a bigon of the \ESC). We follow the bridge disk construction of Lemma~\ref{lem:btfsc}. There are rectangles that are disjoint from $\mathcal{T}$, $\rho_{12}$ and $\rho_{34}$, on the boundaries of the $1$-handle neighborhoods $\nbhd(\arc{12})$ and $\nbhd(\arc{34})$ between corners of $f$ and $g$, such that 
$f \cup \delta \cup g \cup \rho_{12} \cup \rho_{34}$ forms a disk $D$ whose interior may be lifted off $A_{34,12}$ and $\mathcal{T}$.   Note that the $\arc{34}$-corner of $g$ incident to the $\edge{34}$-edge of $g$ cannot lie in the 
rectangle $\rho_{34}$ since otherwise $g$ would intersect the 
interior of $\delta$. Then $\bdry D$ intersects $A_{34,12}$ only along the arc $\arc{34}$ (the $\arc{34}$ corner of $g$ that was disjoint from $\rho_{34}$) and at the vertex $2$.  A slight isotopy pulls $D$ off vertex $2$.  Now attach a bridge disk $D_{34}$ for the $\arc{34}$-arc contained in $\mathcal{T}$ to $D$ along the $\arc{34}$-arc.  Then $D' = D \cup D_{34}$ is a properly embedded disk in $H_B$ that intersects $B_{34,12}$ only in the spanning arc $D_{34} \cap B_{34,12}$.  Hence $D'$ is a primitivizing disk for the component $\bdry A_{41}$ of $\bdry B_{34,12}$.  However, $\bdry A_{41}$ cannot be primitive in $H_B$ by Lemma~\ref{lem:escfortis4}.  

Thus we must assume $f$ lies in $\mathcal{T}$.
Yet as in Case I (though using $g$ instead of $f$ there) there is a Black bigon $h$ of the \ESC\ so that the $\edge{41}$-edges of $g$ and $h$ together bound a disk $\delta'$.  Using $h$ and $\delta'$ in lieu of $f$ and $\delta$ we may apply the previous argument to again conclude that $\bdry A_{41}$ is primitive in $H_B$ contradicting Lemma~\ref{lem:escfortis4}.
\end{proof}

\begin{lemma}\label{lem:noBBBBT}
There cannot be four bigons adjacent to a trigon.
\end{lemma}
\begin{proof}
If there were, then by Lemma~\ref{lem:esc4g} they must form an \ESC\ and a \FESC\ that share a bigon.  Lemma~\ref{lem:esc+fesc} prohibits this configuration.
\end{proof}

\subsection{More configurations of bigons and trigons.}

\begin{lemma}\label{lem:SMTM}
Assume that at a vertex, $x$, of $G_Q$ there is a configuration {\tt SMTM} 
and another \ttS\ on the same corner as the \ttT. 
That is, WLOG assume we have the configurations of Figure~\ref{fig:SMTM}.
Then the edges of the length two and three $\arc{23}$-Scharlemann cycles 
cannot lie in
$\hatF$ in a subsurface which is a 3-punctured sphere or a 1-punctured torus. 
In particular, there cannot be two more $\edge{41}$-edges incident to $x$.
\end{lemma}
\begin{figure}[h]
\centering
\input{SMTM+S.pstex_t}
\caption{}
\label{fig:SMTM}
\end{figure}
\begin{proof}
Assume we have the configuration of Figure~\ref{fig:SMTM}. 
Let $A_{23}$ and $A_{41}$ be the two Black \mobius bands arising from the two \SCs.    Let $A_{12,34}$ be the White annulus formed by joining $f_3$ and $f_4$ along the arcs $\arc{12}$ and $\arc{34}$.  Write $\bdry A_{12,34} = \gamma_{23} \cup \gamma_{41}$ where $\gamma_{23}$ is the component formed from edges of $g$. 

Let $\Theta_{23}$ be the Black ``twisted $\theta$-band'' gotten by identifying
the corners of the Scharlemann cycle trigon along the $\arc{23}$-arc of $K$.  
By $\bdry \Theta_{23}$ we denote the $\theta$-graph formed from the three edges of the Scharlemann cycle trigon and the vertices $2$ and $3$ that is the intersection of $\Theta_{23}$ with $\hatF$.

The edges of $f_1$ and $g$, as edges in $G_F$, cannot lie in a 
3-punctured sphere. For by Lemma~\ref{lem:parallelwithsamelabel},
the edges of $f_1$ would have to be separating in this punctured sphere
and this contradicts the labelling around vertices $2,3$ of the edges
of $g$. 

So we assume for contradiction that edges of $f_1,g$ lie in a 1-punctured
torus in $\hatF$.  But then there is a properly embedded disk $D$ in 
$H_B$ that separates $A_{41}$ from $A_{23}$ and $\Theta_{23}$ and that is
disjoint from $K$ (in the boundary of the 3-manifold gotten by 
thickening the punctured torus, the $\arc{23}$-arc of $K$, and $f_1,g$).

Then $H_B - \nbhd(D)$ is two solid tori $\calT_{41}$ and $\calT_{23}$ containing $A_{41}$ and $A_{23} \cup \Theta_{23}$ respectively.  The subgraph of $G_F$ on $\bdry \calT_{23}$ consisting of the vertices $2$ and $3$ and the edges of the two $\arc{23}$-\SCs\ has three parallel edges, two from $g$ flanking one from $f_1$ (use Lemma~\ref{lem:parallelwithsamelabel}, the fact that $\calT_{23}$ is 
a solid torus, and the labelling at
vertices $2,3$ of the edges of $f_1,g$ on $\bdry \calT_{23}$, also
see Goda-Teragaito \cite{gt:dsokwylsagok}). Furthemore $\bdry A_{23}$ lies in an annulus on $\bdry \calT_{23}$ that runs twice longitudinally and $\bdry \Theta_{23}$ lies
in an annulus running three time longitudinally along $\calT_{23}$ (consider
the lens space resulting from attaching a 2-handle to $\calT_{23}$ along 
these annuli). We may take  $\gamma_{41}$ disjoint from $D$ and contained in 
$\bdry \calT_{41}$.  Since $\gamma_{23} \subset \bdry \Theta_{23}$, it is either trivial on $\bdry \calT_{23}$ or it runs three times longitudinally around $\calT_{23}$.  Furthermore, observe that $\arc{23}$ and $\arc{41}$ have  bridge disks disjoint from $D$ and $A_{23} \cup \Theta_{23}$ and $A_{41}$.

If $\gamma_{23}$ is trivial on $\bdry \calT_{23}$.  Then it must be isotopic to $\bdry D$ on $\hatF$ since otherwise the two edges forming it would be parallel to a $\edge{23}$-edge of $f_1$ violating Lemma~\ref{lem:parallelwithsamelabel}.  Thus $A_{12,34}$ is separating and $\bdry$-parallel (else $H_B \cup D$ contains a lens space summand) onto a neighborhood of $\bdry D \subset \hatF$.  
Since there exists a bridge disk for $\arc{23}$ in $\calT_{23}$ disjoint from $D$, there is an isotopy of the arc $\arc{1234}$ onto $\hatF$ fixing the complementary arc $\arc{41}$.  Hence $K$ is at most $1$-bridge, a contradiction.  

Thus we assume $\gamma_{23}$ runs three times longitudinally around $\calT_{23}$.
If $\gamma_{41}$ bounds a disk $D'$ on $\bdry \calT_{41}$, then $\nbhd(D' \cup A_{12,34} \cup \calT_{23})$ forms a punctured $L(3,1)$.  This cannot occur since $M$ is irreducible and not a lens space.  Hence $\gamma_{41}$ is essential on $\bdry \calT_{41}$.   

Let $\calN = \calT_{23} \cup \nbhd(A_{12,34}) \cup \calT_{41}$.  If $\gamma_{41}$ is longitudinal on $\bdry \calT_{41}$, then  $\calN$ is a solid torus containing $K$ contradicting the hyperbolicity of $K$, that $t=4$, or 
Lemma~\ref{lem:AEGor}. Thus $\calN$ is a Seifert fiber space over the disk with two exceptional fibers.  Hence $M \cut \calN$ is a solid torus.  Let $H_B' = \calT_{23} \cup \nbhd(f_3) \cup \calT_{41}$ and note that, by using the bridge disks disjoint from the \SCs, $K$ is isotopic onto $\bdry H_B'$.  Viewing $\nbhd(f_4)$ as a $1$-handle attached to the solid torus $M \cut \calN$, $M \cut H_B' = (M \cut \calN) \cup \nbhd(f_4)$ is a genus $2$ handlebody.  Hence $K$ is $0$-bridge with respect to this new genus $2$ Heegaard splitting, a contradiction. Thus the edges of $f_1,g$ do not lie on a 1-punctured
torus in $\hatF$.

To prove the last sentence of the Lemma, note that the first part implies that
all $\edge{41}$-edges must lie in an annulus on $\hatF$. 
Given two more $\edge{41}$-edges with an endpoint labeled $x$, then we have at least five such total. There is then a violation of Lemma~\ref{lem:parallelwithsamelabel}.
\end{proof}

\begin{lemma}\label{lem:scBsc+BB+BB} %{\tt (SMS + MS + MS)}
Given a collection of bigons in $G_Q$ as shown in Figure~\ref{fig:scBsc+BB+BB}, then the two $\arc{34}$-\SCs\ must be parallel such that the two $\edge{34}$-edges of $g$ and $h$ are parallel.
\end{lemma}

\begin{figure}
\centering
\input{scBsc+BB+BB.pstex_t}
\caption{}
\label{fig:scBsc+BB+BB}
\end{figure}

\begin{proof}
Assume we do have the collection of bigons shown in Figure~\ref{fig:scBsc+BB+BB}.  Let $f$, $g$, and $h$ denote the bigons as shown.  Let $A_{12}$ and $A_{34}$ be the Black \mobius bands arising from the two Black \SCs\ $h_{12}$ and $h_{34}$ in the run of $3$ bigons.  Let $A_{34}'$ be the Black \mobius band arising from the remaining Black \SC\ $g_{34}$.  Let $A_{41}$ be the White \mobius band arising from the White \SC\ $f_{41}$.     

Chop open $H_B$ along $A_{12}$ and $A_{34}$ to form the genus $2$ handlebody $H_B'$.  These leave annular impressions $\tilde{A}_{12}$ and $\tilde{A}_{34}$ that are each primitive on $H_B'$.  The bigon $f$ becomes a compressing disk that traverses the impressions $\tilde{A}_{12}$ and $\tilde{A}_{34}$ each once.  Further chopping along $f$ leaves a solid torus in which the \SC\ $g_{34}$ may only have two positions.  Figure~\ref{fig:one12one34one1234} shows the two possibilities of $g_{34}$ in $H_B'$ with respect to $f$.  Reforming $H_B$ by gluing $\tilde{A}_{12}$ and $\tilde{A}_{34}$ back into $A_{12}$ and $A_{34}$, we observe that the two $\arc{34}$-\SCs\ either have no two edges parallel or are parallel.

\begin{figure}
\centering
\input{one12one34one1234.pstex_t}
\caption{}
\label{fig:one12one34one1234}
\end{figure}

Assume no pair of edges of the two Black $\edge{34}$-\SCs\ are parallel as in Figure~\ref{fig:one12one34one1234}(b).  The complement in $\hatF$ of the subgraph of $G_F$ induced by the edges of the Black bigons is seen to be  one annulus and two disks.  The annulus does not meet the vertices $1$ or $2$.  Each disk meets each of the four vertices of $G_F$.  Around the boundary of one disk we see the vertices in the cyclic order $143412$; around the other we see $234321$.  The $\edge{41}$-edge of $f$ appears as the subarcs of the boundary of the first disk joining the consecutive $1$ and $4$ vertices.  The $\edge{41}$-edge of $f_{41}$ that is not an edge of $f$ cannot be in the first disk, since then $\bdry A_{41}$ would be isotopic to $\bdry A_{12}$ and a Klein bottle could be formed.  Thus it must be an edge of the second disk, and this choice is unique.

We can now find bridge disks for the arcs $\arc{12}$ and $\arc{34}$ in $H_B$ 
that guide isotopies of these arcs (rel $\bdry$) to arcs on $\bdry H_B$ that
are disjoint from $\bdry A_{41}$ except at vertices $1$ and $4$, and which arcs
are incident to $\bdry A_{41}$ on the same side. Furthermore,
we see that $\bdry A_{41}$ is primitive in $H_B$ (e.g.\ boundary compressing
$A_{12}$ along the above bridge disk for $\arc{12}$ gives a disk intersecting
$\bdry A_{41}$ once). Attaching a neighborhood of 
$A_{41}$ to $H_B$ forms a new genus $2$ handlebody $H_B''$, whose 
complement $H_W'' = H_W - \nbhd(A_{41})$ is a genus $2$ handlebody. 
As in the argument of Lemma~\ref{lem:btfsc}, $K$ can be isotoped to be 1-bridge
with respect to this new splitting (as the union of the arc $(12341)$, properly
isotopic to the bridge arc $(23)$ in $H_W''$, and an arc in $H_B''$ that is
a cocore of the attaching annulus $\nbhd(A_{41}) \cap H_B$). 
This contradicts the minimality of the 
presentation of $K$.

Hence the two $\arc{34}$-\SCs\ are parallel as in Figure~\ref{fig:one12one34one1234}(a).  Assume the  $\edge{34}$-edges of $g$ and $h$ are not parallel.  Then after an isotopy of the $\edge{34}$-edge of $g$, these two edges form $\bdry A_{34}$.  Thus we may regard $g \cup h$ as a White annulus $A_{23,41}$ that has $\bdry A_{34}$ as a boundary component.  Thus $A_{34} \cup A_{23,41}$ is a long \mobius band as if it arose from an \ESC\ centered at a Black \SC.  The argument of Lemma~\ref{lem:esc4} now applies to show that the Black bigon $f$ should have been an \SC.
\end{proof}

\begin{lemma}\label{lem:SMSandSMS}
There cannot be two triples of {\tt SMS} on the same corner at a vertex of $\Lambda$.
\end{lemma}

\begin{proof}
Assume there are two such triples on the corner $\arc{2341}$ of a vertex of $\Lambda$.  Then each triple contains a $\arc{23}$-\SC\ and a $\arc{41}$-\SC.  By Lemma~\ref{lem:4bsc} either each pair of like-labeled \SCs\ has a pair of parallel edges or one pair of the \SCs\ is parallel.  

Thus an outside edge of one triple must be parallel on $G_F$ to an edge of the
middle mixed bigon, $f$, of the other triple. Say this outside edge belongs 
to a $\arc{23}$-\SC\ of the first triple. Then the faces of the first triple,
along with $f$ and the $\arc{41}$-\SC\ of the second triple, can be used
in the argument of Lemma~\ref{lem:esc4g} to find a thinning. (The two mixed
bigons form the equivalent of an \ESC\ about this $\arc{23}$-\SC).
\end{proof}

%\newpage

\section{Lemma~\ref{lem:dbfg} and its proof}
Throughout this section assume $t=4$, there are no Dyck's surfaces embedded in $M$, and we are in \situationnscc.

\begin{lemma}%[Lemma DBFG]
\label{lem:dbfg}
Assume the configurations shown in Figure~\ref{fig:dbfgconfig} appear in $\Lambda$.  Then
\begin{enumerate}
\item $e_3$ is incident to opposite sides of $e_2 \cup e_6$,
\item $e_4$ is incident to opposite sides of $e_1 \cup e_5$,
\item $\bdry A_{34}$ transversely intersects each component of $\bdry A_{12,34}$ once, and 
\item neither $f_0$ nor $f_5$ is a bigon.
\end{enumerate}
Here $A_{34}$ is the Black \mobius band arising from the \SC\ $f_6$ and $A_{12,34}$ is the Black annulus arising from gluing $f_1$ and $f_4$ together along $\arc{12}$ and $\arc{34}$.
\end{lemma}
\begin{figure}
\centering
\input{dbfgconfig.pstex_t}
\caption{}
\label{fig:dbfgconfig}
\end{figure}

\begin{proof}
In addition to forming the Black annulus $A_{12,34}$ and Black \mobius band $A_{34}$ from $f_1$, $f_4$, and $f_6$,  the White \SCs\ $f_2$ and $f_3$ form White \mobius bands $A_{23}$ and $A_{41}$ respectively. By 
Lemma~\ref{lem:parallelwithsamelabel}, no component of $\partial A_{12,34}$ is
trivial in $\hatF$.

\begin{claim}%[Claim 1]
\label{claim:dbfgclaim1}
The annulus $A_{12,34}$ is non-separating in $H_B$.
\end{claim}
\begin{proof}
Assume $A_{12,34}$ is separating in $H_B$.  Then $\bdry A_{12,34}$ bounds an 
annulus $A$ in $\partial H_B$.  Furthermore, no edge of $G_F$ may be incident to opposite sides of either $e_2 \cup e_6$ or $e_1 \cup e_5$.  

{\bf Subclaim 1a:} Both $e_3$ and $e_4$ are disjoint from $A$.

\begin{proof}
By labelings, if $e_3$ is incident to just one side of $e_2 \cup e_6$, then $e_4$ must be incident to just one side of $e_1 \cup e_5$ as indicated in Figure~\ref{fig:dbfgclaim1} (i) where neither $e_3$ nor $e_4$ lie in $A$ or (ii) where both $e_3$ and $e_4$ lie in $A$.   However in Figure~\ref{fig:dbfgclaim1}(ii) either $\bdry A_{23}$ or $\bdry A_{41}$ is trivial in $A$ or both are isotopic to the core of $A$; hence an embedded $\RP^2$ or Klein bottle may be created.  (As drawn in Figure~\ref{fig:dbfgclaim1}(ii), $\bdry A_{41}$ is trivial and a $\RP^2$ may be created.)  Thus the edges must appear as in Figure~\ref{fig:dbfgclaim1}(i).
\begin{figure}
\centering
\input{dbfgclaim1.pstex_t}
\caption{}
\label{fig:dbfgclaim1}
\end{figure}
\end{proof}

{\bf Subclaim 1b:}  Both edges $e_7$ and $e_8$ of $f_6$ are disjoint from $A$.

\begin{proof}
By labelings, if either $e_7$ or $e_8$ were to lie in $A$ then so would the other.  Hence $\bdry A_{34}$ would be isotopic to the core of $A$.  Thus $A_{23}$, $A_{34}$, and $A_{41}$ are three disjoint \mobius bands, each properly embedded in either $H_B$ or $H_W$.  This contradicts Lemma~\ref{lem:3disjointmobiusbands}. 
\end{proof}

{\bf Subclaim 1c:}  For $i=3,4$, let $C_i$ be the corner on vertex $i$ cobounded by the edges $e_7$ and $e_8$ of $f_6$ that is disjoint from $A$.  Then $e_3$ must be incident to $C_3$ and $e_4$ must be incident to $C_4$.
Consequently $e_3$ is not parallel to either $e_2$ or $e_6$ and $e_4$ is not parallel to either $e_1$ or $e_5$.

\begin{proof}
By labelings, $e_3$ is incident to $C_3$ if and only if $e_4$ is incident to $C_4$.  Thus assume neither is incident to $C_3$ or $C_4$.  Then we may perturb $A_{34}$ to be disjoint from $A_{23}$ and $A_{41}$.   Again, this contradicts Lemma~\ref{lem:3disjointmobiusbands}.  

Since the edges $e_7$ and $e_8$ separate both $e_3$ from $e_2$ and $e_6$ at vertex $3$ and $e_4$ from $e_1$ and $e_5$ at vertex $4$, no pairs of these edges may be parallel.
\end{proof}

As a consequence of these subclaims, the subgraph of $G_F$ induced by the edges $e_1, \dots, e_8$ appear as in Figure~\ref{fig:dbfggraph} with possibly $e_7$ and $e_8$ swapped.  
\begin{figure}
\centering
\input{dbfggraph.pstex_t}
\caption{}
\label{fig:dbfggraph}
\end{figure}

Note that $P = \nbhd(A \cup e_3 \cup e_4) \subset \bdry H_B$ is a $4$-punctured sphere whose complement in $\bdry H_B$ is two annuli $A_1$ and $A_2$.  Moreover $\bdry A_{23}$ is isotopic to the core of one of these annuli and $\bdry A_{41}$ is isotopic to the core of the other;  they cannot be isotopic to the same core since together they would form a Klein bottle.  However, since $e_8$ (or $e_7$) lies outside of $P$, it lies in, say, 
$A_1$. But then $A_1$ connects $\bdry A_{23}$ to $\bdry A_{41}$.

This finishes the proof of Claim~\ref{claim:dbfgclaim1}.
\end{proof}

\begin{claim}%[Claim 2]
\label{claim:dbfgclaim2}
$\bdry A_{34}$ transversely intersects each component of $\bdry A_{12,34}$ just once.
\end{claim}

\begin{proof}
Assume otherwise.  Then the subgraph of $G_F$ induced by the edges of $f_1$, $f_4$, and $f_6$ appears as in Figure~\ref{fig:dbfggraph2} (disregard $\delta$ and $D$ for now).
\begin{figure}
\centering
\input{dbfggraph2.pstex_t}
\caption{}
\label{fig:dbfggraph2}
\end{figure}
Note that $\bdry A_{34}$ may be perturbed to be disjoint from $A_{12,34}$.  
Since $A_{12,34}$ is a Black annulus and $A_{34}$ is a Black \mobius band, there is a non-separating compressing disk $D$ for $H_B$ that is disjoint from both.  Cutting $H_B$ along $D$ we obtain a solid torus $\calT$ in which $A_{12,34}$ must be the boundary of a neighborhood around $A_{34}$.  Thus some component of $\bdry A_{12,34}$ must be isotopic on $\bdry H_B$ to $\bdry A_{34}$.  As in Figure~\ref{fig:dbfggraph2}, let $\delta$ be the region of $\bdry H_B$ giving this parallelism.  

Note that if either $e_3$ or $e_4$ lies in $\delta$ then it is parallel to an edge of $\bdry A_{12,34}$.   Then either $A_{23}$ or $A_{41}$ in union with $A_{12,34}$ forms a long \mobius band.  Assuming $\delta$ is as shown in Figure~\ref{fig:dbfggraph2}, then $e_3$ could lie in $\delta$ and $A_{23} \cup A_{12,34}$ would form the long \mobius band.   By Proposition~\ref{prop:primitivemobius} $\bdry A_{23}$ must then be primitive in $H_B$.  Yet since $A_{12,34}$ is the boundary of a neighborhood around $A_{34}$ in $\calT$, neither component of $\bdry A_{12,34}$ may be primitive in $H_B$.  

Thus we may assume neither $e_3$ nor $e_4$ may lie in $\delta$, and that $\delta$ is as pictured in Figure~\ref{fig:dbfggraph2}.  Therefore the two ends  of $e_3$ are incident to the same side of $e_2 \cup e_6$, and thus $\bdry A_{23}$ can be perturbed off $\bdry A_{12,34}$.    Then $A_{23}$, $A_{34}$, and $A_{41}$ are disjoint \mobius bands contrary to Lemma~\ref{lem:3disjointmobiusbands}.
\end{proof}

Without loss of generality, the subgraph of $G_F$ induced by the edges of $f_1$, $f_4$, and $f_6$ appears as in Figure~\ref{fig:dbfggraph3}.
\begin{figure}
\centering
\input{dbfggraph3.pstex_t}
\caption{}
\label{fig:dbfggraph3}
\end{figure}
Let $N = \nbhd( \arc{12} \cup \arc{34} \cup f_1 \cup f_4 \cup f_5) = \nbhd(A_{12,34} \cup A_{34})$.  Thus $B= \bdry N - \bdry H_B $ is an annulus and $A=\bdry H_B - \bdry N$ is an annulus.  Also $B \cup A = \bdry \calT$ where $\calT$ is a solid torus in which $A$ is longitudinal (since a bridge disk for $\arc{12}$, say, can be taken to be disjoint from $f_1 \cup f_4 \cup f_6$).  Thus there is a compressing disk for $H_B$ transversely intersecting $A_{12,34}$ once and each component of $\bdry A_{12,34}$ is primitive in $H_B$.  Finally, note that $N$ is a twisted $I$-bundle over a once-punctured Klein bottle.

\begin{claim}%[Claim 3]
\label{claim:dbfgclaim3}
The edge $e_4$ is incident to opposite sides of the closed curve $e_1 \cup e_5$.
\end{claim}
\begin{proof}
Assume both ends of $e_4$ are incident to the same side of $e_1 \cup e_5$.  Then both endpoints of $e_4$ lie on the same component of $\bdry A$.  Since $\bdry A_{41}$ is not trivial in $\bdry H_B$, either
\begin{enumerate}
\item $e_4$ is parallel to $e_1$,
\item $\bdry A_{41}$ is isotopic to the core of $A$, or
\item $\bdry A_{41}$ is isotopic to the $23$-component of $\bdry A_{12,34}$.
\end{enumerate}

In situation (1), $\bdry A_{41}$ is isotopic to the $41$-component of $\bdry A_{12,34}$.  Thus $A_{41} \cup A_{12,34}$ forms a long \mobius band containing the arc $\arc{3412}$.  Since there is a disk in $H_B$ transversely intersecting $A_{12,34}$ just once, we may form the handlebody $H_B' = H_B \cup_{\bdry A_{41}} \nbhd(A_{41})$ in which the arc $\arc{3412}$ is bridge.  Since the 
White arc $\arc{23}$ has a bridge disk disjoint from $A_{41}$, removing $\nbhd(A_{41})$ from $H_W$ forms the handlebody $H_W' = H_W - \nbhd(A_{41})$ in which $\arc{23}$ is bridge.  Thus together $H_B'$ and $H_W'$ form a genus $2$ Heegaard splitting of $M$ in which $K$ is $1$-bridge.  This contradicts the minimality of $t$.

In situation (2) we may form an embedded Dyck's surface
by taking the $0$-section of the twisted $I$-bundle $N$ in union with $A_{41}$. Its existence is contrary to assumption. 

Thus we are in situation (3) and we have the subgraph of $G_F$ shown in Figure~\ref{fig:dbfggraph4}.  This implies that the endpoints of the edge $e_3$ must lie on the same side of $e_2 \cup e_6$.  Hence the same argument applied to $e_4$ applies to $e_3$ allowing us to conclude that $\bdry A_{23}$ must be isotopic to the $41$-component of $\bdry A_{12,34}$.  Then together $A_{23} \cup A_{12,34} \cup A_{41}$ form a Klein bottle in $M$.
\begin{figure}
\centering
\input{dbfggraph4.pstex_t}
\caption{}
\label{fig:dbfggraph4}
\end{figure}
\end{proof}

\begin{claim}%[Claim 4]
\label{claim:dbfgclaim4}
The edge $e_3$ is incident to opposite sides of the closed curve $e_2 \cup e_6$.
\end{claim}
\begin{proof}
The argument for Claim~\ref{claim:dbfgclaim3} applies analogously.
\end{proof}

\begin{claim}%[Claim 5]
Neither $f_0$ nor $f_5$ is a bigon.
\end{claim}
\begin{proof}
We will show that $f_0$ cannot be a bigon.  The argument for $f_5$ is the same.

Assume $f_0$ is a bigon.  Then it must be a $41$-\SC\ as shown in Figure~\ref{fig:dbfgextrabigon}.  Let $A_{41}'$ be the White \mobius band arising from $f_0$.  We divide the argument into cases according to the relationships among the $\edge{41}$-edges of $f_0$ and $f_3$.
\begin{figure}
\centering
\input{dbfgextrabigon.pstex_t}
\caption{}
\label{fig:dbfgextrabigon}
\end{figure}

{\bf Case I:} No two edges of $f_0$ and $f_3$ are parallel in $\bdry H_B$.

Then $A_{41}$ and $A_{41'}$ intersect transversely and a neighborhood in $\bdry H_B$ of the union of the edges of $f_0$ and $f_3$ is a $4$-punctured sphere. (Otherwise $A_{41}$ and $A_{41}'$ could be isotoped to be disjoint from one another and from $A_{23}$;  two or three of these together then would form an embedded non-orientable surface in the handlebody $H_W$.)  Yet by Claim~\ref{claim:dbfgclaim4}, these edges lie in a $1$-punctured torus on $\bdry H_B$.  Hence one of the components of the $4$-punctured sphere must bound a disk in $\bdry H_B$.  This however implies that two edges of $f_0$ and $f_3$ are parallel.

{\bf Case II:} Two edges of $f_0$ and $f_3$ are parallel in $\bdry H_B$.

By Lemma~\ref{lem:parallelwithsamelabel}, it must be that either  $e_0$ is parallel to $e_5$ or  $e_1$ is parallel to $e_4$ on $\bdry H_B$.  Then either $\bdry A_{41}'$ or $\bdry A_{41}$ respectively is isotopic to the $\edge{41}$-component of $\bdry A_{12,34}$.  Hence we have a long \mobius band and may apply the argument of situation (1) in the proof of Claim~\ref{claim:dbfgclaim3} to obtain a thinning of $K$.
\end{proof}

This completes the proof of Lemma~\ref{lem:dbfg}
\end{proof}

%\newpage

%\begin{center}
%Part V
%\end{center}

\section{\conditionII\ for $t=4$.}\label{sec:tis4scc}
Throughout this section we assume we are in \conditionII\ for $t=4$.  In Theorem~\ref{thm:scctnot4} we will conclude that $t \neq 4$.

We may assume there is a meridian disk $D$ of $\hatF$ disjoint from $K$ and $Q$.
Let $F^*$ be $\hatF$ surgered along $D$.

\begin{lemma}\label{lem:scctorus}
The graph $G_F$ lies in a single component $\hatT$ of $F^*$.
\end{lemma}
\begin{proof}
Otherwise $F^*$ is two tori.  Say $D \subset H_W$ so that $D$ cuts $H_W$ into two solid tori $\calT_{23}$ and $\calT_{41}$, each containing one arc of $K$, say $\arc{23}$ and $\arc{41}$ respectively.  Then every White face of $\Lambda$ is a \SC\ and every Black face is not.   Moreover every Black face is either a mixed
bigon or has at least four sides. Finally, we may surger any disk face
of $G_Q$ so that its interior is disjoint from $F^*$ (by 
Corollary~\ref{cor:AEntscc} and the strong irreducibility of the Heegaard 
splitting, an innermost curve of intersection is either a copy of $D$ or
bounds a meridian disk of $\calT_{23}$ or $\calT_{41}$ that is disjoint 
from $K$. In the former case we can surger the intersection away. The latter
combines with Corollay~\ref{cor:bigonsforall} to give the contradiction that
$M$ contains a lens space summand.)

The edges of any two White Scharlemann cycles of $\Lambda$ of length at most three and on 
the same
label pair, lie in exactly two parallelism classes in the graphs on 
$\partial \calT_{23}$ or $\partial \calT_{41}$. Thus 
Lemma~\ref{lem:parallelwithsamelabelAPE} prevents there from being three or
more bigon,trigon Scharlemann cycles at the $\arc{23}$-corners of a vertex of $\Lambda$ 
or at the 
$\arc{41}$-corners of a vertex of $G_Q$. In the language of special vertices
(see below),
this means there must be at least one true gap at a $\arc{23}$-corner and
at a $\arc{41}$-corner of any special vertex.

Recall that two bigon faces of $G_Q$ are said to be {\it parallel} if each edge
of one is parallel to an edge of the other.

\begin{claim}\label{claim:parallelmixed}
In $\Lambda$ if a Black bigon is adjacent to a White bigon or trigon 
Scharlemann
cycle, then all Black bigons are parallel.  Moreover, at a vertex of $\Lambda$, at most two Black corners have bigons and these would have opposite labels.
\end{claim}
\begin{proof}
To the contrary, assume $\Lambda$ has two non-parallel mixed 
Black bigons $f$ and $g$ such that $f$ shares an edge with a White bigon or trigon Scharlemann cycle.
Note that neither edge of $f$ is parallel on $F^*$ to an edge of $g$, 
else the two Black
faces can be combined to give a disk that contradicts Lemma~\ref{lem:AEGor} and
the strong irreducibility of the Heegaard splitting. 
Then $\calN = \nbhd(f \cup \arc{12} \cup \arc{34} \cup \calT_{23} \cup \calT_{41})$ is a genus $2$ handlebody in which $K$ is isotopic to an arc on $\bdry \calN$ and a bridge arc (using the bigon/trigon Scharlemann cycle 
that can be surgered to lie
entirely in $\calT_{23}$ or $\calT_{41}$).  Attaching $\nbhd(g)$ to $\calN$ 
forms a Seifert fiber space over the disk with two exceptional fibers.  (Otherwise $K$ would be contained in the solid torus, $\nbhd(g) \cup \calN$. As $K$ is
hyperbolic and $M$ is not a lens space, $K$ would have to be a core of this
solid torus. But then $K$ is $0$-bridge with respect to $H_W \cup H_B$).  
Then as usual $M \cut (\calN \cup \nbhd(g))$ is a solid torus $\calT$.   So now $\calN$ and $\calT \cup \nbhd(g)$ form a genus $2$ Heegaard splitting of $M$ in which $K$ is $1$-bridge.

Given that all Black bigons are parallel, a vertex of $\Lambda$ may have 
at most one Black bigon at a $\arc{12}$-corner and one at a 
$\arc{34}$-corner.  
\end{proof}

\begin{remark}
To make the proof of Claim~\ref{claim:parallelmixed} consistent with the proof of Theorem~\ref{thm:changingHS}
we need to sharpen the argument to show that either Claim~\ref{claim:parallelmixed} holds
(without changing the splitting) or  $M$ is a Seifert fiber space with an
exceptional fiber of order $2$ and furthermore that $K$ is $1$-bridge with 
respect to a vertical splitting of $M$. The argument given in Claim~\ref{claim:parallelmixed} 
applied to an \ESC\ shows this. So we must show there is an \ESC.  We have 
shown that any black interval either  belongs to a mixed bigon or
corresponds to a true gap. We have also shown that there
is a true gap at a $\arc{41}$-corner and a $\arc{23}$-corner. If there
is no \ESC\ then Lemma~\ref{lem:t4trulyspecial} shows that $\Gamma$ has a special
vertex of weight $N=4$. The fact that there are at least two
true gaps at this special vertex implies that there are 
at most three true gaps and one more corner which is has a trigon.
Between these four corners every other corner belongs to a bigon
of $\Gamma$.  First, note that if
there is no \ESC, then there is no triple of bigons. Otherwise
on either side of this triple must be black gaps, hence true
gaps -- but then there are four true gaps (two black and two
white). There is only one way that there is no triple of bigons, 
and that is that $\Delta=3$ and there are exactly two bigons between 
each of these four corners. But this implies that two of these
four corners are black, which along with the white gaps makes
four true gaps. Thus there must be an \ESC. 
\end{remark}

To finish the proof of Lemma~\ref{lem:scctorus}, 
Claim~\ref{claim:parallelmixed} and Lemma~\ref{lem:parallelwithsamelabel}
imply that there cannot be $9$ mutually parallel edges in $\Lambda$.  By 
Lemma~\ref{lem:t4trulyspecial}, $\Lambda$ must have a special vertex of 
length $N=4$. Recall from the beginning of section~\ref{sec:t4noscc}, that
a ``true gap'' is a corner of a special vertex that is not known to be a 
bigon or trigon. By Lemma~\ref{lem:N4trulyspecialtypes}, the special vertex 
of $\Lambda$ has at most three true gaps. If around a vertex of $\Lambda$ there is a Black bigon 
adjacent to a White bigon or trigon (which are Scharlemann cycles), then Claim~\ref{claim:parallelmixed} and Lemma~\ref{lem:parallelwithsamelabel} imply at least $4$ Black corners at the special vertex have true gaps (any Black face is a true gap or a
bigon). Thus a special vertex of $\Lambda$ must have any White bigon or trigon flanked by true gaps.  But then, since there are at least two true gaps at White corners, this vertex must have at least $4$ true gaps.
\end{proof}

Let $\hatT$ be as in Lemma~\ref{lem:scctorus} and $\calT$ be the solid torus
it bounds (gotten by surgering the Heegaard handlebody along $D$). By possibly
rechoosing $D$, we may assume that any disk face of $G_Q$ may be surgered
so that its interior is disjoint from $\hatT$ (as argued in the proof of 
Lemma~\ref{lem:scctorus}). 

\begin{lemma}\label{lem:sccnoSandS}
There cannot be two \SCs\ of the same color but with different labels.
\end{lemma}
\begin{proof}
Otherwise the \mobius bands to which they give rise are disjoint and have parallel boundaries on $\hatT$.  From this we may construct an embedded Klein bottle.
\end{proof}

\begin{prop}\label{prop:sccnoesc}
There is no \ESC.
\end{prop}

\begin{proof}
Assume we do have an \ESC\ as in Figure~\ref{fig:esc4}. Let $A_{23}$ be the associated White \mobius band and $A_{12,34}$ be the Black annulus. We may take
both to be properly embedded in $\calT$ or its exterior.

\begin{claim}\label{claim:Disblack}
$D \subset H_B$
\end{claim}

\begin{proof}
If $D \subset H_W$, then $A_{23}$ is contained in $\calT$.

Set $\calN = \calT \cup \nbhd(A_{12,34})$.  Observe that $\calN$ is a Seifert fiber space over the annulus with one exceptional fiber and that $K \subset \calN$.  Then $\bdry \calN$ is two tori and one of these components 
must compress outside of $\calN$.  Such a compression produces a 
$2$-sphere which must bound a $3$-ball $B$.  Since $K \not \subset B$, 
$\calN \not \subset B$.  Therefore this torus bounds a solid torus 
$\calT'$ with interior disjoint from $\calN$.  

Assume $\calN \cup \calT'$ is a solid torus. Since $K$ is hyperbolic and
$M$ does not contain a lens space summand, $K \subset \calN \cup \calT'$ must
be isotopic to its core. That is, $K$ is isotopic to a core curve of $H_W$.
But then $K$ is $0$-bridge.

Thus $\calN \cup \calT'$ must form a Seifert fiber space over the disk with two exceptional fibers.  Hence $M \cut (\calN \cup \calT')$ is a solid torus $\calT''$.  Now we may form a genus $2$ Heegaard splitting of $M$ by taking $H_B' = \calT' \cup \nbhd(f_1) \cup \calT''$ and $H_W' = M \cut H_B' = \calT \cup \nbhd(f_3)$.  Then $K \subset H_W'$ and $K$ may be isotoped so that it is $1$-bridge with respect to this Heegaard splitting.
\end{proof}

Since $D \subset H_B$ by Claim~\ref{claim:Disblack}, $A_{12,34}$ is contained in
$\calT$ of $H_B \cut D$.

\begin{claim}\label{claim:mobbdryislong}
$\bdry A_{23}$ is a longitude of $\calT$.
\end{claim}
\begin{proof}
If it is not, then we may form a Seifert fiber space over the disk with two exceptional fibers $\calN = \calT \cup \nbhd(A_{23})$.  
Now apply Lemma~\ref{lem:sSFS1bridge} to produce a genus $2$ Heegaard splitting of $M$ in which $K$ is $0$-bridge.
\end{proof}

There can be no White mixed bigon.  Otherwise $K$ could be isotoped into 
$\calT \cup \nbhd(A_{23})$ which is a solid torus by Claim~\ref{claim:mobbdryislong}. As $K$ is hyperbolic and $M$ contains no lens space summand, $K$ is
isotopic to a core of $\calT \cup \nbhd(A_{23})$. 
But then the core, $L$, of the solid torus $\calT$ is a (2,1)-cable 
of $K$.  
As $L$ is a core of $H_B$, Claim~\ref{claim:ckt1} contradicts that $t=4$.

There cannot be a White $\arc{41}$-\SC\ by 
Lemma~\ref{lem:sccnoSandS}.  Consequently, there can be no bigon of $\Lambda$
at a $\arc{41}$-corner of a vertex in $\Lambda$.  This prohibits there being $4$ parallel edges at a vertex of $\Lambda$.  Hence by Lemmas~\ref{lem:t4trulyspecial} and \ref{lem:N4trulyspecialtypes} there must be a special vertex $v$ in $\Lambda$.  Such a special vertex has at most $5$ gaps (counting both trigons and true gaps).  

Furthermore, around $v$ there may only be two $\arc{23}$-corners 
that have \SCs.    Otherwise, since $\bdry A_{12,34}$ bounds an annulus on $\hatT$, there would be two edges of $G_F$ that meet a vertex at the same label and are parallel on $\hatT$.  Lemma~\ref{lem:parallelwithsamelabelAPE} prevents this.  

Since $D \subset H_B$ there can be no Black \SCs.  Otherwise 
one would give a \mobius band intersecting the separating annulus $A_{12,34}$ 
transversely in a single arc in $\calT$. 

Thus in total, $v$ must have at least $\Delta$ gaps at all the $\arc{41}$-corners and $\Delta-2$ gaps among the $\arc{23}$-corners.  Since $v$ has at most $5$ gaps, it must be the case that $\Delta = 3$.  Hence with three gaps at the $\arc{41}$-corners and one gap at a $\arc{23}$-corner, either all three $\arc{12}$-corners have a mixed bigon or all three $\arc{34}$-corners have a mixed bigon.    Yet since their edges must be parallel to the edges of $\bdry A_{12,34}$ on $\calT$, there will have to be two edges parallel on $\calT$ that meet a vertex 
at the same label, contrary to Lemma~\ref{lem:parallelwithsamelabelAPE}.
\end{proof}

\begin{lemma}\label{lem:sccnoBBB}
There cannot be three consecutive bigons in $\Lambda$.
\end{lemma}
\begin{proof}
By Proposition~\ref{prop:sccnoesc}, a triplet of bigons must be two \SCs\ 
flanking a mixed bigon.  But then these two \SCs\ will have the same color and different labels, contrary to Lemma~\ref{lem:sccnoSandS}.
\end{proof}

\begin{lemma}\label{lem:sccnofesc}
Assume $\Lambda$ contains an \FESC\ and let $h$ be its \SC. Then any 
bigon in $\Lambda$ of the same color as $h$ (Black or White) 
is a \SC\ on the same label pair.
\end{lemma}
\begin{proof}
WLOG assume there is an \FESC\ as in Figure~\ref{fig:FESC}(a).  The graph 
induced by the edges of the \FESC\ is shown abstractly in Figure~\ref{fig:FESC}(b).  As mentioned above, any face of $\Lambda$ can be taken to have interior
disjoint from $\hatT$. 

First assume $D \subset H_W$.  
Since the graph of Figure~\ref{fig:FESC}(b) lies in $\hatT$, one of $\alpha$, $\beta$, $\alpha\beta$ bounds a disk in $\hatT$.  It cannot be $\alpha$ since $\alpha$ bounds a \mobius band $A_{23}$.   If it were $\beta$ then as in Lemma~\ref{lem:btfsc} we could form bridge disks that guided isotopies of the arcs$\arc{12}$ and $\arc{34}$ onto $\hatT$ so that $K$ would be contained in $\calT$. 
But then $K$ is isotopic to a core of $\calT$, and $K$ is $0$-bridge in the
given Heegaard splitting. Thus $\alpha\beta$ bounds a disk $E$ in $\hatT$.

Let $\calN = \nbhd(\arc{12} \cup \arc{34} \cup f \cup g)$.
As shown in Claim~\ref{claim:primitivity}, $\calN \cup \nbhd(E)$ is a trefoil complement (and the meridian of this trefoil complement is $\bdry A_{23}$).  
Then $\calN' = \calN \cup \nbhd(E) \cup \nbhd(A_{23})$ has incompressible 
boundary $T'$.  Therefore, by Lemma~\ref{lem:MBB}, 
$\calT' = M \cut \calN'$ is a solid torus. 
$K$ intersects $\calT'$ in only the arc $\arc{41}$. 
By Lemma~\ref{lem:1bdryslope} (and that $K$ is not locally knotted), $T' -\nbhd(K)$ compresses in $\calT'-\nbhd(K)$ to show that $\arc{41}$ is $\bdry$-parallel in $\calT'$. ($T'-\nbhd(K)$ cannot compress into $\calN'$ since that would imply the arc $K \cap \calN'$ and hence $K$ is isotopic into $\calT'$.  But then $T'$ would be an essential torus in the exterior of $K$, a contradiction.)

Since $\alpha$ is a primitive curve on $\calN$ by Claim~\ref{claim:primitivity}, $\calN \cup \nbhd(A_{23})$ is a genus $2$ handlebody, $H_B'$. On the other hand, the complement of $H_B'$ is a genus two handlebody, $H_W'$ (the union of $\calT'$ and a $1$-handle dual to $E$).   As in the proof of  Claim~\ref{claim:TBIII}, $K$ can be written as the union of two arcs:  $\arc{34123}$, $\kappa$. Here $\kappa$ is $K \cap H_B'$ and is properly isotopic in $H_B'$ to a cocore of the annulus $\nbhd(\partial A_{23}) \cap \bdry\calN$.  As $\partial A_{23}$ primitive in $\calN$, $\kappa$ is a bridge arc in $H_B'$.  On the other hand, the arc $\arc{34123}$ is $K \cap H_W'$ and is properly isotopic in $H_W'$ to $\arc{41}$. As $\arc{41}$ is bridge in $\calT'$, it is bridge in $H_W'$. Thus $K$ is $1$-bridge with respect to a genus $2$ splitting of $M$. 

\begin{remark} 
This is one of the special cases of the proof of Theorem~\ref{thm:changingHS}.
Here $M$ is $n/2$-surgery on the trefoil (and hence a Seifert fiber space over the
$2$-sphere with an exceptional fiber of order $2$ and one of order $3$.) 
The argument presents $K$ as $1$-bridge with respect to the splitting of 
$M$ gotten from a genus $2$ Heegaard splitting of the trefoil exterior: i.e.\
remove a neighborhood of the unknotting tunnel from the exterior of the trefoil 
for one handlebody
of the splitting, then the filling solid torus in union with a neighborhood of 
the unknotting tunnel is the other.
\end{remark}

If $D \subset H_B$, then the first part of the argument of Claim~\ref{claim:TBIII} shows that $\partial A_{23}$ is 
longitudinal in $\calT$. Thus $\calN' = \calT \cup \nbhd(A_{23})$ is a solid torus. Let $l$ be a bigon of $\Lambda$ of the same color as
$h$, that is a White bigon. If $l$ is a mixed bigon, then $l$ guides an isotopy of $K$ into the solid torus $\calN'$. As $K$ is hyperbolic
and $M$ contains no lens space summand, $K$ can be isotoped to be a core of $\calN'$. 
But the core, $L$, of $\calT$ is a (2,1)-cable  
of $K$ (in $\calN'$). As $L$ is a core of $H_B$, Claim~\ref{claim:ckt1} contradicts that $t=4$. 
Thus $l$ cannot be a mixed bigon. By 
Lemma~\ref{lem:sccnoSandS}, $l$ must be a \SC\ on the same label pair as $h$.         
\end{proof}

\begin{thm}\label{thm:scctnot4}
In \conditionII, $t \neq 4$.
\end{thm}
\begin{proof}
By Lemma~\ref{lem:sccnoBBB}, $\Lambda$ cannot have a triple of bigons. 
Hence by Lemma~\ref{lem:t4trulyspecial} there must be a special vertex $v$ in $\Lambda$.  By Lemma~\ref{lem:N4trulyspecialtypes}, such a special vertex has at most $5$ gaps (counting both trigons and true gaps). That is, there are at most 5
corners at $v$ that do not belong to bigons of $\Lambda$.  
Furthermore $v$ must be of type $[8,1]$ with $\Delta = 3$ or type $[7,4]$ with 
$\Delta=3$, else it must have a triple of bigons.

If the special vertex $v$ has type $[7,4]$, then there are five gaps.  
Having no triple of bigons implies there must be at least two instances of adjacent bigons flanked by gaps. 
First assume there is no sequence {\tt TBBT} at $v$ (notation as described at 
the beginning of section~\ref{sec:t4noscc}). 
Then there must be a {\tt TBBGBBT}. Lemmas~\ref{lem:sccnoSandS} and \ref{lem:sccnofesc} force this
to be {\tt TSMGMST}. Again applying Lemmas~\ref{lem:sccnoSandS} and \ref{lem:sccnofesc}, we must have {\tt TTSMGMSTT}. But then there is a
triple of bigons at $v$, contradicting Lemma~\ref{lem:sccnoBBB}.
So assume there is a sequence {\tt TBBT} at $v$. By Lemmas~\ref{lem:sccnoSandS} and \ref{lem:sccnofesc}, we may assume we have {\tt TMSTg}, where the
{\tt MST} is an \FESC. There
must be at least two more bigon pairs, {\tt BB}. In particular there must be
a sequence {\tt gBBgBBg}, possibly including part of the above sequence.
As argued above, the existence of the \FESC\ forces {\tt ggSMgMSgg} and hence
a triple of bigons, a contradiction.

If the special vertex has type $[8,1]$, then there are four gaps.  Having no triple of bigons implies that the gaps occur at every third corner separating four pairs of adjacent bigons.  There is now no way to label these bigons without violating Lemma~\ref{lem:sccnoSandS}.
\end{proof}

%%%%%%%%%%%%%%%%%%%%%%%%

%\section{Appendix: Small Seifert fiber spaces containing a Dyck's surface.}\label{sec:appendix}
%\appendix{Small Seifert fiber spaces containing a Dyck's surface.}\label{sec:appendix}
\appendix

\appsection{Small Seifert fiber spaces containing a Dyck's surface.}\label{sec:appendix}
This appendix proves Theorem~\ref{SFSDycks}, which restricts the small Seifert fiber
spaces containing a Dyck's surface.

%This appendix proves Theorem~\ref{SFSDycks}, describing those small Seifert fiber spaces containing a Dyck's surface.

In what follows a {\em surface} will always be connected.
\begin{defn} $M=S^2(s_1/t_1,s_2/t_2,s_3/t_3)$ is defined as follows.
Let $\widetilde{M}=S^1 \times F$ where $F$ is a pair of pants. An orientation
on each of the factors induces coordinates $(s,t)$ where $s$ is the 
number of times around the $S^1$ factor. To each component of $\partial 
\widetilde{M}$ attach a solid torus $T_i$ so that the meridian of $T_i$
is identified with the curve $(s_i,t_i)$. The resulting manifold is $M$.
$\widetilde{M}$ is a circle bundle, $p:\widetilde{M} \to F$. $M$ is a Seifert
fiber space over $S^2$ where each $T_i$ is a neighborhood of an exceptional
fiber of order $t_i$. 
\end{defn}

\begin{thm}\label{SFSDycks}
Let $M$ be a SFS over the 2-sphere with three exceptional fibers. 
If $M$ contains an incompressible Dyck's surface, then either
\begin{itemize}
\item[(A)] $M=S^2(2/p_1,2/p_2,2/p_3)$ where each $p_i$ is an odd integer; or
\item[(B)] one of the exceptional fibers of $M$ has order $2$
and a second has order which is a multiple of $4$; or
\item[(C)] $M$ has exceptional fibers of order $2$ and $3$. In fact,
$M$ is $(2/n)$-surgery on a trefoil knot; or
\item[(D)] $M$ has two exceptional fibers of order $2$. In this case
$M$ contains a Klein bottle; or
\item[(E)] $M$ has two exceptional fibers of order $3$.
\end{itemize}
\end{thm}

\begin{remark} The Teragaito examples are in $S^2(-1/2,1/6,2/7)$, which
by the above does not contain a Dyck's surface.
\end{remark}

\begin{proof}
Let $K$ be an incompressible Dyck's surface in $M$. 
Theorem 2.5 of \cite{Fr} shows that $K$ can be isotoped to be 
either pseudovertical or pseudohorizontal. $K$ is said to be 
{\em pseudovertical} if $\widetilde{K} =K \cap \widetilde{M}$ is a 
vertical annulus whose
boundary lies in distinct components of $\partial \widetilde{M}$,
$\partial T_i, \partial T_j$; furthermore,
$K \cap T_i$ and $K \cap T_j$ are one-sided incompressible surfaces
in $T_i$ and $T_j$. $K$ is said to be {\em pseudohorizontal}
if $K \cap \widetilde{M}$ is horizontal under the circle fibration and $K$
intersects each of $T_1,T_2,T_3$ in either a family of meridian disks
or in a one-sided incompressible surface. Note that by Corollary 2.2 of
\cite{Fr}, a one-sided incompressible surface in a solid torus has a
single boundary component.
\end{proof}

\begin{claim}\label{horizontal}
If $K$ is pseudohorizontal then one of the conclusions to
Theorem~\ref{SFSDycks} holds.
\end{claim}

\begin{proof}
Assume $K$ is pseudohorizontal. Then $p: \widetilde{K} \to F$ is a cover of
index $\lambda \geq 1$. Note that $\lambda$ is the intersection number 
of $\widetilde{K}$ with any circle fiber of $\widetilde{M}$. 

Assume $\lambda=1$. Then a component
$c$ of $\partial \widetilde{K}$ would intersect the Seifert fiber in the
neighborhood of the corresponding exceptional fiber once. This immediately
implies that $c$ does not bound a meridian of that solid torus neighorhood.
Thus $K$ intersects the neighborhood of an exceptional fiber in a single
one-sided incompressible surface. As the Euler characteristic of $K$ is
$-1$, it must be the $K$ intersects the neighborhood of each of the exceptional
fibers of $M$ in a \mobius band. As $\widetilde{K}$ is a section for the 
circle bundle we can use it to define the product structure on $\widetilde{M}$. This
gives coordinates on the boundary of each exceptional
fiber so that $\widetilde{K} \cap \partial T_i$ is $(0,1)$ and the circle fiber 
(which is the Seifert fiber of $M$) is $(1,0)$. As $K \cap T_i$ is a 
\mobius band, its boundary must
intersect the meridian of the solid torus twice. Thus in these coordinates,
the meridian is $(2,p_i)$ were $|p_i|$ is the order of the 
exceptional fiber (and odd).
Thus $M = S^2(2/p_1,2/p_2,2/p_3)$ and we have conclusion $(A)$ above.

Assume $\lambda > 1$. As $K$ is 1-sided it cannot intersect all of the
$T_i$ in disks. On the other hand, since $\widetilde{K}$ is a $\lambda$-fold
cover of the pair of pants $F$, it must have Euler characteristic $-\lambda$. 
Thus $K$ must
intersect some $T_i$ in disks. 

Assume first that $K$ intersects only $T_1$ in disks. Let $r \leq 0$ be the 
sum of the Euler characteristics of the one-sided surfaces $K \cap T_2,
K \cap T_3$. Then $ -1 = \chi(K) = -\lambda + \lambda/p + r $
where $p$ is the order of the singular fiber at $T_1$. This implies that
$r=0$ and $\lambda (p-1)=p$.  As $\lambda$ is a multiple of $p$, this
implies that $\lambda=p=2$. But then we conclude that $\widetilde{K}$ has
exactly three boundary components and Euler characteristic $-2$. This
implies that $\widetilde{K}$ is non-orientable. But $\widetilde{K}$ covers
the orientable $F$.

So assume $K$ intersects $T_1,T_2$ in disks and $T_3$ in a 1-sided
surface with Euler characteristic $r \leq 0$. Let $p_1,p_2$ be the 
orders of the singular fibers of $T_1,T_2$. Then we have the following
equality $(*)$:  $-1 = \chi(K) =
-\lambda + \lambda/p_1 + \lambda/p_2 + r$. 

\begin{claim}\label{clm:possibilities}
One of the following must hold.
\begin{itemize}
\item[(1)] $r=-1$ and $p_1=p_2=2$; or
\item[(2)] $r=0$, $p_1=2,p_2=3$, and $\lambda=6$; or
\item[(3)] $r=0$, $p_1=2,p_2=4$, and $\lambda=4$; or
\item[(4)] $r=0$, $p_1=3=p_2$, and $\lambda=3$.
\end{itemize}
\end{claim}

\begin{proof}
Noting that $\lambda$ is a multiple of both $p_1$ and $p_2$,
define the natural numbers $e_1=\lambda/p_1, e_2=\lambda/p_2$.

Assume that $r \leq -1$. Then $(*)$ implies that $p_1p_2 \leq p_1 + p_2$,
hence $p_1=p_2=2$ and $r=-1$, giving conclusion (1). 
%equality. Combined with equation $(*)$, this gives conclusion $(2)$.

We hereafter take $r=0$. WLOG assume $p_2 \geq p_1$ and hence $e_1 \geq e_2$.

First assume $p_1>2$. Then $2e_1 < \lambda = e_1 + e_2 +1$ from $(*)$; hence, 
$e_1 < e_2 + 1$. Thus $e_1=e_2$, $p_1=p_2$. Then $(*)$ becomes 
$\lambda ((p_1-2)/p_1)=1$ or that 
$e_1(p_1-2)=1$. This gives conclusion $(4)$ above.

So assume $p_1=2$. Then we get that $e_2 (p_2-2)=2$. This means that either
$e_2=2, p_2=3, \lambda=6$ or $e_2=1, p_2=4, \lambda=4$. These are conclusions
$(2)$ and $(3)$.
\end{proof} 

Lemma~\ref{horizontal} now follows from Claim~\ref{clm:possibilities}:
Conclusion $(1),(2),(3),(4)$ of Claim~\ref{clm:possibilities} imply conclusions
$(D),(C),(B),(E)$ respectively. Note that in conclusion $(D)$, 
$M$ contains a pseudovertical Klein bottle between the exceptional
fibers of order $2$.
In the context of conclusion $(2)$ of the claim (which is the context of $(C)$
in the Theorem)
$X=M- \hbox{Int}(T_3)$ is the exterior of a trefoil knot and $K \cap X$ is
a 1-punctured torus, hence a Seifert surface for $X$. As $T_3$ intersects
$K$ in a \mobius band, the meridian of $T_3$ intersects the boundary of
this Seifert surface twice. Hence $M$ is an $(2/n)$-filling of $X$ as 
claimed. 
\end{proof}

\begin{claim}
If $K$ is pseudovertical then conclusion $(B)$ of Theorem~\ref{SFSDycks} holds.
\end{claim}

\begin{proof}
Assume $K$ is pseudovertical. As $\chi(K)=-1$, $K$ is the union of 
a vertical annulus $\widetilde{K}$ and a \mobius band $K_1$ in $T_1$ (say) along
with a punctured Klein bottle $K_2$ in $T_2$ (say). 

Now $\partial K_1$ will intersect the meridian of $T_1$ twice. As $\partial
K_1$  is a Seifert fiber of $M$ this says that the order of the
exceptional fiber at $T_1$ is $2$.

By Corollary 2.2 of \cite{Fr},
a one-sided incompressible surface in a solid torus has boundary a single 
$(2k,l)$-curve in longitude, meridian coordinates of the solid torus 
where $k,l$ are integers and $k>0$. In \cite{BW}, a recursive formula is 
developed for $N(2k,l)$, which, as pointed out in \cite{Fr}, is equal to
the cross-cap number of the (unique) 1-sided incompressible 
surface whose boundary
is the $(2k,l)$-curve. By picking the right longitude, we may assume
that $k>l>0$ in the computation of $N(2k,l)$. Then \cite{BW}(6.4) shows
that $N(2k,l)=2$ iff $k$ is even. So let $\partial K_2$ be such a 
$(2k,l)$ curve in $T_2$. Then $2=N(2k,l)$ and $k$ is even. As $\partial K_2$
is a Seifert fiber for $M$, this implies that the exceptional fiber
for $T_2$ has order $2k$ with $k$ even. 

Thus $M$ is as in $(B)$ of the Theorem~\ref{SFSDycks}.
\end{proof}

%\newpage

%%%%%%%%%%%%%%%%%%%%%
%%%%%%%%%%%%

\providecommand{\bysame}{\leavevmode\hbox to3em{\hrulefill}\thinspace}
\providecommand{\MR}{\relax\ifhmode\unskip\space\fi MR }
\providecommand{\MRhref}[2]{%
  \href{http://www.ams.org/mathscinet-getitem?mr=#1}{#2}
}
\providecommand{\href}[2]{#2}

\end{document}